\documentclass{amsart}
\usepackage{latexsym,amsfonts,amssymb,exscale,enumerate,comment}
\usepackage{amsmath,amsthm,amscd}
 \usepackage[percent]{overpic}
\usepackage{mathrsfs}
\usepackage{pifont}

\excludecomment{answer}      

\makeatletter
\tagsleft@false
\makeatother

\usepackage{geometry}
\usepackage{url}
\usepackage[bookmarks=true,%
    colorlinks=true,%
    linkcolor=blue,%
    citecolor=blue,%
    filecolor=blue,%
    menucolor=blue,%
    urlcolor=blue,%
    breaklinks=true]{hyperref} 
    
\newtheorem*{namedtheorem}{\namedtheoremname}
\newcommand{\namedtheoremname}{Theorem}

\makeatletter
\newenvironment{alphatheorem}[1]{%
  \renewcommand{\namedtheoremname}{Theorem #1}%
  \def\@currentlabel{#1}%
  \begin{namedtheorem}%
}{%
  \end{namedtheorem}%
}
\makeatother

\makeatletter
\newenvironment{alphacorollary}[1]{%
  \renewcommand{\namedtheoremname}{Corollary #1}%
  \def\@currentlabel{#1}%
  \begin{namedtheorem}%
}{%
  \end{namedtheorem}%
}
\makeatother

\DeclareRobustCommand{\waterdrop}{%
  \tikz[scale=0.07, baseline=-0.4ex, inner sep=0pt, outer sep=0pt]{%
    \useasboundingbox (-0.9,-1.4) rectangle (0.9,1.5);
    \draw[line width=0.5pt]
      (0,1.3)
      .. controls (-0.1,0.8) and (-0.8,0.2) ..
      (-0.75,-0.4)
      arc(180:360:0.75 and 0.7)
      .. controls (0.8,0.2) and (0.1,0.8) ..
      cycle;
  }%
}

\DeclareRobustCommand{\snowflake}{\makebox[0pt]{%
  \tikz[scale=0.09, baseline=-0.4ex, inner sep=0pt, outer sep=0pt]{%
    \useasboundingbox (-1,-1) rectangle (1,1);
    \foreach \a in {0,60,...,300}{
      \draw[line width=0.5pt] (0,0) -- (\a:1);
      \draw[line width=0.4pt] (\a:0.55) -- +(\a+60:0.3);
      \draw[line width=0.4pt] (\a:0.55) -- +(\a-60:0.3);
    }%
  }%
}}

\input xy
\usepackage[all]{xy}
\xyoption{line}
\xyoption{arrow}
\xyoption{color}
\SelectTips{cm}{}

\usepackage{tikz}
\usetikzlibrary{shapes,snakes}
\usetikzlibrary{decorations.markings}
\usetikzlibrary{decorations.pathreplacing}
\tikzstyle directed=[postaction={decorate,decoration={markings,
    mark=at position #1 with {\arrow{>}}}}]

\newcommand{\hackcenter}[1]{
 \xy (0,0)*{#1}; \endxy}

\tikzset{->-/.style={decoration={
  markings,
  mark=at position #1 with {\arrow{>}}},postaction={decorate}}}

\tikzset{middlearrow/.style={
        decoration={markings,
            mark= at position 0.5 with {\arrow{#1}} ,
        },
        postaction={decorate}
    }
}

\usepackage{graphicx}
\usepackage{color}

\usetikzlibrary{shapes.geometric}

\theoremstyle{plain}
\newtheorem{theorem}{Theorem}

\newtheorem{corollary}[theorem]{Corollary}
\newtheorem{proposition}[theorem]{Proposition}
\newtheorem{lemma}[theorem]{Lemma}

\theoremstyle{definition}
\newtheorem{example}[theorem]{Example}
\newtheorem{definition}[theorem]{Definition}
\newtheorem{conjecture}[theorem]{Conjecture}
\theoremstyle{definition}
\newtheorem{remark}[theorem]{Remark}

\def\height{{\operatorname{ht}}}

\numberwithin{equation}{section}
\numberwithin{theorem}{section}

\newcommand{\refequal}[1]{\xy {\ar@{=}^{#1}
(-1,0)*{};(1,0)*{}};
\endxy}

\newcommand{\Hom}{{\rm Hom}}
\newcommand{\TypeH}{{\rm Type}}

\renewcommand{\to}{\rightarrow}

\def\Res{{\mathrm{Res}}}
\def\Ind{{\mathrm{Ind}}}

\numberwithin{equation}{section}
\let\hat=\widehat
\let\tilde=\widetilde

\let\epsilon=\varepsilon

\usepackage{bbm}
\def\C{{\mathbb{C}}}

\def\ZZ{{\mathbbm Z}}

\def\1{\mathbbm{1}}%
\usepackage{bbm}

\def \k {\mathbbm{k}}
\def \Z {\mathbbm{Z}}

\def \bkap {\boldsymbol{\kappa}}
\def \blam {\boldsymbol{\lambda}}
\def \bi {\boldsymbol{i}}
\def \bj {\boldsymbol{j}}
\def \bk {\boldsymbol{k}}
\newcommand{\bS}{\boldsymbol{\mathsf{S}}}
\def \bmu {\boldsymbol{\mu}}

\def \bgamma {\boldsymbol{\gamma}}

\def \brho {\boldsymbol{\rho}}
\def \beps {\boldsymbol{\varepsilon}}

\def \eps {\varepsilon}
\def \bOm {\boldsymbol{\Omega}}
\def \bom {\boldsymbol{\omega}}
\def \bTheta {\boldsymbol{\Theta}}
\def \bnu {\boldsymbol{\nu}}

\def \PL {\textup{P}^\textup{L}}
\def \PR {\textup{P}^\textup{R}}
\def \NPL {\textup{NP}^\textup{L}}
\def \NPR {\textup{NP}^\textup{R}}

\def \min {{\rm min}}

\newcommand\nc{\newcommand}

\nc\rnc{\renewcommand}
\nc\Kar{\operatorname{Kar}}
\nc\End{\operatorname{End}}
\nc\modQ {{\mathbb Q}}
\nc\modZ {{\mathbb Z}}
\nc\simeqto{\overset{\simeq}{\longrightarrow }}
\nc\K{\mathcal {K}}
\nc\CC{\mathbf{C}}

\nc\tE{\check{E}}
\nc\qh{\mathcal{H}}
\nc\hbm{\mathcal{B}}
\nc\bu{\mathbf{u}}
\nc\bZ{\mathbf{Z}}
\nc\theirs{\mathrm{theirs}}
\nc\ours{\mathrm{ours}}
\nc\hE{\mathcal{\hat E}}
\nc\bK{\mathbf{K}}
\nc\bw{\mathbf{w}}
\nc\bh{\mathbf{h}}
\nc\ba{\mathbf{a}}
\nc\bb{\mathbf{b}}
\nc\SPAR{{\tt spar}}
\nc\TAL{\mathrm{tally}}
\nc\BFC{\mathrm{bfc}}
\nc\ORD{\mathrm{ord}}
\nc\ORTH{\mathrm{orth}}
\nc\EXT{\mathrm{ext}}
\nc\CALC{\mathrm{calc}}
\nc\TRIP{\mathrm{trip}}
\nc\STACK{\mathrm{stack}}
\nc\FLIP{\mathrm{flip}}
\newcommand{\sfD}{\mathsf{D}}
\newcommand{\sfU}{\mathsf{U}}
\newcommand{\sfL}{\mathsf{L}}
\newcommand{\usfD}{\underline{\mathsf{D}}}
\newcommand{\sfC}{\mathsf{C}}
\newcommand{\usfC}{\underline{\mathsf{C}}}

\newcommand{\cross}{%
  \mathord{\begin{tikzpicture}[baseline=0.4ex, scale=0.35, line width=0.5pt]
    \draw[->] (0,0) -- (1,1);
    \draw[->] (1,0) -- (0,1);
  \end{tikzpicture}}}

\newcommand{\uncross}{%
  \mathord{\begin{tikzpicture}[baseline=0.4ex, scale=0.35, line width=0.5pt]
    \draw[->] (0,0) -- (0,1);
    \draw[->] (0.5,0) -- (0.5,1);
  \end{tikzpicture}}}

\newcommand\aMV{\mathcal{MV}}
\newcommand\MP{\mathcal{MP}}
\newcommand\Par{\mathcal{P}}
\newcommand\Parreg{\mathcal{P}_{\mathrm{2reg}}}
\newcommand\Parres{\mathcal{P}_{\mathrm{2res}}}
\newcommand\UL{\mathcal{UL}}
\newcommand\LL{\mathcal{LL}}
\newcommand\MPres{\mathcal{MP}_{\mathrm{2res}}}
\newcommand\MPK{\mathcal{MP}_{\mathrm{K}}}

\newcommand\FLOAT{\mathrm{float}}
\newcommand\SPLIT{\mathrm{split}}

\newcommand\GLUE{\mathrm{glue}}
\newcommand\kres{\mathrm{res}}
\newcommand\SINK{\mathrm{sink}}

\newcommand\re{\mathrm{re}}
\newcommand\im{\mathrm{im}}

\newcommand\succa{\substack{{} \\ {} \\ \textstyle\succ \\ \scriptscriptstyle {}^1}}
\newcommand\succatext{\substack{{} \\ \textstyle\succ \\ \scriptscriptstyle {}^1}}
\newcommand\succb{\substack{{} \\ {} \\ \textstyle\succ \\ \scriptscriptstyle {}^0}}
\newcommand\succbtext{\substack{ {} \\ \textstyle\succ \\ \scriptscriptstyle {}^0}}

\newcommand\succitext{\substack{{} \\ \textstyle\succ \\ \scriptscriptstyle {}^i}}
\newcommand\succieqtext{\substack{{} \\ \textstyle\succeq \\ \scriptscriptstyle {}^i}}

\nc{\urcap}{\textrm{uRCap}}
\nc{\urcup}{\textrm{uRCup}}
\nc{\ulcap}{\textrm{uLCap}}
\nc{\ulcup}{\textrm{uLCup}}
\nc{\drcap}{\textrm{dRCap}}
\nc{\drcup}{\textrm{dRCup}}
\nc{\dlcap}{\textrm{dRCap}}
\nc{\dlcup}{\textrm{dRCup}}

\nc{\cl}{\mathrm{cl}}
\nc{\cala}{\mathcal{A}}
\nc{\calb}{\mathcal{B}}
\nc{\calc}{\mathcal{C}}
\nc{\bx}{\mathbf{x}}
\nc{\by}{\mathbf{y}}
\nc{\bE}{\mathbbm{E}}
\nc{\bElong}{{\mathbbm E^{\bullet, \geq}}}

\nc{\Sk}{\mathrm{Sk}}
\nc{\Hilb}{\mathrm{Hilb}}

\nc\col{\colon\thinspace}

\allowdisplaybreaks

\title[Crystals and KLR representations in type \({\tt A}_1^{(1)}\)]{
Kleshchev multipartitions, affine Mirkovi\'c-Vilonen polytopes, and representations of KLR algebras in type \({\tt A}^{(1)}_1\)
}

\begin{document}
\setcounter{tocdepth}{2}

\author{Samantha Allen}
\email{allens6@duq.edu}
\address{Department of Mathematics and Computer Science \\ Duquesne University \\ Pittsburgh, PA USA}

\author{Jack Isaac}
\email{jci11@pitt.edu}
\address{Department of Mathematics \\ University of Pittsburgh \\ Pittsburgh, PA USA}

\author{Corinne Moscariello}
\email{moscarielloc@duq.edu}
\address{Duquesne University \\ Pittsburgh, PA USA}

\author{Robert Muth}
\email{muthr@duq.edu}
\address{Department of Mathematics and Computer Science \\ Duquesne University \\ Pittsburgh, PA USA}


\author{Bella Deborah Uwase}
\email{uwased@duq.edu}
\address{Duquesne University \\ Pittsburgh, PA USA}

\author{Lucas Walton}
\email{waltonl2@duq.edu}
\address{Duquesne University \\ Pittsburgh, PA USA}

\renewcommand{\shortauthors}{S. Allen, J. Isaac, C. Moscariello, R. Muth, B. D. Uwase, and L. Walton}

\date{\today}

\begin{abstract}
We construct explicit isomorphisms between three models for the $\mathcal{B}(\infty)$ crystal in type ${\tt A}_1^{(1)}$: affine Mirkovi\'{c}--Vilonen 
polytopes, Kleshchev multipartitions, and a new model we call {\em upper ledge diagrams}. We also present some clarifying results on these crystals, giving a direct method for completing an affine MV polytope from the data of one of its boundary root partitions, and a non-iterative recognition theorem which characterizes Kleshchev multipartitions in type ${\tt A}_1^{(1)}$. We apply these results to the representation theory of KLR algebras, where they yield a combinatorial dictionary between cuspidal- and cellular-theoretic frameworks, along with some augmented branching rules for real root functors of induction and restriction.
\end{abstract}

\date{\today}

\maketitle
\setcounter{tocdepth}{1}
\tableofcontents

\section{Introduction}

Crystal bases are a powerful combinatorial framework for studying representations of quantum groups, and the crystal $\mathcal{B}(\infty)$ plays a central organizing role; it encodes the combinatorial skeleton of highest weight representations and arises naturally in the categorification program via Khovanov–Lauda–Rouquier (KLR) algebras.  Herein our focus is on the fundamental type ${\tt A}_1^{(1)}$ crystals for the quantum group \(U_q(\hat{\mathfrak{sl}}_2)\), where $\mathcal{B}(\infty)$ admits a number of well-studied but still rather mysterious combinatorial realizations. Under consideration in this paper are two key examples:
\begin{itemize}
\item \emph{Kleshchev multipartitions} (\S\ref{bigKleshsec}), whose crystal rules originate
in work of Kleshchev \cite{KleshBranchI, KleshBranchII} on modular branching rules for symmetric groups 
and Hecke algebras, and work of Lascoux--Leclerc--Thibon \cite{LLT}
and Misra--Miwa \cite{MisraMiwa} on the basic representation of \(U_q(\hat{\mathfrak{sl}}_e)\), extended to higher levels by Ariki--Mathas \cite{ArikiMathas}.
\item \emph{Rank 2 affine Mirkovi\'{c}--Vilonen (MV) polytopes} (\S\ref{bigAMVsec}), defined in work of Baumann--Dunlap--Kamnitzer--Tingley \cite{aMVrank2} 
(see also  \cite{BKT, Kamnitzer2010, MuthiahTingley2014, MuthiahTingley2018}), which encode PBW data for the quantum group. These polytopes are fundamental building blocks (2-faces) in the characterization of affine MV polytopes in arbitrary symmetric affine type. 
\end{itemize}

Motivating our interest in these two crystals is the deep representation-theoretic content they carry in the world of categorification via KLR algebras (see \S\ref{bigKLRsec}). 
Simple KLR modules are indexed by affine MV polytopes via the cuspidal system framework of \cite{KleshCusp, McNaffine, TW}. On the other hand, Kleshchev multipartitions index simple modules for cyclotomic KLR 
algebras through the graded cellular structure of \cite{HuMathasCycA}. From each point of view we capture complementary structural information; the 
former offers stratification, filtrations, and BGG-type reciprocity \cite{KMstrat, Murata}, the latter offers
ties to modular representation theory 
of Hecke algebras and attendant results on Specht homomorphisms and decomposition numbers \cite{BKisom, BKW, KMR}. 

Creating a direct combinatorial bridge between these perspectives would allow these tool kits to be brought to bear in tandem, yet translating directly between crystals has remained heretofore a difficult task. Happily, our main result (Theorems~\ref{PolyOH} and \ref{isommpul} and Corollary~\ref{PolyMPthm}) constructs such a bridge by way of explicit bijections between these crystals and a new, third, nicely transparent (in both characterization and crystal operation) combinatorial model for \(\mathcal{B}(\infty)\) which we call \emph{upper ledge diagrams} (see \S\ref{bigULsec}).

\begin{figure}[h]
\begin{align*}
\hackcenter{}
\hackcenter{
\begin{overpic}[height=66.5mm]{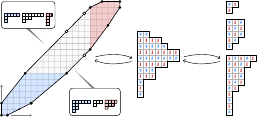}
 \put(1.8,-0.3){\makebox(0,0)[t]{$\scriptstyle \alpha_1$}}    
  \put(-1.2,1.8){\makebox(0,0)[t]{$\scriptstyle \alpha_0$}}
  \put(29.4,8.6){\makebox(0,0)[t]{$\scriptstyle\pi^1$}}    
    \put(37.4,8.6){\makebox(0,0)[t]{$\scriptstyle\pi^\delta$}}    
      \put(42,8.6){\makebox(0,0)[t]{$\scriptstyle\pi^0$}}    
  \put(2.9,42.6){\makebox(0,0)[t]{$\scriptstyle\phi^0$}}    
    \put(9.8,42.6){\makebox(0,0)[t]{$\scriptstyle\phi^\delta$}}    
      \put(19,42.6){\makebox(0,0)[t]{$\scriptstyle\phi^1$}}    
  \put(86,15.8){\makebox(0,0){$\scriptstyle \lambda^{(1)}$}}    
    \put(86,33.5){\makebox(0,0){$\scriptstyle \lambda^{(2)}$}}    
      \put(86,42.8){\makebox(0,0){$\scriptstyle \lambda^{(3)}$}}    
  \put(54.3,34.8){\makebox(0,0){$\scriptstyle \sfD$}}    
             \put(44.6,26){\makebox(0,0){$\scriptstyle \mathcal{X}$}}  
               \put(44.6,19){\makebox(0,0){$\scriptstyle \mathcal{Y}$}}  
                  \put(79,26){\makebox(0,0){$\scriptstyle \mathcal{W}^{\bkap}$}}  
               \put(79,19){\makebox(0,0){$\scriptstyle \mathcal{Z}^{\bkap}$}}  
\end{overpic}
}
\end{align*}
\caption{
The affine Mirkovi\'c-Vilonen polytopes (left, \S\ref{bigAMVsec}), upper ledge diagrams (center,  \S\ref{bigULsec}), and Kleshchev multipartitions (right, \S\ref{bigKleshsec}) are models for \(\mathcal{B}(\infty)\) in type \({\tt A}_1^{(1)}\). We present direct isomorphisms \(\mathcal{X}, \mathcal{Y},  \mathcal{Z}^{\bkap}, \mathcal{W}^{\bkap}\) between these crystals in \S \ref{crystalisomsec}.}
\label{maineverythingfig}
\end{figure}
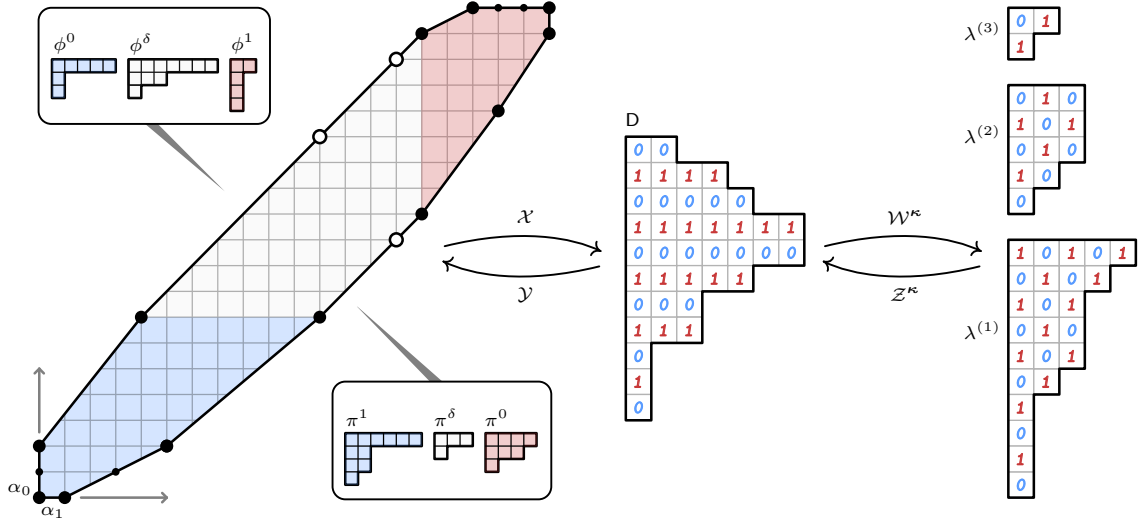

Along the way, we establish several results of independent interest, the first two of which address significant practical difficulties in working with these crystals:
\begin{itemize}
\item There are two convex preorders on the root system of type \({\tt A}_1^{(1)}\), and root partitions in either order correspond to northwest/southeast boundaries of polytopes (as seen in Figure~\ref{maineverythingfig}). Working with affine MV polytopes and their crystal operators requires one to generate, from a given root partition in one order, the root partition in the other order which uniquely completes an affine MV polytope---a rather delicate iterative balancing act featuring the satisfaction of many tropical-flavored inequalities. We simplify matters by providing an explicit tableau-theoretic polytope completion algorithm in type \({\tt A}_1^{(1)}\) (Theorem~\ref{rpartranslate}). 
\item Recognizing whether a given multipartition is Kleshchev is also an iterative process which in general requires a number of steps roughly equal to the rank of the multipartition itself. We provide a new non-iterative recognition theorem for Kleshchev multipartitions in type \({\tt A}^{(1)}_1\) for arbitrary level and multicharge (Theorem~\ref{Kleshtest}), the first of its kind in affine type {\tt A}.
\item A key ingredient in our combinatorial setup is a bijection between partitions and 2-regular partitions (Theorem~\ref{tallyprop}) with interesting properties; it can be seen for instance that it restricts to a bijection between partitions of \(n\) with \(k\) parts and 2-regular partitions whose odd-numbered parts sum to \(n\) and even-numbered parts sum to \(n-k\). That the cardinalities of these sets are equal is known thanks to Euler and refinements by Sylvester and Bessenrodt (see for instance \cite{eulerbijections,BessSyl, Paksurvey}) but our bijection is apparently new. 
\end{itemize}

Next we consider in \S\ref{bigKLRsec} applications of these results to type \({\tt A}_1^{(1)}\) KLR representation theory. Corollary~\ref{PolyMPthm}'s crystal isomorphisms lift to isomorphisms of simple modules in the associated cuspidal/cellular labeling regimes mentioned above, as shown in \S\ref{simpletranslations}. Exploiting this new connection allows us to freely share representation-theoretic information between regimes, and in this direction we establish several related results and pose some new conjectures:
\begin{itemize}
\item We give a remarkably simple ``glue-then-split'' procedure for translating multipartition labels of simple modules between any two choices of multicharge, allowing us to pass between distinct cellular regimes for cyclotomic KLR algebras (\S\ref{transmulticharges}).
\item We provide a direct means of calculating the {\em jump} statistic on simple KLR modules via upper ledge diagram combinatorics (\S\ref{jumpstatsec}).
\item We present some `augmented' branching rules (\S\ref{augmentedmulti}) for Kleshchev multipartitions describing heads/socles of simple modules under  {\em  real root} functors of induction/restriction (generalizing the usual notion of branching by {\em simple roots}, see \S\ref{rootfuncsec}).
\item We describe a combinatorial `tile-factor' theorem (\S\ref{tilingsec}) which identifies some multiplicity-one simple factors of quasirestricted Specht modules in level one.
\item Inspired by the crystal combinatorics, we pose some new conjectures: a level-two analogue of James's famous regularization theorem  (\S\ref{conjjames2}), and an alternate construction of simple KLR modules closely tied to the upper ledge diagram crystal (\S\ref{conjlabelreg}).
\end{itemize}

\subsection{Future directions} It would be interesting to construct explicit crystal isomorphisms between upper ledge diagrams and other combinatorial realizations of $\mathcal{B}(\infty)$ in type ${\tt A}_1^{(1)}$ using the upper ledge diagram crystal as a combinatorial hub. Natural candidates include Young walls \cite{KangYW, KimShinYW}; Nakajima monomials \cite{KimMono}; and rigged configurations \cite{SSSrigB}. At the level of 
highest weight crystals $\mathcal{B}(\Lambda)$, it would also be interesting to connect upper ledge diagrams to Uglov/FLOTW multipartitions \cite{FLOTW, Gerber, JaconLecouvey, Uglov}, which provide alternatives to the Kleshchev labeling in higher levels.

Another important direction would be to extend this work to types \({\tt A}_{>1}^{(1)}\), 
constructing analogous notions of upper ledge diagrams, associated crystal isomorphisms, and dictionaries between KLR representation theories. This work is a crucial first step in that direction, since the \({\tt A}_{1}^{(1)}\) polytopes (and their rank 2 finite-type kin) govern the key 2-face conditions in the characterization of higher-rank affine MV polytopes \cite{BKT}.
There are new difficulties that arise in higher ranks---for instance, the associated Weyl group combinatorics become more complex and there are infinitely many convex preorders to contend with (as opposed to two). Certain aspects of Kleshchev multipartition combinatorics also become considerably thornier---for instance, the Mullineux map is nontrivial in higher ranks (see e.g.\ the increase in complexity in characterizing Kleshchev bipartitions in \cite[Propositions 9.7, 9.8]{AKT}).

\subsection{On the structure of the paper} In order to maximize readability/usability, we adopt a slightly idiosyncratic structure. In \S\ref{partitioncombsec}--\S\ref{crystalisomsec} we define the objects under study, present the key combinatorial algorithms and functions, and {\em state} the main (lettered) theorems. The combinatorial heavy lifting is then done below decks in the long and technical \S\ref{dirtywork}--\S\ref{proofsmainthmssec}, where more ancillary objects and functions are defined, and proofs of the lettered theorems may be found. In \S\ref{bigKLRsec} we devote our attention to applying the main combinatorial theorems to KLR algebras.

\subsection{ArXiv Version}\label{SS:ArxivVersion} We relegate some of the more routine/repetitive calculations to the \texttt{arXiv} version of the paper.  The reader interested in seeing these additional details can download the \LaTeX~source file from the \texttt{arXiv}.  Near the beginning of the file is a toggle which allows one to compile the paper with these calculations included.

\subsection{Acknowledgements}
We thank Farnaz Ariyanfar, Nicholas Davidson, Matthew Fayers, Alexander Kleshchev, Andrew Mathas, Travis Scrimshaw, George Seelinger, Liron Speyer, Arthur Sugden, Louise Sutton and Peter Tingley for helpful conversations and insights. Some of our nomenclature in \S\ref{partitioncombsec} was inspired by 
\cite{pynchon2006}.
The first author was supported by NSF Grant No. DMS-2532488. The first and fourth authors were supported by a Duquesne Faculty Development Fund grant. The fourth author was supported by an AMS-Simons PUI Research Enhancement Grant, and by ICERM for the {\em Categorification and Computation in Algebraic Combinatorics} semester program. The sixth author was supported by a Duquesne University School of Science and Engineering Dean's Grant. 

\section{Partition combinatorics}\label{partitioncombsec}
We will freely associate integer partitions \(\lambda = (\lambda_1, \ldots, \lambda_m)\) with their Young diagrams, displayed in the English convention as top-left aligned box arrangements, where the \(k\)th row consists of \(\lambda_k\) boxes. We write 
\begin{itemize}
\item \(|\lambda| := \lambda_1 + \cdots + \lambda_m\) for the number of boxes in \(\lambda\);
\item  \(p(\lambda):=m\) for the number of nonzero rows (or {\em parts}) in \(\lambda\);
\item \(c(\lambda)\) for the number of nonzero columns in \(\lambda\); 
\item \(c_t(\lambda)\) for the length of the \(t\)th column in \(\lambda\);
\item \(n_k(\lambda) := \#\{t \mid \lambda_t = k\}\) for the number of rows of length \(k\) in \(\lambda\); \item \(\lambda'\) for the {\em conjugate} of \(\lambda\), given by \(\lambda'_t = c_t(\lambda)\) for all \(t\);
\item \(\min(\lambda):= \lambda_{p(\lambda)}\) for the minimal row length in \(\lambda\) if \(\lambda\) is nonempty, and \(\min(\lambda) = \infty\) if \(\lambda = \varnothing\);
\item \(\textup{double}(\lambda)\) for the partition given by doubling all column lengths in \(\lambda\);
\item  \(\textup{halve}(\lambda)\) for the partition given by halving all column lengths (assuming the column lengths of \(\lambda\) are even of course).
\end{itemize}
We say that \(\lambda\) is {\em 2-regular} (resp. {\em 2-restricted}) provided that \(\lambda\) has no two rows (resp. columns) of equal length. We write \(\Par, \Parreg, \Parres\) for the set of partitions, 2-regular partitions, and 2-restricted partitions respectively. 
In some settings a partition \(\lambda\) will come equipped with a distinguished {\em charge} \(\kappa \in \Z_2\) which defines {\em residues} associated to each box in \(\lambda\): if \(\lambda\) has charge \(\kappa\) then the residue of a box \(u\) in the \(r\)th row and \(c\)th column of \(\lambda\) is given by \(\textup{res}(u) :=\kappa +  \bar c - \bar r \in \Z_2\). We write \(\Par^{\kappa}\) for the set of partitions equipped with charge \(\kappa\).

\subsection{Spar tableaux}\label{traversesec}
Given a partition \(\lambda\), we define the {\em spar tableau} \(\SPAR(\lambda)\) to be the labeling of boxes in the Young diagram for \(\lambda\) with positive integers as follows. Traversing through columns from right to left, and within each column from bottom to top, label each box with the smallest positive integer which exceeds the label of any boxes directly below and to the right, and appears nowhere in the column immediately to the right. See Figure~\ref{makereeffig} for an example. We set \({\tt w}(\lambda)\) to be the label in the upper left corner of \(\SPAR(\lambda)\) (or 0 if \(\lambda = \varnothing\)). We now define a pair of functions which refract out of spar tableau combinatorics.

\subsubsection{Ordinary/extraordinary functions} Given \(\lambda \in \Par\), we define \(\mu = \ORD(\lambda) \in \Parreg\) to be the unique 2-regular partition such that \(c_t(\mu)\) is odd if and only if \(t\) appears in the first column of \(\SPAR(\lambda)\). More concretely, we have \(c_{{\tt w}(\lambda)}(\mu) = 1\), and moving to the left, the \(t\)th column length in \(\mu\) increases by one precisely when exactly one of \(t, t+1\) is in the first column of \(\SPAR(\lambda)\).

For \(w \geq {\tt w}(\lambda)\), we then define \(\EXT^w(\lambda) \in \Parreg\) by adding a row of length \(w\) to the top of \(\ORD(\lambda)\) if \(w> {\tt w}(\lambda)\), or deleting the top row of \(\ORD(\lambda)\) if \(w = {\tt w}(\lambda)\). It follows from definitions that \(\ORD(\lambda)\) and \(\EXT^w(\lambda)\) are `orthogonal' in a particular sense; for each \(t \in [1,w]\), the \(t\)th columns in each are of opposite parity. 
See Figure~\ref{makereeffig} for an example.

\begin{figure}[h]
\begin{align*}
\hackcenter{}
\hackcenter{
\begin{overpic}[height=28mm]{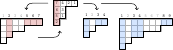}
 \put(31,30){\makebox(0,0)[b,l]{$\scriptstyle \SPAR(\lambda)$}}    
  \put(8.5,8){\makebox(0,0)[l]{$\scriptstyle \ORD(\lambda)$}}    
    \put(53.2,8){\makebox(0,0)[l]{$\scriptstyle \EXT^7(\lambda)$}}    
  \put(84.5,8){\makebox(0,0)[l]{$\scriptstyle \EXT^9(\lambda)$}}    
   \put(20,28){\makebox(0,0)[b,l]{$\scriptstyle\ORD$}}    
    \put(49,28){\makebox(0,0)[b,l]{$\scriptstyle\EXT^7$}}    
    \put(65,28){\makebox(0,0)[b,l]{$\scriptstyle\EXT^9$}}    
      \put(27.9,6.1){\makebox(0,0)[t]{$\scriptstyle \CALC$}}    
\end{overpic}
}
\end{align*}
\caption{At top, a partition \(\lambda = (4,2^2,1)\) and its spar tableau. Below are shown the 2-regular partitions \(\ORD(\lambda) = (7,4,2,1) \), \(\EXT^7(\lambda) = (4,2,1)\), and  \(\EXT^9(\lambda)=(9,7,4,2,1)\). Note that \(\CALC \circ \ORD(\lambda) = \CALC(7,4,2,1) = (4,2^2,1) = \lambda\).}
\label{makereeffig}
\end{figure}
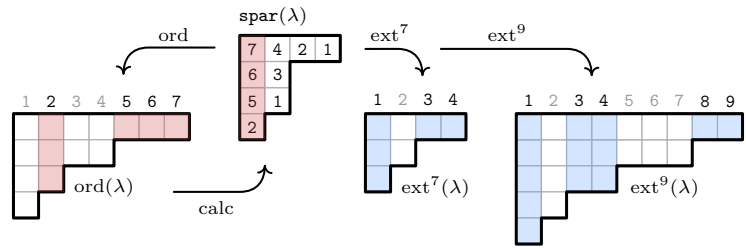

\begin{conjecture}\label{conjsparcontain}
If \(\lambda, \nu \in \Par\) are such that \(\nu\) is obtained by deleting the last row of \(\lambda\), then \(\ORD(\nu) \subseteq \ORD(\lambda)\). 
\end{conjecture}

\subsubsection{The calc function} 
Given \(\mu \in \Parreg\), we construct a (tableau for a) partition \(\lambda = \CALC(\mu)\) as follows. Set the first column of \(\lambda\) to consist of boxes labeled by \(\{ t \in [1, c(\mu)] \mid c_t(\mu) \text{ is odd}\}\) in decreasing order from top to bottom. Then, beginning from the second column in \(\lambda\), traverse through columns from left to right, and within each column from top to bottom, placing labeled boxes according to the following rule. For each new box, label it with the largest positive integer which is strictly less than any entry directly above or to the left of the new box, and appears nowhere in the column immediately to the left. If no such integer exists---or if there is no box immediately to its left---terminate this column without placing this box and move on to the next column to the right. The construction halts when a column receives no new boxes, and the resulting partition is \(\lambda\). See Figure~\ref{makereeffig} for an example.

The following result is proved in Theorem~\ref{appendreefundo}.

\begin{alphatheorem}{A}\label{tallyprop}
The functions \(\ORD\colon \Par \to \Parreg\) and \(\CALC\colon \Parreg \to \Par\) are mutually inverse bijections. Moreover, the tableau generated in the construction of \(\CALC(\mu)\) is exactly \(\SPAR(\CALC(\mu))\).
\end{alphatheorem}

\subsection{The involution \(\ORTH^w\)}\label{opsec}
For  \(\lambda \in \Par\) and \(w \geq {\tt w}(\lambda)\), we set:
\begin{align*}
\ORTH^w(\lambda) = \CALC(\EXT^w(\lambda)) \in \Par.
\end{align*}
It follows that \(\ORTH^w(\lambda)\) is the (unique, thanks to Theorem~\ref{tallyprop}) partition such that \([1,w]\) is the disjoint union of the labels in the first columns of \(\SPAR(\lambda)\) and \(\SPAR(\ORTH^w(\lambda))\). It also follows that \(\ORTH^w\) is an involution with no fixed points on the set of (the \(2^w\) many) partitions \(\lambda\) with \({\tt w}(\lambda) \leq w\).

\section{The type \({\tt A}^{(1)}_1\) root system}
Throughout the paper, it will be convenient to write \(\hat i = i + \bar 1\) for \(i \in \Z_2\).
We work with the root system of type \({\tt A}_1^{(1)}\) for the Kac-Moody Lie algebra \(\hat{\mathfrak{sl}}_2(\C)\) (see \cite{Kac1990} for a thorough exposition). We write \(I = \{\alpha_0, \alpha_1\}\) for the simple roots, \(\Z_{\geq 0}I\) for the root lattice, and define 
\begin{align*}
\alpha_{1:k} := k \alpha_1 + (k-1) \alpha_0 ;\qquad \alpha_{0:k} := (k-1) \alpha_1 + k \alpha_0 ; \qquad \delta = \alpha_1 + \alpha_0,
\end{align*}
in \(\Z_{\geq 0}I\) for all \(k \in \Z_{>0}\), where \(\delta\) is the null root. For \(\theta = a_1 \alpha_1 + a_0 \alpha_0 \in \Z_{\geq 0}I\) we write \(\height(\theta)= a_1 + a_0\) for the height of \(\theta\). We have
\begin{align*}
\Phi_+^\re = \{  \alpha_{1:k},  \; \alpha_{0:k}\mid k \in \Z_{> 0} \}; 
\qquad
\Phi_+^\im= \{ k \delta \mid k \in \Z_{>0}\};
\qquad
\Phi_+ = \Phi_+^\re \sqcup \Phi_+^\textup{im};
\qquad
\Psi_+ = \Phi_+^\re \sqcup \{\delta\},
\end{align*}
for the sets of real, imaginary, positive, and indivisible roots respectively. We write \(P = \Z\{\Lambda_1, \Lambda_0, \delta\}\) for the weight lattice. 

\subsection{Root partitions}
Given \(\theta \in \Z_{\geq 0}I\), a {\em root partition} \(\pi\) of \(\theta\) is a triple of integer partitions \(\pi = \{\pi^1,  \pi^\delta, \pi^0\}\) such that
\begin{align}\label{rpar1}
\theta = \sum_{k=1}^\infty n_k(\pi^1) \alpha_{1:k} + | \pi^\delta| \delta + \sum_{k=1}^\infty n_k(\pi^0)\alpha_{0:k}.
\end{align}
We write \(\Pi(\theta)\) for the set of root partitions of \(\theta\) and set \(\Pi = \bigsqcup_{\theta \in \Z_{\geq 0}I} \Pi(\theta)\).

\subsection{Crystals of type \({\tt A}_1^{(1)}\)} 
\begin{definition}\label{crystaldef}
A {\em highest weight \({\tt A}_1^{(1)}\)-crystal} is the data of a set \(B\) equipped with a {\em highest weight element \(b^+ \in B\)} and functions \(\textup{wt}: B \to P\), \(\varepsilon_i, \varphi_i : B \to \Z \), \(f_i, e_i: B \to B \sqcup \{0\}\) for \(i \in \Z_2\), such that for \(b, b' \in B\),
\begin{enumerate}
\item \(\varphi_i(b) = \varepsilon_i(b) + \langle \textup{wt}(b), \alpha_i^\vee\rangle\);
\item \( e_i\) increases \(\varphi_i\) by one, decreases \(\varepsilon_i\) by one, and increases \(\textup{wt}\) by \(\alpha_i\) (when \(e_ib \neq 0\));
\item \(f_i b = b'\) if and only if \(e_i b' = b\);
\item \(\varepsilon_i(b) = \max\{m \mid e_i^m b \neq 0\}\);
\item $e_i b_+ = 0$;
\item \(b_+\) can be reached from any \(b \in B\) by applying a sequence of \(e_1\)'s and \(e_0\)'s.
\end{enumerate}
\end{definition}

\begin{definition}
An {\em \({\tt A}_1^{(1)}\)-bicrystal} is a set \(B\) equipped with two distinct crystal structures whose weight functions agree. We indicate the secondary structure with stars, writing \(e_i^*, f_i^*, \varepsilon_i^*, \varphi_i^*\). 
\end{definition}

Kashiwara's theory of crystals \cite{Kashiwara1995} attaches to each integrable highest weight representation \(V(\Lambda)\)
of the quantum group \(U_q(\hat{\mathfrak{sl}}_2)\)
a highest weight crystal \(\mathcal{B}(\Lambda)\)
which records the leading-order \((q \to 0)\) 
combinatorial behavior of \(V(\Lambda)\). 
The negative half of the quantum group $U_q^-(\hat{\mathfrak{sl}}_2)$ has a crystal basis $\mathcal{B}(\infty)$ given as the direct limit of (weight shifted) $\mathcal{B}(\Lambda)$, and  $\mathcal{B}(\Lambda)$ can be obtained from cutting out a certain part of $\mathcal{B}(\infty)$.

 \section{Affine Mirkovi\'c-Vilonen polytopes}\label{bigAMVsec}
\subsection{Convex orders and polytopes}\label{cvxordersec}
We work with two total orders on the indivisible roots \(\Psi\) which are restrictions of the two possible convex preorders on \(\Phi_+\):
 \begin{align}\label{favorder1}
\alpha_{1:1} \hspace{1mm}\succa \hspace{1mm}\alpha_{1:2} \hspace{1mm}\succa \hspace{1mm}\alpha_{1:3} \hspace{1mm}\succa \hspace{1mm}\cdots \hspace{1mm}\succa\hspace{1mm} \delta \hspace{1mm}\succa \hspace{1mm}\cdots\hspace{1mm} \succa \hspace{1mm}\alpha_{0:3} \hspace{1mm}\succa \hspace{1mm}\alpha_{0:2} \hspace{1mm}\succa \hspace{1mm}\alpha_{0:1},
 \end{align}
 and
  \begin{align}\label{favorder2}
\alpha_{0:1} \hspace{1mm} \succb \hspace{1mm} \alpha_{0:2}\hspace{1mm} \succb \hspace{1mm}\alpha_{0:3} \hspace{1mm}\succb \hspace{1mm}\cdots\hspace{1mm} \succb \hspace{1mm}\delta\hspace{1mm} \succb\hspace{1mm} \cdots\hspace{1mm} \succb \hspace{1mm}\alpha_{1:3} \hspace{1mm}\succb \hspace{1mm}\alpha_{1:2} \hspace{1mm}\succb\hspace{1mm} \alpha_{1:1}.
 \end{align}
We may associate a root partition \(\pi \in \Pi(\theta)\) with finite sets of points in the root lattice \(\ZZ_{\geq 0}I\), letting \(S(\pi, \succatext)\) (resp. \(S(\pi, \succbtext)\)) denote the set of partial sums of (\ref{rpar1}) taken in \(\succatext\)-decreasing (resp. \(\succbtext\)-decreasing) order with respect to (\ref{favorder1},\ref{favorder2}). Then, for \(\pi, \phi \in \Pi(\theta)\) we may visualize the ordered pair \((\pi | \phi)\) as the (marked) convex hull of the points \(S(\pi, \succatext) \sqcup S(\phi, \succbtext)\), see for instance Example~\ref{polytopefirstex}.
For this reason we abuse terminology somewhat and refer to ordered pairs of root partitions \((\pi | \phi)\) as `polytopes'. While we will make use of the geometric visualization of \((\pi | \phi)\) for aid of intuition occasionally, we will work with polytopes in this paper from a primarily combinatorial and tableau-theoretic standpoint.

\begin{figure}[h]
\begin{align*}
\hackcenter{}
\hackcenter{
\begin{overpic}[height=70mm]{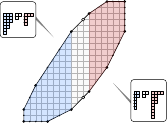}
 \put(22,-2){\makebox(0,0)[]{$\scriptstyle 4 \alpha_{1:1}$}}    
  \put(39.5,2){\makebox(0,0)[]{$\scriptstyle 2 \alpha_{1:2}$}}    
    \put(48.5,10.5){\makebox(0,0)[]{$\scriptstyle 2 \delta$}}    
    \put(53.2,15.5){\makebox(0,0)[]{$\scriptstyle  \delta$}}    
      \put(62.3,24.7){\makebox(0,0)[]{$\scriptstyle  \alpha_{0:4}$}}    
         \put(74,43){\makebox(0,0)[]{$\scriptstyle 3 \alpha_{0:2}$}}    
            \put(80,63.5){\makebox(0,0)[]{$\scriptstyle 5 \alpha_{0:1}$}}    
     \put(9.5,5){\makebox(0,0)[]{$\scriptstyle 2 \alpha_{0:1}$}}          
       \put(13,16.5){\makebox(0,0)[]{$\scriptstyle 2 \alpha_{0:2}$}}       
         \put(28,40){\makebox(0,0)[]{$\scriptstyle 3 \alpha_{0:3}$}}     
            \put(44,60){\makebox(0,0)[]{$\scriptstyle 2 \delta$}}    
    \put(50,65.5){\makebox(0,0)[]{$\scriptstyle  \delta$}}   
     \put(57,71.5){\makebox(0,0)[]{$\scriptstyle \alpha_{1:3}$}}     
       \put(69.5,75.5){\makebox(0,0)[]{$\scriptstyle 3\alpha_{1:1}$}}     
        \put(12.5,-1){\makebox(0,0)[]{$\scriptstyle \color{gray}0$}}   
         \put(77,74.5){\makebox(0,0)[]{$\scriptstyle \color{gray}\theta$}}   
              \put(82,20){\makebox(0,0)[b]{$\scriptstyle \pi^1$}}   
               \put(87,20){\makebox(0,0)[b]{$\scriptstyle \pi^\delta$}}   
                \put(92.5,20){\makebox(0,0)[b]{$\scriptstyle \pi^0$}}   
                  \put(3.5,66.5){\makebox(0,0)[b]{$\scriptstyle \phi^0$}}   
                    \put(10.5,66.5){\makebox(0,0)[b]{$\scriptstyle \phi^\delta$}}   
                      \put(16,66.5){\makebox(0,0)[b]{$\scriptstyle \phi^1$}}   
\end{overpic}
}
\end{align*}
\caption{The polytope \((\pi |\phi) \in \Pi(\theta)\), for \(\theta = 17 \alpha_1 + 20 \alpha_0\), as discussed in Example~\ref{polytopefirstex}.}
\label{polytopeex5fig}
\end{figure}
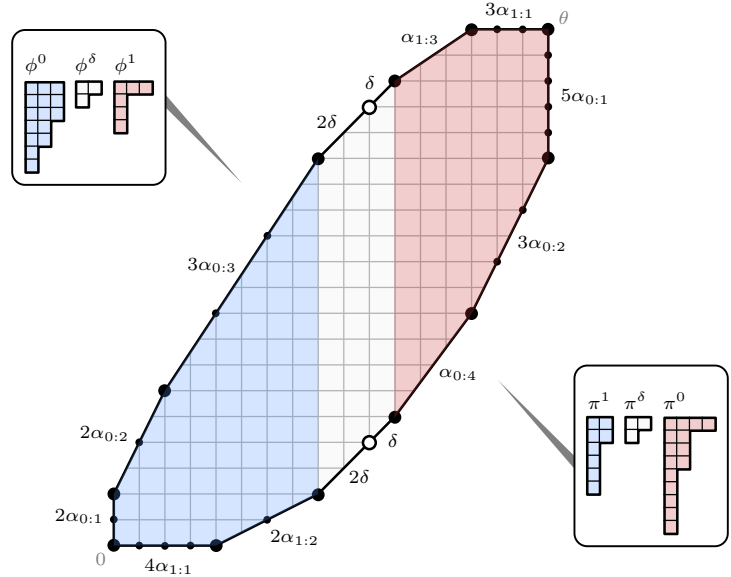

\begin{example}\label{polytopefirstex}
Let \(\theta = 17 \alpha_1 + 20 \alpha_0\), and consider the root partitions \(\pi, \phi \in \Pi(\theta)\) shown in Figure~\ref{polytopeex5fig}. Since
\begin{align*}
n_1(\pi^1) = 4;\quad
n_2(\pi^1) = 2;\quad
|\pi^\delta| = 3;\quad
n_4(\pi^0) = 1; \quad
n_2(\pi^0) = 3; \quad
n_1(\pi^0) = 5,
\end{align*}
with \(n_k(\pi^i) = 0\) for all other \(k \in \Z_{>0}, i \in \Z_2\), we have that \(S(\pi, \succatext) \) consists of the points
\begin{align*}
0 \xrightarrow{+ 4 \alpha_{1:1}}4 \alpha_1 \xrightarrow{+ 2 \alpha_{1:2}}8 \alpha_1 \hspace{-0.6mm}+\hspace{-0.6mm} 2 \alpha_0\xrightarrow{+ 3 \delta}11 \alpha_1 \hspace{-0.6mm}+\hspace{-0.6mm} 5 \alpha_0\xrightarrow{+ \alpha_{0:4}}14 \alpha_1 \hspace{-0.6mm}+ \hspace{-0.6mm}9 \alpha_0\xrightarrow{+3 \alpha_{0:2}} 17 \alpha_1\hspace{-0.6mm} +\hspace{-0.6mm} 15 \alpha_0 \xrightarrow{+ 5\alpha_{0:1}}17 \alpha_1\hspace{-0.6mm} + \hspace{-0.6mm}20 \alpha_0.
\end{align*}
By convention we will represent the simple roots \(\alpha_1, \alpha_0\) 
by unit vectors in the \(x\)- and \(y\)-directions respectively, so \(S(\pi, \succatext)\) makes up the vertices along the southeast boundary of the geometric visualization of the polytope \((\pi | \phi)\) shown in Figure~\ref{polytopeex5fig}. We use white dots to subdivide the edge parallel to \(\delta\) into segments \(\pi_1^\delta \delta, \pi_2^\delta \delta, \ldots\), so that the root partition \(\pi\) can be completely recovered from the shape together with these markings. 
The points \(S(\phi, \succbtext)\) are similarly computed and make up the northwest boundary of the polytope. \end{example}
 
 \subsection{Affine Mirkovi\'c-Vilonen polytopes}
Following \cite{aMVrank2}, we will only be interested in those polytopes which satisfy some rather delicate tropical-flavored conditions.
For \(\pi, \phi \in \Pi(\beta)\) and \(m \in \Z_{>0}\) we define integers
\begin{align*}
&c_m^{\uparrow}(\pi | \phi) = \sum_{k=1}^m k(n_{k+1}(\phi^0) - n_k(\pi^1));
&c_m^{\rightarrow}(\pi | \phi) = \sum_{k=1}^m k(n_{k+1}(\pi^1) - n_k(\phi^0));\\
&c_m^{\downarrow}(\pi | \phi) = \sum_{k=1}^m k(n_{k+1}(\pi^0) - n_k(\phi^1));
&c_m^{\leftarrow}(\pi | \phi) = \sum_{k=1}^m k(n_{k+1}(\phi^1) - n_k(\pi^0)).
\end{align*}

\begin{definition}\label{defaMV}
Let \(\pi, \phi \in \Pi(\theta)\). 
Then \((\pi | \phi)\) is an {\em affine Mirkovi\'c-Vilonen (MV)} polytope (of content \(\theta\)) provided that both of the following conditions are satisfied:
\begin{enumerate}
\item For every \(m \in \Z_{>0}\), we have
\begin{align*}
0 \in \{c_m^{\uparrow}(\pi | \phi), c_m^{\rightarrow}(\pi | \phi)\} \subseteq \Z_{\leq 0}
\quad
\textup{and}
\quad
0 \in \{c_m^{\downarrow}(\pi | \phi), c_m^{\leftarrow}(\pi | \phi)\} \subseteq \Z_{\leq 0}
\end{align*}
\item We have \(\pi_1^\delta \leq p(\pi^1) + p(\phi^0) = p(\pi^0) + p(\phi^1)\) and
\begin{align*}
\phi^\delta =
\begin{cases}
\pi^\delta & \textup{if } |\pi^\delta| = |\phi^\delta|;\\
(\pi^\delta_2, \pi^\delta_3, \ldots, \pi^\delta_{p(\pi^\delta)}) & \textup{if } |\pi^\delta| > |\phi^\delta|;\\
(p(\pi^1) + p(\phi^0), \pi^\delta_1, \ldots, \pi^\delta_{p(\pi^\delta)}) & \textup{if } |\pi^\delta| < |\phi^\delta|.
\end{cases}
\end{align*}
\end{enumerate}
\end{definition}

The above definition is a rephrasing of \cite[Definition~3.4]{aMVrank2} adapted to our notational choices---to translate between our polytope \((\pi | \phi)\) notation and the Lusztig data \((a_t, \lambda, a^t, \bar a_t, \bar \lambda, \bar a^t)\) of \cite{aMVrank2}, one takes:
\begin{align*}
n_t(\pi^1) = a_t; \quad  \pi^\delta = \lambda; \quad n_t(\pi^0) = a^t; \quad \quad n_t(\phi^0) = \bar a_t;\quad \phi^\delta = \bar \lambda; \quad n_t(\phi^1) = \bar a^t.
\end{align*}
We write \(\aMV(\theta)\) for the set of affine MV polytopes of content \(\theta\):
\begin{align*}
\aMV(\theta) = \{ (\pi | \phi) \mid \pi, \phi \in \Pi(\theta) \textup{ and } (\pi | \phi) \textup{ is an affine MV polytope}\},
\end{align*}
and set \(\aMV = \bigsqcup_{\theta \in \Z_{\geq 0}I} \aMV(\theta)\).

\begin{example}\label{polyex67}
Consider the polytope \((\pi | \phi)\) as in Figure~1. We have
\begin{align*}
\begin{array}{lllllllll}
n_1(\pi^1) = 4 && n_2(\pi^1) = 2 && n_3(\pi^1) = 0 && n_m(\pi^1) = 0\\
n_1(\phi^0) = 2 && n_2(\phi^0) = 2 && n_3(\phi^0) = 3  && n_m(\phi^0) = 0
\end{array}
\end{align*}
for \(m > 3\), and thus
\begin{align*}
\begin{array}{lllllll}
c_1^{\uparrow}(\pi | \phi) = -2 & &c_2^{\uparrow}(\pi | \phi) = 0  && c_m^{\uparrow}(\pi | \phi) = 0\\
c_1^{\rightarrow}(\pi | \phi) = 0 & &c_2^{\rightarrow}(\pi | \phi) = -4  && c_m^{\rightarrow}(\pi | \phi) = -13
\end{array}
\end{align*}
for all \(m \geq 3\), and so \(0 \in \{c_m^{\uparrow}(\pi | \phi), c_m^{\rightarrow}(\pi | \phi)\} \subseteq \Z_{\leq 0}\). A similar check shows that \(0 \in \{c_m^{\downarrow}(\pi | \phi), c_m^{\leftarrow}(\pi | \phi)\} \subseteq \Z_{\leq 0}\), and so Definition~\ref{defaMV}(1) is satisfied. Since \(\pi^\delta_1 = 2 < 13 = p(\pi^1) + p(\phi^0)\) and \(\pi^\delta = \phi^\delta\), we have that Definition~\ref{defaMV}(2) is satisfied as well, and therefore \((\pi | \phi)\) is an affine MV polytope.
\end{example}

\begin{theorem}\label{exunloz} \cite[Theorem~3.11]{aMVrank2}
Let \(\pi \in \Pi(\theta)\). Then there exists a unique \(\phi \in \Pi(\theta)\) such that \((\pi | \phi)\) is an affine Mirkovi\'c-Vilonen polytope.
\end{theorem}

Given \(\pi \in \Pi(\theta)\), we will write \(\pi_{\lozenge}\) for the unique root partition of \(\theta\) such that \((\pi | \pi_\lozenge)\) is an affine MV polytope. It follows from the symmetricity of Definition~\ref{defaMV} then that \((\pi_\lozenge | \pi)\) is an affine MV polytope as well, so the function \(\lozenge: \pi \mapsto \pi_\lozenge\) is an involution on \(\Pi(\theta)\).

 \subsection{The affine Mirkovi\'c-Vilonen polytope crystal}\label{aMVcrysdef} 
 
 Let \(\pi \in \Pi(\theta)\) and \(i \in \Z_2\). We let \(\pi_{+i:k} \in \Pi(\theta + \alpha_{i:k})\) be defined by inserting a row of length \(k\) into \(\pi^i\). If \(n_k(\pi^i) >0\), we let \(\pi_{-i:k} \in \Pi(\theta - \alpha_{i:k})\) be defined by deleting a row of length \(k\) from \(\pi^i\). 
For \(i \in \Z_2\), \(k \in \Z_{>0}\) we define operators 
\begin{align*}
 f_{i:k},  f_{i:k}^* \colon \aMV(\theta) \to \aMV(\theta + \alpha_{i:k}); 
 \qquad
e_{i:k},  e_{i:k}^* \colon \aMV(\theta) \to \aMV(\theta - \alpha_{i:k}) \sqcup \{0\}
 \end{align*}
 as follows. 
\begin{enumerate}
 \item The action of the operators \( f_{i:k},  f_{i:k}^*\) is given by setting
  \begin{align*}
 \begin{array}{llllll}
  (\pi | \phi) \xrightarrow{ f_{1:k}} ((\phi_{+1:k})_{\lozenge} \,|\, \phi_{+1:k});
 &&
 (\pi | \phi) \xrightarrow{ f_{1:k}^*} (\pi_{+1:k}\,|\, (\pi_{+1:k})_{\lozenge});\\
  (\pi | \phi) \xrightarrow{ f_{0:k}} (\pi_{+0:k} \,|\, (\pi_{+0:k} )_{\lozenge});
 &&
 (\pi | \phi) \xrightarrow{ f^*_{0:k}} ((\phi_{+0:k})_{\lozenge} \,|\, \phi_{+0:k}).
 \end{array}
 \end{align*}
 \item The action of the operators \( e_{i:k},  e_{i:k}^*\) is given by setting
 \begin{align*}
 \begin{array}{llllll}
  (\pi | \phi) \xrightarrow{ e_{1:k}} ((\phi_{-1:k})_{\lozenge} \,|\, \phi_{-1:k});
 &&
 (\pi | \phi) \xrightarrow{ e_{1:k}^*} (\pi_{-1:k}\,|\, (\pi_{-1:k})_{\lozenge});\\
  (\pi | \phi) \xrightarrow{ e_{0:k}} (\pi_{-0:k} \,|\, (\pi_{-0:k} )_{\lozenge});
 &&
 (\pi | \phi) \xrightarrow{ e^*_{0:k}} ((\phi_{-0:k})_{\lozenge} \,|\, \phi_{-0:k}).
 \end{array}
 \end{align*}
if \(n_k(\phi^1), n_k(\pi^1), n_k(\pi^0), n_k(\phi^0)\) are nonzero, respectively. Otherwise we set
 \((\pi | \phi) \to 0\).
 \end{enumerate}
 We abbreviate \(e_i = e_{i:1}\), \(f_i = f_{i:1}\), \(e_i^* = e^*_{i:1}\), \(f_i^* = f^*_{i:1}\).

 Following \cite{aMVrank2} we consider \(\aMV\) as an \({\tt A}_1^{(1)}\)-bicrystal with operators \(e_i, f_i, e_i^*, f_i^*\) as defined above, and statistics defined by \(\textup{wt}(\pi | \phi) = -\theta\), and
 \begin{align}\label{epsvaluesamv}
 \varepsilon_1(\pi | \phi) = n_1(\phi^1);
 \quad
 \varepsilon_0(\pi | \phi) = n_1(\pi^0);
 \quad
 \varepsilon_1^*(\pi | \phi) = n_1(\pi^1);
 \quad
 \varepsilon_0^*(\pi | \phi) = n_1(\phi^0),
 \end{align}
 for all \((\pi | \phi) \in \aMV(\theta)\). The statistics \(\varphi_i, \varphi_i^*\) are implicitly defined by Definition~\ref{crystaldef}(1).

We include the starred operators for completeness; we will work primarily with \(\aMV\) as a highest weight crystal (rather than bicrystal) with highest weight element \((\varnothing | \varnothing) \in \aMV(0)\) using the unstarred operators. 
The initial layers of the \(\aMV\) crystal are shown in Figure~\ref{aMVcrystal11fig}. The key result on affine MV polytopes is the following

\begin{figure}[h]
\begin{align*}
\hackcenter{}
\hackcenter{
\begin{overpic}[height=80mm]{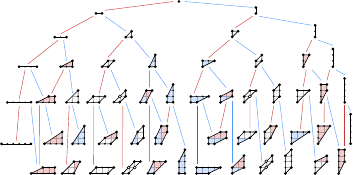}
\end{overpic}
}
\end{align*}
\caption{Initial layers of the \(\aMV\) crystal, with elements displayed as marked polytopes (see Example~\ref{polyex67} for an explanation of the notation). The operators \(e_1/f_1\) appear in red as upward/downward segments respectively. The operators \(e_0/f_0\) appear in blue.}
\label{aMVcrystal11fig}
\end{figure}

\begin{theorem}\cite[Theorem~4.5]{aMVrank2}\label{PolyBinf}
The  crystal \(\aMV\) is isomorphic to the crystal \(\mathcal{B}(\infty)\).
\end{theorem}

\begin{remark}\label{starinvpoly} Under the isomorphism of Theorem~\ref{PolyBinf}, Kashiwara's \(*\)-involution on \(\mathcal{B}(\infty)\) corresponds to negation at the level of the geometric polytope, and to the exchange \((\pi |\phi) \mapsto (\phi | \pi) = (\pi_\lozenge | \phi_\lozenge)\) in our combinatorial notation.
\end{remark}

\subsection{The highest weight crystal \(\aMV^\Lambda\)}
Let \(t_1, t_0 \in \Z_{\geq 0}\), so \(\Lambda = t_1 \Lambda_1 + t_0 \Lambda_0\) is a dominant integral weight. Let \(\aMV^\Lambda \subset \aMV\) be the subset of polytopes \((\pi | \phi)\) such that \(\varepsilon_i^*(\pi | \phi) \leq t_i\) for \(i \in \Z_2\). Then \((\aMV^\Lambda, e_i^\Lambda, f_i^\Lambda, \textup{wt}^\Lambda,  \varepsilon_i^\Lambda, \varphi_i^\Lambda)\) is a highest-weight combinatorial \({\tt A}_1^{(1)}\)-crystal, where  \(e_i^\Lambda, f_i^\Lambda\) are the operators inherited from \(\aMV\) by setting all \((\pi | \phi) \notin \aMV^\Lambda\) equal to \(0\), and we have \(\textup{wt}^\Lambda = \Lambda +  \textup{wt}\), \(\varepsilon_i^\Lambda = \varepsilon_i\),  \(\varphi_i^\Lambda = \varphi_i + t_i\).
The following is \cite[Corollary~4.6]{aMVrank2}.
\begin{theorem}\label{polycycthm}
The  crystal \(\aMV^\Lambda\) is isomorphic to the highest weight crystal \(\mathcal{B}(\Lambda)\).
\end{theorem}

\subsection{Working with the crystal \(\aMV\)}
The main inconvenience in working with the crystal \(\aMV\) (or \(\aMV^\Lambda)\) is the cumbersome definition of the involution \(\lozenge: \pi \mapsto \pi_{\lozenge}\). The existence/uniqueness definition of \(\lozenge\) implied by Theorem~\ref{exunloz} does not yield a very practical way to work with affine MV polytopes, particularly since each application of an operator \(e_i/f_i\) in \S\ref{aMVcrysdef} involves recalculating \(\lozenge\) anew. Currently, the best approach to constructing  \(\pi_{\lozenge}\)  involves iteratively tweaking/mending polytope shapes `one trapezoid at a time' until the myriad \(\aMV\) conditions are satisfied (see \cite[Remark~3.22]{aMVrank2}). We now rectify this issue.

\subsubsection{A combinatorial description of the \(\lozenge\) involution}\label{combloz}
The tableau combinatorics of \S\ref{traversesec}, \ref{opsec} can be used to directly describe the key involution \(\lozenge: \Pi(\theta) \to \Pi(\theta)\). This appears as Theorem~\ref{amvrecogthm}.

\begin{alphatheorem}{B}\label{rpartranslate}
Let \(\pi = \{\pi^1, \pi^\delta, \pi^0\} \in \Pi(\beta)\) be a root partition in type \({\tt A}_1^{(1)}\). Then \(\pi_{\lozenge} = \{\pi_\lozenge^1, \pi_\lozenge^\delta, \pi_\lozenge^0\} \) may be directly constructed as follows. Set \({\tt w}(\pi) = \max\{{\tt w}(\pi^{0}), \pi_1^\delta, {\tt w}(\pi^{1})\}\). Then we have
\begin{align*}
\pi_\lozenge^1 &= \ORTH^{{ \tt w}(\pi)}(\pi^0);\\
\pi_\lozenge^0 &= \ORTH^{{ \tt w}(\pi)}(\pi^1);\\
\pi_{\lozenge}^\delta &= \begin{cases}
({\tt w}(\pi), \pi_1^\delta, \ldots, \pi_{p(\pi^\delta)}^\delta) &\textup{if } {\tt w}(\pi)={\tt w}(\pi^{0}) = {\tt w}(\pi^{1}); \\
(\pi^\delta_2, \ldots, \pi^\delta_{p(\pi^\delta)}) & \textup{if } {\tt w}(\pi) > {\tt w}(\pi^{0}), {\tt w}(\pi^{1});\\
\pi^\delta & \textup{otherwise}.
\end{cases}
\end{align*}
\end{alphatheorem}

\section{Kleshchev multipartitions}\label{bigKleshsec}
\subsection{Multicharges and multipartitions}
A {\em (semi-infinite) multicharge} \(\bkap = (\kappa_1, \kappa_2, \ldots)\) is a sequence of elements of \(\Z_2\) containing infinitely many 0's and infinitely many 1's. A {\em \(\bkap\)-multipartition} \(\blam = (\lambda^{(1)}, \lambda^{(2)}, \ldots)\) is a sequence of partitions \(\lambda^{(t)} \in \Par^{\kappa_t}\), with \(\lambda^{(u)} = \varnothing\) for \(u \gg 0\). 
 If \(\lambda^{(m)} = \varnothing\) for \(m > \ell\) then we say that \(\blam\) is of {\em level \(\ell\)}. 
We write \(|\blam| = \sum_{t=1}^\infty |\lambda^{(t)}|\). The {\em content} of \(\blam\) is given by
\begin{align*}
\textup{cont}(\blam) = \#\{\textup{boxes of residue }0\textup{ in }\blam\}\cdot \alpha_0 +  \#\{\textup{boxes of residue }1\textup{ in }\blam\}\cdot \alpha_1 \in \Z_{\geq 0}I.
\end{align*}
We write \(\MP^{\bkap}(\theta)\) for the set of \(\bkap\)-multipartitions of content \(\theta \in \Z_{\geq0}I\), and set \(\MP^{\bkap} = \bigsqcup_{\theta \in \Z_{\geq 0}I} \MP^{\bkap}(\theta)\).
We depict \(\blam \in \MP^{\bkap}\) as a vertical array of partitions with \(\lambda^{(1)}\) at the bottom and with residues shown within each box; see Figure~\ref{kleshcrysexfig} for examples.
The dominance order \(\trianglerighteq\) on \(\MP^{\bkap}\) is defined by setting \(\bmu \trianglerighteq \blam\) provided
\begin{align*}
\sum^\infty_{s=t+1} |\mu^{(s)}| + \sum_{r=1}^m \mu_r^{(t)}
\geq 
\sum^\infty_{s=t+1} |\lambda^{(s)}| + \sum_{r=1}^m \lambda_r^{(t)}
\qquad
\textup{for all }t \in \Z_{>0}, m \in [1,p(\lambda^{(t)})].
\end{align*}

\begin{remark}
We use semi-infinite multicharges in this paper because our primary focus is the \(\mathcal{B}(\infty)\) crystal; one easily recovers the more familiar {\em finite} multicharge \((\kappa_1, \ldots, \kappa_\ell)\) setting by selecting any semi-infinite extension \(\bkap = (\kappa_1, \ldots, \kappa_\ell, \ldots)\) and restricting to multipartitions of level \(\ell\), see \S\ref{cycKleshcrystal}. Because we choose to index our multipartition components vertically with the first component at the bottom, readers accustomed to indexing beginning from the top component will also need to reorder to match our convention.
\end{remark}

For \(i \in \Z_2\), a box \(u \in \blam\) is said to be {\em \(i\)-removable} in \(\blam\) if \(\kres(u) = i\) and \(\blam \backslash u\) is a \(\bkap\)-multipartition. A box \(u \notin \blam\) is said to be {\em \(i\)-addable} in \(\blam\) if \(\kres(u) = i\) and \(\blam \sqcup u\) is a \(\bkap\)-multipartition. We note that by the conditions on the semi-infinite multicharge \(\bkap\), there will always be infinitely many \(i\)-addable boxes for \(\blam\), and finitely many (perhaps zero) \(i\)-removable boxes.

\begin{figure}[h]
\begin{align*}
\hackcenter{}
\hackcenter{
\begin{overpic}[height=45.5mm]{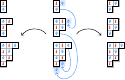}
 \put(27,41.7){\makebox(0,0)[b]{$\scriptstyle e_0$}}    
  \put(71,41.7){\makebox(0,0)[b]{$\scriptstyle f_0$}}    
    \put(42,27){\makebox(0,0)[r]{$\scriptstyle \lambda^{(1)}$}}    
        \put(42,46){\makebox(0,0)[r]{$\scriptstyle \lambda^{(2)}$}}    
            \put(42,60.5){\makebox(0,0)[r]{$\scriptstyle \lambda^{(3)}$}}    
                \put(44.5,65){\makebox(0,0)[b]{$\scriptstyle \blam$}}    
                    \put(2.5,65){\makebox(0,0)[b]{$\scriptstyle e_0\blam$}}    
                        \put(88,65){\makebox(0,0)[b]{$\scriptstyle f_0\blam$}}    
\end{overpic}
}
\end{align*}
\caption{Take the multicharge \(\bkap\) with \(\kappa_1 =\kappa_2 = 0, \kappa_3 = 1\), and \(\theta = 7 \alpha_1 + 8 \alpha_0\). A level 3 multipartition \(\blam \in \MP^{\bkap}(\theta)\) is shown in the center, where \(\lambda^{(1)} = (3,2^2,1)\), \(\lambda^{(2)} = (2^2,1)\), \(\lambda^{(3)} = (1^2)\). The \(0\)-addable boxes in the first three components are highlighted, as is the \(0\)-arc arrangement \(\blam[0]\) for \(\blam\). As \(\blam\) has one unarced \(0\)-removable box, we have \(\varepsilon_0(\blam) = 1\). The multipartitions \(e_0 \blam\) and \(f_0 \blam\) are shown to the left and right respectively.}
\label{kleshcrysexfig}
\end{figure}
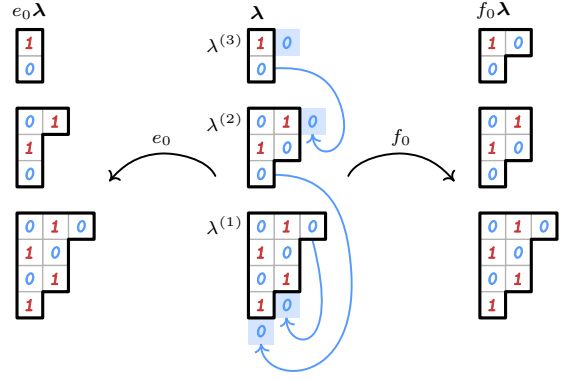

\subsection{The Kleshchev multipartition crystal}\label{Kleshcrys}
Let \(\blam \in \MP^{\bkap}\) and \(i \in \Z_2\). We define the {\em \(i\)-arc arrangement \(\blam[i]\) for \(\blam\)} as follows. Reading top to bottom, whenever an \(i\)-removable box is followed by an \(i\)-addable box with no unarced \(i\)-removable/\(i\)-addable boxes between them, draw an arc connecting them, and repeat until no more such arcs are possible. 

For \(i \in \Z_2\), we define operators 
\begin{align*}
f_i: \MP^{\bkap}(\theta)  \to \MP^{\bkap}(\theta + \alpha_i);
\qquad
e_i: \MP^{\bkap}(\theta) \to \MP^{\bkap}(\theta - \alpha_i)  \sqcup \{0\}
\end{align*}
as follows.
\begin{enumerate}
\item Set \(f_i \blam\) to be the \(\bkap\)-multipartition given by adding the bottommost unarced \(i\)-addable box in \(\blam[i]\).
\item Set \(e_i \blam\) to be the \(\bkap\)-multipartition given by deleting the topmost unarced \(i\)-removable box in \(\blam[i]\). If no such box exists, set \(e_i \blam = 0\).
\end{enumerate}
We show an example application of the operators \(e_0/f_0\) in Figure~\ref{kleshcrysexfig}.

We say that \(\blam \in \MP^{\bkap}\) is {\em Kleshchev} provided that \(\blam = f_{i_m} \cdots f_{i_2} f_{i_1} \varnothing\) (or equivalently \(\varnothing = e_{i_1} e_{i_2} \cdots e_{i_m} \blam\)) for some sequence \((i_1, i_2, \ldots, i_m) \in \Z_2^m\).
We write \(\MPK^{\bkap}\) for the set of Kleshchev multipartitions. Then \(\MPK^{\bkap}\) is a highest weight combinatorial \({\tt A}^{(1)}_1\)-crystal, where \(e_i, f_i\) are simply the restrictions to \(\MPK^{\bkap}\) of the operators defined above, the weight function is given by \(\textup{wt}(\blam) = -\textup{cont}(\blam)\), the statistic \(\varepsilon_i(\blam)\) is the number of unarced \(i\)-removable boxes in \(\blam[i]\), and \(\varphi_i\) is implicitly defined by Definition~\ref{crystaldef}(1). The first few layers of the crystal \(\MPK^{\bkap}\) for \(\bkap = (0,1,0,1,\ldots)\) are shown in Figure~\ref{Kleshcrystal11fig}.

\begin{figure}[h]
\begin{align*}
\hackcenter{}
\hackcenter{
\begin{overpic}[height=92mm]{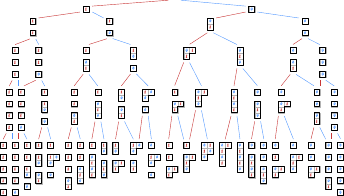}
 \put(50.8,57){\makebox(0,0)[]{$\scriptstyle \varnothing$}}   
\end{overpic}
}
\end{align*}
\caption{Initial layers of the \(\MPK^{\bkap}\) crystal for the multicharge \(\bkap = (0,1,0,1,\ldots)\). Empty components of multipartitions are not shown. The operators \(e_1/f_1\) appear in red as upward/downward segments respectively. The operators \(e_0/f_0\) appear in blue.}
\label{Kleshcrystal11fig}
\end{figure}

\subsection{The highest weight crystal \(\MPK^{(\bkap, \ell)}\)}\label{cycKleshcrystal}
Fix \(\ell \in \Z_{>0}\). We define \(\MPK^{(\bkap, \ell)}\) to be the set of Kleshchev multipartitions in \(\MPK^{\bkap}\) of level \(\ell\) (i.e., those \(\blam \in \MPK^{\bkap}\) with \(\lambda^{(m)} = \varnothing\) for \(m > \ell\)), and we define the associated dominant integral weight \(\Lambda(\bkap, \ell) := \sum_{t = 1}^\ell \Lambda_{\kappa_t}\). Then  \((\MPK^{(\bkap, \ell)}, e_i^{(\bkap, \ell)}, f_i^{(\bkap, \ell)}, \textup{wt}^{(\bkap, \ell)},  \varepsilon_i^{(\bkap, \ell)}, \varphi_i^{(\bkap, \ell)})\)
is a highest weight combinatorial \({\tt A}_1^{(1)}\)-crystal, with operators \(e_i^{(\bkap, \ell)}, f_i^{(\bkap, \ell)}\) inherited from \(\MPK^{\bkap}\) by setting all \(\blam \notin \MPK^{(\bkap, \ell)}\) equal to 0. We have \(\textup{wt}^{(\bkap, \ell)} = \Lambda(\bkap, \ell) +  \textup{wt}\), \(\varepsilon_i^{(\bkap, \ell)} = \varepsilon_i\), and \(\varphi_i^{(\bkap, \ell)}(\blam)\) is the number of unarced \(i\)-addable boxes in components \(1\) through \(\ell\) of \(\blam[i]\).
The following is due to work of \cite{KleshBranchI, KleshBranchII, LLT, MisraMiwa} in level one, and extended to higher levels in \cite{ArikiMathas}. See also \cite{Fayers2010, Uglov}.

\begin{theorem}\label{KleshCrystalCycThm}
Let \(\Lambda\) be a dominant integral weight, and assume the multicharge \(\bkap\) and \(\ell \in \Z_{>0}\) are such that \(\Lambda = \Lambda(\bkap, \ell)\). Then the  crystal \(\MPK^{(\bkap, \ell)}\) is isomorphic to the highest weight crystal \(\mathcal{B}(\Lambda)\).
\end{theorem}

The following theorem, well-known to experts, can be seen as an \(\ell \to \infty\) limit of Theorem~\ref{KleshCrystalCycThm}, or as a result of Theorems~\ref{PolyBinf}, \ref{mainamvulthm} and \ref{WKisommainproof}.

\begin{theorem}
Let \(\bkap\) be a multicharge. Then the crystal \(\MPK^{\bkap}\) is isomorphic to \(\mathcal{B}(\infty)\).
\end{theorem}

\subsection{A recognition theorem for Kleshchev multipartitions}
One of the main difficulties in working with the crystal \(\MPK^{\bkap}\) lies in the iterative definition of Kleshchevness. It is well-known that if \(\blam\) is of level one, then \(\blam\) is Kleshchev if and only if \(\lambda^{(1)}\) is 2-restricted. Under a restrictive choice of multicharge and multipartition (\(\kappa_t \neq \kappa_{t+1}\) for at most one \(t\) less than the level of \(\blam\)), a non-iterative condition for Kleshchevness was given in \cite[Proposition 4.9]{MathasAK} (see also \cite{AKT, Jacon} for other settings and approaches).
But for general multicharge and multipartition, no similar characterizing condition is known, and a priori it  requires a \(|\blam|\)-step process (either via repeated applications of the operators \(e_0/e_1\) or by an approach as in \cite{Jacon}) to determine whether \(\blam\) is Kleshchev or not. We rectify this with the following non-iterative recognition theorem, which appears as Theorem~\ref{Kleshtestappend}.

\begin{alphatheorem}{C}\label{Kleshtest}
Working in type \({\tt A}_1^{(1)}\), let \(\bkap\) be a multicharge, and let \(\blam\) be a multipartition of level \(\ell\). Then \(\blam\) is Kleshchev if and only if it satisfies the following conditions:
\begin{enumerate}
\item The partition \(\lambda^{(t)}\) is 2-restricted for all \(t \in [1,\ell]\).
\item For all \(t \in [1,\ell-1]\), we have \(\lambda^{(t)}_1 - (\lambda^{(t+1)})'_1 \geq \delta_{\kappa_t, \kappa_{t+1}} - 1\).
\item If \(s \in [1,\ell-2]\), \(u \in [s+2,\ell]\) are such that \(\lambda^{(u)} \neq \varnothing\) and \(\kappa_s \neq \kappa_{s+1} = \cdots =  \kappa_{u-1} \neq \kappa_u\), then there exists  \(t \in [s,u-1]\) such that  \(\lambda^{(t)}_1 - (\lambda^{(t+1)})'_1 \geq \delta_{\kappa_t, \kappa_{t+1}}+1\). 
\end{enumerate}
\end{alphatheorem}

\begin{remark}
We learned after completing this work that Theorem~\ref{Kleshtest} was also discovered independently by Fayers, and will appear in his forthcoming book \cite{FayersBook}.
\end{remark}

\begin{remark}
Conditions (1) and (2) appear in Mathas's characterization in the restricted multicharge setting \cite{MathasAK}. Condition (3) was identified with the aid of machine learning. We trained decision tree classifiers with an emphasis on features known to distinguish Kleshchevness in special cases, and interpreting the results helped identify the criterion for the general case. It would be interesting to explore whether machine learning could similarly assist in the study of Kleshchev multipartitions beyond the ${\tt A}_1^{(1)}$ case.
\end{remark}

\section{Upper ledge diagrams}\label{bigULsec}

\subsection{Ledge diagrams}\label{ledgeintrosec}
We now define some new combinatorial objects.
A {\em ledge diagram} \(\sfD\) is a left-aligned arrangement of boxes with rows of length \((d_1, \ldots, d_m) \in \Z_{>0}^m\) from top to bottom, with each box `colored' by \(i \in \Z_2\), such that:
\begin{enumerate}
\item for some \(t \in [1,m]\), \((d_1, \ldots, d_t)\) is weakly increasing and \((d_t, \ldots, d_m)\) is weakly decreasing, and;
\item all boxes within a given row share the same color, and colors alternate down columns.
\end{enumerate}
We refer to any row of maximal length in a ledge diagram \(\sfD\) as a {\em peak row} of \(\sfD\), and any row which is of maximal length among \(i\)-colored rows as an \(i\)-peak row. We say an \(i\)-colored box in \(\sfD\) is \(i\)-bottom (resp. \(i\)-top) if there is no box directly below (resp. above) it. We define the {\em content} of \(\sfD\) to be
\begin{align*}
\textup{cont}(\sfD) = \#\{\textup{boxes of color }0\textup{ in }\sfD\}\cdot \alpha_0 +  \#\{\textup{boxes of color }1\textup{ in }\sfD\}\cdot \alpha_1 \in \Z_{\geq 0}I.
\end{align*}
An {\em upper} (resp. {\em lower}) ledge diagram is one wherein we replace the `weakly' in front of `increasing' (resp. `decreasing') with `strictly' in (1) above.
We write \(\UL(\theta)\) (resp. \(\LL(\theta)\)) for the upper (resp. lower) ledge diagrams of content \(\theta \in \Z_{\geq0}I\). We set \(\UL = \bigsqcup_{\theta \in \Z_{>0} I} \UL(\theta)\) and \(\LL = \bigsqcup_{\theta \in \Z_{>0} I} \LL(\theta)\). Example upper/lower ledge diagrams appear in Figure~\ref{sinkfloatexfig}.

Given an upper (resp. lower) ledge diagram \(\sfD\), there is a unique way to rearrange \(\sfD\) into a lower (resp. upper) ledge diagram \(\sfD' \) by shifting columns of \(\sfD\) in increments of two so that box colors are preserved. We write \(\SINK \colon\LL(\theta) \to \UL(\theta)\) and \(\FLOAT \colon \UL(\theta) \to \LL(\theta)\) for these inverse bijections. See Figure~\ref{sinkfloatexfig} for an example. We will also write \(\textup{flip}\) for the involution on ledge diagrams given by reflecting across a horizontal axis. For \(\lambda \in \Par\), we will write \(\textup{color}_i(\lambda) \in \UL\) for the upper ledge diagram with shape \(\lambda\) and with top row colored by \(i\). If a ledge diagram has weakly decreasing rows, then we write \(\textup{forget}(\sfD) \in \Par\) for the partition given by forgetting the box colors of \(\sfD\).

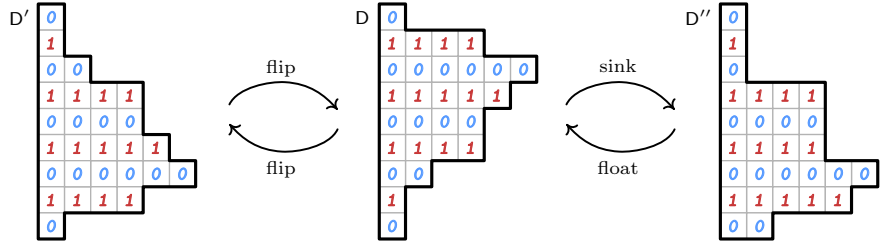
\begin{figure}[h]
\begin{align*}
\hackcenter{}
\hackcenter{
\begin{overpic}[height=31.5mm]{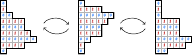}
 \put(-0.7,26.5){\makebox(0,0)[r]{$\scriptstyle \sfD'$}}    
  \put(39.5,26.5){\makebox(0,0)[r]{$\scriptstyle \sfD$}}    
   \put(80.2,26.5){\makebox(0,0)[r]{$\scriptstyle \sfD''$}}    
  \put(69,20){\makebox(0,0)[b]{$\scriptstyle \SINK$}}  
    \put(69,9.4){\makebox(0,0)[t]{$\scriptstyle \FLOAT$}}  
      \put(29,19.5){\makebox(0,0)[b]{$\scriptstyle \textup{flip}$}}  
    \put(29,9.4){\makebox(0,0)[t]{$\scriptstyle \textup{flip}$}}  
\end{overpic}
}
\end{align*}
\caption{Take \(\theta = 14 \alpha_1 + 14 \alpha_0\). An upper ledge diagram \(\sfD \in \UL(\theta)\) is shown at center, and lower ledge diagrams \(\sfD' = \textup{flip}(\sfD)\) and \(\sfD'' = \SINK(\sfD)\) are shown to either side.}
\label{sinkfloatexfig}
\end{figure}

\subsection{The upper ledge diagram crystal}\label{ULcrystaldefworkings}
Let \(i \in \Z_2\) and \(\sfD \in \UL\). We define the {\em \(i\)-arc arrangement \(\sfD[i]\) for \(\sfD\)} as follows. Reading right to left, whenever an \(i\)-top (resp. \( i\)-bottom) box is followed by an \(\hat{i}\)-top (resp. \(\hat i\)-bottom) box with no unarced top (resp. bottom) boxes between them, draw an arc connecting them. Repeat until no more such arcs are possible. An example is shown in Figure~\ref{ulex11fig}.

We give \(\UL\) the structure of an \({\tt A}_1^{(1)}\)-bicrystal as follows. Let \(i \in \Z_2\) and \(\sfD \in \UL\). We set \(\textup{wt}(\sfD) = -\textup{cont}(\sfD)\). The statistic \(\varepsilon_i(\sfD)\) is the number of unarced \(i\)-bottom boxes in \(\sfD[ i]\), and \(\varepsilon_i^*(\sfD)\) is the number of unarced \(i\)-top boxes in \(\sfD[i]\). The statistics \(\varphi_i, \varphi_i^*\) are implicitly defined by Definition~\ref{crystaldef}(1).
Then
\begin{enumerate}
\item \(e_{ i} \sfD\) is given by deleting the rightmost unarced \( i\)-bottom box in \(\sfD[ i]\). If none exists, set \(e_{ i} \sfD = 0\).
\item \(f_{ i} \sfD\) is given by adding an \( i\)-colored box directly below the leftmost unarced \(\hat i\)-bottom box in \(\sfD[ i]\). If none exists, append an \( i\)-colored box to the uppermost \(i\)-peak row in \(\sfD\).
\item \(e_{ i}^* \sfD\) is given by deleting the rightmost unarced \( i\)-top box in \(\sfD[i]\). If none exists, set \(e_i^* \sfD = 0\).
\item \(f_i^*\sfD\) is given by adding an \(i\)-colored box directly above the leftmost unarced \(\hat{i}\)-top box in \(\sfD[i]\). If none exists, append an \(i\)-colored box to the uppermost \(i\)-peak row in \(\sfD\).
\end{enumerate}
An example application of these crystal operators is shown in Figure~\ref{ulex11fig}.  
As with affine MV polytopes, we will primarily treat \(\UL\) as a crystal (rather than bicrystal) with unstarred operators. The first few layers of the crystal \(\UL\) appear in Figure~\ref{ledgecrystal11fig}.
The following theorem is a result of Theorems~\ref{PolyBinf} and \ref{mainamvulthm}.
\begin{alphatheorem}{D}\label{BinfUL}
The \({\tt A}_1^{(1)}\)-crystal \(\UL\) is isomorphic to \(\mathcal{B}(\infty)\) .
\end{alphatheorem}

\begin{remark}
It is straightforward to check from Remark~\ref{starinvpoly} and Theorem~\ref{PolyBinf} that Kashiwara's \(*\)-involution corresponds to the map \(\textup{float} \circ \textup{flip}: \UL \to \UL\) under the isomorphism of Theorem~\ref{BinfUL}.
\end{remark}

\begin{figure}[h]
\begin{align*}
\hackcenter{}
\hackcenter{
\begin{overpic}[height=48.5mm]{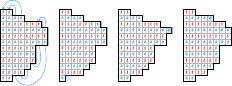}
 \put(-0.5,32){\makebox(0,0)[r]{$\scriptstyle \sfD$}}    
  \put(29.2,34.7){\makebox(0,0)[r]{$\scriptstyle e_0 \sfD$}}  
    \put(54.7,34.7){\makebox(0,0)[r]{$\scriptstyle f_0 \sfD$}}  
       \put(83.2,34.7){\makebox(0,0)[r]{$\scriptstyle f_0^* \sfD$}}  
\end{overpic}
}
\end{align*}
\caption{At left, an upper ledge diagram \(\sfD \in \UL\) shown with \(0\)-arc arrangement \(\sfD[0]\). Shown to the right are the upper ledge diagrams \(e_0\sfD\), \(f_0\sfD\), and \(f_0^*\sfD\) (note \(e_0^*\sfD = 0\)).}
\label{ulex11fig}
\end{figure}

\begin{figure}[h]
\begin{align*}
\hackcenter{}
\hackcenter{
\begin{overpic}[height=75mm]{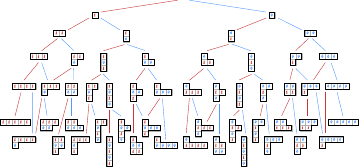}
 \put(51.2,46.5){\makebox(0,0)[]{$\scriptstyle \varnothing$}}    
\end{overpic}
}
\end{align*}
\caption{Initial layers of the \(\UL\) crystal. The operators \(e_1/f_1\) appear in red as upward/downward segments respectively. The operators \(e_0/f_0\) appear in blue.}
\label{ledgecrystal11fig}
\end{figure}

\subsection{The highest weight crystal \(\UL^\Lambda\)}
Let \(t_1, t_0 \in \Z_{\geq 0}\), so \(\Lambda = t_1 \Lambda_1 + t_0 \Lambda_0\) is a dominant integral weight. Let \(\UL^\Lambda \subset \UL\) be the subset of upper ledge diagrams \(\sfD\) such that \(\varepsilon_i^*(\sfD) \leq t_i\) for \(i \in \Z_2\). Then \((\UL^\Lambda, e_i^\Lambda, f_i^\Lambda, \textup{wt}^\Lambda,  \varepsilon_i^\Lambda, \varphi_i^\Lambda)\) is a highest-weight combinatorial \({\tt A}_1^{(1)}\)-crystal, where  \(e_i^\Lambda, f_i^\Lambda\) are the operators inherited from \(\UL\) by setting all \(\sfD\notin \UL^\Lambda\) equal to \(0\), and we have \(\textup{wt}^\Lambda = \Lambda +  \textup{wt}\), \(\varepsilon_i^\Lambda = \varepsilon_i\),  \(\varphi_i^\Lambda = \varphi_i + t_i\).
The following is immediate from Theorems~\ref{polycycthm} and \ref{mainamvulthm}.
\begin{theorem}
The \({\tt A}_1^{(1)}\)-crystal \(\UL^\Lambda\) is isomorphic to the highest weight crystal \(\mathcal{B}(\Lambda)\).
\end{theorem}

\section{Crystal isomorphisms}\label{crystalisomsec}
In this section we describe the main results of the paper---a tool kit of explicit isomorphisms, couched in the language of partition combinatorics, between the \({\tt A}_1^{(1)}\)-crystal structures described in \S\ref{bigAMVsec}, \ref{bigKleshsec}, and \ref{bigULsec}.

\subsection{Crystal isomorphisms between affine MV polytopes and upper ledge diagrams} \label{MVULisomsintro}
In this subsection we provide an isomorphism between the \({\tt A}_1^{(1)}\)-bicrystals \(\aMV\) and \(\UL\).

\subsubsection{The \(i\)-triple upper ledge diagrams} We will first need to introduce a key combinatorial notion. We say \((\sfD^\uparrow_i, \sfD^\delta_i, \sfD^\downarrow_i) \in \UL^3\) is an {\em \(i\)-triple upper ledge diagram} provided:
\begin{itemize}
\item the row lengths of \(\sfD^\uparrow_i\) strictly increase from top to bottom, and the bottom row has color \(i\);
\item the row lengths of \(\sfD^\delta_i\) weakly decrease from top to bottom, all top boxes in \(\sfD^\delta_i\) have color \(\hat i\), and all bottom boxes in \(\sfD^\delta_i\) have color \(i\);
\item the row lengths of \(\sfD^\downarrow_i\) strictly decrease from top to bottom, and the top row has color \(\hat i\).
\end{itemize}
Note then that \(D_i^\delta\) will always have even-length columns. We write \(\UL^3_i\) for the set of \(i\)-triple upper ledge diagrams.

To \((\sfD^\uparrow_i, \sfD^\delta_i, \sfD^\downarrow_i) \in \UL_i^3\) we may associate an upper ledge diagram \( \STACK_i(\sfD^\uparrow_i, \sfD^\delta_i, \sfD^\downarrow_i) \in \UL\) as follows. Place \(\sfD^\uparrow_i\) directly above \(\sfD^\delta_i\) and \(\sfD^\downarrow_i\) directly below \(\sfD^\delta_i\) with left columns aligned. Then shift the columns of \(\sfD^\downarrow_i\) upward until they meet the bottom of \(\sfD^\delta_i\) or else align with its top row. Then \( \STACK_i(\sfD^\uparrow_i, \sfD^\delta_i, \sfD^\downarrow_i)\) is the resulting diagram. See Figure~\ref{decompex123fig} for an example.

For \(\sfD \in \UL\), we define an associated \(i\)-triple upper ledge diagram \(\TRIP_i(\sfD):= (\sfD^\uparrow_i, \sfD^\delta_i, \sfD^\downarrow_i)\) as follows. 
Take \(\sfD_i^\uparrow\) to be all boxes weakly northwest of the rightmost \(i\)-top box in \(\sfD\).
To define \(\sfD_i^\downarrow\),  highlight the rightmost \(\hat i\)-bottom box in \(\sfD\), and then working leftward, highlight the same number of boxes at the bottom of each successive column in \(\sfD\), increasing by one each time the bottom box color changes. 
Shifting the highlighted columns down so they are top-aligned yields \(\sfD_i^\downarrow\).
(Equivalently, one could take \(\sfD_1^\downarrow\) to be all boxes weakly southwest of the rightmost \(\hat i\)-bottom box in \(\SINK(\sfD)\)).
We take the remaining diagram after removing the boxes in \(\sfD^\uparrow_i\) and \(\sfD^\downarrow_i\) from \(\sfD\) to be \(\sfD^\delta_i\). See Figure~\ref{decompex123fig} for an example.
It is easy to check that \(\TRIP_i: \UL \to \UL^3_i\) and \(\STACK_i: \UL^3_i \to \UL\) are mutually inverse bijections.

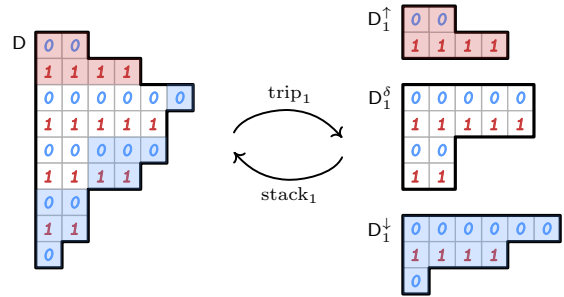
\begin{figure}[h]
\begin{align*}
\\
\hackcenter{}
\hackcenter{
\begin{overpic}[height=38.5mm]{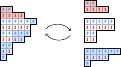}
 \put(-1.5, 48){\makebox(0,0)[r]{$\scriptstyle \sfD$}}  
  \put(68, 52.5){\makebox(0,0)[r]{$\scriptstyle \sfD_1^\uparrow$}}    
   \put(68, 37.5){\makebox(0,0)[r]{$\scriptstyle \sfD_1^\delta$}}    
    \put(68, 12.5){\makebox(0,0)[r]{$\scriptstyle \sfD_1^\downarrow$}}      
         \put(48.5,36.5){\makebox(0,0)[b]{$\scriptstyle \TRIP_1$}}    
           \put(48.5,20.5){\makebox(0,0)[t]{$\scriptstyle \STACK_1$}}    
\end{overpic}
}
\end{align*}
\caption{At left, an upper ledge diagram \(\sfD \in \UL\), and at right its decomposition into the \(1\)-triple upper ledge diagram \(\TRIP_1(\sfD) = (\sfD_1^\uparrow, \sfD^\delta_1, \sfD^\downarrow_1) \in \UL^3_1\).}
\label{decompex123fig}
\end{figure}

\subsubsection{From root partitions to upper ledge diagrams}
Let \(\theta \in \Z_{\geq 0}I\), and \(i \in \Z_2\). We define a map \(\mathcal{X}_i:\Pi(\theta) \to \UL(\theta)\) as follows. For \(\pi = \{\pi^1, \pi^\delta, \pi^0\} \in \Pi(\theta)\), we set 
\begin{align*}
\sfD^\uparrow_i = \textup{flip} \circ \textup{color}_i \circ \ORD(\pi^i);
\qquad
\sfD^\delta_i = \textup{color}_{\,\hat i} \circ \textup{double}(\pi^\delta);
\qquad 
\sfD^\downarrow_i =\textup{color}_{\,\hat i} \circ \ORD(\pi^{\hat i}),
\end{align*}
and define
\(
\mathcal{X}_i(\pi)= \STACK_i(\sfD^\uparrow_i, \sfD^\delta_i, \sfD^\downarrow_i).
\)
An example is shown in Figure~\ref{rptoledge2fig}.

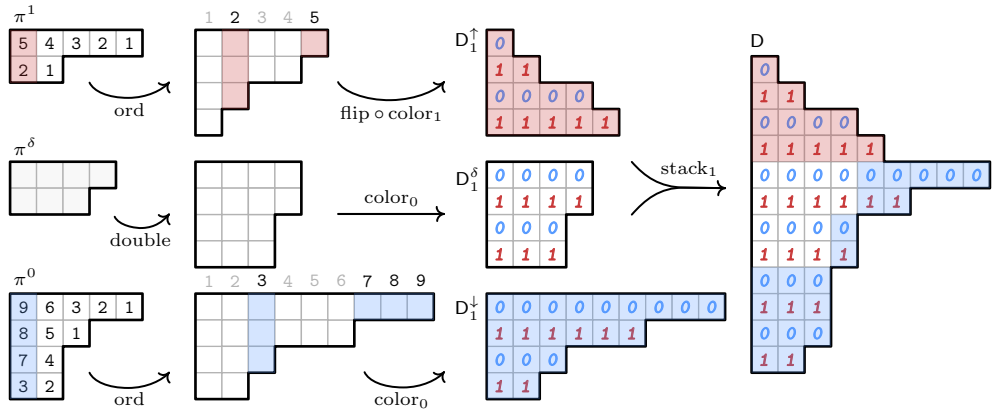
\begin{figure}[h]
\begin{align*}
\\
\hackcenter{}
\hackcenter{
\begin{overpic}[height=51.5mm]{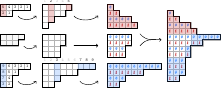}
 \put(0.5,38.5){\makebox(0,0)[l,b]{$\scriptstyle \pi^1$}}    
  \put(0.5,25){\makebox(0,0)[l,b]{$\scriptstyle \pi^\delta$}}    
    \put(0.5,11.7){\makebox(0,0)[l,b]{$\scriptstyle \pi^0$}}   
        \put(76.2,36){\makebox(0,0)[b]{$\scriptstyle \sfD$}}   
         \put(12.3,30.3){\makebox(0,0)[t]{$\scriptstyle \ORD$}}    
           \put(12.3,0.8){\makebox(0,0)[t]{$\scriptstyle \ORD$}}   
             \put(13.4,17){\makebox(0,0)[t]{$\scriptstyle \textup{double}$}}   
            \put(39,30){\makebox(0,0)[t]{$\scriptstyle \textup{flip} \, \circ \, \textup{color}_1$}}  
          \put(39, 19.8){\makebox(0,0)[b]{$\scriptstyle \textup{color}_0$}}  
                 \put(40,0.8){\makebox(0,0)[t]{$\scriptstyle \textup{color}_0$}}   
                            \put(69.3,22.8){\makebox(0,0)[b]{$\scriptstyle \STACK_1$}}   
      \put(48,36.7){\makebox(0,0)[r]{$\scriptstyle \sfD^\uparrow_1$}}   
          \put(48,22.5){\makebox(0,0)[r]{$\scriptstyle \sfD^\delta_1$}}   
              \put(48,9.4){\makebox(0,0)[r]{$\scriptstyle \sfD^\downarrow_1$}}   
\end{overpic}
}
\end{align*}
\caption{Computing the upper ledge diagram \(\mathcal{X}_1(\pi) = \sfD\) (at right) associated to the root partition \(\pi = \{\pi^1, \pi^\delta, \pi^0\}\) (at left).}
\label{rptoledge2fig}
\end{figure}

\subsubsection{From upper ledge diagrams to root partitions}
Let \(\theta \in \Z_{\geq 0}I\), and \(i \in \Z_2\).
We now define a map \(\mathcal{Y}_i: \UL(\theta) \to \Pi(\theta)\) as follows. Let \(\sfD \in \UL(\theta)\), and set \((\sfD^\uparrow_i, \sfD^\delta_i, \sfD^\downarrow_i) = \TRIP_i(\sfD)\). Then define
\begin{align*}
\pi^i = \CALC \circ \textup{forget} \circ \textup{flip}(\sfD^\uparrow_i);
\qquad
\pi^\delta = \textup{halve} \circ \textup{forget}(\sfD^\delta_i);
\qquad
\pi^{\hat i} = \CALC \circ \textup{forget}(\sfD^\downarrow_i),
\end{align*}
and set \(\mathcal{Y}_i(\sfD) = \pi = \{\pi^1, \pi^\delta, \pi^0\}\). 
An example is shown in Figure~\ref{ledgetorp1fig}.

\begin{figure}[h]
\begin{align*}
\\
\hackcenter{}
\hackcenter{
\begin{overpic}[height=48mm]{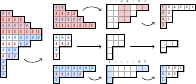}
 \put(-1, 36){\makebox(0,0)[r]{$\scriptstyle \sfD$}}  
  \put(26.8, 38.7){\makebox(0,0)[r]{$\scriptstyle \sfD_1^\uparrow$}}    
   \put(26.8, 23.2){\makebox(0,0)[r]{$\scriptstyle \sfD_1^\delta$}}    
    \put(26.8, 7.8){\makebox(0,0)[r]{$\scriptstyle \sfD_1^\downarrow$}}      
          \put(17,10){\makebox(0,0)[t]{$\scriptstyle \TRIP_1$}}  
              \put(46.7,40){\makebox(0,0)[b]{$\scriptstyle  \textup{forget}\,\circ\, \textup{flip}$}}       
           \put(45,19.5){\makebox(0,0)[b]{$\scriptstyle \textup{forget}$}}    
             \put(48,1.2){\makebox(0,0)[t]{$\scriptstyle \textup{forget}$}}    
                \put(72.5,19.6){\makebox(0,0)[b]{$\scriptstyle \textup{halve}$}}    
                  \put(75.5,29){\makebox(0,0)[t]{$\scriptstyle \CALC$}}    
                    \put(75.5,1.2){\makebox(0,0)[t]{$\scriptstyle \CALC$}}    
                       \put(83,41.3){\makebox(0,0)[b]{$\scriptstyle \pi^1$}}    
                          \put(83,22.7){\makebox(0,0)[b]{$\scriptstyle \pi^\delta$}}    
                             \put(83,10.5){\makebox(0,0)[b]{$\scriptstyle \pi^0$}}    
\end{overpic}
}
\end{align*}
\caption{Computing the root partition \(\mathcal{Y}_1(\sfD) = \pi = \{\pi^1, \pi^\delta, \pi^0\}\) associated to an upper ledge diagram \(\sfD\).}
\label{ledgetorp1fig}
\end{figure}
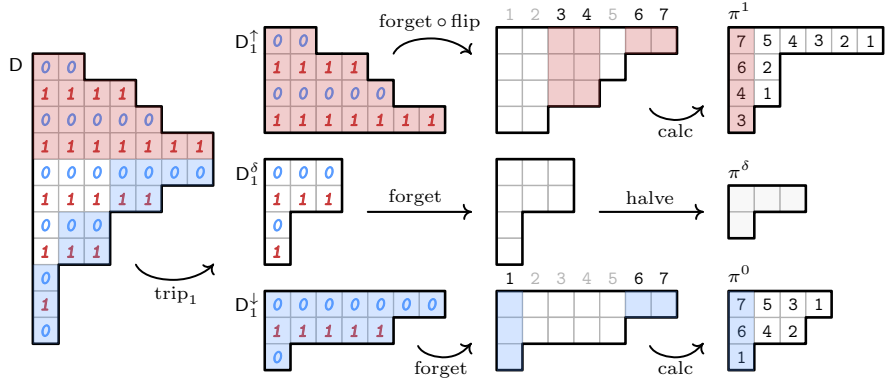

\subsubsection{The crystal isomorphism \(\aMV \leftrightarrow \UL\)} 
The following result appears as Theorem~\ref{mainamvulthm}.

\begin{alphatheorem}{E}\label{PolyOH}
For \(\theta \in \Z_{\geq 0}I\) and \(i \in \Z_2\), the maps \(\mathcal{X}_i: \Pi(\theta) \to \UL(\theta)\) and \(\mathcal{Y}_i: \UL(\theta) \to \Pi(\theta)\) are mutually inverse bijections. Moreover, we have mutually inverse \({\tt A}_1^{(1)}\)-bicrystal isomorphisms
\begin{align*}
&\mathcal{X}: \aMV \xrightarrow{\sim} \UL, \quad (\pi | \phi) \mapsto \mathcal{X}_1(\pi) = \mathcal{X}_0(\phi)
\\
&\mathcal{Y}: \UL \xrightarrow{\sim} \aMV, \quad
\sfD \mapsto (\mathcal{Y}_1(\sfD) \,|\, \mathcal{Y}_0(\sfD)).
\end{align*}
\end{alphatheorem}

\begin{alphacorollary}{F}
For any dominant integral weight \(\Lambda\), the maps \(\mathcal{X}, \mathcal{Y}\) restrict to mutually inverse  \({\tt A}_1^{(1)}\)-crystal isomorphisms
\begin{align*}
\mathcal{X}^\Lambda: \aMV^{\Lambda} \xrightarrow{\sim} \UL^\Lambda
\qquad \textup{and} \qquad
\mathcal{Y}^\Lambda: \UL^{\Lambda} \xrightarrow{\sim} \aMV^{\Lambda}.
\end{align*}
\end{alphacorollary}

The following would be implied by  Conjecture~\ref{conjsparcontain}.

\begin{conjecture}
Let \(i \in \Z_2\), \(\pi, \tilde{\pi} \in \Pi\), and assume \(\tilde{\pi}\) is obtained by deleting the last row of \(\pi^1\) or \(\pi^0\). Then \(\mathcal{X}_i(\tilde{\pi}) \subseteq \mathcal{X}_i(\pi)\).\end{conjecture}

\subsection{Crystal isomorphisms between Kleshchev multipartitions and upper ledge diagrams}\label{kmulisomsec}
In this subsection we provide an isomorphism between the \({\tt A}_1^{(1)}\)-bicrystals \(\MPK^{\bkap}\) and \(\UL\). 

\subsubsection{From upper ledge diagrams to Kleshchev multipartitions}\label{defthesplitmap}
Let \(i \in \Z_2\). We define now a map
\begin{align*}
\SPLIT^i \colon \UL \to \UL \times \UL, \quad \sfD \mapsto (\SPLIT^i(\sfD)_\downarrow, \SPLIT^i(\sfD)_\uparrow)
\end{align*}
as follows. 
Let \(\sfD \in \UL\), and let \(u\) be the uppermost box of color \(i\) in the first column of \(\sfD\). If there is no such \(u\), we set \(\SPLIT^i(\sfD) = (\varnothing, \sfD)\). Otherwise, extending a ray \(r\) directly southeast (i.e. at a \(-45^\circ\) angle) from the centroid of \(u\), let \(\mathsf{L}\) be those boxes in \(\sfD\) whose centroid lies weakly below \(r\), and let \(\mathsf{U}\) be those whose centroid lies strictly above \(r\). Now, \(\mathsf{L}\) is already an upper ledge diagram, and left-justifying the rows of \(\mathsf{U}\) gives a lower ledge diagram \(\mathsf{U}'\). We set \(\SPLIT^i(\sfD)_\downarrow = \mathsf{L}\) and \(\SPLIT^i_\uparrow(\sfD) = \FLOAT(\mathsf{U}')\). An example is shown in Figure~\ref{splitexfig}.

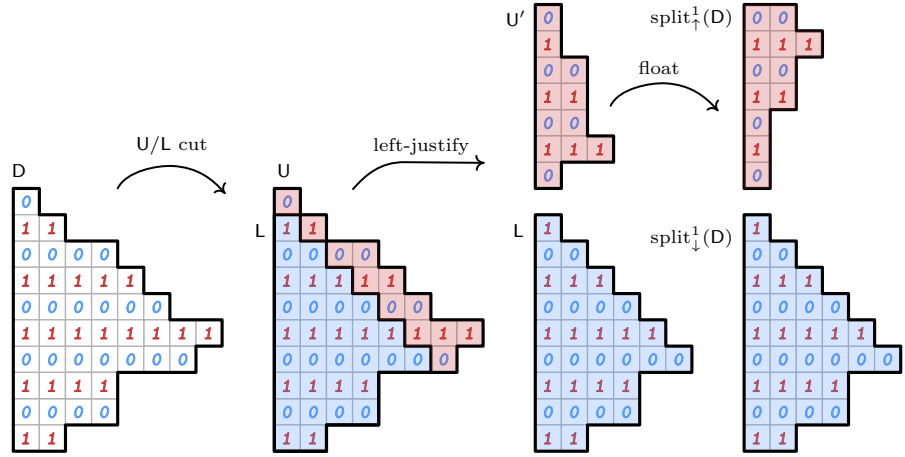
\begin{figure}[h]
\begin{align*}
\\
\hackcenter{}
\hackcenter{
\begin{overpic}[height=59.5mm]{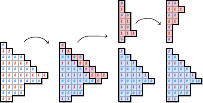}
 \put(0,31){\makebox(0,0)[l,b]{$\scriptstyle \sfD$}}    
  \put(29.8,31){\makebox(0,0)[l,b]{$\scriptstyle \mathsf{U}$}}    
    \put(28.5,25){\makebox(0,0)[r]{$\scriptstyle \mathsf{L}$}}    
      \put(57.8,48.8){\makebox(0,0)[r]{$\scriptstyle \mathsf{U}'$}}  
          \put(57.5,25){\makebox(0,0)[r]{$\scriptstyle \mathsf{L}$}}  
       \put(81,48.8){\makebox(0,0)[r]{$\scriptstyle \textup{split}_\uparrow^1(\sfD)$}}  
        \put(81,24){\makebox(0,0)[r]{$\scriptstyle \textup{split}_\downarrow^1(\sfD)$}}  
        \put(18,33.5){\makebox(0,0)[b]{$\scriptstyle \textup{$\mathsf{U}/\mathsf{L}$ cut}$}}  
          \put(46,33.5){\makebox(0,0)[b]{$\scriptstyle \textup{left-justify}$}}  
             \put(72.7,42.5){\makebox(0,0)[b]{$\scriptstyle \FLOAT$}}  
\end{overpic}
}
\end{align*}
\caption{An application of \(\SPLIT^1\) to an upper ledge diagram \(\sfD\) is shown. The resulting upper ledge diagrams \(\SPLIT^1(\sfD)_\downarrow\), \(\SPLIT^1(\sfD)_\uparrow\) are shown on the right.}
\label{splitexfig}
\end{figure}

More generally, if \(\bkap = (\kappa_1, \kappa_2, \ldots)\) is a multicharge, we define an infinite sequence of upper ledge diagrams \(\SPLIT^{\bkap}(\sfD) = (\sfD_1, \sfD_2, \ldots)\) by iteratively splitting the upper component repeatedly; i.e., we set 
\begin{align*}
 \sfD_1 = \SPLIT^{\kappa_1}(\sfD)_\downarrow, \quad \textup{and} \quad (\sfD_2, \sfD_3, \ldots) = \SPLIT^{(\kappa_2, \kappa_3, \ldots)}(\SPLIT^{\kappa_1}(\sfD)_\uparrow),
\end{align*}
noting that \(\sfD_k = \varnothing \) for \(k \gg 0\) by the definition of multicharges and the map \(\SPLIT^i\). 

Now, we define a map \(\mathcal{W}^{\bkap}: \UL(\theta) \to \MPK^{\bkap}(\theta)\) as follows. For \(\sfD \in \UL(\theta)\), assume \(\SPLIT^{\bkap}(\sfD) = (\sfD_1, \sfD_2, \ldots)\). Then, for \(t \in \Z_{>0}\), let \(\lambda^{(t)}\) be the partition constructed by top-aligning every column in \(\sfD_t\), and set \(\mathcal{W}^{\bkap}(\sfD) = (\lambda^{(1)}, \lambda^{(2)}, \ldots)\). An example is shown in Figure~\ref{applyWkap2fig}.

\begin{figure}[h]
\begin{align*}
\\
\hackcenter{}
\hackcenter{
\begin{overpic}[height=73.5mm]{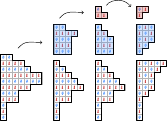}
 \put(1,41){\makebox(0,0)[b]{$\scriptstyle \sfD$}}   
  \put(19, 48){\makebox(0,0)[b]{$\scriptstyle \SPLIT^1$}}   
    \put(44,66){\makebox(0,0)[b]{$\scriptstyle \SPLIT^0$}}   
     \put(70.7, 73.2){\makebox(0,0)[b]{$\scriptstyle\textup{top-justify}$}}   
     \put(56.2, 34.2){\makebox(0,0)[r]{$\scriptstyle \sfD_1$}}   
       \put(56.2, 55){\makebox(0,0)[r]{$\scriptstyle \sfD_2$}}   
         \put(56.2, 66){\makebox(0,0)[r]{$\scriptstyle \sfD_3$}}   
            \put(81, 34.2){\makebox(0,0)[r]{$\scriptstyle \lambda^{(1)}$}}   
       \put(81, 55){\makebox(0,0)[r]{$\scriptstyle \lambda^{(2)}$}}   
         \put(81, 66){\makebox(0,0)[r]{$\scriptstyle \lambda^{(3)}$}}   
\end{overpic}
}
\end{align*}
\caption{Let \(\theta = 19 \alpha_1 + 19 \alpha_0\), and \(\bkap\) be a multicharge with \(\kappa_1 = 1\), \(\kappa_2 = 0\), \(\kappa_3 = 0\). 
At left, an upper ledge diagram \(\sfD \in \UL(\theta)\). The computation of \(\mathcal{W}^{\bkap}(\sfD) = \blam = (\lambda^{(1)}, \lambda^{(2)}, \lambda^{(3)}, \varnothing, \ldots) \in \MP^{\bkap}(\theta)\) is shown from left to right. The second-to-rightmost diagram is the sequence of upper ledge diagrams \((\sfD_1, \sfD_2, \sfD_3, \varnothing, \ldots) = \SPLIT^{\bkap}(\sfD)\).}
\label{applyWkap2fig}
\end{figure}

\subsubsection{From Kleshchev multipartitions to upper ledge diagrams}\label{gluingsec}

Let \(\sfD_1, \sfD_2 \in \UL\). We define an upper ledge diagram \(\GLUE(\sfD_1, \sfD_2) \in \UL\) as follows. First set \(\sfD_2\) to the right of \(\sfD_1\) with top rows aligned if they have matching color, or so that the top row of \(\sfD_2\) is one notch higher otherwise. Then, starting from the leftmost column in \(\sfD_2\) whose bottom box is above the uppermost peak row of \(\sfD_1\) (if such exists), work to the right, shifting columns down as far as possible in steps of two so that the bottom box is still (a) not below the uppermost peak row of \(\sfD_1\) and (b) is not below the bottom box in the column directly to its left. 
Then combine diagrams, left-justifying all rows to yield the diagram \(\GLUE(\sfD_1, \sfD_2)\). It is straightforward to check that this resulting diagram is upper ledge.
See Figure~\ref{gengluefig} for an example.

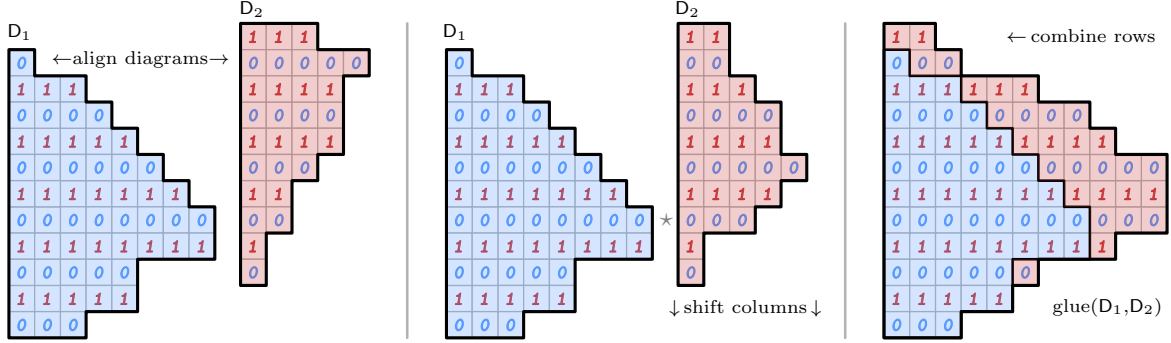
\begin{figure}[h]
\begin{align*}
\\
\hackcenter{}
\hackcenter{
\begin{overpic}[height=42mm]{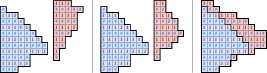}
 \put(3.8,23.4){\makebox(0,0)[l,b]{$\scriptstyle\leftarrow \textup{align diagrams} \rightarrow$}}    
  \put(57,2){\makebox(0,0)[l,b]{$\scriptstyle\downarrow \, \textup{shift columns} \,\downarrow$}}    
       \put(56.2,9.8){\makebox(0,0)[l,b]{$  \color{gray}\star \color{black}$}}    
         \put(86,25.6){\makebox(0,0)[l,b]{$\scriptstyle\leftarrow \, \textup{combine rows}$}}    
          \put(0,25.8){\makebox(0,0)[l,b]{$\scriptstyle \sfD_1$}}   
           \put(20,27.8){\makebox(0,0)[l,b]{$\scriptstyle \sfD_2$}}   
              \put(37.5,25.8){\makebox(0,0)[l,b]{$\scriptstyle \sfD_1$}}   
           \put(57.5,27.8){\makebox(0,0)[l,b]{$\scriptstyle \sfD_2$}}   
            \put(90,2){\makebox(0,0)[l,b]{$\scriptstyle \GLUE(\sfD_1, \sfD_2)$}}  
\end{overpic}
}
\end{align*}
\caption{The gluing process for upper ledge diagrams \(\sfD_1, \sfD_2\). The alignment of diagrams is shown in the first pane, and the column shift is shown in the second pane, with the uppermost peak row of \(\sfD_1\) marked with a star. The last pane is the resulting upper ledge diagram \(\GLUE(\sfD_1, \sfD_2)\).}
\label{gengluefig}
\end{figure}

More generally, let \(\underline{\sfD} = (\sfD_1, \sfD_2, \ldots)\) be an infinite sequence of upper ledge diagrams such that \(\sfD_k = \varnothing\) for \(k \gg 0\). If all \(\sfD_k\) are empty, set \(\GLUE(\underline{\sfD}) = \varnothing\). Otherwise, let \(M\) be maximal such that \(\sfD_M \neq \varnothing\). Now, set \(\sfD'_{M+1} = \varnothing\) and proceeding downwards iteratively define \(\sfD_t' = \GLUE(\sfD_t, \sfD_{t+1}')\) for \(t=M, M-1, M-2, \ldots, 1\). Set \(\GLUE(\underline{\sfD}) = \sfD_1'\).

Now we define a map \(\mathcal{Z}^{\bkap}: \MP^{\bkap}(\theta) \to \UL(\theta)\) as follows. Let \(\blam = (\lambda^{(1)}, \lambda^{(2)}, \ldots) \in \MP^{\bkap}(\theta)\). We convert each \(\lambda^{(t)}\) into an upper ledge diagram \(\sfD_t\) by shifting (for all \(k\)) the \(k\)th column of \(\lambda^{(t)}\) down by \(k-1\) notches and then left aligning all rows---the color of each box in \(\sfD_t\) being given by its pre-shift residue in \(\lambda^{(t)}\). 
Taking \(\underline{\sfD} = (\sfD_1, \sfD_2, \ldots)\), we set \(\mathcal{Z}^{\bkap}(\blam) = \GLUE(\underline{\sfD})\). An example is shown in Figure~\ref{Zkapexfig2}.

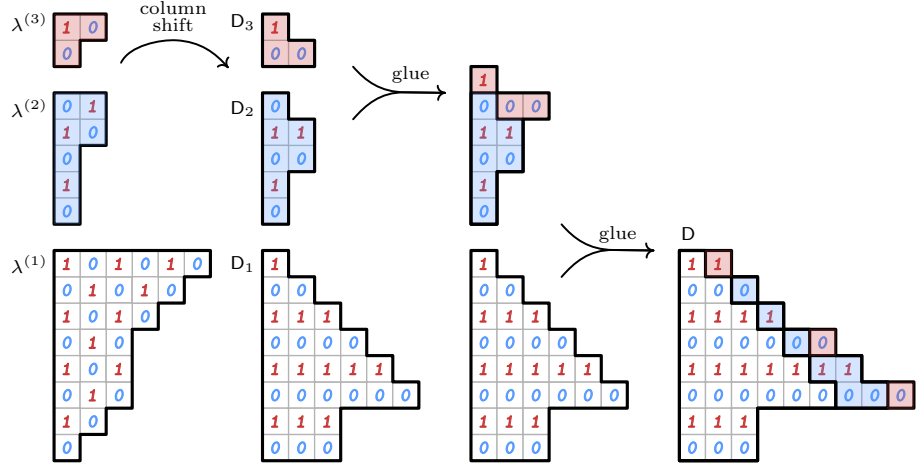
\begin{figure}[h]
\begin{align*}
\\
\hackcenter{}
\hackcenter{
\begin{overpic}[height=59.5mm]{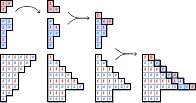}
 \put(-0.5,23){\makebox(0,0)[r]{$\scriptstyle \lambda^{(1)}$}}   
  \put(-0.5,41){\makebox(0,0)[r]{$\scriptstyle \lambda^{(2)}$}}   
    \put(-0.5,51){\makebox(0,0)[r]{$\scriptstyle \lambda^{(3)}$}}   
       \put(74.5,26.5){\makebox(0,0)[r]{$\scriptstyle \sfD$}}   
        \put(23.5,23){\makebox(0,0)[r]{$\scriptstyle \sfD_1$}}   
  \put(23.5,41){\makebox(0,0)[r]{$\scriptstyle \sfD_2$}}   
    \put(23.5,51){\makebox(0,0)[r]{$\scriptstyle \sfD_3$}}   
     \put(14,52){\makebox(0,0)[b]{$\scriptstyle \textup{column}$}}  
        \put(14,50){\makebox(0,0)[b]{$\scriptstyle \textup{shift}$}}  
           \put(41.5,44){\makebox(0,0)[b]{$\scriptstyle \GLUE$}}  
            \put(65.5,25.5){\makebox(0,0)[b]{$\scriptstyle \GLUE$}}  
\end{overpic}
}
\end{align*}
\caption{Let \(\theta = 16 \alpha_1 + 21\alpha_0\) and \(\bkap\) be a multicharge with \(\kappa_1 = 1\), \(\kappa_2 = 0\), \(\kappa_3 = 1\). At left, a Kleshchev multipartition \(\blam = (\lambda^{(1)}, \lambda^{(2)}, \lambda^{(3)}, \varnothing, \ldots) \in \MPK^{\bkap}(\theta)\). The computation of \(\mathcal{Z}^{\bkap}(\blam) = \sfD \in \UL(\theta)\) is shown from left to right. The second-to-leftmost diagram is the sequence of upper ledge diagrams \(\underline{\sfD} = (\sfD_1, \sfD_2, \sfD_3, \varnothing, \ldots)\) with \(\GLUE(\underline{\sfD}) = \sfD\).}
\label{Zkapexfig2}
\end{figure}

\subsubsection{The crystal isomorphism \(\MPK^{\bkap} \leftrightarrow \UL\)}

The following result appears as Theorem~\ref{WKisommainproof}.

\begin{alphatheorem}{G}\label{isommpul}
For any multicharge \(\bkap\) we have \(\mathcal{W}^{\bkap}(\UL) = \MPK^{\bkap}\), and the maps 
\begin{align*}
\mathcal{W}^{\bkap}:  \UL \to \MPK^{\bkap}
\qquad \textup{and} \qquad
\mathcal{Z}^{\bkap}: \MPK^{\bkap} \to \UL
\end{align*}
are mutually inverse crystal isomorphisms. 
\end{alphatheorem}

\begin{alphacorollary}{H}
Let \(\Lambda\) be a dominant integral weight, and assume the multicharge \(\bkap\) and \(\ell \in \Z_{>0}\) are such that \(\Lambda = \Lambda(\bkap, \ell)\). Then the maps \(\mathcal{W}^{\bkap}, \mathcal{Z}^{\bkap}\) restrict to mutually inverse crystal isomorphisms
\begin{align*}
\mathcal{W}^{(\bkap, \ell)}:  \UL^{\Lambda} \to \MPK^{(\bkap, \ell)}
\qquad \textup{and} \qquad
\mathcal{Z}^{(\bkap, \ell)}: \MPK^{(\bkap, \ell)} \to \UL^{\Lambda}.
\end{align*}
\end{alphacorollary}

\subsection{A crystal isomorphism between affine MV polytopes and Kleshchev multipartitions}
Finally, the upper ledge diagrams act as an intermediary, yielding a straightforward combinatorial isomorphism between  affine MV polytopes and Kleshchev multipartitions:

\begin{alphacorollary}{I}\label{PolyMPthm}
The maps 
\begin{align*}
\mathcal{W}^{\bkap} \circ \mathcal{X}: \aMV \to \MPK^{\bkap} 
\qquad
\textup{ and }
\qquad
\mathcal{Y} \circ \mathcal{Z}^{\bkap}: \MPK^{\bkap} \to \aMV
\end{align*} are mutually inverse crystal isomorphisms.
\end{alphacorollary}

\begin{alphacorollary}{J}\label{PolyMPthmcyc}
Let \(\Lambda\) be a dominant integral weight, and assume the multicharge \(\bkap\) and \(\ell \in \Z_{>0}\) are such that \(\Lambda = \Lambda(\bkap, \ell)\). Then the maps of Corollary~\ref{PolyMPthm} restrict to mutually inverse crystal isomorphisms
\begin{align*}
\mathcal{W}^{(\bkap, \ell)} \circ \mathcal{X}^{\Lambda}: \aMV^{\Lambda} \to \MPK^{(\bkap, \ell)} 
\qquad
\textup{ and }
\qquad
\mathcal{Y}^{\Lambda} \circ \mathcal{Z}^{(\bkap, \ell)}: \MPK^{(\bkap, \ell)} \to \aMV^{\Lambda}.
\end{align*}
\end{alphacorollary}

\section{Bicompositions and binary sequences}\label{dirtywork}
In this long and technical section we roll up our sleeves and prepare to prove the main results of \S\ref{partitioncombsec}--\ref{crystalisomsec}. Most of these are downstream of the key combinatorial result (Theorem~\ref{DGthm}) proved in this section which establishes a hidden binary tree structure on bicompositions associated with the `real' components of affine MV polytopes, see \S\ref{secmvbicomp}.

\subsection{Compositions}
A {\em composition} \(\mu\) is a finitely supported function \(\mu: \ZZ \to \ZZ_{\geq 0}\), which we write as \(\mu = (\mu_k)_{k \in \ZZ}\), where \(\mu_k = \mu(k)\). We write \(|\mu| = \sum_{k \in \ZZ} \mu_k \in \ZZ_{\geq 0}\). In practice we will usually have \(\mu_k = 0\) for \(k \leq 0\), but it is convenient for our purposes to allow \(\mu\) to be defined more generally. We write \(\Omega\) for the set of all compositions. We also will occasionally add decorations to this set, writing
\begin{align*}
\Omega_{\geq k} &= \{ \mu \in \Omega \mid \mu_\ell = 0 \textup{ for all } \ell <k\} \qquad \textup{and} \qquad
\Omega(n) = \{ \mu \in \Omega \mid |\mu| = n\},
\end{align*}
with multiple decorations like \(\Omega_{\geq k}(n)\) indicating the obvious intersections of the corresponding sets.

For each \(\mu \in \Omega, k \in \ZZ\), define a statistic \(\vartheta_k(\mu) \in \ZZ_{\geq 0}\) to be the {\em minimal nonnegative integer} \(m\) such that \(\mu_{k-m} = 0\), and define a statistic \(\eta_k(\mu) \in \ZZ_{\geq 0}\) to be the {\em minimal nonnegative integer} \(m\) such that \(\mu_{k+m} = 0\). Note in particular that \(\eta_k(\mu) = \vartheta_k(\mu) = 0\) if and only if \(\mu_k = 0\). 

\subsection{The functions \(\mathcal{G}\) and \(\mathcal{H}\)}
We define a function \(\mathcal{G}: \Omega \to \Omega\) by setting \(\mathcal{G}\mu = \nu\), where
\begin{align}
\nu_k = \begin{cases}
\mu_k - 1 & \textup{if \(\vartheta_k(\bmu)\) is odd};\\
\mu_k + 1 & \textup{if \(\vartheta_{k-1}(\bmu)\) is odd};\\
0 & \textup{if \(\mu_k = 0\) and \(\vartheta_{k-1}(\bmu)\) is even.}\\
\end{cases}
\label{Gcases}
\end{align}
We also define a dual function \(\mathcal{H}: \Omega \to \Omega\) by setting \(\mathcal{H}\mu = \nu\), where
\begin{align}
\nu_k = \begin{cases}
\mu_k - 1 & \textup{if \(\eta_k(\bmu)\) is odd};\\
\mu_k + 1 & \textup{if \(\eta_{k+1}(\bmu)\) is odd};\\
0 & \textup{if \(\mu_k = 0\) and \(\eta_{k+1}(\bmu)\) is even}.\\
\end{cases}
\label{Hcases}
\end{align}
We also
define maps \(\textup{zero}_0: \Omega\to \Omega\) and \(\textup{one}_0 : \Omega \to \Omega\) by setting
\begin{align*}
(\textup{zero}_0(\mu))_k = \begin{cases}
0 & \textup{if }k = 0;\\
\mu_k & \textup{otherwise}
\end{cases}
\qquad
\textup{and}
\qquad
(\textup{one}_0(\mu))_k = \begin{cases}
1 & \textup{if }k = 0;\\
\mu_k & \textup{otherwise}
\end{cases}
\end{align*}
In other words,  \(\textup{zero}_0(\mu)\) is given by replacing \(\mu_0\) with zero, and \(\textup{one}_0(\mu)\) is given by replacing \(\mu_0\) with one.
Then we define maps \(\bar{\mathcal{G}}\), \(\bar{\mathcal{H}}: \Omega \to \Omega\) by setting 
\begin{align*}
\bar{\mathcal{G}} = \mathcal{G}  \circ \textup{one}_0 \qquad \textup{and} \qquad  \bar{\mathcal{H}} = \textup{zero}_0 \circ \mathcal{H}.
\end{align*}

Let \(s: \Omega \to \Omega\) be the `flip' involution given by setting \(s\mu = \nu\), where \(\nu_k = \mu_{-k}\) for all \(k \in \ZZ\). It follows then that \(\vartheta_k(s\mu) = \eta_{-k}(\mu)\) and \(\eta_k(s\mu) = \vartheta_{-k}(\mu)\), and therefore that \(s  \mathcal{H}= \mathcal{G}  s\).
The following lemma provides an alternate description of the maps \(\mathcal{G}, \mathcal{H}: \Omega \to \Omega\). 

\begin{lemma}\label{altdesc}
Let \(\mu \in \Omega\). Then we have that 
\begin{enumerate}[(1)]
\item \(\mathcal{H}\mu\) is the unique element \(\nu \in \Omega\) with the property that, for all \(k \in \ZZ\):
\begin{enumerate}
\item  \(\nu_k \in \{\mu_k \pm 1\}\) if \(\mu_k >0\);
\item \(\nu_k \in \{0,1\}\) if \(\mu_k = 0\);
\item \(\nu_{k+1} = \mu_{k+1}-1\) if and only if \(\nu_k = \mu_k + 1\).
\end{enumerate}
Moreover if \(\mu \in \Omega_{\geq 1}\), then \(\bar{\mathcal{H}}\mu\) is the unique element \(\nu \in \Omega_{\geq 1}\) which satisfies (1a, 1b, 1c) for all \(k \geq 1\).
\item  \(\mathcal{G}\mu\) is the unique element \(\nu \in \Omega\) with the property that, for all \(k \in \ZZ\):
\begin{enumerate}
\item  \(\nu_k \in \{\mu_k \pm 1\}\) if \(\mu_k >0\);
\item \(\nu_k \in \{0,1\}\) if \(\mu_k = 0\);
\item \(\nu_{k+1} = \mu_{k+1}+1\) if and only if \(\nu_k = \mu_k - 1\).
\end{enumerate}
Moreover if \(\mu \in \Omega_{\geq 1}\), then \(\bar{\mathcal{G}}\mu\) is the unique element \(\nu \in \Omega_{\geq 1}\) which satisfies (2a, 2b, 2c) for all \(k \geq 1\), and \(\nu_1 = \mu_1 + 1\).
\end{enumerate}
\end{lemma}

\begin{proof}
Let \(\nu\) be the element defined as in (1a, 1b, 1c) above. 
It is straightforward to see that \(\nu\) exists and is uniquely defined---if \(m\) is maximal such that \(\mu_m > 0\), then we have \(\nu_{m'} = 0\) for all \(m' > m\), \(\nu_m = \mu_m -1\), and rules (1a, 1b, 1c) above iteratively force the choice of \(\nu_{m''}\) for all \(m''< m\), proceeding in the downward direction. 

We now claim that for all \(k \in \ZZ\), we have
\begin{align}
\nu_k = \begin{cases}
\mu_k -1 & \textup{if }\eta_k(\mu) \textup{ is odd};\\
\mu_k +1 & \textup{if }\eta_{k+1}(\mu) \textup{ is odd};\\
0 & \textup{if }\mu_k = 0 \textup{ and }\eta_{k+1}(\mu) \textup{ is even}.
\end{cases}
\label{casesnuparx}
\end{align}
Clearly (\ref{casesnuparx}) holds for \(k \gg 0\), since in this case  \(\mu_k = \nu_k = \eta_{k+1}(\mu) = 0\). Now, going by downward induction on \(k \in \ZZ_{\geq 0}\), we assume that (\ref{casesnuparx}) holds for \(k' = k+1\), and prove that (\ref{casesnuparx}) holds for \(k\).\\

\noindent{\em Case 1. Assume \(\mu_k = 0 \textup{ and }\eta_{k+1}(\mu) \textup{ is even}.\)} Then by (1b) we have \(\nu_k \in \{0,1\}\). Since \(\eta_{k'}(\mu)\) is even, by induction assumption \(\nu_{k'} \neq \mu_{k'} -1\). Then by (1c) we have \(\nu_k \neq \nu_k +1 = 1\), so \(\nu_k = 0\), as desired.\\

\noindent{\em Case 2. Assume \(\eta_k(\mu)\) is odd.} Then \(\mu_k>0\), so \(\nu_k \in \{\mu_k \pm 1\}\) by (1a). 
We have that  \(\eta_{k'}(\mu)\) is even, so by induction assumption \(\nu_{k'} \neq \mu_{k'} -1\). Then by (1c) we have \(\nu_k \neq \nu_k +1 \), so \(\nu_k= \mu_k - 1\), as desired.\\

\noindent{\em Case 3.  Assume \(\eta_{k+1}(\mu)\) is odd.} Then since \(\eta_{k'}(\mu)\) is odd, we have \(\nu_{k'} = \mu_{k'}-1\) by induction assumption. Then by (1c) we have \(\nu_k = \mu_k+1\), as desired.\\

Since (\ref{casesnuparx}) holds for all \(k\), we have \(\nu = \mathcal{H} \mu\). The `moreover' clause in (1) is immediate then as well, completing the proof of (1). To prove (2), note that \(\nu\) satisfies (2a, 2b, 2c) for \(\mu\) if and only if \(\nu' = s(\nu)\) satisfies (1a, 1b, 1c) for \(\mu' = s(\mu)\). Therefore using (1) we have
\begin{align*}
\mathcal{G} \mu = \mathcal{G}s\mu' = s\mathcal{H} \mu' = s \nu' = \nu,
\end{align*}
proving the initial statement in (2). The `moreover' clause in (2) follows directly since \(\vartheta(\textup{add}_0(\mu)) = 1\) in this setting.
\end{proof}

\begin{corollary}\label{pressum}
For all \(\mu \in \Omega\), we have \(|\mu| = |\mathcal{H} \mu| = |\mathcal{G} \mu|\).
\end{corollary}
\begin{proof}
By construction, the maps \(\mathcal{G}\) and \(\mathcal{H}\) either preserve entries, or else raise or lower entries by one. 
One sees from Lemma~\ref{altdesc}(1c, 2c) that raised/lowered entries appear in neighboring pairs, so the overall sum of entries must be preserved.
\end{proof}

\begin{lemma}\label{GHisom}
For \(n \in \ZZ_{\geq 0}\), the maps \(\mathcal{G}, \mathcal{H}\) are mutually inverse on \(\Omega(n)\).
\end{lemma}
\begin{proof}
We have by Corollary~\ref{pressum} that 
\(\mathcal{G}, \mathcal{H}\) restrict to maps on \(\Omega(n)\).
Let \(\mu \in \Omega\), and set \(\nu = \mathcal{H}\mu\), \(\omega = \mathcal{G}\nu\). Applications of \(\mathcal{G}\) and \(\mathcal{H}\) change entries of compositions by at most one, so we have \(\omega_m \in [\mu_m - 2, \mu_m + 2]\) for all \(m\). \\

\noindent{\em Claim I. There exists no \(m \in \ZZ\) such that \(\omega_m = \mu_m + 2\).} Indeed, assume by way of contradiction that there exists an occurrence of such \(m\). We assume moreover that \(m\) is minimal with this property. Then in consideration of Lemma~\ref{altdesc}, we must have \(\nu_m = \mu_m + 1\) and \(\omega_m = \nu_m + 1\). Then by Lemma~\ref{altdesc}(1c, 2c) we have \(\nu_{m-1} \neq \mu_{m-1} + 1\) and \(\omega_{m-1} = \nu_{m-1}-1\). Hence \(\nu_{m-1} > 0\). If \(\mu_{m-1} = 0\), then Lemma~\ref{altdesc}(1b) would force \(\nu_{m-1} = 0\) or \(\nu_{m-1} = 1 = \mu_{m-1} + 1\), but as neither of these are the case, it must be that \(\mu_{m-1} > 0\), and hence \(\nu_{m-1} = \mu_{m-1} -1\) by Lemma~\ref{altdesc}(1a). Then \(\nu_{m-2} = \mu_{m-2} +1\) by Lemma~\ref{altdesc}(1c). Then \(\nu_{m-2} > 0\), so \(\omega_{m-2} \in \{\nu_{m-2} \pm 1\}\) by Lemma~\ref{altdesc}(2a). But since \(\omega_{m-1} = \nu_{m-1} - 1\), it follows from  by Lemma~\ref{altdesc}(2c) that \(\omega_{m-2} = \nu_{m-2} +1\). Therefore
\begin{align*}
\omega_{m-2} = \nu_{m-2} + 1 = \mu_{m-2} + 2,
\end{align*}
contradicting the minimality assumption on \(m\), and proving Claim I.\\

\noindent{\em Claim II. There exists no \(m \in \ZZ\) such that \(\omega_m = \mu_m + 1\).} Indeed, assume by way of contradiction that there exists an occurrence of such \(m\). It follows then that either
\begin{align*}
\nu_m = \mu_m + 1, \;\omega_m = \nu_m, \qquad 
\textup{or}
\qquad
\nu_m = \mu_m, \;\omega_m = \nu_m +1.
\end{align*}
The former case is impossible however, since in this setting we have \(\nu_m >0\), which would imply by Lemma~\ref{altdesc}(2a) that \(\omega_m \in \{\nu_m \pm 1\}\). Thus we assume the latter; \(\nu_m = \mu_m, \;\omega_m = \nu_m +1\). Then \(\mu_m = \nu_m = 0\) by Lemma~\ref{altdesc}(1b). By Lemma~\ref{altdesc}(2c) we have \(\omega_{m-1} = \nu_{m-1} -1\), so \(\nu_{m-1} > 0\). Since \(\nu_m \neq \mu_m -1\), we cannot have that \(\nu_{m-1} = \mu_{m-1} +1\) by Lemma~\ref{altdesc}(1c). Thus either \(\nu_{m-1} = \mu_{m-1} = 0\), or else \(\nu_{m-1} = \mu_{m-1} -1\) by Lemma~\ref{altdesc}(1a,1b). But since \(\nu_{m-1} > 0\), it follows then that \(\nu_{m-1} = \mu_{m-1} -1\), and so \(\nu_{m-2} = \mu_{m-2} + 1\) by Lemma~\ref{altdesc}(1c). Thus \(\nu_{m-2} > 0\), so \(\omega_{m-2} \in \{\nu_{m-2} \pm 1\}\). Since 
\(\omega_{m-1} = \nu_{m-1}-1\), it cannot be that \(\omega_{m-2} = \nu_{m-2} - 1\) by Lemma~\ref{altdesc}(2c). Therefore \(\omega_{m-2} = \nu_{m-2} + 1\), and so we have
\begin{align*}
\omega_{m-2} = \nu_{m-2} + 1 = \mu_{m-2} + 2,
\end{align*}
which contradicts Claim I, and thereby proves Claim II.

Now, by Claim I and II, we have \(\omega_m \geq \mu_m\) for all \(m \in \ZZ\). But then since \(|\mu| = |\omega|\) by Corollary~\ref{pressum}, it must be that \(\omega = \mu\). Thus \(\mathcal{G} \mathcal{H} \mu = \mu\) for all \(\mu \in \Omega\). Therefore
\begin{align*}
\mathcal{H} \mathcal{G} \mu = \mathcal{H} \mathcal{G} ss\mu  = \mathcal{H} s \mathcal{H} s\mu = s\mathcal{G}  \mathcal{H} s\mu = ss \mu = \mu
\end{align*}
for all \(\mu \in \Omega\) as well, completing the proof.
 \end{proof}
 
 \begin{corollary}\label{leftinvcor}
 Restricted to \(\Omega_{\geq 1}(n)\), the map \(\overline{\mathcal{H}}\) is a left inverse to both \(\mathcal{G}\) and \(\overline{\mathcal{G}}\).
 \end{corollary}

 \subsection{The \(\ORD\) and \(\CALC\) functions}\label{RKsecappend}
 Recall the spar tableau, and the \(\ORD\) and \(\CALC\) functions defined in \S\ref{partitioncombsec}.
For \(\mu \in \Parreg\), it will be useful to define a (tableau for a) partition \(\lambda = \CALC_\textup{alt}(\mu)\) as follows. Start by setting \(\lambda = \varnothing\), and then for each \(t = c(\mu), c(\mu)-1, \ldots, 1\), work from left to right and add a box labeled \(t\) at addable positions in \(\lambda\) if there is {\em not} a box labeled \(t\) in the column immediately to the left---starting from the first column if \(c_t(\mu)\) is odd, and starting from the second column if \(c_t(\mu)\) is even. The following lemma is straightforward to check, and we leave it as an exercise for the reader. 

\begin{lemma}\label{altdesccalc}
For \(\mu \in \Parreg\), the partitions (and associated tableau) \(\CALC(\mu)\) and \(\CALC_\textup{alt}(\mu)\) are identical.
\end{lemma}

Let \(\textup{mc}: \Omega_{\geq 1} \to \Par\) be the bijection defined by setting \(\textup{mc}(\mu) = \lambda\), where \(n_k(\lambda) = \mu_k\) for all \(k \in \Z_{>0}\). 
Define operators \(\mathcal{G}_\Par, \overline{\mathcal{G}}_\Par, \overline{\mathcal{H}}_\Par : \Par \to \Par\) by setting
\begin{align*}
\mathcal{G}_\Par = \textup{mc} \circ \mathcal{G} \circ \textup{mc}^{-1};
\qquad
\overline{\mathcal{G}}_\Par= \textup{mc} \circ \overline{\mathcal{G}} \circ \textup{mc}^{-1};
\qquad
\overline{\mathcal{H}}_\Par= \textup{mc} \circ \overline{\mathcal{H}} \circ \textup{mc}^{-1}.
\end{align*}
Note that if \(\lambda = \textup{mc}(\mu)\), then adding an addable box to a row of length \(k\) in \(\lambda\) is tantamount to subtracting one from \(\mu_k\) and adding one to \(\mu_{k+1}\), and removing a removable box from a row of length \(k\) in \(\lambda\) is tantamount to subtracting one from \(\mu_k\) and (when \(k>1\)) adding one to \(\mu_{k-1}\). 
In view of Lemma~\ref{altdesc}, we thus may straightforwardly describe the workings of these operators:
\begin{enumerate}
\item[(GP1)] \(\mathcal{G}_\Par\) acts by adding a box at each addable position in \(\lambda\) if and only if there was not a box added in the column directly to the left, starting from the second column.
\item[(GP2)] \(\overline{\mathcal{G}}_\Par\) acts by adding a box at each addable position in \(\lambda\) if and only if there was not a box added in the column directly to the left, starting from the first column.
\item[(HP)] \(\overline{\mathcal{H}}_\Par\) acts by deleting the rightmost removable box, and then working to the left, deleting each removable box if and only if a box was not deleted in the column directly to the right.
\end{enumerate}
We may record one additional feature of these operators, thanks to Corollary~\ref{leftinvcor}:
\begin{enumerate}
\item[(HG)] The operator \(\overline{\mathcal{H}}\) removes from \(\mathcal{G}(\omega)\) (resp. \(\overline{\mathcal{G}}(\omega)\)) exactly those boxes which are added to \(\omega\) by \(\mathcal{G}\) (resp. \(\overline{\mathcal{G}}\)).
\end{enumerate}

\begin{theorem}\label{appendreefundo}
The functions \(\ORD\colon \Par \to \Parreg\) and \(\CALC\colon \Parreg \to \Par\) are mutually inverse bijections. Moreover, the tableau generated in the construction of \(\CALC(\nu)\) is exactly \(\SPAR(\CALC(\nu))\).
\end{theorem}
\begin{proof}
Let \(\nu \in \Parreg\), and set \(\lambda = \CALC_\textup{alt}(\nu)\). Then from the definition of \(\CALC_\textup{alt}\) and (GP1),(GP2), we have that \(\lambda\) is obtained by iterative applications of \(\mathcal{G}_\Par\), \(\overline{\mathcal{G}}_\Par\) to \(\varnothing\); applying \(\overline{\mathcal{G}}_\Par\) in the \(t\)th step if \(c_t(\nu)\) is odd, applying \(\mathcal{G}_\Par\) instead if \(c_t(\nu)\) is even. On the other hand, 
we have that \(\overline{\mathcal{H}}_\Par^k(\lambda) = \varnothing\) for \(k \gg 0\). It is clear from (HP) that a box is labeled \(t\) in \(\SPAR(\lambda)\) if it is ousted in the \(t\)th application of \(\overline{\mathcal{H}}_\Par\) to \(\lambda\). Hence by (HG), \(t\) appears in the first column of \(\SPAR(\lambda)\) if and only if \(c_t(\nu)\) is odd. Thus the tableau generated in the construction of \(\CALC_\textup{alt}(\nu)\) is exactly \(\SPAR(\lambda)\), and so \(\nu = \ORD(\lambda)\). Then the result follows from Lemma~\ref{altdesccalc}.
\end{proof}

\subsection{Bicompositions}
A {\em bicomposition} \(\bmu = (\mu^{1}, \mu^{0}) \in \Omega^2\) is an ordered pair of compositions. We write \(\bOm\) for the set of all bicompositions. 
Recall that we write the complementary element of \(i \in \ZZ_2\) with a hat, i.e. \(\hat 0 = 1\), \(\hat 1 = 0\). 

Let \(i \in \ZZ_2, k \in \ZZ, a, b,n\in \ZZ_{\geq 0}\). We define some important subsets of \(\bOm\) as follows:
\begin{align*}
\bOm_{\geq k} &= \{ \bmu \in \bOm \mid \mu^{0}, \mu^{1} \in \Omega_{\geq k}\}\\
\bOm_{\geq 0,i} &= \{ \bmu \in \bOm \mid \mu^{0}, \mu^{1} \in \Omega_{\geq 0}, \mu_0^i = 1, \mu_0^{\hat i} = 0\}\\
\bOm_{\geq 0, \bullet} &= \bOm_{\geq 0,0}  \sqcup \bOm_{\geq 0,1}\\
\bOm(a,b) &= \{\bmu \in \bOm \mid |\mu^0| = a, |\mu^1| = b\}\\
\bOm(n) &= \{\bmu \in \bOm \mid |\mu^0| + |\mu^1| = n\},
\end{align*}
and we will also mix decorations, writing \(\bOm_{\geq 0, \bullet}(a,b)\) for instance, to be interpreted in the obvious way.

\subsubsection{Functions of bicompositions}
Abusing notation, we extend the functions \(\mathcal{G}, \mathcal{H}, \bar{\mathcal{H}}\) to bicompositions \(\bOm\) by acting diagonally, so that, for instance \(\mathcal{G} \bmu = \bnu\), where \(\nu^i = \mathcal{G} \mu^i\). For \(i \in \ZZ_2\), we further define the function \(\bar{\mathcal{G}}^i: \bOm \to \bOm\) by setting \(\bar{\mathcal{G}}^i \bmu = \bnu\), where
\begin{align*}
\nu^i = \bar{\mathcal{G}}\mu^i \qquad 
\textup{and}
\qquad
\nu^{\hat i} = \mathcal{G} \mu^{\hat i}.
\end{align*}

\subsection{Binary words}
We call a binary string \(\beps = \varepsilon_1 \cdots \varepsilon_n \in \ZZ_2^n\) a  {\em word}, and the entries of \(\beps\) {\em letters}.  We write \(\beps^\textup{rev}:=\eps_n \cdots \eps_1\).
Let \(\Theta(n)\) be the set of all words of length \(n\). A {\em subword} of \(\beps\) is a substring \(\beps^{[s,t]}:=\varepsilon_s \cdots  \varepsilon_t\) for some \(1 \leq s \leq t \leq n\). For \(i \in \ZZ_2\), we write \(\#_i\beps\) for the number of entries in \(\beps\) equal to \(i\). We write
\begin{align*}
\Theta(a,b) = \{ \beps \in \Theta \mid \#_0 \beps = a, \#_1 \beps = b\} \subseteq \Theta(a+b).
\end{align*} 

\subsubsection{Peaks}
We say \(k \in [1,n]\) is a {\em left \(i\)-peak} for \(\beps \in \Theta(n)\) provided
\(
\#_i\beps^{[j,k]}> \#_{\hat i}\beps^{[j,k]}
\)
for all \(j \in [1,k]\). We say \(k\) is a {\em right \(i\)-peak} for \(\beps\) provided
\(
\#_i\beps^{[k,j]}> \#_{\hat i}\beps^{[k,j]}
\)
for all \(j \in [k,n]\). We write \(\textup{P}^{\textup{L}}_i(\beps)\) and \(\textup{P}^{\textup{R}}_i(\beps)\) for the set of left and right \(i\)-peaks of \(\beps\), respectively. If \(\eps_k = i\) but \(k\) is not a left (resp. right) \(i\)-peak for \(\beps\), then we say \(k\) is a {\em left (resp. right) \(i\)-nonpeak} for \(\beps\).
We write \(\textup{NP}^{\textup{L}}_i(\beps)\) and \(\textup{NP}^{\textup{R}}_i(\beps)\) for the set of left and right \(i\)-nonpeaks of \(\beps\), respectively.
It follows then that
\begin{align*}
\textup{P}^{\textup{L}}_i(\beps) \sqcup \textup{P}^{\textup{L}}_{\hat i}(\beps) \sqcup \textup{NP}^{\textup{L}}_i(\beps) \sqcup \textup{NP}^{\textup{L}}_{\hat i}(\beps) = [1,n],
\end{align*}
and similarly for the right.

\begin{lemma}\label{NPbij}
Let \(\beps \in \Theta\), \(i \in \ZZ_2\).
There is a bijection
\begin{align*}
\zeta: \NPL_{\hat i}(\beps) \to \NPR_{ i}(\beps).
\end{align*}
given by sending \(k \in \NPL_{\hat i}(\beps)\) to the unique \(\zeta(k) \in [1,k-1]\) such that \(\#_i \beps^{[\zeta(k),k]} = \#_{\hat i} \beps^{[\zeta(k),k]}\), \(\#_i \beps^{[t,k]} < \#_{\hat i} \beps^{[t,k]}\) for all \(\zeta(k) < t \leq k\), and \(\#_i \beps^{[\zeta(k), u]} > \#_{\hat i}\beps^{[\zeta(k),u]}\) for all \(\zeta(k) \leq u < k\).
\end{lemma}

\begin{proof}
Let \(k \in \NPL_{\hat i}(\beps)\). Then by \(k\)'s nonpeakness there exists a maximal \(k' \in [1,k-1]\) such that \(\#_i \beps^{[k',k]} \geq \#_{\hat i} \beps^{[k',k]}\). Then by assumption \(\#_i \beps^{[t,k]} < \#_{\hat i} \beps^{[t,k]}\) for all \(k' < t \leq k\). Note then that \(\eps_{k'} = i\) and \(\#_i \beps^{[k',k]} = \#_{\hat i} \beps^{[k',k]}\). If \(\#_i \beps^{[k',u]} \leq \#_{\hat i} \beps^{[k',u]}\) for some \(k' < u < k\), then 
\begin{align*}
\#_i \beps^{[k',k]} = \#_i \beps^{[k',u]} + \#_i \beps^{[u+1,k]} < \#_{\hat i} \beps^{[k',u]} + \#_{\hat i} \beps^{[u+1,k]} = \#_{\hat i} \beps^{[k',k]},
\end{align*}
a contradiction. Therefore \(\#_i \beps^{[k',u]} > \#_{\hat i} \beps^{[k',u]}\) for all \(k' \leq u < k\). Thus we may set \(\zeta(k) = k'\), so the map \(\zeta\) of the lemma statement is well-defined. There is plainly an inverse map constructed  symmetrically, so \(\zeta\) is a bijection.
\end{proof}

\subsection{Functions on binary words}
Fix \(n \in \ZZ_{\geq 0}\). We define a map \(\mathcal{A}^i : \Theta(n) \to \Theta(n+1)\) which sends \(\beps \in \Theta(n)\) to \(\mathcal{A}^i\beps \in \Theta(n+1)\) by appending an \(i\) to the end of \(\beps\). We define a map \(\mathcal{R}: \Theta(n+1) \to \Theta(n)\) which sends \(\beps \in \Theta(n+1)\) to \(\mathcal{R}\beps \in \Theta_{n}\) by deleting the last letter in \(\beps\).

\begin{lemma}\label{PtoR}
Let \(n \in \ZZ_{\geq 1}\), \(i \in \ZZ_2\). Let \(\beps \in \Theta(n)\). Then
\begin{align*}
\PR_i (\mathcal{R} \beps) = 
\begin{cases}
\PR_i (\beps) \backslash \{n\} & \textup{if } \eps_n = i;\\
\PR_i (\beps) \sqcup \{\zeta(n)\} & \textup{if } n \in \NPL_{\hat i}(\beps);\\
\PR_i (\beps) = \varnothing & \textup{if } n \in \PL_{\hat i}(\beps).\\
\end{cases}
\end{align*}
\end{lemma}
\begin{proof}
This is a routine definition chase, occasionally invoking Lemma~\ref{NPbij}, so we omit the proof here.
 Full details can be found in the \texttt{arXiv} version of the paper as explained in \S\ref{SS:ArxivVersion}.
 \begin{answer}
It follows from definitions that 
\begin{align}
\PR_i(\beps) \cap [1,n-1] \subseteq \PR_i(\mathcal{R}\beps).\label{prsubset}
\end{align}
We consider now the cases \(\eps_n = i\) or \(\eps_n = \hat i\), the latter of which we subdivide into the two cases \(n \in \NPL_{\hat i}(\beps)\) and \(n \in \PL_{\hat i}(\beps)\). \\

\noindent{\em Case 1. Assume \(\eps_n = i\).} Then \(\PR_i(\beps) = (\PR_i(\beps) \cap [1,n-1]) \cup \{n\}\), so by (\ref{prsubset}) we have \(\PR_i(\beps) \backslash \{n\} \subseteq \PR_i(\mathcal{R}\beps)\). On the other hand, if \(q \in \PR_i(\mathcal{R}\beps)\), then \(q \in [1,n-1]\), and \(\#_i\beps^{[q,t]} > \#_{\hat i}\beps^{[q,t]}\) for all \(t \in [q,n-1]\), and therefore 
\begin{align*}
\#_i \beps^{[q,n]} = \#_i \beps^{[q,n-1]} + 1 >  \#_{\hat i} \beps^{[q,n-1]} =  \#_{\hat i} \beps^{[q,n]}
\end{align*}
so \(q \in \PR_i(\beps) \backslash \{n\}\). Thus \(\PR_i (\mathcal{R} \beps) = \PR_i (\beps) \backslash \{n\} \).\\

\noindent{\em Case 2.  Assume \(n \in \NPL_{\hat i}(\beps)\).} Then \(n \notin \PR_i(\beps)\), so \(\PR_i(\beps) \subseteq \PR_i(\mathcal{R}\beps)\) by (\ref{prsubset}). By Lemma~\ref{NPbij} we have \(q:=\zeta(n) \in \NPR_i(\beps)\), where \(\#_i \beps^{[q,n]} = \#_{\hat i} \beps^{[q,n]}\), \(\#_i \beps^{[t,n]} < \#_{\hat i} \beps^{[t,n]}\) for all \(q < t \leq n\), and \(\#_i \beps^{[q, u]} > \#_{\hat i}\beps^{[q,u]}\) for all \(q \leq u <n\). Then it follows that \(q \in \PR_i(\mathcal{R} \beps)\), so \(\PR_i(\beps) \sqcup \{\zeta(n)\} \subseteq \PR_i(\mathcal{R}\beps)\). 

On the other hand, assume \(v \in \PR_i(\mathcal{R}\beps) \backslash \PR_i(\beps)\). By Lemma~\ref{NPbij}, \(v<\zeta^{-1}(v) \in \NPL_{\hat i}(\beps)\) is such that \(\#_i \beps^{[v, \zeta^{-1}(v)]} = \#_{\hat i} \beps^{[v, \zeta^{-1}(v)]}\),  \(\#_i \beps^{[t,\zeta^{-1}(v)]} < \#_{\hat i} \beps^{[t,\zeta^{-1}(v)]}\) for all \(v< t \leq \zeta^{-1}(v)\), and \(\#_i \beps^{[v, u]} > \#_{\hat i}\beps^{[v,u]}\) for all \(v\leq u < \zeta^{-1}(v)\).
Since \(v \in \PR_i(\mathcal{R}\beps)\), we have \(\#_i\beps^{[v,t]} > \#_{\hat i} \beps^{[v,t]}\) for \(t \in [v,n-1]\), which implies that \(\zeta^{-1}(v) = n\), and therefore \(v = \zeta(n)\). 
Therefore we have \(\PR_i(\mathcal{R} \beps) \subseteq \PR_i(\beps) \sqcup \{\zeta(n)\}\), so \(\PR_i(\mathcal{R} \beps) = \PR_i(\beps) \sqcup \{\zeta(n)\}\).\\

\noindent{\em Case 3. Assume \(n \in \PL_{\hat i}(\beps)\).} Then for all \(k \in [1,n]\) we have \(\#_i \beps^{[k,n]} < \#_{\hat i} \beps^{[k,n]}\), and thus \(\PR_i(\beps) = \varnothing\). Therefore, for \(t \in [1,n-1]\), we have 
\begin{align*}
\#_i \beps^{[t,n-1]} = \#_i \beps^{[t,n]} <  \#_{\hat i} \beps^{[t,n]} =  \#_{\hat i} \beps^{[t,n-1]} +1,
\end{align*}
which implies that \(\#_i \beps^{[t,n-1]} \leq \#_{\hat i} \beps^{[t,n-1]}\), so \(t \notin \PR_i(\mathcal{R}\beps)\), and thus \(\PR_i(\mathcal{R}\beps) = \varnothing\).
\end{answer}
\end{proof}

\subsubsection{The \(\mathcal{T}_i\) map}
Now we define a function \(\mathcal{T}_i: \Theta(n) \to \Theta(n)\) by setting
\(
\mathcal{T}_i\beps = \brho
\),
where 
\begin{align}
\rho_r = \begin{cases}
\eps_r & \textup{if } r \in \textup{P}^\textup{L}_{\hat i}(\beps);\\
\hat{\eps}_r & \textup{otherwise}.
\end{cases}\label{Tmapdef}
\end{align}

\begin{lemma}
Let \(n \in \ZZ_{\geq 1}\). We have \(\mathcal{R} \circ \mathcal{T}_i = \mathcal{T}_i \circ \mathcal{R}\) as functions on \(\Theta(n)\).
\end{lemma}
\begin{proof}
Let \(\beps \in \Theta(n)\). For \(j \in [1,n-1]\), we have \(j \in \PL_{\hat i}(\beps)\) if and only if \(j \in \PL_{\hat i}( \mathcal{R} \beps)\). From this it immediately follows that \(\mathcal{R}  \mathcal{T}_i\beps = \mathcal{T}_i  \mathcal{R} \beps\).
\end{proof}

\begin{lemma}\label{countis}
For \(\beps \in \Theta\), we have \(\#_i \mathcal{T}_i \beps = \#_i \beps - \# \PR_i(\beps)\).
\end{lemma}
\begin{proof}
By Lemma~\ref{NPbij} and the definition of the map \(\mathcal{T}_i\), we have
\begin{align*}
\#_i \beps - \# \PR_i(\beps) = \#\NPR_i(\beps) = \#\NPL_{\hat i}(\beps) = \#_i \mathcal{T}_i \beps,
\end{align*}
which yields the result.
\end{proof}

\begin{lemma}\label{TsTerm}
For \(\beps \in \Theta(n)\), we have \(\mathcal{T}_i^k\beps = {\hat i}^n\) for \(k \geq n\). 
\end{lemma}
\begin{proof}
Let \(\beps \in \Theta(n)\). If \(\beps = \hat{i}^n\), then by definition of the map \(\mathcal{T}_i\) we have \(\mathcal{T}_i^k\beps = \hat{i}^n\) for all \(k \in \ZZ_{\geq 0}\), so the claim follows. Otherwise, let \(m\) be minimal such that \(\eps_m = i\). Setting \(\brho = \mathcal{T}_i \beps\), we have \([1,m-1] \in \PL_{\hat i}(\beps)\), so \(\rho_1 = \cdots = \rho_m = \hat i\). Thus each application of \(\mathcal{T}_i\) moves the first instance of \(i\) to strictly later in the string, so by the \(n\)th application of \(\mathcal{T}_i\) there can be no more \(i\)'s. Therefore  \(\mathcal{T}_i^n \beps = {\hat i}^n\), and the claim follows.
\end{proof}

\begin{lemma}\label{PRL}
Let \(\beps \in \Theta\), \(i \in \ZZ_2\). Then \(\PR_{\hat i}(\mathcal{T}_i \beps) = \PL_{\hat i}(\beps) \sqcup \PR_i(\beps)\), and \(j<k\) for all \(j \in \PL_{\hat i}(\beps)\), \(k \in \PR_i(\beps)\).
\end{lemma}
\begin{proof}
This is another routine definition chase making use of Lemma~\ref{NPbij}, so we omit the proof here.
 Full details can be found in the \texttt{arXiv} version of the paper as explained in \S\ref{SS:ArxivVersion}.
 \begin{answer}
The secondary statement is immediate from the definitions of the sets \(\PL_{\hat i}(\beps)\) and \( \PR_i(\beps)\).
We prove the primary statement. Set \(\brho = \mathcal{T}_i\beps\).

We first show \(\PR_{\hat i}(\brho) \subseteq \PL_{\hat i}(\beps) \sqcup \PR_i(\beps)\).
Let \(j \in \PR_{\hat i}(\brho)\). Then \(\rho_j = \hat i\), so either \(j \in \PL_{\hat i}(\beps)\) or \(\eps_j = i\). In the former case, we are done, so assume \(\eps_j = i\). Assume by way of contradiction that \(j \in \NPR_i(\beps)\). Then by Lemma~\ref{NPbij} there exists  \(k = \zeta^{-1}(j) \in \NPL_{\hat i}\) with \(k \in [j+1,n]\), \(\#_i \beps^{[j,k]} = \#_{\hat i} \beps^{[j,k]}\), and \(\#_i \beps^{[j,u]} > \#_{\hat i} \beps^{[j,u]}\) for all \(u \in [j,k-1]\). Note then that \(u \notin \PL_{\hat i}(\beps)\) for \(u \in [j,k-1]\), and \(j,k \notin \PL_{\hat i}(\beps)\), so therefore we have
\begin{align*}
\#_{\hat i} \brho^{[j,k]} = \#_i \beps^{[j,k]} = \#_{\hat i} \beps^{[j,k]} = \#_i \brho^{[j,k]},
\end{align*}
which contradicts \(j \in \PR_{\hat i}(\brho)\). Thus \(j \in \PR_i(\beps)\), and so  \(\PR_{\hat i}(\brho) \subseteq \PL_{\hat i}(\beps) \sqcup \PR_i(\beps)\).

Next we show \(\PL_{\hat i}(\beps) \subseteq \PR_{\hat i}(\brho)\).
Assume by way of contradiction that \(\PL_{\hat i}(\beps) \cap \NPR_{\hat i}(\brho) \neq \varnothing\). Let \(j\) be the maximal element of this set. 
Then we have \(\eps_j = \rho_j =  \hat i\). Since \(j \in \NPR_{\hat i}(\brho)\), by Lemma~\ref{NPbij} there exists \(k = \zeta^{-1}(j) \in \NPL_{i}(\brho)\) with \(k \in [j+1,n]\) such that \(\#_{\hat i} \brho^{[j,k]} = \#_{ i} \brho^{[j,k]}\), \(\#_{\hat i} \brho^{[t,k]} < \#_{ i} \brho^{[t,k]}\) for all \(t \in [j+1,k]\), 
We claim that, for each \(t \in [j+1,k]\), we have \(\rho_t = \hat \eps_t\). Indeed, if it were the case that \(\rho_t = \eps_t\), then we would have \(t \in \PL_{\hat i}(\beps)\) and \(\rho_t = \eps_t = \hat i\), so by maximality of \(j\) we would have that \(t \in \PR_{\hat i}(\brho)\). But since \(\#_{\hat i} \brho^{[t,k]} < \#_{ i} \brho^{[t,k]}\), this cannot be the case.
Thus for all \(t \in [j+1,k]\), we have \(\rho_t = \hat{\eps}_t\), so
\begin{align}\label{tkeq}
\#_i \beps^{[t,k]}  = \#_{\hat i} \brho^{[t,k]} < \#_i \brho^{[t,k]} = \#_{\hat i} \beps^{[t,k]}.
\end{align}
Moreover, for \(u \in [1,j]\), by (\ref{tkeq}) and the fact that \(j \in \PL_{\hat i}(\beps)\), we have
\begin{align*}
\#_i \beps^{[u,k]}  = \#_i \beps^{[u,j]} + \#_i \beps^{[j+1,k]} <
\#_{\hat i} \beps^{[u,j]} + \#_{\hat i} \beps^{[j+1,k]}  = \#_{\hat i} \beps^{[u,k]} 
\end{align*}
Therefore \(k \in \PL_{\hat i}(\beps)\). But then we should have \(\rho_k = \hat i\), which contradicts \(k \in \NPL_i(\brho)\). Thus \(\PL_{\hat i}(\beps) \cap \NPR_{\hat i}(\brho) = \varnothing\).
Since \(\PL_{\hat i}(\beps) \subseteq \PR_{\hat i}(\brho) \sqcup \NPR_{\hat i}(\brho)\) by definition of the map \(\mathcal{T}^i\), we have then that  \(\PL_{\hat i}(\beps) \subseteq \PR_{\hat i}(\brho)\).

Finally we show \(\PR_i(\beps) \subseteq \PR_{\hat i}(\brho)\).
Assume \(k \in \PR_i(\beps)\). It follows from definitions that \(j<k\) for any \(j \in \PL_{\hat i}(\beps)\). Therefore for all \(u \in [k,n]\), we have \(\rho_u = \hat{\eps}_u\). Therefore we have
\begin{align*}
\#_{\hat i} \brho^{[k,u]} = \#_i \beps^{[k,u]} > \#_{\hat i} \beps^{[k,u]} = \#_i \brho^{[k,u]},
\end{align*}
and so \(k \in \PR_{\hat i}(\brho)\).
Therefore \(\PL_{\hat i}(\beps) \sqcup \PR_i(\beps) \subseteq\PR_{\hat i}(\brho)\), completing the proof.
\end{answer}
\end{proof}

\subsubsection{The map \(\mathcal{U}_{i,m}\)}
Let \(\beps \in \Theta(n)\), \(m \in \ZZ_{\geq 0}\). Set \(\PR_{\hat i}(\beps)_{-m}\) to be the subset of \(\PR_{\hat i}(\beps)\) given by deleting the \(m\) largest elements from \(\PR_{\hat i}(\beps)\). We define a function \(\mathcal{U}_{i, m}:\Theta(n) \to \Theta(n)\) by setting \(\mathcal{U}_{i,m}(\beps) = \brho\), where
\begin{align}
\rho_r = \begin{cases}
\eps_r & \textup{if } r \in \PR_{\hat i}(\beps)_{-m}\\
\hat{\eps}_r & \textup{otherwise}.
\end{cases}\label{Umapdef}
\end{align}
We write \(\mathcal{U}_i := \mathcal{U}_{i,0}\).

\begin{lemma}\label{Uundo}
Let \(\beps \in \Theta\).  Then \(\mathcal{U}_{i,\#\PR_i(\beps)}  \mathcal{T}_i \beps = \beps\).
\end{lemma}

\begin{proof}
By Lemma~\ref{PRL}, we have \(\PR_{\hat i}(\mathcal{T}_i \beps)_{-\#\PR_i(\beps)} = \PL_{\hat i}(\beps)\), and so it immediately follows from (\ref{Tmapdef}, \ref{Umapdef}) that \(\mathcal{U}_{i,\#\PR_i(\beps)}  \mathcal{T}_i\beps = \beps\).
\end{proof}

\subsection{From binary sequences to bicompositions}

We now define a function \(\mathcal{J}: \Theta \to \bOm_{\geq 1}\) by setting \(\mathcal{J}\beps = \bmu\), where
\begin{align}\label{Jdef}
\mu^i_k = \#\textup{P}^{\textup{R}}_i(\mathcal{T}_i^{k-1} \beps)
\end{align}
for \(i \in \ZZ_2\), \(k \in \ZZ_{\geq 1}\). This map is well defined, since for \(k \gg 0\), we have \(\mu_k^i = \#\PR_i(\mathcal{T}_i^k\beps) = \#\PR_i(\hat i^n) = 0\) by Lemma~\ref{TsTerm}. In fact, it follows from Lemma~\ref{countis} that \(|\mu^i| = \#_i \beps\), so \(\mathcal{J}\) restricts to a map 
\begin{align*}
\mathcal{J}: \Theta(a,b) \to \bOm_{\geq 1}(a,b).
\end{align*}

\begin{lemma}\label{commaps}
Let \(n \in \ZZ_{\geq 0}, i \in \ZZ_2\). We have equalities
\begin{align}
\label{comsq}
\bar{\mathcal{H}} \mathcal{J}= \mathcal{J} \mathcal{R}
\qquad
\textup{and}
\qquad
\bar{\mathcal{G}}^i \mathcal{J}= \mathcal{J} \mathcal{A}^i
\end{align}
of functions \(\Theta(n+1) \to \bOm_{\geq 1}(n)\), and  \(\Theta(n) \to \bOm_{\geq 1}(n+1)\) respectively.
\end{lemma}
\begin{proof}
We prove the left equality first.
Let \(\beps \in \Theta(n+1)\). 
Let \(\bmu = \mathcal{J} \beps\), and \(\bnu = \mathcal{J} \mathcal{R}\beps\). Let \(i \in \ZZ_2\), \(k \in \ZZ_{\geq 1}\). Set \(\brho = \mathcal{T}_i^{k-1}\beps\) and \(\bar \brho = \mathcal{T}_i^k \beps = \mathcal{T}_i \brho\). Then \(\mu_k^i = \#\PR_i(\brho)\), \(\mu_{k+1}^i = \#\PR_i(\bar \brho)\), \(\nu_k^i = \#\PR_i(\mathcal{T}_i^{k-1}\mathcal{R}\beps)  = \#\PR_i(\mathcal{R}\brho)\), and similarly \(\nu_{k+1}^i = \#\PR_i(\mathcal{R}\bar \brho)\). Then it follows from Lemma~\ref{PtoR} that 
\begin{align}
\nu_k^i =
\begin{cases}
\mu_k^i-1 & \textup{if } \rho_n = i;\\
\mu_k^i+1 & \textup{if } n \in \NPL_{\hat i}(\brho);\\
  \mu_k^i= 0 & \textup{if } n \in \PL_{\hat i}(\brho),\\
\end{cases}
\qquad
\textup{and}
\qquad
\nu_{k+1}^i =
\begin{cases}
\mu_{k+1}^i-1 & \textup{if } \bar\rho_n = i;\\
\mu_{k+1}^i+1 & \textup{if } n \in \NPL_{\hat i}(\bar \brho);\\
  \mu_{k+1}^i= 0 & \textup{if } n \in \PL_{\hat i}(\bar \brho).\\
\end{cases}\label{munucases}
\end{align}
We have the following by the definition of the map \(\mathcal{T}_i\).
\begin{itemize}
\item If \(\bar \rho_n = i\), then \(n \in \NPL_{\hat i}(\brho)\).
\item If \(n \in \NPL_{\hat i}(\bar \brho)\), then \(\rho_n = i\) or \(n \in \PL_{\hat i}(\bar \brho)\).
\item If \( n \in \PL_{\hat i}(\bar \brho)\), then \(\rho_n = i\) or \(n \in \PL_{\hat i}(\bar \brho)\).
\end{itemize}
By (\ref{munucases}) we must have \(\nu^i_{k+1} = \mu^i_{k+1} \pm 1\), or else \(\nu^i_{k+1} = \mu^i_{k+1} = 0\), and in view of (\ref{munucases}) and the above bulleted facts it follows that:
\begin{itemize}
\item If \(\nu^i_{k+1} = \mu^i_{k+1} - 1\), then \(\nu^i_k = \mu_k + 1\).
\item If \(\nu^i_{k+1} = \mu^i_{k+1} + 1\), then  \(\nu^i_k = 0\) if \(\mu^i_k = 0\), and otherwise \(\nu^i_k = \mu^i_k -1\).
\item If \(\nu^i_{k+1} = \mu^i_{k+1} =0\), then \(\nu^i_k = 0\) if \(\mu^i_k = 0\), and otherwise \(\nu^i_k = \mu^i_k -1\).
\end{itemize}
But then this situation is logically equivalent to (1a, 1b, 1c) of Lemma~\ref{altdesc} holding for \(\nu^i\) and \(\mu^i\) and all \(k \in \ZZ_{\geq 1}\), so it immediately follows from that Lemma~\ref{altdesc}(1) that \(\nu^i = \bar{\mathcal{H}}\mu^i\), and thus \(\bar{\mathcal{H}} \bmu = \bnu\), verifying the left side of (\ref{comsq}). 

Now we work on the right side of (\ref{comsq}).
Let \(\bgamma \in \Theta(n-1)\), and set \(\beps = \mathcal{A}^j \bgamma\). Then \(\eps_n = j\) and \(\bgamma = \mathcal{R}\beps\). Set \(\bmu = \mathcal{J} \beps\), and \(\bnu = \mathcal{J} \mathcal{R}\beps = \mathcal{J} \bgamma\). Then by (\ref{munucases}) we have, for \(i \in \ZZ_2\), that 
\begin{align*}
(\mathcal{H}\bmu)^i_1=(\bar{\mathcal{H}}\bmu)^i_1=(\bar{\mathcal{H}}\mathcal{J} \beps)^i_1=(\mathcal{JR\beps})^i_1 = \nu_1^i = \mu_1^i -1 \textup{ if }i=j \qquad \textup{and} \qquad (\mathcal{H}\bmu)^i_1 =\nu_1^i \in \{\mu_1^i+1, \mu_1^i\} \textup{ if } i = \hat j. 
\end{align*}
Therefore by (\ref{Hcases}) we have that 
\(\eta^i_1(\bmu)\) is odd if \(i = j\) and even if \(i = \hat j\), and so 
\begin{align*}
(\mathcal{H}\bmu)^j_0 = \mu_0^j + 1 = 1 \qquad \textup{and} \qquad (\mathcal{H}\bmu)^{\hat j}_0 = \mu_0^{\hat j} = 0
\end{align*}
by (\ref{Hcases}). It follows then that \(\mathcal{H}\bmu = \textup{add}^j_0 \circ \textup{del}_0  \circ \mathcal{H}\bmu\), and so we have
\(
\bar{\mathcal{G}}^j\bar{\mathcal{H}}\bmu  =\mathcal{G}\mathcal{H} \bmu.
\)
Therefore, applying the left side of (\ref{comsq}) and Lemma~\ref{GHisom}, we have
\begin{align*}
\bar{\mathcal{G}}^j  \mathcal{J} \bgamma =
\bar{\mathcal{G}}^j  \mathcal{J} \mathcal{R} \mathcal{A}^j \bgamma =
\bar{\mathcal{G}}^j \bar{\mathcal{H}} \mathcal{J} \mathcal{A}^j \bgamma = \bar{\mathcal{G}}^j \bar{\mathcal{H}} \bmu = \mathcal{G}\mathcal{H} \bmu = \bmu = \mathcal{J}\mathcal{A}^j \bgamma
\end{align*}
as desired.
\end{proof}

\subsection{From bicompositions to binary sequences}
For \(i \in \ZZ_2\), define a map \(\mathcal{K}^i : \bOm_{\geq 1}(n) \to \Theta(n)\) as follows. If \(\mu^i = \varnothing\), set \(\mathcal{K}^i(\bmu) = \hat i^n\). Otherwise, let \(m\) be maximal such that \(\mu^i_m > 0\), and define
\begin{align}\label{Kidef}
\mathcal{K}^i(\bmu) = \mathcal{U}_{i,\mu^i_1} \cdots \mathcal{U}_{i,\mu^i_m}(\hat i^n).
\end{align}

\subsection{MV bicompositions}\label{secmvbicomp}
Let \(\bmu = (\mu^1, \mu^0) \in \bOm\), \(i \in \ZZ_2\) and \(m \in \Z_{\geq 1}\). 
We define points \(\{ \beta^i_m(\bmu) \mid i \in \ZZ_2, m \in \ZZ_{\geq 0}\}\) in the lattice \(\ZZ_{\geq 0}I\):
\begin{align*}
\beta_m^i (\bmu) = \sum_{k = 1}^m \mu_k((k-1)\delta + \alpha_i) = (\mu_1 + 2 \mu_2 + \cdots + m \mu_m)\alpha_i + (\mu_2 + 2 \mu_3 + \cdots +(m-1)\mu_m) \alpha_{\hat i}.
\end{align*}
We define \(P(\bmu)\) to be the convex hull of the points \(\{ \beta^i_m(\bmu) \mid i \in \ZZ_2, m \in \ZZ_{\geq 0}\}\). 
(see Example~\ref{bicomgeomex}).
We further define integers:
\begin{align}
c^i_m(\bmu) := \sum_{k=1}^m k(\mu_{k+1}^{\hat i} - \mu^i_k) 
\qquad
d^i_m(\bmu) := m \mu_m^{\hat i} + (m+1)\mu_{m+1}^{\hat i} + \sum_{k=1}^{m-1} k(\mu_k^{\hat i} - \mu_{k+1}^i) 
\label{cdef}
\end{align}
so that
\begin{align*}
\beta^{\hat i}_{m+1}(\bmu) - \beta^{i}_m(\bmu):=  c_m^i(\bmu)\alpha_i + d_m^i(\bmu)\alpha_{\hat i}
\end{align*}
for all \(i \in \ZZ_2, m \in \ZZ_{\geq 0}\).
It will be convenient to record some equalities for future use:
\begin{align}
d_m^i(\bmu) - c_m^i(\bmu) &= (\mu_2^i + \cdots + \mu_m^i) + (\mu_1^{\hat i} + \cdots + \mu_{m+1}^{\hat i})\label{dcdiff}\\
c_{m}^{i}(\bmu) - c_{m-1}^{i}(\bmu)&=  m(\mu_{m+1}^{\hat{i}} - \mu_m^{i}) \label{downonec}\\
d_{m}^{i}(\bmu) - d_{m-1}^{i}(\bmu)&=   (m-1)\mu^{i}_m + (m+1)\mu^{\hat{i}}_{m+1}\label{downoned}\\
c_{m}^{i}(\bmu) + d_{m}^{\hat{i}}(\bmu)&=  m\mu_{m+1}^{i} + (m+1)\mu_{m+1}^{\hat{i}}\label{ctodsame}\\
c_{m}^{i}(\bmu) + d_{m-1}^{\hat{i}}(\bmu)&=  (m-1)\mu_m^{\hat{i}} + m\mu_{m+1}^{\hat{i}}\label{ctoddown}\\
c^{\hat{i}}_{m-1}(\bmu) + d_{m}^{i}(\bmu) &= m\mu^{\hat{i}}_m + (m+1)\mu^{\hat{i}}_{m+1}\label{dtocdown}
\end{align}

Finally, we define \(\eta^i_m(\bmu)\) to be smallest non-negative integer \(t\) such that \(\mu^i_{m+t} = 0\), and we set \(\vartheta^i_m(\bmu)\) to be the smallest non-negative integer \(t\) such that \(\mu^i_{m-t} = 0\).
We will occasionally write \(\beta^i_m, c^i_m, d^i_m, \eta^i_m, \vartheta^i_m\) when \(\bmu\) is clear from context.

\begin{definition}\label{mvbicompdef}
Let \(\bmu = (\mu^{1}, \mu^{0}) \in \bOm\). Then we say \(\bmu\) is  {\em MV} provided that for all \(m \in \ZZ_{>0}\) we have
\begin{align}
0 \in \{c_m^1(\bmu), c_m^0(\bmu) \} \subseteq \ZZ_{\leq 0}.
\end{align}
We write \(\bOm^{\textup{MV}}\) for the set of all MV bicompositions. We say that \(\bmu \in \bOm^{\textup{MV}}\) is {\em nontrivial} if there exists some \(k > 0\), \(i \in \ZZ_2\) with \(\mu_k^i > 0\). For \(j \in \ZZ_2\) we define subsets
\begin{align*}
\bOm_{\geq 1,+}^{\textup{MV}} &= \{\bmu \in \bOm^{\textup{MV}} \mid \bmu \textup{ is nontrivial and } \mu^i_m = 0 \textup{ for all } i \in \ZZ_2, m \leq 0\}\\
\bOm_{\geq 0, j}^{\textup{MV}} &= \{\bmu \in \bOm^{\textup{MV}} \mid \mu^j_0 = 1, \mu^{\hat j}_0 = 0, \mu^i_m = 0 \textup{ for all } i \in \ZZ_2, m < 0\}\\
\bOm_{\geq 0, \bullet}^{\textup{MV}} &= \{\bmu \in \bOm^{\textup{MV}} \mid \mu^0_0 + \mu^1_0 = 1, \mu^i_m = 0 \textup{ for all } i \in \ZZ_2, m < 0\} = \bOm_{\geq 0, 0}^{\textup{MV}}  \sqcup \bOm_{\geq 0, 1}^{\textup{MV}}. 
\end{align*}
\end{definition}

\begin{lemma}\label{mixedpar}
Let \(\bmu \in \bOm^{\textup{MV}}\) be nontrivial. Then \(|\eta^1_1(\bmu) - \eta^0_1(\bmu)| = 1\).
\end{lemma}
\begin{proof}
Let \(i \in \ZZ_2\) be such that \(s = \eta^i_1(\bmu) \leq \eta^{\hat i}_1(\bmu) = t\). Since \(\bmu\) is nontrivial, there exists some \(j \in \ZZ_2, r \geq 0\) such that 
\begin{align*}
\mu^i_1 = \cdots =\mu^i_r = 0, \qquad \mu^{\hat i}_1 = \cdots = \mu^{\hat i}_r = 0, \qquad \mu^j_{r+1} > 0.
\end{align*}
If \(r > 0\), then by (\ref{cdef}) we would have \(c^{\hat j}_r = r \mu^j_{r+1} > 0\), a contradiction. Therefore \(r = 0\), and therefore \(t\geq 1\).

First assume \(s = 0\). Then \(\mu^i_1 =0\), and we have 
\begin{align*}
0 \geq c^{ i}_1 = \mu^{\hat i}_2 - \mu^i_1 = \mu^{\hat i}_2, 
\end{align*}
and thus \(\mu^{\hat i}_2 = 0\), so \(t \leq 1\), which implies that \(t=1\), and we are done.

Next we assume \(s>0\). 
We have \(\mu^{\hat i}_{s} > 0\) and \(\mu^i_{s+1} = 0\). Then by (\ref{downonec}) we have
\begin{align*}
c^{\hat i}_{s}= c^{\hat i}_{s-1} + s(\mu^i_{s+1} - \mu^{\hat i}_s) = c^i_{s-1} - s\mu^{\hat i}_s < c^{\hat i}_{s-1} \leq 0.
\end{align*}
Thus \(c^{\hat i}_s < 0\), so \(c^i_{s} = 0\). Thus, noting that \(\mu^i_s > 0\) and \(c^i_{s-1} \leq 0\), we have by (\ref{downonec}) that
\begin{align*}
0 = c^i_s = c^i_{s-1} + s(\mu^{\hat i}_{s+1} - \mu^i_s) = c^i_{s-1} +s\mu^{\hat i}_{s+1} -s \mu^i_s < c^i_{s-1}  +s\mu^{\hat i}_{s+1} \leq s\mu^{\hat i}_{s+1},
\end{align*}
which implies that \(\mu^{\hat i}_{s+1} > 0\), so \(t \geq s+1\).
Next we have by (\ref{downonec}) that
\begin{align*}
0 \geq c^i_{s+1} = c^i_{s} + (s+1)(\mu^{\hat i}_{s+2} - \mu^i_{s+1}) = (s+1) \mu^{\hat i}_{s+2},
\end{align*}
and so \(\mu^{\hat i}_{s+2}=0\). Thus \(t= s+1\), completing the proof.
\end{proof}

\begin{lemma}\label{onemustbebigger}
Let \(\bmu \in \bOm^{\textup{MV}}\) be nontrivial. Then there exists \(i \in \ZZ_2\), \(M \in \Z_{>0}\) such that \(\mu^i_M >0\) and \(\mu^{\hat i}_m = 0\) for all \(m \geq M\).
\end{lemma}
\begin{proof}
By way of contradiction assume that there exists \(N \in \Z_{>0}\) such that \(\mu^1_N, \mu^0_N > 0\) and \(\mu^j = 0\) for all \(m >N\). Since \(\bmu   \in \bOm^{\textup{MV}}\), we have \(c^1_{N-1}(\bmu), c^0_{N-1}(\bmu) \leq 0\). Therefore:
\begin{align*}
c^1_N(\bmu) &= c^1_{N-1}(\bmu) + N(\mu^0_{N+1} - \mu^1_N) = c^1_{N-1}(\bmu) - N \mu^1_N < 0;\\
c^0_N(\bmu) &= c^0_{N-1}(\bmu) + N(\mu^1_{N+1} - \mu^0_N) = c^0_{N-1}(\bmu) - N \mu^0_N < 0,
\end{align*}
which contradicts the fact that \(0 \in \{c^1_N(\bmu), c^0_N(\bmu)\}\).
\end{proof}

\subsection{Diagonalization}
Let \(\bmu \in \bOm^{\textup{MV}}\). For \(i \in \ZZ_2\), \(m \in \ZZ_{\geq 1}\), 
if \(c^{i}_m(\bmu) = 0\), then \(\beta^{\hat i}_{m+1} - \beta^i_m = d^i_m(\bmu) \alpha_{\hat i}\), and we refer to the line segment connecting the points \( \beta^i_m, \beta^{\hat i}_{m+1}\) as a {\em diagonal}. 
A {\em diagonalization} of \(\bmu \in \bOm^{\textup{MV}}\) is a function 
\(
\xi: \ZZ_{\geq 0} \to \ZZ_2
\)
such that \(c^{\xi(m)}_m(\bmu) = 0\) for all \(m \in \ZZ_{\geq 0}\); i.e. \(\xi\) gives a specific choice of diagonal for each \(m \in \ZZ_{\geq 0}\). By the definition of MV bicompositions, diagonalizations always exist.

Given a diagonalization \(\xi\) of \(\bmu\), we refer to the portion of \(P(\bmu)\) between the \((m-1)\)th and \(m\)th diagonals in \(\xi\) as the {\em \(m\)th segment of \(\xi\)}. In other words, the \(m\)th segment of \(\xi\) is the convex hull of the points 
\begin{align*}
\beta^{\xi(m-1)}_{m-1}, \beta^{\widehat{\xi(m-1)}}_m, \beta^{\xi(m)}_m, \beta^{\widehat{\xi(m)}}_{m+1}, \beta^{\xi(m-1)}_m.
\end{align*}

\subsection{Downward diagonalization}\label{downdiag}
We inductively construct a specific `downward' diagonalization \(\xi_{\bmu}\) of \(\bmu \in \bOm^{\textup{MV}}\) as follows. 
First, let \(M \in \ZZ_{\geq 0}\) be minimal such that \(\mu^1_k = \mu^0_k = 0\) for all \(k \geq M\). Note that for \(i \in \ZZ_2\) and \(m \geq M-1\), we have
\begin{align}\label{c01meq}
c^0_m =  c^0_{M-2}  -(M-1)\mu^0_{M-1} \qquad \textup{and} \qquad c^1_m =  c^1_{M-2}  -(M-1)\mu^1_{M-1}.
\end{align}
Since at least one of \(c^0_m\), \(c^1_m\) is equal to zero and \(c^i_{M-2} \leq 0\), it follows from the minimality of \(M\) that \(\mu^i_{M-1} = 0\) for exactly one \(i \in \ZZ_2\). Letting \(i \in \ZZ_2\) be such that \(\mu^i_{M-1} = 0\) and  \(\mu^{\hat i}_{M-1} > 0\), it follows from (\ref{c01meq}) that \(c^{\hat i}_m < 0\) for all \(m \geq M-1\). Hence by the properties of MV bicompositions, we have that \(c^i_m = 0\) for all \(m \geq M-1\). We therefore may construct a well-defined diagonalization \(\xi_{\bmu}\) of \(\bmu\) by downward induction, setting:
\begin{align}\label{downdiagdef}
\xi_{\bmu}(m) = 
\begin{cases}
\; i & \textup{if }m \geq M-1;\\
\;\xi_{\bmu}(m+1) &\textup{if }m < M-1 \textup{ and } c^{\xi_{\bmu}(m+1)}_{m} = 0;\\
\;\widehat{\xi_{\bmu}(m+1)} &\textup{if }m < M-1 \textup{ and } c^{\xi_{\bmu}(m+1)}_{m}  < 0;\\
\end{cases}
\end{align}
for all \(m \in \ZZ_{\geq 1}\). This choice of diagonalization gives us some control over the structure of its segments, as noted in the next subsection.

\subsubsection{Typifying segments}\label{deftypes}
Let \(\xi\) be a diagonalization of \(\bmu \in \bOm^{\textup{MV}}\). We define a function
\begin{align*}
\TypeH_{\xi, \bmu} : \ZZ_{\geq 1} \to \{\textup{A}_\varnothing, \textup{A}_+, \textup{A}_-, \textup{B}_+, \textup{B}_-, \textup{Z}\}
\end{align*}
given by setting
\begin{align*}
\TypeH_{\xi, \bmu}(m) = 
\begin{cases}
\textup{A}_\varnothing & \textup{if \(\xi(m-1) = \xi(m)\), 
\(\mu^{\xi(m)}_m = \mu^{\hat{\xi(m)}}_{m+1} = 0\), and \(\eta^{\xi(m)}_{m+1}(\bmu) , \eta^{\hat{\xi(m)}}_{m+2}(\bmu) \) are even.}\\
\textup{A}_+ & \textup{if \(\xi(m-1) = \xi(m)\) and \(\eta^{\xi(m)}_{m+1}(\bmu) , \eta^{\hat{\xi(m)}}_{m+2}(\bmu) \) are odd.}\\
\textup{A}_- & \textup{if \(\xi(m-1) = \xi(m)\) and \(\eta^{\xi(m)}_{m}(\bmu) , \eta^{\hat{\xi(m)}}_{m+1}(\bmu) \) are odd.}\\
\textup{B}_+ & \textup{if \(\xi(m-1) = \hat{\xi(m)}\), \(c^{\xi(m)}_{m-1}(\bmu) \neq 0\), \(\eta^{\hat{ \xi(m)}}_{m+1}(\bmu) \) is odd and \(\eta^{\xi(m)}_{m+1}(\bmu) \) is even.}\\
\textup{B}_- & \textup{if \(\xi(m-1) = \hat{\xi(m)}\), \(c^{ \xi(m)}_{m-1}(\bmu)  \neq  0\), \(\mu^{ \hat{\xi(m)}}_{m+1} > 0\), and  \(\eta^{ \hat{\xi(m)}}_{m}(\bmu) \), \(\eta^{ \xi(m)}_{m+1}(\bmu) \) are odd.}\\
\textup{Z} & \textup{otherwise},
\end{cases}
\end{align*}
and we refer to \(\TypeH_{\xi, \bmu}(m)\) as the {\em (downward) type of the \(m\)th segment of \(\xi\) in \(\bmu\)}.

\begin{figure}[h]
\begin{align*}
\hackcenter{}
\hackcenter{
\begin{overpic}[height=35mm]{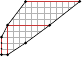}
 \put(5.8,-3){\makebox(0,0)[]{$\scriptstyle \alpha_{1:1}$}}    
  \put(28,8){\makebox(0,0)[]{$\scriptstyle  \alpha_{1:3}$}}    
    \put(54,26){\makebox(0,0)[]{$\scriptstyle  \alpha_{1:4}$}}   
      \put(87,51){\makebox(0,0)[]{$\scriptstyle  \alpha_{1:5}$}}   
     \put(-6,12){\makebox(0,0)[]{$\scriptstyle 3 \alpha_{0:1}$}}          
       \put(-1,32){\makebox(0,0)[]{$\scriptstyle \alpha_{0:2}$}}       
         \put(15,56){\makebox(0,0)[]{$\scriptstyle  \alpha_{0:4}$}}     
        \put(-2,-1){\makebox(0,0)[]{$\scriptstyle \color{gray}0$}}   
           \put(12,0){\makebox(0,0)[l]{$\scriptstyle \beta^1_1 = \beta^1_2$}}    
            \put(34.5,16){\makebox(0,0)[l]{$\scriptstyle \beta^1_3$}}    
            \put(64,37){\makebox(0,0)[l]{$\scriptstyle \beta^1_4$}}    
             \put(102,70){\makebox(0,0)[l]{$\scriptstyle \beta^1_5 = \beta^1_6 =  \cdots$}}    
              \put(-1.5,23){\makebox(0,0)[r]{$\scriptstyle \beta^0_1$}}    
               \put(6,41){\makebox(0,0)[r]{$\scriptstyle \beta^0_3 = \beta^0_2$}}    
                \put(27,70){\makebox(0,0)[r]{$\scriptstyle \cdots  = \beta^0_5 = \beta^0_4$}}    
\end{overpic}
}
\end{align*}
\caption{The polytope \(P(\bmu)\) associated to the MV bicomposition \(\bmu\) as considered in Example~\ref{bicomgeomex}.}
\label{bicomgeomfig}
\end{figure}
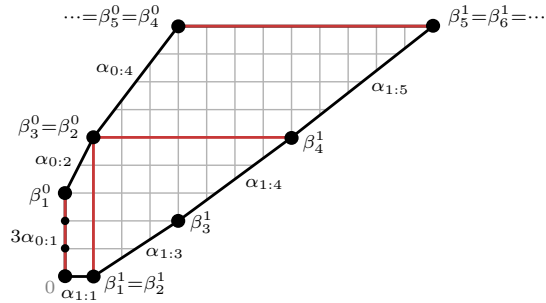

\begin{example}\label{bicomgeomex}
Consider the case of \(\bmu \in \bOm^{\textup{MV}}\), with \(\mu\) given by:
\begin{align*}
\mu^1_1 = 1; \quad \mu^1_3=1;\quad  \mu^1_4 = 1; \quad \mu^1_5=1;\qquad \mu^0_1=3; \quad \mu^0_2=1; \quad \mu^0_4 = 1,
\end{align*}
and \(\mu^i_m = 0\) otherwise. Then the polytope \(P(\bmu)\) which is the convex hull of the points \(\beta^i_m(\bmu)\) is shown in Figure~\ref{bicomgeomfig}. Following \S\ref{downdiag}, the downward diagonalization \(\xi_{\bmu}\) gives us \(\xi_{\bmu}(0) = \xi_{\bmu}(1) = \xi_{\bmu}(2) = 1\) and \(\xi_{\bmu}(m) = 0\) otherwise. The associated \(m\)th diagonals---line segments connecting \(\beta_m^{\xi_{\bmu}(m)}\) and \(\beta_{m+1}^{\widehat{\xi_{\bmu}(m)}}\)---are shown in red in the figure.
The (downward) types of the segments in the diagonalization \(\xi_{\bmu}\) are given by
\begin{align*}
\textup{Type}_{\xi_{\bmu}, \bmu}(1) = \textup{A}_-;
\,\,\,
\textup{Type}_{\xi_{\bmu}, \bmu}(2) = \textup{A}_+;
\,\,\,
\textup{Type}_{\xi_{\bmu}, \bmu}(3) = \textup{B}_-;
\,\,\,
\textup{Type}_{\xi_{\bmu}, \bmu}(4) = \textup{A}_-;
\,\,\,
\textup{Type}_{\xi_{\bmu}, \bmu}(m) = \textup{A}_\varnothing,
\end{align*}
for \(m \geq 5\).  The reader should compare these segment types against (\ref{allcombotypes}).
\end{example}

\begin{lemma}\label{downdiagtypeslem}
Let \(\bmu \in \bOm^{\textup{MV}}\). 
For all \(m \in \ZZ_{\geq 1}\), we have \(\TypeH_{\xi_{\bmu}, \bmu}(m) \in \{\textup{A}_\varnothing, \textup{A}_+, \textup{A}_-, \textup{B}_+, \textup{B}_-\}\). Moreover,
 for all \(m \in \ZZ_{\geq 2}\), we have that 
 \begin{align*}
 (\TypeH_{\xi_{\bmu}, \bmu}(m-1), \TypeH_{\xi_{\bmu}, \bmu}(m))
 \end{align*}
 is equal to one of 
 \begin{align}
 &\{
(\textup{A}_\varnothing, \textup{A}_\varnothing), (\textup{A}_-,\textup{A}_\varnothing), (\textup{B}_+, \textup{A}_\varnothing), (\textup{A}_\varnothing, \textup{A}_+), (\textup{A}_-, \textup{A}_+),(\textup{B}_+, \textup{A}_+), (\textup{A}_\varnothing, \textup{B}_+),\notag\\ & \qquad\;\;\;(\textup{A}_-, \textup{B}_+), (\textup{B}_+, \textup{B}_+), (\textup{A}_+, \textup{A}_-), (\textup{B}_-, \textup{A}_-), (\textup{A}_+, \textup{B}_-),(\textup{B}_-, \textup{B}_-)
 \}.\label{allcombotypes}
 \end{align}
\end{lemma}
\begin{proof}
Write \(\xi = \xi_{\bmu}\). 
By construction of the diagonalization \(\xi\), it is immediate that \(\TypeH_{\xi, \bmu}(m) = \textup{A}_{\varnothing}\) 
for all \(m \geq M-1\). Now assume that \(m \in \ZZ_{\geq 2}\) and \(\TypeH_{\xi, \bmu}(m) \in  \{\textup{A}_\varnothing, \textup{A}_+, \textup{A}_-, \textup{B}_+, \textup{B}_-\}\). We prove accordingly that \(\TypeH_{\xi, \bmu}(m-1) \in  \{\textup{A}_\varnothing, \textup{A}_+, \textup{A}_-, \textup{B}_+, \textup{B}_-\}\), and satisfies (\ref{allcombotypes}). We consider the possibilities 
\begin{align*}
\mu^{\xi(m-1)}_{m-1} = \mu^{\hat{\xi(m-1)}}_m =0, \qquad \mu^{\xi(m-1)}_{m-1} = \mu^{\hat{\xi(m-1)}}_m > 0, \qquad \mu^{\xi(m-1)}_{m-1} \neq \mu^{\hat{\xi(m-1)}}_m
\end{align*}
as three individual cases.

\noindent{\em Case 1. Assume that \(\mu^{\xi(m-1)}_{m-1} = \mu^{\hat{\xi(m-1)}}_m =0\).} Then by (\ref{downonec}) we have \(c^{\xi(m-1)}_{m-2} = c^{\xi(m-1)}_{m-1} = 0\), and thus by (\ref{downdiagdef}) we have \(\xi(m-2) = {\xi(m-1)}\). Now, note that
\begin{align*}
\eta^{\xi(m-1)}_m = \begin{cases}
1+ \eta^{\xi(m-1)}_{m+1} & \textup{if }\mu^{\xi(m-1)}_m > 0\\
0 & \textup{if }\mu^{\xi(m-1)}_m = 0
\end{cases}
\qquad
\textup{and}
\qquad
\eta^{\hat{\xi(m-1)}}_{m+1} = \begin{cases}
1+ \eta^{\hat{\xi(m-1)}}_{m+2} & \textup{if }\mu^{\hat{\xi(m-1)}}_{m+1} > 0\\
0 & \textup{if }\mu^{\hat{\xi(m-1)}}_{m+1} = 0.
\end{cases}
\end{align*}
We consider some subcases:
\begin{enumerate}[(1a)]
\item[(1a)] 
If \(\TypeH_{\xi,\bmu}(m) = \textup{A}_\varnothing\), then we have that \(\mu^{\xi(m-1)}_m = \mu^{\hat{\xi(m-1)}}_{m+1} =0\), so \(\eta^{\xi(m-1)}_m= \eta^{\hat{\xi(m-1)}}_{m+1} = 0\), and thus \(\TypeH_{\xi,\bmu}(m-1) = \textup{A}_\varnothing\).
\item[(1b)]
If \(\TypeH_{\xi,\bmu}(m) = \textup{A}_+\), then we have that  \(\eta^{\xi(m-1)}_{m+1}\), \(\eta^{\hat{\xi(m-1)}}_{m+2}\) are odd, and therefore in any case \(\eta^{\xi(m-1)}_m, \eta^{\hat{\xi(m-1)}}_{m+1}\) are even, so \(\TypeH_{\xi,\bmu}(m-1) = \textup{A}_\varnothing\).
\item[(1c)] If the \(\TypeH_{\xi,\bmu}(m) = \textup{B}_+\), then we have that \(\eta^{\xi(m-1)}_{m+1}\) is odd and \(\eta^{\hat{\xi(m-1)}}_{m+1}\) is even. Thus \(\eta^{\xi(m-1)}_{m}\) is even in any case, and so \(\TypeH_{\xi,\bmu}(m-1) = \textup{A}_\varnothing\).
\item[(1d)]
If \(\TypeH_{\xi,\bmu}(m) \in\{ \textup{A}_-, \textup{B}_-\}\), then we have that \(\eta^{\xi(m-1)}_{m}, \eta^{\hat{\xi(m-1)}}_{m+1}\) are odd, which implies that 
\(\TypeH_{\xi,\bmu}(m-1) = \textup{A}_+\).
\end{enumerate}

\noindent{\em Case 2. Assume that \(\mu^{\xi(m-1)}_{m-1}=\mu^{\hat{\xi(m-1)}}_m > 0\).} Then by (\ref{downonec}) we have \(c^{\xi(m-1)}_{m-2} = c^{\xi(m-1)}_{m-1} = 0\), and thus by (\ref{downdiagdef}) we have \(\xi(m-2) = {\xi(m-1)}\). Now, note that
\begin{align*}
\eta^{\xi(m-1)}_{m-1} = \begin{cases}
2+ \eta^{\xi(m-1)}_{m+1} & \textup{if }\mu^{\xi(m-1)}_m > 0\\
1 & \textup{if }\mu^{\xi(m-1)}_m = 0
\end{cases}
\qquad
\textup{and}
\qquad
\eta^{\hat{\xi(m-1)}}_{m} = \begin{cases}
2+ \eta^{\hat{\xi(m-1)}}_{m+2} & \textup{if }\mu^{\hat{\xi(m-1)}}_{m+1} > 0\\
1 & \textup{if }\mu^{\hat{\xi(m-1)}}_{m+1} = 0
\end{cases}
\end{align*}
We consider some subcases:
\begin{enumerate}[(2a)]
\item 
If \(\TypeH_{\xi,\bmu}(m) = \textup{A}_\varnothing\), then we have that \(\mu^{\xi(m-1)}_m = \mu^{\hat{\xi(m-1)}}_{m+1} =0\), so \(\eta^{\xi(m-1)}_{m-1}= \eta^{\hat{\xi(m-1)}}_{m} = 1\), and thus 
\(\TypeH_{\xi,\bmu}(m-1) = \textup{A}_-\).
\item 
If \(\TypeH_{\xi,\bmu}(m) = \textup{A}_+\), then we have that  \(\eta^{\xi(m-1)}_{m+1}\), \(\eta^{\hat{\xi(m-1)}}_{m+2}\) are odd, and therefore in any case \(\eta^{\xi(m-1)}_{m-1}, \eta^{\hat{\xi(m-1)}}_{m}\) are odd, so
\(\TypeH_{\xi,\bmu}(m-1) = \textup{A}_-\).
\item If 
\(\TypeH_{\xi,\bmu}(m) = \textup{B}_+\), then we have that \(\eta^{\xi(m-1)}_{m+1}\) is odd and \(\eta^{\hat{\xi(m-1)}}_{m+1}\) is even. Thus either \(\mu^{\hat{\xi(m-1)}}_{m+1} = 0\) or \(\eta^{\hat{\xi(m-1)}}_{m+2}\) is odd. Thus \(\eta^{\xi(m-1)}_{m-1}, \eta^{\hat{\xi(m-1)}}_m\) are odd in any case, and so
\(\TypeH_{\xi,\bmu}(m-1) = \textup{A}_-\).
\item
\(\TypeH_{\xi,\bmu}(m) \in\{ \textup{A}_-, \textup{B}_-\}\), then we have that \(\eta^{\xi(m-1)}_{m}, \eta^{\hat{\xi(m-1)}}_{m+1}\) are odd, which implies that \(\TypeH_{\xi,\bmu}(m-1) = \textup{A}_+\)
\end{enumerate}

\noindent{\em Case 3. Assume that \(\mu^{\xi(m-1)}_{m-1}  \neq \mu^{\hat{\xi(m-1)}}_m\).} Then by (\ref{downonec})  we have \(c^{\xi(m-1)}_{m-2} \neq 0\). Thus, since \(\bmu \in \bOm^{\textup{MV}}\) we have that \(c^{\xi(m-1)}_{m-2} < 0\), so \(c^{\hat{\xi(m-1)}}_{m-2} = 0\) and \(\mu^{\hat{\xi(m-1)}}_m > \mu^{\xi(m-1)}_{m-1} \geq 0\). By (\ref{downdiagdef}) we have \(\xi(m-2) = \hat{\xi(m-1)}\). Note that
\begin{align*}
\eta^{\hat{\xi(m-1)}}_m = 
\begin{cases}
2 + \eta_{m+2}^{\hat{\xi(m-1)}} &\textup{if }\mu_{m+1}^{\hat{\xi(m-1)}} > 0\\
1&\textup{if }\mu_{m+1}^{\hat{\xi(m-1)}} = 0
\end{cases}
\qquad
\textup{and}
\qquad
\eta^{\xi(m-1)}_m = 
\begin{cases}
1 + \eta_{m+1}^{\xi(m-1)} &\textup{if }\mu_{m}^{\xi(m-1)} > 0\\
0&\textup{if }\mu_{m}^{\xi(m-1)} = 0.
\end{cases}
\end{align*}
We again consider some subcases:
\begin{enumerate}[(3a)]
\item 
If \(\TypeH_{\xi,\bmu}(m) = \textup{A}_\varnothing\), then 
\(\xi(m-1) = \xi(m)\) and 
\(\mu^{\xi(m)}_m = \mu^{\hat{\xi(m)}}_{m+1} = 0\).  We have that \(\eta_m^{\hat{\xi(m-1)}} = 1\) and \(\eta^{\xi(m-1)}_{m} = 0\), so 
\(\TypeH_{\xi,\bmu}(m-1) = \textup{B}_+\).

\item If 
\(\TypeH_{\xi,\bmu}(m) = \textup{A}_+\), then 
\(\xi(m-1) = \xi(m)\) and \(\eta^{\xi(m)}_{m+1}, \eta^{\hat{\xi(m)}}_{m+2}\) are odd. Then in any case we have that \(\eta^{\hat{\xi(m-1)}}_m\) is odd, and \(\eta^{\xi(m-1)}_m\) is even, so
\(\TypeH_{\xi,\bmu}(m-1) = \textup{B}_+\).

\item If 
\(\TypeH_{\xi,\bmu}(m) = \textup{A}_-\), then
\(\xi(m-1) = \xi(m)\) and \(\eta^{\xi(m)}_{m}, \eta^{\hat{\xi(m)}}_{m+1}\) are odd.
Since \(\eta_{m+1}^{\hat{\xi(m-1)}}\) is odd, it follows that \(\mu_m^{\xi(m-1)}, \mu_{m+1}^{\hat{\xi(m-1)}}> 0\) and \(\eta_{m+2}^{\hat{\xi(m-1)}}\) is even. Moreover, we have
\begin{align*}
0 \geq c^{\hat{\xi(m-1)}}_{m-1} = c^{\hat{\xi(m-1)}}_{m-2} + (m-1)(\mu_m^{\xi(m-1)} - \mu_{m-1}^{\hat{\xi(m-1)}}) = (m-1)(\mu_m^{\xi(m-1)} - \mu_{m-1}^{\hat{\xi(m-1)}}),
\end{align*}
which implies that \(\mu_{m-1}^{\hat{\xi(m-1)}} \geq \mu_m^{\xi(m-1)} > 0\). Thus \(\eta_{m-1}^{\hat{\xi(m-1)}} = 1 + \eta_{m}^{\hat{\xi(m-1)}}\) is odd. 
Therefore 
\(\TypeH_{\xi,\bmu}(m-1) = \textup{B}_-\).

\item If 
\(\TypeH_{\xi,\bmu}(m) = \textup{B}_+\),  we have that \(\xi(m-1) = \hat{\xi(m)}\), \(c^{\xi(m)}_{m-1} \neq 0\), \(\eta^{\hat{ \xi(m)}}_{m+1}\) is odd and \(\eta^{\xi(m)}_{m+1}\) is even. Since \(\eta^{\hat{\xi(m-1)}}_{m+1}\) is even, either \(\mu_{m+1}^{\hat{\xi(m-1)}} = 0\) or  \(\eta^{\hat{\xi(m-1)}}_{m+2}\) is odd. In any case then it follows that \(\eta^{\hat{ \xi(m-1)}}_{m}\) is odd and \(\eta_m^{\xi(m-1)}\) is even. Therefore 
\(\TypeH_{\xi,\bmu}(m-1) = \textup{B}_+\)

\item
If
\(\TypeH_{\xi,\bmu}(m) = \textup{B}_-\), we have that \(\xi(m-1) = \hat{\xi(m)}\), \(c^{ \xi(m)}_{m-1} \neq  0\), \(\mu^{ \hat{\xi(m)}}_{m+1} > 0\), and  \(\eta^{ \hat{\xi(m)}}_{m}\), \(\eta^{ \xi(m)}_{m+1}\) are odd. 
Since \(\eta_m^{\hat{\xi(m)}}\) is odd, we have that \(\mu_m^{\hat{\xi(m)}} = \mu_m^{\xi(m-1)}>0\).
Moreover, we have
\begin{align*}
0 \geq c^{\hat{\xi(m-1)}}_{m-1} = c^{\hat{\xi(m-1)}}_{m-2} + (m-1)(\mu_m^{\xi(m-1)} - \mu_{m-1}^{\hat{\xi(m-1)}}) = (m-1)(\mu_m^{\xi(m-1)} - \mu_{m-1}^{\hat{\xi(m-1)}}),
\end{align*}
which implies that \(\mu_{m-1}^{\hat{\xi(m-1)}} \geq \mu_m^{\xi(m-1)} >0\). Moreover, since \(\mu_m^{\hat{\xi(m-1)}} >0\), we have that \(\eta^{ \hat{\xi(m-1)}}_{m-1} = 2 +\eta^{ \hat{\xi(m-1)}}_{m+1} = 2 +\eta^{\xi(m)}_{m+1} \) is odd. Therefore 
\(\TypeH_{\xi,\bmu}(m-1) = \textup{B}_-\).
\end{enumerate}

Thus in every case we have that \( (\TypeH_{\xi, \bmu}(m-1), \TypeH_{\xi, \bmu}(m))\) is in the set (\ref{allcombotypes}), completing the proof. 
\end{proof}

\subsection{The \(\mathcal{H}\) map and MV bicompositions}\label{Hmapbicompssec}
The downward diagonalization (and associated control over segment types), allows us to prove the following key result, working on a `segment-by-segment' basis.
\begin{lemma}\label{Hisom}
The map \(\mathcal{H}\) restricts to a map of MV bicompositions \(\mathcal{H}: \bOm^{\textup{MV}}_{\geq 1,+} \to \bOm^{\textup{MV}}_{\geq 0,\bullet}\). Moreover, if \(\bmu \in \bOm^{\textup{MV}}_{\geq 1,+}\), then \(\xi_{\bmu}\) is a diagonalization of \(\bnu = \mathcal{H}(\bmu)\), and:
\begin{align}
d^{\xi_{\bmu}(m-1)}_{m-1}(\bnu) &= 
\begin{cases}
d^{\xi_{\bmu}(m-1)}_{m-1}(\bmu)-1 &\textup{if } \TypeH_{\xi_{\bmu}, \bmu}(m) \in \{\textup{A}_\varnothing, \textup{A}_+, \textup{B}_+\}\\
d^{\xi_{\bmu}(m-1)}_{m-1}(\bmu)+1 &\textup{if } \TypeH_{\xi_{\bmu}, \bmu}(m)\in \{\textup{A}_-, \textup{B}_-\}\\
\end{cases}\label{dunder}\\
d^{\xi_{\bmu}(m)}_{m}(\bnu) &= 
\begin{cases}
d^{\xi_{\bmu}(m)}_{m}(\bmu)-1 &\textup{if } \TypeH_{\xi_{\bmu}, \bmu}(m) \in \{\textup{A}_\varnothing, \textup{A}_-, \textup{B}_+\}\\
d^{\xi_{\bmu}(m)}_{m}(\bmu)+1 &\textup{if } \TypeH_{\xi_{\bmu}, \bmu}(m) \in \{\textup{A}_+, \textup{B}_-\}.\\
\end{cases}\label{dover}
\end{align}
\end{lemma}

\begin{proof}
Let \(\bmu \in \bOm^{\textup{MV}}_{\geq 1,+}\). 
We will write  \(\xi:= \xi_{\bmu}\). For \(i \in \ZZ_2\), \(m \in \ZZ_{\geq 1}\), the value \(\nu^i_m\) is directly related to the type of the segment in the diagonalization \(\xi\) of \(\bmu\) which the edge corresponding to \(\mu^i_m\) belongs to. Unpacking the definitions in \S\ref{deftypes} alongside (\ref{Hcases}), we have the following facts, which we will use repeatedly:
\begin{enumerate}[(i)]
\item If \(\TypeH_{\xi, \bmu}(m) = \textup{A}_\varnothing\) then \(\nu^{\xi(m)}_m = \nu^{\hat{\xi(m)}}_{m+1} = 0\);
\item If \(\TypeH_{\xi, \bmu}(m) = \textup{A}_+\) then \(\nu^{\xi(m)}_m =\mu^{\xi(m)}_m +1\) and \(\nu^{\hat{\xi(m)}}_{m+1} = \mu^{\hat{\xi(m)}}_{m+1} +1\);
\item If \(\TypeH_{\xi, \bmu}(m) = \textup{A}_-\) then \(\nu^{\xi(m)}_m =\mu^{\xi(m)}_m -1\) and \(\nu^{\hat{\xi(m)}}_{m+1} = \mu^{\hat{\xi(m)}}_{m+1} -1\);
\item If \(\TypeH_{\xi, \bmu}(m) = \textup{B}_+\) then \(\nu^{\hat{\xi(m)}}_m = \mu^{\hat{\xi(m)}}_m+1\) and \(\nu^{\hat{\xi(m)}}_{m+1} = \mu^{\hat{\xi(m)}}_{m+1}-1\);
\item If \(\TypeH_{\xi, \bmu}(m) = \textup{B}_-\) then \(\nu^{\hat{\xi(m)}}_m = \mu^{\hat{\xi(m)}}_m-1\) and \(\nu^{\hat{\xi(m)}}_{m+1} = \mu^{\hat{\xi(m)}}_{m+1}+1\).
\end{enumerate} We proceed via a number of claims.\\

\noindent {\em Claim I: Let \(m \in \ZZ_{\geq 1}\), and set \(m' = m+1\). If (\ref{dover}) holds for \(m\), then (\ref{dunder}) holds for \(m'\).}
We have from Lemma~\ref{downdiagtypeslem} that \(\TypeH_{\xi, \bmu}(m') \in \{\textup{A}_\varnothing, \textup{A}_+, \textup{B}_+\}\) if and only if \(\TypeH_{\xi, \bmu}(m) \in \{\textup{A}_\varnothing, \textup{A}_-, \textup{B}_+\}\), and \(\TypeH_{\xi, \bmu}(m') \in \{\textup{A}_-, \textup{B}_-\}\) if and only if \(\TypeH_{\xi, \bmu}(m) \in \{\textup{A}_+, \textup{B}_-\}\), which immediately implies the claim.\\

\noindent {\em Claim II: The equalities \(c_{m-1}^{\xi(m-1)}(\bnu) = 0\), and (\ref{dunder}) hold when \(m=1\).} 
The first equality is obvious, since \(c_0^{\xi(m-1)}(\bnu)\) is definitionally zero. 
Since \(c^i_0(\bmu) = 0\) for both \(i \in \ZZ_2\), it follows from (\ref{downdiagdef}) and Lemma~\ref{downdiagtypeslem} that \(\xi(0) = \xi(1)\) and \(\TypeH_{\xi, \bmu}(1) \in \{\textup{A}_\varnothing, \textup{A}_+,\textup{A}_-\}\). If the type is \(\textup{A}_\varnothing\) or \(\textup{A}_+\) then \(\eta_1^{\xi(1)}\) is even, so by Lemma~\ref{mixedpar},  \(\eta_1^{\hat{\xi(1)}}\) is odd. Thus
\begin{align*}
d_0^{\xi(0)}(\bnu) &= \nu^{\hat{\xi(1)}}_1 = \mu^{\hat{\xi(1)}}_1-1 = d_0^{\xi(0)}(\bmu)-1.
\end{align*}
On the other hand, if \(\TypeH_{\xi, \bmu}(1)=\textup{A}_-\), then \(\eta_1^{\xi(1)}\) is odd, so by Lemma~\ref{mixedpar},  \(\eta_1^{\hat{\xi(1)}}\) is even. Thus 
\begin{align*}
d_0^{\xi(0)}(\bnu) &= \nu^{\hat{\xi(1)}}_1 = \mu^{\hat{\xi(1)}}_1+1 = d_0^{\xi(0)}(\bmu)+1.
\end{align*}
Thus we have verified Claim II.\\

\noindent {\em Claim III: If \(c^{\xi(m-1)}_{m-1}(\bnu) = 0\) and (\ref{dunder}) hold for a given \(m \in \ZZ_{\geq 1}\), then \(c^{\xi(m)}_{m}(\bnu) = 0\) and (\ref{dover}) both hold for \(m\).} 
Assume that \(c^{\xi(m-1)}_{m-1}(\bnu) = 0\) and (\ref{dunder}) hold. We have by Lemma~\ref{downdiagtypeslem} that \(\TypeH_{\xi_{\bmu}, \bmu}(m) \in \{\textup{A}_\varnothing, \textup{A}_+, \textup{A}_-, \textup{B}_+, \textup{B}_-\}\). We consider these cases separately.

\noindent{\em Case 1. Assume 
\(\TypeH_{\xi,\bmu}(m) = \textup{A}_\varnothing\).} Then \(\xi(m-1) = \xi(m)\), so in view of (\ref{dunder}) and (i) above we have
\begin{align*}
\nu_m^{\xi(m)} = 0 \qquad \nu_{m+1}^{\hat{\xi(m)}} = 0 \qquad d_{m-1}^{\xi(m-1)}(\bnu) = d_{m-1}^{\xi(m-1)}(\bmu) - 1
\end{align*}
so in consideration of (\ref{downonec}, \ref{downoned}) we have
\begin{align*}
c_{m}^{\xi(m)}(\bnu) &= c_{m-1}^{\xi(m-1)}(\bnu) + m(\nu_{m+1}^{\hat{\xi(m)}} - \nu_m^{\xi(m)}) = 0\\
d_{m}^{\xi(m)}(\bnu) &= d_{m-1}^{\xi(m-1)}(\bnu) - (m-1)\nu^{\xi(m)}_m + (m+1)\nu^{\hat{\xi(m)}}_{m+1} \\
&= d_{m-1}^{\xi(m-1)}(\bmu)-1 + (m-1)\mu^{\xi(m)}_m - (m+1)\mu^{\hat{\xi(m)}}_{m+1} =  d_{m}^{\xi(m)}(\bmu) - 1. 
\end{align*}

\noindent{\em Case 2. Assume 
\(\TypeH_{\xi,\bmu}(m) = \textup{A}_+\).} Then \(\xi(m-1) = \xi(m)\), so in view of (\ref{dunder}) and (ii) above we have
\begin{align*}
\nu_m^{\xi(m)} =\mu_m^{\xi(m)} +1 \qquad \nu_{m+1}^{\hat{\xi(m)}} = \mu_{m+1}^{\hat{\xi(m)}}+1 \qquad d_{m-1}^{\xi(m-1)}(\bnu) = d_{m-1}^{\xi(m-1)}(\bmu) - 1
\end{align*}
so in consideration of (\ref{downonec}, \ref{downoned}) we have
\begin{align*}
c_{m}^{\xi(m)}(\bnu) &= c_{m-1}^{\xi(m-1)}(\bnu) + m(\nu_{m+1}^{\hat{\xi(m)}} - \nu_m^{\xi(m)}) = c_{m-1}^{\xi(m-1)}(\bmu) + m(\mu_{m+1}^{\hat{\xi(m)}} - \mu_m^{\xi(m)}) = c_{m}^{\xi(m)}(\bmu) = 0 \\
d_{m}^{\xi(m)}(\bnu) &= d_{m-1}^{\xi(m-1)}(\bnu) - (m-1)\nu^{\xi(m)}_m + (m+1)\nu^{\hat{\xi(m)}}_{m+1}\\
&=d_{m-1}^{\xi(m-1)}(\bmu) - (m-1)\mu^{\xi(m)}_m + (m+1)\mu^{\hat{\xi(m)}}_{m+1} +1 = d_{m}^{\xi(m)}(\bmu) +1.
\end{align*}

\noindent{\em Case 3. Assume 
\(\TypeH_{\xi,\bmu}(m) = \textup{A}_-\).} Then \(\xi(m-1) = \xi(m)\) so in view of (\ref{dunder}) and (iii) above we have
\begin{align*}
\nu_m^{\xi(m)} =\mu_m^{\xi(m)} -1 \qquad \nu_{m+1}^{\hat{\xi(m)}} = \mu_{m+1}^{\hat{\xi(m)}}-1 \qquad d_{m-1}^{\xi(m-1)}(\bnu) = d_{m-1}^{\xi(m-1)}(\bmu) + 1,
\end{align*}
so in consideration of (\ref{downonec}, \ref{downoned}) we have
\begin{align*}
c_{m}^{\xi(m)}(\bnu) &= c_{m-1}^{\xi(m-1)}(\bnu) + m(\nu_{m+1}^{\hat{\xi(m)}} - \nu_m^{\xi(m)}) = c_{m-1}^{\xi(m-1)}(\bmu) + m(\mu_{m+1}^{\hat{\xi(m)}} - \mu_m^{\xi(m)}) = c_{m}^{\xi(m)}(\bmu) = 0 \\
d_{m}^{\xi(m)}(\bnu) &= d_{m-1}^{\xi(m-1)}(\bnu) - (m-1)\nu^{\xi(m)}_m + (m+1)\nu^{\hat{\xi(m)}}_{m+1}\\
&=d_{m-1}^{\xi(m-1)}(\bmu) - (m-1)\mu^{\xi(m)}_m + (m+1)\mu^{\hat{\xi(m)}}_{m+1} -1 = d_{m}^{\xi(m)}(\bmu) -1.
\end{align*}

\noindent{\em Case 4. Assume
\(\TypeH_{\xi,\bmu}(m) = \textup{B}_+\).} Then  \(\xi(m-1) = \hat{\xi(m)}\), so in view of (\ref{dunder}) and (iv) above we have
\begin{align*}
\nu_{m}^{\hat{\xi(m)}} =\mu_{m}^{\hat{\xi(m)}} +1 \qquad \nu_{m+1}^{\hat{\xi(m)}} =\mu_{m+1}^{\hat{\xi(m)}} -1 \qquad d_{m-1}^{\xi(m-1)}(\bnu) = d_{m-1}^{\xi(m-1)}(\bmu) - 1,
\end{align*}
so in consideration of (\ref{ctoddown}, \ref{dtocdown}) we have
\begin{align*}
c_{m}^{\xi(m)}(\bnu) &=  (m-1)\nu_m^{\hat{\xi(m)}} + m\nu_{m+1}^{\hat{\xi(m)}} -d_{m-1}^{\hat{\xi(m)}}(\bnu)\\
&=(m-1)\mu_m^{\hat{\xi(m)}} + m\mu_{m+1}^{\hat{\xi(m)}} -d_{m-1}^{\hat{\xi(m)}}(\bmu)= c_{m}^{\xi(m)}(\bmu) = 0\\
d_{m}^{\xi(m)}(\bnu) &= m\nu^{\hat{\xi(m)}}_m + (m+1)\nu^{\hat{\xi(m)}}_{m+1} - c^{\hat{\xi(m)}}_{m-1}(\bnu)\\
&= m\mu^{\hat{\xi(m)}}_m + (m+1)\mu^{\hat{\xi(m)}}_{m+1} - c_{m-1}^{\hat{\xi(m)}}(\bmu) - 1 = d_{m}^{\xi(m)}(\bmu)-1.
\end{align*}

\noindent{\em Case 5. Assume
\(\TypeH_{\xi,\bmu}(m) = \textup{B}_-\).} Then  \(\xi(m-1) = \hat{\xi(m)}\), so in view of (\ref{dunder}) and (v) above we have
\begin{align*}
\nu_{m}^{\hat{\xi(m)}} =\mu_{m}^{\hat{\xi(m)}} -1 \qquad \nu_{m+1}^{\hat{\xi(m)}} =\mu_{m+1}^{\hat{\xi(m)}} +1 \qquad d_{m-1}^{\xi(m-1)}(\bnu) = d_{m-1}^{\xi(m-1)}(\bmu) + 1,
\end{align*}
so in consideration of (\ref{ctoddown}, \ref{dtocdown}) we have
\begin{align*}
c_{m}^{\xi(m)}(\bnu) &=  (m-1)\nu_m^{\hat{\xi(m)}} + m\nu_{m+1}^{\hat{\xi(m)}} -d_{m-1}^{\hat{\xi(m)}}(\bnu)\\
&=(m-1)\mu_m^{\hat{\xi(m)}} + m\mu_{m+1}^{\hat{\xi(m)}} -d_{m-1}^{\hat{\xi(m)}}(\bmu)= c_{m}^{\xi(m)}(\bmu) = 0\\
d_{m}^{\xi(m)}(\bnu) &= m\nu^{\hat{\xi(m)}}_m + (m+1)\nu^{\hat{\xi(m)}}_{m+1} - c^{\hat{\xi(m)}}_{m-1}(\bnu)\\
&= m\mu^{\hat{\xi(m)}}_m + (m+1)\mu^{\hat{\xi(m)}}_{m+1} - c_{m-1}^{\hat{\xi(m)}}(\bmu) + 1 = d_{m}^{\xi(m)}(\bmu)+1.
\end{align*}
This completes the proof of Claim III.\\

Combining Claims I, II, III, we have that \(c^{\xi(m)}_m(\bnu) = 0\), and (\ref{dunder}, \ref{dover}) hold for all \(m \in \ZZ_{\geq 1}\). 
Moreover, by Lemma~\ref{mixedpar}, we have for some \(i \in \ZZ_2\) that \(\eta^i_1(\bmu)\) is odd and \(\eta^{\hat i}_1(\bmu)\) is even. Thus by (\ref{Hcases}) we have that \(\nu^i_0 = \mu^i_0 + 1 = 1\), and \(\nu^{\hat i}_0 = 0\). It remains to finish the check that \(\bnu\) is an MV bicomposition; i.e., we must check that \(c^{\hat{\xi(m)}}_m(\bnu) \leq 0\) for all \(m \in \ZZ_{\geq 1}\).
Fix \(m \in \ZZ_{\geq 1}\), and set
\begin{align*}
(\textup{X},\textup{Y}) =  (\TypeH_{\xi_{\bmu}, \bmu}(m), \TypeH_{\xi_{\bmu}, \bmu}(m+1)).
\end{align*}
By (\ref{Hcases}, \ref{dover}), we have
\begin{align*}
\nu_{m+1}^{\hat{\xi(m)}} = \mu_{m+1}^{\hat{\xi(m)}}+ \varepsilon_1
\qquad
\nu_{m+1}^{\xi(m)} = \mu_{m+1}^{\xi(m)} + \varepsilon_2
\qquad
d_m^{\xi(m)}(\bnu) = d_m^{\xi(m)}(\bmu) + \varepsilon_3
\end{align*}
for some \(\varepsilon_1, \varepsilon_2, \varepsilon_3 \in \{-1,0,1\}\). Therefore applying  (\ref{ctodsame}) we have
\begin{align*}
c^{\hat{\xi(m)}}_m(\bnu) &= (m+1)\nu^{\hat{\xi(m)}}_{m+1}  + m\nu_{m+1}^{\xi(m)}  - d^{\xi(m)}_{m}(\bnu)  \\
&=(m+1)(\mu^{\hat{\xi(m)}}_{m+1} + \varepsilon_1) + m(\mu_{m+1}^{\xi(m)} + \varepsilon_2) - (d^{\xi(m)}_{m}(\bmu) + \varepsilon_3)
\\&=(m+1)\mu^{\hat{\xi(m)}}_{m+1}+ m\mu_{m+1}^{\xi(m)} - d^{\xi(m)}_{m}(\bmu)  + m(\varepsilon_1 + \varepsilon_2) + (\varepsilon_1 - \varepsilon_3)\\
&=c^{\hat{\xi(m)}}_m(\bmu) + m(\varepsilon_1 + \varepsilon_2) + (\varepsilon_1 - \varepsilon_3).
\end{align*}
We use the above, together with the specific values for \(\varepsilon_1, \varepsilon_2, \varepsilon_3\) taken from (\ref{dover}) and facts (i)--(v) noted at the start of the proof to compute \(c^{\hat{\xi(m)}}_m(\bnu) \) for the 13 possible cases for \((\textup{X},\textup{Y})\) in Lemma~\ref{downdiagtypeslem}. This yields:
\begin{align}
c^{\hat{\xi(m)}}_m(\bnu)=
\begin{cases}
c^{\hat{\xi(m)}}_m(\bmu)+1 & \textup{if } (\textup{X},\textup{Y}) = (\textup{A}_\varnothing, \textup{A}_\varnothing);\\
c^{\hat{\xi(m)}}_m(\bmu) + m +1 & \textup{if } (\textup{X},\textup{Y}) \in \{(\textup{A}_\varnothing, \textup{A}_+), (\textup{A}_\varnothing, \textup{B}_+)\};\\
c^{\hat{\xi(m)}}_m(\bmu) -m & \textup{if } (\textup{X},\textup{Y}) \in  \{(\textup{A}_-, \textup{A}_\varnothing),(\textup{B}_+, \textup{A}_\varnothing) \};\\
c^{\hat{\xi(m)}}_m(\bmu) & \textup{otherwise}.
\end{cases}
\label{chatcases2}
\end{align}
We have \(c^{\hat{\xi(m)}}_m(\bmu) \leq 0\) since \(\bmu \in \bOm_{\textup{MV}}\), so \(c^{\hat{\xi(m)}}_m(\bnu) \leq 0\) in all cases except perhaps for \begin{align*}
(\textup{X},\textup{Y}) \in  \{(\textup{A}_\varnothing, \textup{A}_\varnothing),(\textup{A}_\varnothing, \textup{A}_+), (\textup{A}_\varnothing, \textup{B}_+) \}.
\end{align*} We consider these cases now. 
First assume \((\textup{X},\textup{Y}) = (\textup{A}_\varnothing, \textup{A}_\varnothing)\). Then applying (\ref{ctodsame}) we have 
\begin{align*}
c^{\hat{\xi(m)}}_m(\bnu) &= (m+1)\nu^{\hat{\xi(m)}}_{m+1}  + m\nu_{m+1}^{\xi(m)}  - d^{\xi(m)}_{m}(\bnu) =  - d^{\xi(m)}_{m}(\bnu) \leq 0, 
\end{align*}
thanks to (\ref{dcdiff}), since \(c^{\xi(m-1)}_{m-1}(\bnu) = 0\) and \(d^{\xi(m)}_{m-1}(\bnu) = d^{\xi(m-1)}_{m-1}(\bnu) \geq 0\).
Finally assume \((\textup{X},\textup{Y}) \in  \{(\textup{A}_\varnothing, \textup{A}_+), (\textup{A}_\varnothing, \textup{B}_+) \}\).
In both settings we have \(\mu_{m+2}^{\xi(m)} > 0\) and \(\mu^{\hat{\xi(m)}}_{m+1} = 0\). Therefore
\begin{align*}
0 \geq c_{m+1}^{\hat{\xi(m)}}(\bmu) = c_m^{\hat{\xi(m)}}(\bmu) + (m+1)(\mu^{\xi(m)}_{m+2} - \mu^{\hat{\xi(m)}}_{m+1}) = c_m^{\hat{\xi(m)}}(\bmu) + (m+1)\mu^{\xi(m)}_{m+2} \geq c_m^{\hat{\xi(m)}}(\bmu) + m+1,
\end{align*}
so again in these cases we have \(c^{\hat{\xi(m)}}_m(\bnu) \leq 0\) in view of (\ref{chatcases2}), completing the proof.
\end{proof}

\subsection{Upward diagonalization}\label{updiag}
Now we repeat the approach of \S\ref{downdiag}, \ref{Hmapbicompssec}, but with diagonalization and segment type defintions adjusted to play well with the \(\mathcal{G}\) map instead.
Let \(j \in \ZZ_2\), and \(\bmu \in \bOm_{\geq 0, j}^{\textup{MV}}\).
We inductively construct a specific `upward' diagonalization \(\xi^{\bmu}_j\) of \(\bmu\) as follows, by setting
\begin{align}\label{updiagdef}
\xi_j^{\bmu}(m)
=
\begin{cases}
\hat j & \textup{if } m=0;\\
\xi^{\bmu}_j(m-1) &\textup{if }m> 0 \textup{ and } c_m^{\xi(m-1)}=0;\\
\hat{\xi_j^{\bmu}(m-1)} &\textup{if }m> 0 \textup{ and } c_m^{\xi(m-1)}<0.
\end{cases}
\end{align}
This choice of diagonalization gives us some control over the structure of its segments, as noted in the next subsection.

\subsubsection{Typifying segments, part deuce}\label{deftypesup}
Let \(\xi\) be a diagonalization of \(\bmu \in \bOm^{\textup{MV}}\). We define a function
\begin{align*}
\TypeH^{\xi, \bmu} : \ZZ_{\geq 1} \to \{\textup{A}^\varnothing, \textup{A}^+, \textup{A}^-, \textup{B}^+, \textup{B}^-, \textup{Z}\}
\end{align*}
given by setting
\begin{align*}
\TypeH^{\xi, \bmu}(m) = 
\begin{cases}
\textup{A}^\varnothing & \textup{if \(\xi(m-1) = \xi(m)\), 
\(\mu^{\xi(m)}_m = \mu^{\hat{\xi(m)}}_{m+1} = 0\), and \(\vartheta^{\xi(m)}_{m-1}(\bmu) , \vartheta^{\hat{\xi(m)}}_{m}(\bmu) \) are even.}\\
\textup{A}^+ & \textup{if \(\xi(m-1) = \xi(m)\) and \(\vartheta^{\xi(m)}_{m-1}(\bmu) , \vartheta^{\hat{\xi(m)}}_{m}(\bmu) \) are odd.}\\
\textup{A}^- & \textup{if \(\xi(m-1) = \xi(m)\) and \(\vartheta^{\xi(m)}_{m}(\bmu) , \vartheta^{\hat{\xi(m)}}_{m+1}(\bmu) \) are odd.}\\
\textup{B}^+ & \textup{if \(\xi(m-1) = \hat{\xi(m)}\), \(c^{\xi(m-1)}_{m}(\bmu) \neq 0\), \(\vartheta^{\hat{ \xi(m)}}_{m}(\bmu) \) is odd and \(\vartheta^{\xi(m)}_{m}(\bmu) \) is even.}\\
\textup{B}^- & \textup{if \(\xi(m-1) = \hat{\xi(m)}\), \(c^{ \xi(m-1)}_{m}(\bmu)  \neq  0\), \(\mu^{ \hat{\xi(m)}}_{m} > 0\), and  \(\vartheta^{ \hat{\xi(m)}}_{m+1}(\bmu) \), \(\vartheta^{ \xi(m)}_{m}(\bmu) \) are odd.}\\
\textup{Z} & \textup{otherwise},
\end{cases}
\end{align*}
and we refer to \(\TypeH^{\xi, \bmu}(m)\) as the {\em (upward) type of the \(m\)th segment of \(\xi\) in \(\bmu\)}.

\begin{lemma}\label{updiagtypeslem}
Let \(j \in \ZZ_2\), \(\bmu \in\bOm_{\geq 0, j}^{\textup{MV}}\).
For all \(m \in \ZZ_{\geq 1}\), we have \(\TypeH^{\xi_j^{\bmu}, \bmu}(m) \in \{\textup{A}^\varnothing, \textup{A}^+, \textup{A}^-, \textup{B}^+, \textup{B}^-\}\). Moreover,
 for all \(m \in \ZZ_{\geq 1}\), we have that 
 \begin{align*}
 (\TypeH^{\xi_j^{\bmu}, \bmu}(m), \TypeH^{\xi_j^{\bmu}, \bmu}(m+1))
 \end{align*}
 is equal to one of 
 \begin{align}
 &\{
(\textup{A}^\varnothing, \textup{A}^\varnothing), 
(\textup{A}^\varnothing, \textup{A}^-), 
(\textup{A}^\varnothing, \textup{B}^+), 
(\textup{A}^+, \textup{A}^\varnothing), 
(\textup{A}^+, \textup{A}^-), 
(\textup{A}^+, \textup{B}^+), 
(\textup{A}^-, \textup{A}^+),
\notag\\ & \qquad\;\;\;
(\textup{A}^-, \textup{B}^-),
(\textup{B}^+, \textup{A}^\varnothing),
(\textup{B}^+, \textup{A}^-),
(\textup{B}^+, \textup{B}^+),
(\textup{B}^-, \textup{A}^+),
(\textup{B}^-, \textup{B}^-)
 \}.\label{allcombotypesup}
 \end{align}
\end{lemma}

\begin{proof}
We omit the proof as it is very similar in spirit to the proof of Lemma~\ref{downdiagtypeslem}, with minor alteration---one proceeds by upward induction rather than downward, and so the base case of induction is somewhat different. Full details can be found in the \texttt{arXiv} version of the paper as explained in \S\ref{SS:ArxivVersion}.
\begin{answer}
Let \(\bmu \in \bOm_{\geq 0, j}^{\textup{MV}}\), and 
write \(\xi = \xi^{\bmu}_j\). 
First assume \(c_1^{\hat j} = 0\). Then \(\xi(0) = \xi(1) = \hat j\). By (\ref{cdef}) we have \(\mu^j_2 = \mu^{\hat j}_1\). We have
\begin{align*}
\vartheta_0^{\hat j} = 0, \qquad \vartheta_1^j = \begin{cases}
0 & \textup{if } \mu_1^j = 0;\\
2 & \textup{if } \mu_1^j > 0,
\end{cases}
\qquad
\vartheta_1^{\hat j} = \begin{cases}
1 & \textup{if } \mu_1^{\hat j} >0;\\
0 & \textup{if } \mu_1^{\hat j} = 0,
\end{cases}
\qquad
\vartheta_2^{j} = 
\begin{cases}
3 & \textup{if } \mu_2^j, \mu_1^j > 0;\\
1 &\textup{if } \mu_2^j >0, \mu_1^j = 0;\\
0 & \textup{if }\mu_2^j = 0.
\end{cases}
\end{align*}
If \(\mu_2^j = \mu_1^{\hat j} = 0\), then \(\vartheta_0^{\hat j} = 0\) and \(\vartheta_1^j\) is even, so 
\(\TypeH^{\xi, \bmu}(1) = \textup{A}^\varnothing\). If \(\mu_2^j = \mu_1^{\hat j} > 0\), then \(\vartheta_1^{\hat j}\), \(\vartheta_2^j\) are odd, so \(\TypeH^{\xi, \bmu}(1) = \textup{A}^-\).

Now assume \(c_1^{\hat j} < 0\). Then \(c_1^j = 0\), \(\xi(0) = \hat j\), \(\xi(1) = j\). By (\ref{cdef}) we have \((\mu^j_2 - \mu^{\hat j}_1) = c_1^{\hat j}<0\) so \(\mu^{\hat j}_1 > \mu^j_2 \geq 0\). Then 
\begin{align*}
\vartheta_1^{\hat j} = 1,
\qquad
\vartheta_1^j = 
\begin{cases}
2 & \textup{if }\mu_1^j > 0;\\
0 & \textup{if } \mu_1^j = 0.
\end{cases}
\end{align*}
which implies that  \(\TypeH^{\xi, \bmu}(1) = \textup{B}^+\). Thus in any case we have that \(\TypeH^{\xi, \bmu}(1) \in \{\textup{A}^\varnothing, \textup{A}^-, \textup{B}^+\}\).

Now assume that \(m\geq 1\) and  \(\TypeH^{\xi, \bmu}(m) \in \{\textup{A}^\varnothing, \textup{A}^+, \textup{A}^-, \textup{B}^+, \textup{B}^-\}\). We show that  \(\TypeH^{\xi, \bmu}(m+1) \in \{\textup{A}^\varnothing, \textup{A}^+, \textup{A}^-, \textup{B}^+, \textup{B}^-\}\) in keeping with (\ref{allcombotypesup}).
We consider the possibilities 
\begin{align*}
\mu^{\xi(m)}_{m+1} = \mu^{\hat{\xi(m)}}_{m+2} =0, \qquad \mu^{\xi(m)}_{m+1} = \mu^{\hat{\xi(m)}}_{m+2} > 0, \qquad \mu^{\xi(m)}_{m+1} \neq \mu^{\hat{\xi(m)}}_{m+2}
\end{align*}
as three individual cases.

\noindent{\em Case 1. Assume that \(\mu^{\xi(m)}_{m+1} = \mu^{\hat{\xi(m)}}_{m+2} =0\).} Then by (\ref{downonec}) we have \(c^{\xi(m)}_{m+1} = c^{\xi(m)}_{m} = 0\), and thus by (\ref{updiagdef}) we have \(\xi(m+1) = {\xi(m)}\). Now, note that
\begin{align*}
\vartheta^{\xi(m)}_m = \begin{cases}
1+ \vartheta^{\xi(m)}_{m-1} & \textup{if }\mu^{\xi(m)}_m > 0\\
0 & \textup{if }\mu^{\xi(m)}_m = 0
\end{cases}
\qquad
\textup{and}
\qquad
\vartheta^{\hat{\xi(m)}}_{m+1} = \begin{cases}
1+ \vartheta^{\hat{\xi(m)}}_{m} & \textup{if }\mu^{\hat{\xi(m)}}_{m+1} > 0\\
0 & \textup{if }\mu^{\hat{\xi(m)}}_{m+1} = 0.
\end{cases}
\end{align*}
We consider some subcases:
\begin{enumerate}[(1a)]
\item[(1a)] 
If \(\TypeH^{\xi,\bmu}(m) = \textup{A}^\varnothing\), then we have that 
\(\mu^{\xi(m)}_m = \mu^{\hat{\xi(m)}}_{m+1} = 0\), so \(\vartheta_m^{\xi(m)}, \vartheta_{m+1}^{\hat{\xi(m)}}\) are even, and thus \(\TypeH^{\xi,\bmu}(m+1) = \textup{A}^\varnothing\).
\item[(1b)]
If \(\TypeH^{\xi,\bmu}(m) = \textup{A}^+\), then we have that  \(\vartheta^{\xi(m)}_{m-1} , \vartheta^{\hat{\xi(m)}}_{m} \) are odd, and therefore in any case \(\vartheta_m^{\xi(m)}, \vartheta_{m+1}^{\hat{\xi(m)}}\) are even, and thus \(\TypeH^{\xi,\bmu}(m+1) = \textup{A}^\varnothing\).
\item[(1c)] If the \(\TypeH^{\xi,\bmu}(m) = \textup{B}^+\), then we have that \(\vartheta^{\hat{ \xi(m)}}_{m} \) is odd and \(\vartheta^{\xi(m)}_{m} \) is even. Thus \(\vartheta^{\xi(m)}_{m}\) is even in any case, and so \(\TypeH^{\xi,\bmu}(m+1) = \textup{A}^\varnothing\).
\item[(1d)]
If \(\TypeH^{\xi,\bmu}(m) \in\{ \textup{A}^-, \textup{B}^-\}\), then we have that \(\vartheta^{\xi(m)}_{m} , \vartheta^{\hat{\xi(m)}}_{m+1} \) are odd. are odd, which implies that 
\(\TypeH^{\xi,\bmu}(m-1) = \textup{A}^+\).
\end{enumerate}

\noindent{\em Case 2. Assume that \(\mu^{\xi(m)}_{m+1} = \mu^{\hat{\xi(m)}}_{m+2} >0\).} Then by (\ref{downonec}) we have \(c^{\xi(m)}_{m+1} = c^{\xi(m)}_{m} = 0\), and thus by (\ref{updiagdef}) we have \(\xi(m+1) = {\xi(m)}\). Now, note that
\begin{align*}
\vartheta^{\xi(m)}_{m+1} = \begin{cases}
2+ \vartheta^{\xi(m)}_{m-1} & \textup{if }\mu^{\xi(m)}_m > 0\\
1 & \textup{if }\mu^{\xi(m)}_m = 0
\end{cases}
\qquad
\textup{and}
\qquad
\vartheta^{\hat{\xi(m)}}_{m+2} = \begin{cases}
2+ \vartheta^{\hat{\xi(m)}}_{m} & \textup{if }\mu^{\hat{\xi(m)}}_{m+1} > 0\\
1 & \textup{if }\mu^{\hat{\xi(m)}}_{m+1} = 0
\end{cases}
\end{align*}
We consider some subcases:
\begin{enumerate}[(2a)]
\item 
If \(\TypeH^{\xi,\bmu}(m) = \textup{A}^\varnothing\), then we have that
\(\mu^{\xi(m)}_m = \mu^{\hat{\xi(m)}}_{m+1} = 0\), so \(\vartheta^{\xi(m)}_{m} , \vartheta^{\hat{\xi(m)}}_{m+1}=1\), and thus 
\(\TypeH^{\xi,\bmu}(m+1) = \textup{A}^-\).
\item 
If \(\TypeH^{\xi,\bmu}(m) = \textup{A}^+\), then we have that  \(\vartheta^{\xi(m)}_{m-1} , \vartheta^{\hat{\xi(m)}}_{m}\) are odd, and therefore in any case we have that \(\vartheta^{\xi(m+1)}_{m+1} , \vartheta^{\hat{\xi(m+1)}}_{m+2} \) are odd, so
\(\TypeH^{\xi,\bmu}(m+1) = \textup{A}^-\).
\item If 
\(\TypeH^{\xi,\bmu}(m) = \textup{B}^+\), then we have that \(\vartheta^{\hat{ \xi(m)}}_{m} \) is odd and \(\vartheta^{\xi(m)}_{m} \) is even. Thus either \(\mu^{\xi(m)}_{m} = 0\) or \(\vartheta^{\xi(m)}_{m-1}\) is odd. Thus \(\vartheta^{\xi(m+1)}_{m+1} , \vartheta^{\hat{\xi(m+1)}}_{m+2}\) are odd in any case, and so
\(\TypeH^{\xi,\bmu}(m+1) = \textup{A}^-\).
\item If
\(\TypeH^{\xi,\bmu}(m) \in\{ \textup{A}^-, \textup{B}^-\}\), then we have that \(\vartheta^{\xi(m)}_{m} , \vartheta^{\hat{\xi(m)}}_{m+1} \) are odd, which implies that \(\TypeH^{\xi,\bmu}(m+1) = \textup{A}^+\)
\end{enumerate}

\noindent{\em Case 3. Assume that \(\mu^{\xi(m)}_{m+1}  \neq \mu^{\hat{\xi(m)}}_{m+2}\).} Then by (\ref{downonec})  we have \(c^{\xi(m)}_{m+1} \neq 0\). Thus, since \(\bmu \in \bOm^{\textup{MV}}\) we have that \(c^{\xi(m)}_{m+1} < 0\), so \(c^{\hat{\xi(m)}}_{m+1} = 0\) and \(\mu^{\xi(m)}_{m+1} > \mu^{\hat{\xi(m)}}_{m+2} \geq 0\). By (\ref{updiagdef}) we have \(\xi(m+1) = \hat{\xi(m)}\). Note that
\begin{align*}
\vartheta^{\xi(m)}_{m+1}=
\begin{cases}
2+ \vartheta_{m-1}^{\xi(m)} & \textup{if } \mu_m^{\xi(m)} > 0\\
1 & \textup{if } \mu_m^{\xi(m)} =0.
\end{cases}
\qquad
\textup{and}
\qquad
\vartheta^{\hat{\xi(m)}}_{m+1} =
\begin{cases}
1 + \vartheta^{\hat{\xi(m)}}_{m} &\textup{if } \mu_{m+1}^{\hat{\xi(m)}} > 0\\
0 & \textup {if } \mu_{m+1}^{\hat{\xi(m)}} = 0.
\end{cases}
\end{align*}

We again consider some subcases:
\begin{enumerate}[(3a)]
\item 
If \(\TypeH^{\xi,\bmu}(m) = \textup{A}^\varnothing\), then 
\(\mu^{\xi(m)}_m = \mu^{\hat{\xi(m)}}_{m+1} = 0\).
We have then that \(\vartheta^{\hat{ \xi(m+1)}}_{m+1} = \vartheta^{\xi(m)}_{m+1} =1 \) and \(\vartheta^{\xi(m+1)}_{m+1} =\vartheta^{\hat{\xi(m)}}_{m+1} = 0\), so \(\TypeH^{\xi,\bmu}(m+1) = \textup{B}^+\).

\item If 
\(\TypeH^{\xi,\bmu}(m) = \textup{A}^+\), then 
\(\vartheta^{\xi(m)}_{m-1} , \vartheta^{\hat{\xi(m)}}_{m} \) are odd.  Then in any case we have that \(\vartheta^{\hat{ \xi(m+1)}}_{m+1} \) is odd and \(\vartheta^{\xi(m+1)}_{m+1} \) is even, so
\(\TypeH^{\xi,\bmu}(m+1) = \textup{B}^+\).

\item If 
\(\TypeH^{\xi,\bmu}(m) = \textup{A}^-\), then
\(\vartheta^{\xi(m)}_{m} , \vartheta^{\hat{\xi(m)}}_{m+1} \) are odd. Since \(\vartheta^{\hat{\xi(m)}}_{m+1}\) is odd, we have \(\mu^{\hat{\xi(m)}}_{m+1}> 0\). Therefore by (\ref{cdef}) we have 
\begin{align*}
0 \geq c^{\hat{\xi(m)}}_m = c^{\hat{\xi(m)}}_{m+1} - (m+1)(\mu^{\xi(m)}_{m+2} - \mu^{\hat{\xi(m)}}_{m+1}) = - (m+1)(\mu^{\xi(m)}_{m+2} - \mu^{\hat{\xi(m)}}_{m+1}),
\end{align*}
which implies that \(\mu^{\xi(m)}_{m+2} \geq \mu^{\hat{\xi(m)}}_{m+1} > 0\). Since \(\mu^{\xi(m)}_{m+2}, \mu^{\xi(m)}_{m+1} >0\) and \(\vartheta^{\xi(m)}_m\) is odd, we have that \(\vartheta^{\xi(m)}_{m+2}\) is odd. Thus we have that \(\vartheta^{ \hat{\xi(m+1)}}_{m+2} \), \(\vartheta^{ \xi(m+1)}_{m+1} \) are odd, so \(\TypeH^{\xi,\bmu}(m+1) = \textup{B}^-\).

\item If 
\(\TypeH^{\xi,\bmu}(m) = \textup{B}^+\),  we have that  \(\vartheta^{\hat{ \xi(m)}}_{m} \) is odd and \(\vartheta^{\xi(m)}_{m} \) is even. Since \(\vartheta^{\xi(m)}_{m} \) is even, either \(\mu_{m}^{\xi(m)} = 0\) or  \(\vartheta^{\xi(m)}_{m-1}\) is odd. In any case then it follows that \(\vartheta^{\xi(m)}_{m+1}\) is odd. Thus \(\vartheta^{\hat{ \xi(m+1)}}_{m+1} \) is odd and \(\vartheta^{\xi(m+1)}_{m+1} \) is even. Therefore
\(\TypeH^{\xi,\bmu}(m+1) = \textup{B}^+\).

\item
If
\(\TypeH^{\xi,\bmu}(m) = \textup{B}^-\), we have that \(\vartheta^{ \hat{\xi(m)}}_{m+1} \), \(\vartheta^{ \xi(m)}_{m} \) are odd. Since \(\vartheta^{\hat{\xi(m)}}_{m+1}\) is odd, we have \(\mu^{\hat{\xi(m)}}_{m+1}> 0\). Therefore by (\ref{cdef}) we have 
\begin{align*}
0 \geq c^{\hat{\xi(m)}}_m = c^{\hat{\xi(m)}}_{m+1} - (m+1)(\mu^{\xi(m)}_{m+2} - \mu^{\hat{\xi(m)}}_{m+1}) = - (m+1)(\mu^{\xi(m)}_{m+2} - \mu^{\hat{\xi(m)}}_{m+1}),
\end{align*}
which implies that \(\mu^{\xi(m)}_{m+2} \geq \mu^{\hat{\xi(m)}}_{m+1} > 0\). Since \(\mu^{\xi(m)}_{m+2}, \mu^{\xi(m)}_{m+1} >0\) and \(\vartheta^{\xi(m)}_m\) is odd, we have that \(\vartheta^{\xi(m)}_{m+2}\) is odd. Thus we have that \(\vartheta^{ \hat{\xi(m+1)}}_{m+2} \), \(\vartheta^{ \xi(m+1)}_{m+1} \) are odd, so \(\TypeH^{\xi,\bmu}(m+1) = \textup{B}^-\).
\end{enumerate}

Thus in every case we have that \( (\TypeH^{\xi, \bmu}(m), \TypeH^{\xi, \bmu}(m+1))\) is in the set (\ref{allcombotypes}), completing the proof. 
\end{answer}
\end{proof}

\subsection{The \(\mathcal{G}\) map and bicompositions}

\begin{lemma}\label{Gisom}
The map \(\mathcal{G}\) restricts to a map of MV bicompositions \(\mathcal{G}: \bOm_{\geq 0,\bullet}^{\textup{MV}} \to \bOm^{\textup{MV}}_{\geq 1, +}\). Moreover, if \(\bmu \in \bOm_{\geq 0,j}^{\textup{MV}} \), then \(\xi^{\bmu}_j\) is a diagonalization of \(\bnu = \mathcal{G}(\bmu)\), and:
\begin{align}
d^{\xi^{\bmu}_j(m-1)}_{m-1}(\bnu) &= 
\begin{cases}
d^{\xi^{\bmu}_j(m-1)}_{m-1}(\bmu)-1 &\textup{if } \TypeH^{\xi^{\bmu}_j, \bmu}(m) \in \{\textup{A}^+, \textup{B}^-\}\\
d^{\xi^{\bmu}_j(m-1)}_{m-1}(\bmu)+1 &\textup{if } \TypeH^{\xi^{\bmu}_j, \bmu}(m)\in \{\textup{A}^\varnothing, \textup{A}^-, \textup{B}^+\}\\
\end{cases}\label{updunder}\\
d^{\xi_{\bmu}(m)}_{m}(\bnu) &= 
\begin{cases}
d^{\xi^{\bmu}_j(m)}_{m}(\bmu)-1 &\textup{if } \TypeH^{\xi^{\bmu}_j, \bmu}(m) \in \{ \textup{A}^-, \textup{B}^-\}\\
d^{\xi^{\bmu}_j(m)}_{m}(\bmu)+1 &\textup{if } \TypeH^{\xi^{\bmu}_j, \bmu}(m) \in \{\textup{A}^\varnothing, \textup{A}^+, \textup{B}^+\}.\\
\end{cases}\label{updover}
\end{align}
\end{lemma}

\begin{proof}
We omit this proof as well as it proceeds along the same lines as Lemma~\ref{Hisom}, utilizing Lemma~\ref{updiagtypeslem} in similar fashion to the way the proof of Lemma~\ref{Hisom} utilizes Lemma~\ref{downdiagtypeslem}. Full details can be found in the \texttt{arXiv} version of the paper as explained in \S\ref{SS:ArxivVersion}.
\begin{answer}
Let \(\bmu \in \bOm_{\geq 0,j}^{\textup{MV}} \). 
We will write  \(\xi:= \xi^{\bmu}_j\). For \(i \in \ZZ_2\), \(m \in \ZZ_{\geq 1}\), the value \(\nu^i_m\) is directly related to the type of the segment in the diagonalization \(\xi\) of \(\bmu\) which the edge corresponding to \(\mu^i_m\) belongs to. Unpacking the definitions in \S\ref{deftypesup} alongside (\ref{Gcases}), we have the following facts, which we will use repeatedly:
\begin{enumerate}[(i)]
\item If \(\TypeH^{\xi, \bmu}(m) = \textup{A}^\varnothing\) then \(\nu^{\xi(m)}_m = \nu^{\hat{\xi(m)}}_{m+1} = 0\);
\item If \(\TypeH^{\xi, \bmu}(m) = \textup{A}^+\) then \(\nu^{\xi(m)}_m =\mu^{\xi(m)}_m +1\) and \(\nu^{\hat{\xi(m)}}_{m+1} = \mu^{\hat{\xi(m)}}_{m+1} +1\);
\item If \(\TypeH^{\xi, \bmu}(m) = \textup{A}^-\) then \(\nu^{\xi(m)}_m =\mu^{\xi(m)}_m -1\) and \(\nu^{\hat{\xi(m)}}_{m+1} = \mu^{\hat{\xi(m)}}_{m+1} -1\);
\item If \(\TypeH^{\xi, \bmu}(m) = \textup{B}^+\) then \(\nu^{\hat{\xi(m)}}_m = \mu^{\hat{\xi(m)}}_m-1\) and \(\nu^{\hat{\xi(m)}}_{m+1} = \mu^{\hat{\xi(m)}}_{m+1}+1\);
\item If \(\TypeH^{\xi, \bmu}(m) = \textup{B}^-\) then \(\nu^{\hat{\xi(m)}}_m = \mu^{\hat{\xi(m)}}_m+1\) and \(\nu^{\hat{\xi(m)}}_{m+1} = \mu^{\hat{\xi(m)}}_{m+1}-1\).
\end{enumerate}
We proceed via a number of claims.\\

\noindent {\em Claim I: Let \(m \in \ZZ_{\geq 1}\), and set \(m' = m+1\). If (\ref{updover}) holds for \(m\), then (\ref{updunder}) holds for \(m'\).}
We have from Lemma~\ref{updiagtypeslem} that \(\TypeH^{\xi, \bmu}(m') \in \{\textup{A}^\varnothing, \textup{A}^-, \textup{B}^+\}\) if and only if \(\TypeH^{\xi, \bmu}(m) \in \{\textup{A}^\varnothing, \textup{A}^+, \textup{B}^+\}\), and \(\TypeH^{\xi, \bmu}(m') \in \{\textup{A}^+, \textup{B}^-\}\) if and only if \(\TypeH^{\xi, \bmu}(m) \in \{\textup{A}^-, \textup{B}^-\}\), which immediately implies the claim.\\

\noindent {\em Claim II: The equalities \(c_{m-1}^{\xi(m-1)}(\bnu) = 0\), and (\ref{updunder}) hold when \(m=1\).} 
The first equality is obvious, since \(c_0^{\xi(m-1)}(\bnu)\) is definitionally zero. We have as in the proof of Lemma~\ref{updiagtypeslem} that \(\TypeH^{\xi, \bmu}(1) \in \{\textup{A}^\varnothing, \textup{A}^-, \textup{B}^+\}\). We have \(\xi(0) = \hat j\), \(\mu_0^j = 1\), and \(\mu_0^{\hat j} = 0\). Note that \(\vartheta_0^j(\bmu) =1\), so 
following (\ref{Gcases}) we have
\begin{align*}
d^{\xi(0)}_0(\bnu) = d^{\hat j}_0(\bnu) =  \nu_1^{j} = \mu_1^j+1 = d^{\xi(0)}_0(\bmu) + 1,
\end{align*}
which agrees with (\ref{updunder}) when \(m=1\). Thus we have verified Claim II.\\

\noindent {\em Claim III: If \(c^{\xi(m-1)}_{m-1}(\bnu) = 0\) and (\ref{updunder}) hold for a given \(m \in \ZZ_{\geq 1}\), then \(c^{\xi(m)}_{m}(\bnu) = 0\) and (\ref{updover}) both hold for \(m\).} 
Assume that \(c^{\xi(m-1)}_{m-1}(\bnu) = 0\) and (\ref{updunder}) hold. We have by Lemma~\ref{updiagtypeslem} that \(\TypeH^{\xi, \bmu}(m) \in \{\textup{A}^\varnothing, \textup{A}^+, \textup{A}^-, \textup{B}^+, \textup{B}^-\}\). We consider these cases separately.

\noindent{\em Case 1. Assume 
\(\TypeH^{\xi,\bmu}(m) = \textup{A}^\varnothing\).} Then \(\xi(m-1) = \xi(m)\), so in view of (\ref{updunder}) and fact (i) above, we have
\begin{align*}
\nu_m^{\xi(m)} = 0 \qquad \nu_{m+1}^{\hat{\xi(m)}} = 0 \qquad d_{m-1}^{\xi(m-1)}(\bnu) = d_{m-1}^{\xi(m-1)}(\bmu) + 1
\end{align*}
so in consideration of (\ref{downonec}, \ref{downoned}) we have
\begin{align*}
c_{m}^{\xi(m)}(\bnu) &= c_{m-1}^{\xi(m-1)}(\bnu) + m(\nu_{m+1}^{\hat{\xi(m)}} - \nu_m^{\xi(m)}) = 0\\
d_{m}^{\xi(m)}(\bnu) &= d_{m-1}^{\xi(m-1)}(\bnu) - (m-1)\nu^{\xi(m)}_m + (m+1)\nu^{\hat{\xi(m)}}_{m+1} \\
&= d_{m-1}^{\xi(m-1)}(\bmu)+1 + (m-1)\mu^{\xi(m)}_m - (m+1)\mu^{\hat{\xi(m)}}_{m+1} =  d_{m}^{\xi(m)}(\bmu) + 1. 
\end{align*}

\noindent{\em Case 2. Assume 
\(\TypeH^{\xi,\bmu}(m) = \textup{A}^+\).} Then \(\xi(m-1) = \xi(m)\), 
so in view of (\ref{updunder}) and fact (ii) above, we have
\begin{align*}
\nu_m^{\xi(m)} =\mu_m^{\xi(m)} +1 \qquad \nu_{m+1}^{\hat{\xi(m)}} = \mu_{m+1}^{\hat{\xi(m)}}+1 \qquad d_{m-1}^{\xi(m-1)}(\bnu) = d_{m-1}^{\xi(m-1)}(\bmu) - 1
\end{align*}
so in consideration of (\ref{downonec}, \ref{downoned}) we have
\begin{align*}
c_{m}^{\xi(m)}(\bnu) &= c_{m-1}^{\xi(m-1)}(\bnu) + m(\nu_{m+1}^{\hat{\xi(m)}} - \nu_m^{\xi(m)}) = c_{m-1}^{\xi(m-1)}(\bmu) + m(\mu_{m+1}^{\hat{\xi(m)}} - \mu_m^{\xi(m)}) = c_{m}^{\xi(m)}(\bmu) = 0 \\
d_{m}^{\xi(m)}(\bnu) &= d_{m-1}^{\xi(m-1)}(\bnu) - (m-1)\nu^{\xi(m)}_m + (m+1)\nu^{\hat{\xi(m)}}_{m+1}\\
&=d_{m-1}^{\xi(m-1)}(\bmu) - (m-1)\mu^{\xi(m)}_m + (m+1)\mu^{\hat{\xi(m)}}_{m+1} +1 = d_{m}^{\xi(m)}(\bmu) +1.
\end{align*}

\noindent{\em Case 3. Assume 
\(\TypeH_{\xi,\bmu}(m) = \textup{A}^-\).} Then \(\xi(m-1) = \xi(m)\), so in view of (\ref{updunder}) and fact (iii) above, we have
\begin{align*}
\nu_m^{\xi(m)} =\mu_m^{\xi(m)} -1 \qquad \nu_{m+1}^{\hat{\xi(m)}} = \mu_{m+1}^{\hat{\xi(m)}}-1 \qquad d_{m-1}^{\xi(m-1)}(\bnu) = d_{m-1}^{\xi(m-1)}(\bmu) + 1,
\end{align*}
so in consideration of (\ref{downonec}, \ref{downoned}) we have
\begin{align*}
c_{m}^{\xi(m)}(\bnu) &= c_{m-1}^{\xi(m-1)}(\bnu) + m(\nu_{m+1}^{\hat{\xi(m)}} - \nu_m^{\xi(m)}) = c_{m-1}^{\xi(m-1)}(\bmu) + m(\mu_{m+1}^{\hat{\xi(m)}} - \mu_m^{\xi(m)}) = c_{m}^{\xi(m)}(\bmu) = 0 \\
d_{m}^{\xi(m)}(\bnu) &= d_{m-1}^{\xi(m-1)}(\bnu) - (m-1)\nu^{\xi(m)}_m + (m+1)\nu^{\hat{\xi(m)}}_{m+1}\\
&=d_{m-1}^{\xi(m-1)}(\bmu) - (m-1)\mu^{\xi(m)}_m + (m+1)\mu^{\hat{\xi(m)}}_{m+1} -1 = d_{m}^{\xi(m)}(\bmu) -1.
\end{align*}

\noindent{\em Case 4. Assume
\(\TypeH_{\xi,\bmu}(m) = \textup{B}^+\).} Then  \(\xi(m-1) = \hat{\xi(m)}\), so in view of (\ref{updunder}) and fact (iv) above, we have
\begin{align*}
\nu_{m}^{\hat{\xi(m)}} =\mu_{m}^{\hat{\xi(m)}} -1 \qquad \nu_{m+1}^{\hat{\xi(m)}} =\mu_{m+1}^{\hat{\xi(m)}} +1 \qquad d_{m-1}^{\xi(m-1)}(\bnu) = d_{m-1}^{\xi(m-1)}(\bmu) + 1,
\end{align*}
so in consideration of (\ref{ctoddown}, \ref{dtocdown}) we have
\begin{align*}
c_{m}^{\xi(m)}(\bnu) &=  (m-1)\nu_m^{\hat{\xi(m)}} + m\nu_{m+1}^{\hat{\xi(m)}} -d_{m-1}^{\hat{\xi(m)}}(\bnu)\\
&=(m-1)\mu_m^{\hat{\xi(m)}} + m\mu_{m+1}^{\hat{\xi(m)}} -d_{m-1}^{\hat{\xi(m)}}(\bmu)= c_{m}^{\xi(m)}(\bmu) = 0\\
d_{m}^{\xi(m)}(\bnu) &= m\nu^{\hat{\xi(m)}}_m + (m+1)\nu^{\hat{\xi(m)}}_{m+1} - c^{\hat{\xi(m)}}_{m-1}(\bnu)\\
&= m\mu^{\hat{\xi(m)}}_m + (m+1)\mu^{\hat{\xi(m)}}_{m+1} - c_{m-1}^{\hat{\xi(m)}}(\bmu) + 1 = d_{m}^{\xi(m)}(\bmu)+1.
\end{align*}

\noindent{\em Case 5. Assume
\(\TypeH^{\xi,\bmu}(m) = \textup{B}^-\).} Then 
\(\xi(m-1) = \hat{\xi(m)}\), \(c^{ \xi(m-1)}_{m}(\bmu)  \neq  0\), 
so in view of (\ref{updunder}) and fact (v) above, we have
\begin{align*}
\nu_{m}^{\hat{\xi(m)}} =\mu_{m}^{\hat{\xi(m)}} +1 \qquad \nu_{m+1}^{\hat{\xi(m)}} =\mu_{m+1}^{\hat{\xi(m)}} -1 \qquad d_{m-1}^{\xi(m-1)}(\bnu) = d_{m-1}^{\xi(m-1)}(\bmu) - 1,
\end{align*}
so in consideration of (\ref{ctoddown}, \ref{dtocdown}) we have
\begin{align*}
c_{m}^{\xi(m)}(\bnu) &=  (m-1)\nu_m^{\hat{\xi(m)}} + m\nu_{m+1}^{\hat{\xi(m)}} -d_{m-1}^{\hat{\xi(m)}}(\bnu)\\
&=(m-1)\mu_m^{\hat{\xi(m)}} + m\mu_{m+1}^{\hat{\xi(m)}} -d_{m-1}^{\hat{\xi(m)}}(\bmu)= c_{m}^{\xi(m)}(\bmu) = 0\\
d_{m}^{\xi(m)}(\bnu) &= m\nu^{\hat{\xi(m)}}_m + (m+1)\nu^{\hat{\xi(m)}}_{m+1} - c^{\hat{\xi(m)}}_{m-1}(\bnu)\\
&= m\mu^{\hat{\xi(m)}}_m + (m+1)\mu^{\hat{\xi(m)}}_{m+1} - c_{m-1}^{\hat{\xi(m)}}(\bmu) - 1 = d_{m}^{\xi(m)}(\bmu)-1.
\end{align*}
This completes the proof of Claim III.\\

Combining Claims I, II, III, we have that \(c^{\xi(m)}_m(\bnu) = 0\), and (\ref{updunder}, \ref{updover}) hold for all \(m \in \ZZ_{\geq 1}\). 
Moreover, it follows from (\ref{Gcases}) and the fact that \(\bmu \in \bOm^{\textup{MV}}_{\geq 1, j}\) that \(\nu^j_1 >0\), and \(\nu^i_k = 0\) for \(i \in \ZZ_2\), \(k \leq 0\). To show \(\bnu \in \bOm^{\textup{MV}}_{\geq 1, +}\), 
it thus remains to finish the check that \(\bnu\) is an MV bicomposition; i.e., we must check that \(c^{\hat{\xi(m)}}_m(\bnu) \leq 0\) for all \(m \in \ZZ_{\geq 1}\).
Fix \(m \in \ZZ_{\geq 1}\), and set
\begin{align*}
(\textup{X},\textup{Y}) =  (\TypeH^{\xi, \bmu}(m), \TypeH^{\xi, \bmu}(m+1)).
\end{align*}
By (\ref{Gcases}, \ref{updover}), we have
\begin{align*}
\nu_{m+1}^{\hat{\xi(m)}} = \mu_{m+1}^{\hat{\xi(m)}}+ \varepsilon_1
\qquad
\nu_{m+1}^{\xi(m)} = \mu_{m+1}^{\xi(m)} + \varepsilon_2
\qquad
d_m^{\xi(m)} = d_m^{\xi(m)} + \varepsilon_3
\end{align*}
for some \(\varepsilon_1, \varepsilon_2, \varepsilon_3 \in \{-1,0,1\}\). Therefore applying  (\ref{ctodsame}) we have
\begin{align*}
c^{\hat{\xi(m)}}_m(\bnu) &= (m+1)\nu^{\hat{\xi(m)}}_{m+1}  + m\nu_{m+1}^{\xi(m)}  - d^{\xi(m)}_{m}(\bnu)  \\
&=(m+1)(\mu^{\hat{\xi(m)}}_{m+1} + \varepsilon_1) + m(\mu_{m+1}^{\xi(m)} + \varepsilon_2) - (d^{\xi(m)}_{m}(\bmu) + \varepsilon_3)
\\&=(m+1)\mu^{\hat{\xi(m)}}_{m+1}+ m\mu_{m+1}^{\xi(m)} - d^{\xi(m)}_{m}(\bmu)  + m(\varepsilon_1 + \varepsilon_2) + (\varepsilon_1 - \varepsilon_3)\\
&=c^{\hat{\xi(m)}}_m(\bmu) + m(\varepsilon_1 + \varepsilon_2) + (\varepsilon_1 - \varepsilon_3).
\end{align*}
We use the above, together with the specific values for \(\varepsilon_1, \varepsilon_2, \varepsilon_3\) taken from (\ref{updover}) and facts (i)--(v) noted at the start of the proof to compute \(c^{\hat{\xi(m)}}_m(\bnu) \) for the 13 possible cases for \((\textup{X},\textup{Y})\) in Lemma~\ref{updiagtypeslem}. This yields:
\begin{align}
c^{\hat{\xi(m)}}_m(\bnu)=
\begin{cases}
c^{\hat{\xi(m)}}_m(\bmu)-1 & \textup{if } (\textup{X},\textup{Y}) = (\textup{A}^\varnothing, \textup{A}^\varnothing);\\
c^{\hat{\xi(m)}}_m(\bmu) - m -1 & \textup{if } (\textup{X},\textup{Y}) \in \{(\textup{A}^\varnothing, \textup{A}^-), (\textup{A}^\varnothing, \textup{B}^+)\};\\
c^{\hat{\xi(m)}}_m(\bmu) +m & \textup{if } (\textup{X},\textup{Y}) \in  \{(\textup{A}^+, \textup{A}^\varnothing),(\textup{B}^+, \textup{A}^\varnothing) \};\\
c^{\hat{\xi(m)}}_m(\bmu) & \textup{otherwise}.
\end{cases}
\label{chatcases3}
\end{align}
We have \(c^{\hat{\xi(m)}}_m(\bmu) \leq 0\) since \(\bmu \in \bOm_{\textup{MV}}\), so \(c^{\hat{\xi(m)}}_m(\bnu) \leq 0\) in all cases except perhaps \((\textup{X},\textup{Y}) \in  \{(\textup{A}^+, \textup{A}^\varnothing),(\textup{B}^+, \textup{A}^\varnothing) \}\). We consider this case now. In both settings we have \(\mu_m^{\hat{\xi(m)}} > 0\) and \(\mu^{\xi(m)}_{m+1} = 0\). Therefore
\begin{align*}
0 \geq c_{m-1}^{\hat{\xi(m)}}(\bmu) = c_m^{\hat{\xi(m)}}(\bmu) - m(\mu^{\xi(m)}_{m+1} - \mu^{\hat{\xi(m)}}_m) = c_m^{\hat{\xi(m)}}(\bmu) + m\mu^{\hat{\xi(m)}}_m \geq c_m^{\hat{\xi(m)}}(\bmu) + m,
\end{align*}
so again in these cases we have \(c^{\hat{\xi(m)}}_m(\bnu) \leq 0\) in view of (\ref{chatcases3}), completing the proof.
\end{answer}
\end{proof}

\begin{corollary}\label{GHisomMV}
The functions 
\begin{align*}
\mathcal{G}: \bOm_{\geq 0,\bullet}^{\textup{MV}} \to \bOm_{\geq 1,+}^{\textup{MV}} \qquad
\textup{and}
\qquad \mathcal{H}: \bOm_{\geq 1,+}^{\textup{MV}} \to \bOm_{\geq 0,\bullet}^{\textup{MV}}
\end{align*}
are mutual inverses. 
\end{corollary}
\begin{proof}
Follows immediately by Lemmas~\ref{GHisom},  \ref{Hisom}, \ref{Gisom}.
\end{proof}

\subsection{Directed binary graphs}\label{dgdefs}
Let \(\textup{DG}(\bOm^{\textup{MV}}_{\geq 1})\) be the directed graph whose vertices are the MV bicompositions \(\bOm^{\textup{MV}}_{\geq 1}\), with an oriented edge \(\bmu \to \bar{\mathcal{G}}^i\bmu\) for each \(\bmu \in \bOm^{\textup{MV}}_{\geq 1}\), \(i \in \ZZ_2\). Note that if \(\bmu \in \bOm^{\textup{MV}}_{\geq 1}(n)\), then \(\bar{\mathcal{G}}^i\bmu \in \bOm^{\textup{MV}}_{\geq 1}(n+1)\), so \(\bmu \neq \bar{\mathcal{G}}^i \bmu\). By definition of the map \(\bar{\mathcal{G}}^i\) we also have
\begin{align}\label{diff01}
(\bar{\mathcal{G}}^0\bmu)^0_1 = \mu^0_1 + 1;
\qquad
(\bar{\mathcal{G}}^0\bmu)^1_1 \neq \mu^1_1 + 1;
\qquad
(\bar{\mathcal{G}}^1\bmu)^0_1 \neq \mu^0_1 + 1;
\qquad
(\bar{\mathcal{G}}^1\bmu)^1_1 = \mu^1_1 + 1,
\end{align}
and therefore \(\bar{\mathcal{G}}^0 \bmu \neq \bar{\mathcal{G}}^1 \bmu\) for all \(\bmu \in \bOm^{\textup{MV}}_{\geq 1}\).  Thus \(\textup{DG}(\bOm^{\textup{MV}}_{\geq 1})\) is a simple directed graph. 

Let \(\textup{DG}(\Theta)\) be the simple directed graph whose vertices are the binary words \(\Theta\), with an oriented edge \(\beps \to\mathcal{A}^i\beps\) for each \(\beps \in \Theta\), \(i \in \ZZ_2\).

\begin{theorem}\label{DGthm}
Let \(i \in \ZZ_2\). We have inverse isomorphisms of directed graphs
\begin{align*}
\mathcal{J}: \textup{DG}(\Theta) \to  \textup{DG}(\bOm^{\textup{MV}}_{\geq 1}) 
\qquad
\textup{and}
\qquad
\mathcal{K}^i:   \textup{DG}(\bOm^{\textup{MV}}_{\geq 1}) \to  \textup{DG}(\Theta).
\end{align*}
\end{theorem}

\begin{proof}
Let \(n \in \ZZ_{\geq 1}\), and \(\bmu \in \bOm^{\textup{MV}}_{\geq 1}(n)\). We proceed via a number of claims.\\

\noindent{\em Claim I.
We have \(\mathcal{J}(\Theta(n)) \supseteq \bOm^{\textup{MV}}_{\geq 1}(n)\).} We go by induction on \(n\). The base case \(n=0\) is immediate, since 
\begin{align}
\label{bc1}\bOm^{\textup{MV}}_{\geq 1}(0) =  \bOm_{\geq 1}(0) = \{(\varnothing, \varnothing)\} = \{\mathcal{J}\varnothing\}.
\end{align}
Now assume the claim holds for \(n \geq0\). Let \(\bmu \in \bOm^{\textup{MV}}_{\geq 1}(n+1)\). By Lemma~\ref{GHisom} we have \(\mathcal{H}\bmu \in \bOm^{\textup{MV}}_{\geq 0, j}\) for some \(j \in \ZZ_2\), and so \(\textup{one}_0^j \circ \textup{zero}_0 \circ \mathcal{H} \bmu = \mathcal{H}\bmu\). Thus
\begin{align}\label{calc3434}
\bar{\mathcal{G}}^j\bar{\mathcal{H}} \bmu = \mathcal{G} \circ \textup{one}_0^j \circ \textup{zero}_0 \circ \mathcal{H} \bmu = \mathcal{G} \mathcal{H} \bmu = \bmu,
\end{align}
by Lemma~\ref{GHisom}.
Since \(\bar{\mathcal{H}} \bmu \in  \bOm^{\textup{MV}}_{\geq 1}(n)\), by induction assumption there exists \(\beps \in \Theta(n)\) such that \(\mathcal{J} \beps = \bar{\mathcal{H}}\bmu\). But then \(\mathcal{A}^j \beps \in \Theta(n+1)\), so we have by Lemma~\ref{commaps} and (\ref{calc3434}) that 
\begin{align*}
\mathcal{J}\mathcal{A}^j \beps = \bar{ \mathcal{G}}^j \mathcal{J} \beps =  \bar{ \mathcal{G}}^j \bar{\mathcal{H}} \bmu = \bmu,
\end{align*}
and so \(\mathcal{J}(\Theta(n+1)) \supseteq \bOm^{\textup{MV}}_{\geq 1}(n+1)\), proving the claim.\\

\noindent{\em Claim II. We have  \(\mathcal{J}(\Theta(n)) \subseteq \bOm^{\textup{MV}}_{\geq 1}(n)\).} We again go by induction on \(n\), with the base case \(n=0\) following from (\ref{bc1}). Assume the claim holds for \(n\geq 0\). Let \(\beps \in \Theta(n+1)\). Then \(\mathcal{R} \beps \in \Theta(n)\), so by induction assumption \(\mathcal{J}\mathcal{R} \beps \in \bOm^{\textup{MV}}_{\geq 1}(n)\). Then \(\textup{one}_0^{\eps_{n+1}} \circ \mathcal{J}\mathcal{R} \beps \in \bOm^{\textup{MV}}_{\geq 0, \eps_{n+1}}(n+1)\), so by Lemma~\ref{Gisom} we have that \(\mathcal{G} \circ \textup{one}_0^{\eps_{n+1}} \circ \mathcal{J}\mathcal{R} \beps \in \bOm^{\textup{MV}}_{\geq 1}(n+1)\). 
But then Lemma~\ref{commaps} implies that 
\begin{align*}
\mathcal{J} \beps = \mathcal{J} \mathcal{A}^{\eps_{n+1}} \mathcal{R} \beps = \bar{\mathcal{G}}^{\eps_{n+1}} \mathcal{J} \mathcal{R} \beps = \mathcal{G} \circ \textup{one}_0^{\eps_{n+1}} \circ \mathcal{J}\mathcal{R} \beps \in \bOm^{\textup{MV}}_{\geq 1}(n+1), 
\end{align*}
and so \(\mathcal{J}(\Theta(n+1)) \subseteq \bOm^{\textup{MV}}_{\geq 1}(n+1)\), proving the claim.\\

\noindent{\em Claim III. The map \(\mathcal{J}: \Theta(n) \to \bOm_{\geq 1}^{\textup{MV}}(n)\) is injective.}
We again go by induction on \(n\), with the base case \(n=0\) following from (\ref{bc1}). Assume the claim holds for \(n \geq 0\), and let \(\beps, \brho \in \Theta(n+1)\), and assume that \(\mathcal{J} \beps = \mathcal{J} \brho\). Then by Lemma~\ref{commaps} we have
\begin{align*}
\mathcal{J} \mathcal{R} \beps = \bar{\mathcal{H}} \mathcal{J} \beps = \bar{\mathcal{H}} \mathcal{J} \brho = \mathcal{J} \mathcal{R} \brho.
\end{align*}
Since \(\mathcal{R} \beps, \mathcal{R} \brho \in \Theta(n)\), the induction assumption implies that \(\mathcal{R} \beps = \mathcal{R} \brho\). Therefore, writing \(\bgamma = \mathcal{R} \beps = \mathcal{R} \brho\), we have \(\mathcal{A}^i \bgamma = \beps\) and \(\mathcal{A}^j \bgamma = \brho\) for some \(i,j \in \ZZ_2\). Then we have
\begin{align}\label{calc35454}
\bar{\mathcal{G}}^i \mathcal{J} \bgamma = \mathcal{J} \mathcal{A}^i \bgamma = \mathcal{J}\beps = \mathcal{J} \brho = \mathcal{J} \mathcal{A}^j \bgamma = \bar{\mathcal{G}}^j \mathcal{J} \bgamma.
\end{align}
But then one may see as in (\ref{diff01}) that  \(\bar{\mathcal{G}}^i \mathcal{J} \bgamma = \bar{\mathcal{G}}^j \mathcal{J} \bgamma\) only if \(i=j\), and so \(\beps = \brho\), proving the claim.\\

\noindent{\em Claim IV. The map \(\mathcal{J}: \textup{DG}(\Theta) \to  \textup{DG}(\bOm^{\textup{MV}}_{\geq 1}) \) is an isomorphism of directed graphs.} We have by Claims I--III that \(\mathcal{J}\) is a bijection of vertices of the graphs. Thus it remains to show that \(\beps \to \brho\) is an arrow in \(\textup{DG}(\Theta) \) if and only if \(\mathcal{J} \beps \to \mathcal{J} \brho\) is an arrow in \( \textup{DG}(\bOm^{\textup{MV}}_{\geq 1}) \).
 First assume  \(\beps \to \brho\) is an arrow in \(\textup{DG}(\Theta) \). Then \(\brho = \mathcal{A}^i \beps\) for some \(i \in \ZZ_2\). Then by Lemma~\ref{commaps} we have
\begin{align*}
\mathcal{J} \brho = \mathcal{J} \mathcal{A}^i \beps = \bar{\mathcal{G}}^i \mathcal{J} \beps,
\end{align*}
so we have an arrow \(\mathcal{J} \beps \to \mathcal{J} \brho\) in \( \textup{DG}(\bOm^{\textup{MV}}_{\geq 1}) \). On the other hand, let \(\mathcal{J} \beps \to \mathcal{J}\brho\) be an arrow in \( \textup{DG}(\bOm^{\textup{MV}}_{\geq 1}) \). Then \(\mathcal{J}\brho= \bar{\mathcal{G}}^j \mathcal{J} \beps\) for some \(j \in \ZZ_2\), so by Lemma~\ref{commaps} we have
\begin{align*}
\mathcal{J}\mathcal{A}^j \beps = \bar{\mathcal{G}}^j \mathcal{J} \beps = \mathcal{J} \brho.
\end{align*}
By injectivity of \(\mathcal{J}\), we have that \(\mathcal{A}^j \beps = \brho\), and so \(\beps \to \brho\) is an arrow in \(\textup{DG}(\Theta)\), proving the claim.\\

\noindent{\em Claim V.
We have \(\mathcal{K}^i = \mathcal{J}^{-1}\).
} 
Let \(\beps \in \Theta(n)\). If \(n=0\), then \(\beps = \varnothing\) and we have \(\mathcal{K}^i \mathcal{J} \beps = \mathcal{K}^i(\varnothing, \varnothing) = \varnothing = \beps\). Assume \(n >0\). 
 Then set \(\bmu = \mathcal{J} \beps  \in \bOm^{\textup{MV}}_{\geq 1}(n)\). We have \(
\mu^i_k = \#\textup{P}^{\textup{R}}_i(\mathcal{T}_i^{k-1} \beps)\) for all \(i \in \ZZ_2, k \in \ZZ_{\geq 1}\). Let \(m\) be maximal such that \(\mu^i_m > 0\). Then \(\mathcal{T}_i^{m+1} \beps = \hat i^n\). Then by (\ref{Kidef}) we have
\begin{align*}
\mathcal{K}^i \bmu 
= \mathcal{U}_{i,\mu_1} \cdots \mathcal{U}_{i,\mu_m}(\hat i^n)
=\mathcal{U}_{i, \# \PR_i(\mathcal{T}_i^0\beps)} \cdots \mathcal{U}_{i,\# \PR_i(\mathcal{T}_i^m\beps)}(\mathcal{T}_i^{m+1}\beps) = \beps
\end{align*}
by repeated applications of Lemma~\ref{Uundo}. Therefore \(\mathcal{K}^i\) is a left inverse for \(\mathcal{J}\), and so \(\mathcal{K}^i = \mathcal{J}^{-1}\) since \(\mathcal{J}\) is an isomorphism.
\end{proof}

\begin{lemma}\label{mutallest}
Let \(i \in \ZZ_2\), \(n \in \Z_{>0}\), \(\bmu \in \bOm_{\geq 1}^{\textup{MV}}(n)\), and set \(\beps = \mathcal{K}^i\bmu \in \bTheta(n)\). Then  there exists \(M \in \Z_{>0}\) such that \(\mu_M^{\eps_1} > 0\) and \(\mu_m^{\hat \eps_1} = 0\) for all \(m\geq M\).
\end{lemma}
\begin{proof}
When \(n=1\), we have \(\beps = j\) if and only if \(\mu^j_1=1\), \(\mu^j_k = 0\) for \(k >1\), and \(\mu_{\hat j} = \varnothing\), so the claim holds. Now assume \(n >1\). Set \(\beps' = \mathcal{R} \beps\) and \(\bnu =\overline{\mathcal{H}}\bmu\), so that
\begin{align*}
\mathcal{K}^i \bnu = \mathcal{K}^i \overline{\mathcal{H}} \bmu = \mathcal{K}^i \overline{\mathcal{H}} \mathcal{J} \beps = \mathcal{K}^i \overline{\mathcal{H}} \mathcal{J} \mathcal{A}^{\eps_n} \beps' = \mathcal{K}^i\overline{\mathcal{H}} \overline{\mathcal{G}}^{\eps_n} \mathcal{J} \beps' = \mathcal{K}^i \mathcal{J} \beps' = \beps',
\end{align*}
where the second, fourth and sixth equalities are by Theorem~\ref{DGthm}, and the fifth equality is by Corollary~\ref{GHisomMV}. Then \(\eps_1 = \eps_1'\), and by induction assumption there exists \(M' \in \Z_{>0}\) such that \(\nu^{\eps_1}_{M'}>0\) and \(\nu^{\hat \eps_1}_m = 0\) for all \(m \geq M'\). 
By Lemma~\ref{altdesc}, since \(\nu^{\eps_1}_{M'} >0\), it must be that \(\mu^{\eps_1}_{M'} >0\) or \(\mu^{\eps_1}_{M'+1} >0\), and since \(\nu^{\hat \eps_1}_m = 0\) for \(m \geq M'\), it must be that \(\mu^{\hat \eps_1}_m = 0\) for \(m > M'\). But then by Lemma~\ref{onemustbebigger} this occurs if and only if there exists \(M \in \{M', M'+1\}\) such that \(\mu^{\eps_1}_{M} >0\) and \(\mu^{\hat \eps_1}_m = 0 \) for all \(m \geq M\).
\end{proof}

\section{Proofs of the main theorems}\label{proofsmainthmssec}
\subsection{Characterizing affine MV polytopes}
Recall now the definitions of affine MV polytopes in \S\ref{defaMV}.

\begin{lemma}\label{MVlongertop}
Let \((\pi | \phi) \in \aMV(\theta)\). Let \(w = p(\pi^1) + p(\phi^0)\). Then
\begin{enumerate}
\item \(\pi^1_1 \neq \phi^0_1\), and \(\pi^1_1 < \phi^0_1\) implies 
\(
\sum_{k =1}^\infty \left(n_k(\phi^0)\alpha_{0:k} - n_k(\pi^1) \alpha_{1:k} \right)= w\alpha_0;
\)
\item \(\pi^0_1 \neq \phi^1_1\), and \(\pi^0_1 < \phi^1_1\) implies 
\(
\sum_{k =1}^\infty \left(n_k(\phi^1)\alpha_{1:k} - n_k(\pi^0) \alpha_{0:k} \right)= w\alpha_1;
\)
\item \(\pi^1_1 < \phi^0_1\) and \(\pi^0_1 < \phi^1_1\) implies \(\pi^\delta = (w, \phi^\delta_1, \ldots, \phi^\delta_{p(\phi^\delta)})\).
\item \(\pi^1_1 > \phi^0_1\) and \(\pi^0_1 > \phi^1_1\) implies \(\phi^\delta = (w, \pi^\delta_1, \ldots, \pi^\delta_{p(\phi^\delta)})\).
\item \(\pi^1_1 > \phi^0_1\) and \(\pi^0_1 < \phi^1_1\), or \(\pi^1_1 < \phi^0_1\) and \(\pi^0_1 > \phi^1_1\) implies \(\pi^\delta = \phi^\delta\).
\end{enumerate}
\end{lemma}
\begin{proof}
Assume \(\pi_1^1  \leq \phi^0_1 = M\). Then  \(n_{M+1}(\pi^1) = n_{M+1}(\phi^0) = 0\) and \(n_M(\phi^0) >0\). We also have that  \(c_{M-1}^\rightarrow(\pi | \phi) , c_{M-1}^\uparrow(\pi | \phi) \leq 0\) since \((\pi | \phi) \in \aMV\). Therefore
\begin{align*}
c_M^\rightarrow(\pi | \phi) = -M n_M(\phi^0) + c_{M-1}^\rightarrow(\pi | \phi) < 0 
\qquad
\textup{and}
\qquad
c_M^\uparrow(\pi | \phi)  = -Mn_M(\pi^1) + c_{M-1}^\uparrow(\pi | \phi),
\end{align*}
which implies that \(c_M^\uparrow(\pi | \phi) = 0\) since  \(0 \in \{ c_M^\rightarrow(\pi | \phi) , c_M^\uparrow(\pi | \phi) \}\). But then the above implies that \(c_{M-1}^\uparrow(\pi | \phi) = 0\) and \(n_M(\pi^1) = 0\). Therefore \(\pi_1^1 < \phi^0_1\), and so we have:
\begin{align*}
\sum_{k =1}^\infty \left(n_k(\phi^0)\alpha_{0:k} - n_k(\pi^1) \alpha_{1:k} \right) &= \sum_{k =1}^M \left(n_k(\phi^0)\alpha_{0:k} - n_k(\pi^1) \alpha_{1:k} \right)\\
&= \sum_{k=1}^M \left( n_k(\phi^0)(k \alpha_0 + (k-1) \alpha_1) - n_k(\pi^1)(k \alpha_1 + (k-1) \alpha_0) \right) \\
&= c_{M-1}^\uparrow(\pi| \phi) (\alpha_1 + \alpha_0) + \sum_{k=1}^M(n_k(\pi^1) + n_k(\phi^0)) \alpha_0 = w \alpha_0.
\end{align*}
This completes the proof of (1). The proof of (2) is similar. For (3) note that
\begin{align*}
\sum_{k=1}^\infty (n_k(\pi^1)\alpha_{1:k} + n_k(\pi^0)\alpha_{0:k}) + |\pi^\delta| \delta = \theta = \sum_{k=1}^\infty (n_k(\phi^1)\alpha_{1:k} + n_k(\phi^0)\alpha_{0:k}) + |\phi^\delta| \delta
\end{align*}
since \(\pi, \phi\) are root partitions of \(\theta\). But on the other hand by (1) and (2) we have that 
\begin{align*}
\sum_{k=1}^\infty (n_k(\phi^1)\alpha_{1:k} + n_k(\phi^0)\alpha_{0:k}) - \sum_{k=1}^\infty (n_k(\pi^1)\alpha_{1:k} + n_k(\pi^0)\alpha_{0:k}) = w \delta
\end{align*}
which implies that \(|\pi^\delta| = |\phi^\delta| + w\). Then the result follows by Definition~\ref{defaMV}(2). Part (4) follows by symmetry with (3). 

For (5), we note that by (1) and (2), we have in either case that 
\begin{align*}
\sum_{k=1}^\infty (n_k(\phi^1)\alpha_{1:k} + n_k(\phi^0)\alpha_{0:k}) - \sum_{k=1}^\infty (n_k(\pi^1)\alpha_{1:k} + n_k(\pi^0)\alpha_{0:k}) = 0,
\end{align*}
which implies that \(|\pi^\delta| = |\phi^\delta|\). Then the result follows by Definition~\ref{defaMV}(2).
\end{proof}

Now
recall the bijection \(\textup{mc}: \Omega_{\geq 1} \to \Par\) from \S\ref{RKsecappend}, given by 
setting \(\textup{mc}(\mu) = \lambda\), where \(n_k(\lambda) = \mu_k\) for all \(k \in \Z_{>0}\). We abuse notation to extend \(\textup{mc}\) to a bijection:
\begin{align*}
\textup{mc}: \bOm_{\geq 1} \to \Par^2, \qquad (\mu^1, \mu^0) \mapsto (\lambda^1:= \textup{mc}(\mu^1), \lambda^0:= \textup{mc}(\mu^0)),
\end{align*}
and write \(\Par^{\textup{MV}}:= \textup{mc}(\bOm_{\geq 1}^{\textup{MV}})\) for the image of the MV-bicompositions under this bijection.
In view of \S\ref{RKsecappend}, the actions of \(\overline{\mathcal{G}}^i\), \(\overline{\mathcal{H}}\) on \(\bOm_{\geq 1}^\textup{MV}\) induce actions of operators \(\overline{\mathcal{G}}^i_{\Par^2}\), \(\overline{\mathcal{H}}_{\Par^2}\) on \(\Par^\textup{MV}\), where
\begin{align*}
\overline{\mathcal{G}}^1_{\Par^2}: (\lambda^1, \lambda^0) \mapsto (\overline{\mathcal{G}}_\Par\lambda^1, \mathcal{G}_\Par \lambda^0);
\qquad
\overline{\mathcal{G}}^0_{\Par^2}: (\lambda^1, \lambda^0) \mapsto (\mathcal{G}_\Par\lambda^1, \overline{\mathcal{G}}_\Par \lambda^0);
\qquad
\overline{\mathcal{H}}_{\Par^2}: (\lambda^1, \lambda^0) \mapsto (\overline{\mathcal{H}}_\Par \lambda^1, \overline{\mathcal{H}}_\Par \lambda^0).
\end{align*}

\begin{lemma}\label{mvbiparsdesc}
Let \((\lambda^1, \lambda^0) \in \Par^2\), and set \({\tt W} = p(\lambda^1) + p(\lambda^0)\). Then \((\lambda^1, \lambda^0) \in \Par^\textup{MV}\) if and only if \(\lambda^0 = \ORTH^{\tt W}(\lambda^1)\). Moreover, if \((\lambda^1, \lambda^0) \in \Par^\textup{MV}\) then we have
\begin{align*}
(\lambda^1, \lambda^0) = \textup{mc} ( \mathcal{J}\brho^\textup{rev}),
\end{align*}
where \(\brho \in \bTheta({\tt W})\) and \(\rho_k = j\) exactly when \(k\) appears in the first column of \(\SPAR(\lambda^j)\). 
\end{lemma}
\begin{proof}
\((\implies)\) 
Assume \((\lambda^1, \lambda^0) \in \Par^\textup{MV}\). Then \((\lambda^1, \lambda^0) = \textup{mc}(\bmu)\) for some \(\bmu \in \bOm^\textup{MV}_{\geq 1}({\tt W})\). 
Set \(\beps = \mathcal{K}^1\bmu\). Then by Theorem~\ref{DGthm}, we have that 
\begin{align*}
(\lambda^1, \lambda^0) &= \textup{mc}(\bmu) = \textup{mc}(\mathcal{J} \beps) = \textup{mc}(\mathcal{J}  \mathcal{A}^{\eps_{\tt W}} \cdots \mathcal{A}^{\eps_1} \varnothing) = \textup{mc}(\overline{\mathcal{G}}^{\eps_{\tt W}} \cdots \overline{\mathcal{G}}^{\eps_1}  \mathcal{J} \varnothing)\\
&= \textup{mc}(\overline{\mathcal{G}}^{\eps_{\tt W}} \cdots \overline{\mathcal{G}}^{\eps_1} (\varnothing, \varnothing)) 
= \overline{\mathcal{G}}_{\Par^2}^{\eps_{\tt W}} \cdots \overline{\mathcal{G}}_{\Par^2}^{\eps_1} (\varnothing, \varnothing)
=\overline{\mathcal{G}}_{\Par^2}^{\rho_1} \cdots \overline{\mathcal{G}}_{\Par^2}^{\rho_{\tt W}} (\varnothing, \varnothing),
\end{align*}
where \(\brho := \beps^\textup{rev}\). Then \((\lambda^1, \lambda^0) = \textup{mc} (\mathcal{J}\brho^\textup{rev})\) is constructed by iteratively adding boxes to the empty bipartition. If one labels the boxes that are added in the application of \(\overline{\mathcal{G}}_{\Par^2}^{\rho_k}\) with \(k\), it follows from \S\ref{RKsecappend}(GP1), (GP2) and Proposition~\ref{appendreefundo} that this yields exactly the spar tableau for each \(\lambda^j\), with \(k\) in the first column of \(\lambda^j\) exactly when \(\rho_k = j\). Then the entries in the first columns of \(\SPAR(\lambda^1), \SPAR(\lambda^0)\) are disjoint and their union is \([1,{\tt W}]\), so  \(\lambda^0 = \ORTH^{\tt W}(\lambda^1)\), as desired.

\((\impliedby)\) In the other direction, assume \(\lambda^0 = \ORTH^{\tt W}(\lambda^1)\). Then we may define \(\brho \in \bTheta({\tt W})\) by setting \(\rho_k = j\) exactly when \(k\) appears in the first column of \(\SPAR(\lambda^j)\). It follows then as in the remarks above that \((\lambda^1, \lambda^0) = \textup{mc} ( \mathcal{J}\brho^\textup{rev})\), which implies that \((\lambda^1, \lambda^0) \in \Par^\textup{MV}\) since \(\mathcal{J}\brho^\textup{rev} \in \bOm_{\geq 1}^\textup{MV}\) by Theorem~\ref{DGthm}.
\end{proof}

Now we are in position to provide a tableau-theoretic description of affine MV-polytopes, which gives a combinatorial description of the lozenge map.

\begin{theorem}\label{amvrecogthm}
Let \(\pi, \phi\) be root partitions. Write \({\tt W} = p(\pi^1) + p(\phi^0)\) and \({\tt W}' = p(\pi^0) + p(\phi^1)\). Then \((\pi | \phi) \in \aMV\) if and only if:
\begin{enumerate}
\item \({\tt W} = {\tt W}'\);
\item \(\ORTH^{\tt W}(\pi^1) = \phi^0\) and \(\ORTH^{\tt W}(\pi^0) = \phi^1\);
\item We have
\begin{align*}
\phi^\delta = 
 \begin{cases}
({\tt W}, \pi_1^\delta, \ldots, \pi_{p(\pi^\delta)}^\delta) &\textup{if } {\tt W}={\tt w}(\pi^{0}) = {\tt w}(\pi^{1}); \\
(\pi^\delta_2, \ldots, \pi^\delta_{p(\pi^\delta)}) & \textup{if } {\tt W} > {\tt w}(\pi^{0}), {\tt w}(\pi^{1});\\
\pi^\delta & \textup{otherwise}.
\end{cases}
\end{align*}
\end{enumerate}
\end{theorem}
\begin{proof}
It will suffice to demonstrate the `only if' direction, thanks to Theorem~\ref{exunloz}. Assume \((\pi| \phi) \in \aMV(\theta)\) for some \(\theta \in \ZZ_{\geq 0}I\). Then since \(\pi, \phi \in \Pi(\theta)\) we have
\begin{align*}
\Lambda_1(\theta) = p(\pi^1) - p(\pi^{0}) = p(\phi^1) - p(\phi^{0}),
\end{align*}
which implies (1). Now, define \(\bmu, \bnu \in \bOm_{\geq 1}({\tt W})\) by setting 
\begin{align*}
\mu^1 = \textup{mc}^{-1}(\pi^1); \qquad \mu^0 = \textup{mc}^{-1}(\phi^0); \qquad \nu^1 = \textup{mc}^{-1}(\phi^1); \qquad \nu^0 = \textup{mc}^{-1}(\pi^0).
\end{align*}
By Definition~\ref{defaMV}(1) we have that
\begin{align*}
0 \in \{c^\uparrow_k(\pi | \phi), c^\rightarrow_k(\pi | \phi)\} \subseteq \ZZ_{\leq 0};
\qquad
\textup{and}
\qquad
0 \in \{c^\downarrow_k(\pi | \phi), c^\leftarrow_k(\pi | \phi)\} \subseteq \ZZ_{\leq 0},
\end{align*}
for all \(k \in \ZZ_{>0}\), which, translating through the bijection \(\textup{mc}\) and considering Definition~\ref{mvbicompdef} gives us that \(\bmu, \bnu \in \bOm^\textup{MV}_{\geq 1}\). Thus (2) follows from Lemma~\ref{mvbiparsdesc}. 

Finally, we confirm (3), in three cases.\\

\noindent{\em Case 1. Assume first that \({\tt w}(\pi^1), {\tt w}(\pi^0) < {\tt W}\).} Note that by Lemma~\ref{mvbiparsdesc} we have  \((\pi^1, \phi^0) = \textup{mc} (\mathcal{J}\brho^\textup{rev})\) for some \(\brho \in \bTheta({\tt W})\), where
\(\rho_{\tt W} =0\) since \({\tt w}(\pi^1) < {\tt W}\). Thus \(\rho^\textup{rev}_1 = 0\), and so it follows from Lemma~\ref{mutallest} that \(\pi^1_1 < \phi^0_1\). In similar fashion we see that \(\pi^0_1 < \phi^1_1\). But then by Lemma~\ref{MVlongertop} we have that 
\begin{align*}
\pi^\delta = ({\tt W}, \phi_1^\delta, \ldots, \phi^\delta_{p(\phi^\delta)})
\qquad
\textup{and so}
\qquad
\phi^\delta = (\pi^\delta_2, \ldots, \pi^\delta_{p(\pi^\delta)}).
\end{align*}

\noindent{\em Case 2. Assume that \({\tt w}(\pi^1) = {\tt w}(\pi^0) = {\tt W}\).} Then we have \({\tt w}(\phi^1), {\tt w}(\phi^0) < {\tt W}\). Then appealing to symmetry and swapping \(\pi, \phi\) in the Case 1 argument above, we have 
\begin{align*}
\phi^\delta = ({\tt W}, \pi_1^\delta, \ldots, \pi^\delta_{p(\pi^\delta)})
\qquad
\textup{and}
\qquad
\pi^\delta = (\phi^\delta_2, \ldots, \phi^\delta_{p(\phi^\delta)}).
\end{align*}

\noindent{\em Case 3. Assume \({\tt w}(\pi^1) < {\tt W}\) and \({\tt w}(\pi^0) = {\tt W}\), or \({\tt w}(\pi^1) = {\tt W}\) and \({\tt w}(\pi^0) < {\tt W}\).} 
Then again as in the cases above, we see that either \(\pi^1_1 < \phi^0_1\) and \(\phi^1_1 < \pi^0_1\), or  \(\pi^1_1 > \phi^0_1\) and \(\phi^1_1 > \pi^0_1\). Then we have by Lemma~\ref{MVlongertop} that \(\phi^\delta = \pi^\delta\), as desired. 
\end{proof}

\subsection{Establishing the crystal isomorphisms in \S\ref{MVULisomsintro}}
Recall now the definition of upper ledge diagrams in \S\ref{bigULsec} and the functions \(\mathcal{X}, \mathcal{Y}\) and associated combinatorial data introduced in \S\ref{MVULisomsintro}. For an upper ledge diagram \(\sfD \in \UL\), we will write \(\beps^\uparrow(\sfD)\) (resp. \(\beps^\downarrow(\sfD)\)) for the colors of top (resp. bottom) boxes in \(\sfD\), read from left to right.

\begin{lemma}\label{Xsidesmatch}
Let \(\sfD \in \UL\). Then \((\mathcal{X}_1(\sfD) \mid \mathcal{X}_0(\sfD)) \in \aMV\), and so the map
\begin{align*}
\mathcal{X}: \UL \to \aMV, \qquad \sfD \mapsto (\mathcal{X}_1(\sfD) \mid \mathcal{X}_0(\sfD))
\end{align*}
is well-defined.
\end{lemma}

\begin{proof}
For a given upper ledge diagram \(\sfD\), it is straightforward to translate between the 1-triple  upper ledge diagram decomposition \(\TRIP_1(\sfD) = (\sfD^\uparrow_1, \sfD^\delta_1, \sfD^\downarrow_1)\) and the 0-triple upper ledge diagram decomposition \(\TRIP_0(\sfD) = (\sfD^\uparrow_0, \sfD^\delta_0, \sfD^\downarrow_0)\) for \(\sfD\). 
Writing \({\tt W}\) for the length of a peak row of \(\sfD\), we have the following:
\begin{enumerate}
\item[(a)] If \(\beps^\uparrow(\sfD)_{\tt W} = \beps^\downarrow(\sfD)_{\tt W} = i\), 
then 
\begin{itemize}
\item \(\sfD_i^\uparrow\) is achieved by adding a \({\tt W}\)-length row of color \(i\) to the bottom of \(\sfD_{\hat i}^\uparrow\);
\item \(\sfD_{ i}^\downarrow\) is achieved by deleting a \({\tt W}\)-length row of color \( i\) from the top of \(\sfD_{\hat i}^\downarrow\);
\item \(\sfD_i^\delta\) and \(\sfD_{\hat i}^\delta\) have identical shape, with opposite colors.
\end{itemize}
\item[(b)] If \(\beps^\uparrow(\sfD)_{\tt W} = i\) and \( \beps^\downarrow(\sfD)_{\tt W} = \hat i\), 
then 
\begin{itemize}
\item \(\sfD_i^\uparrow\) is achieved by adding a \({\tt W}\)-length row of color \(i\) to the bottom of \(\sfD_{\hat i}^\uparrow\);
\item \(\sfD_{ i}^\downarrow\) is achieved by adding a \({\tt W}\)-length row of color \( \hat i\) to the top of \(\sfD_{\hat i}^\downarrow\);
\item \(\sfD_i^\delta\) is achieved by deleting the top two (\({\tt W}\)-length) rows of \(\sfD_{\hat i}^\delta\).
\end{itemize}
\end{enumerate}

Write \(\pi = \mathcal{X}_1(\sfD)\) and \(\phi = \mathcal{X}_0(\sfD)\). In any case, it follows from (a) and (b) above that the \(k\)th column of \(\sfD^\uparrow_i\) (resp. \(\sfD^\downarrow_{ i}\)) will have odd length if and only if the \(k\)th column of \(\sfD^\uparrow_{\hat i}\) (resp. \(\sfD^\downarrow_{\hat i}\)) has even length. Therefore we have
\begin{align}\label{condAAA}
\phi^0 = 
\CALC \circ \textup{forget}(\sfD^\uparrow_0) = \ORTH^{\tt W}( \CALC \circ \textup{forget}(\sfD^\uparrow_1)  ) = \ORTH^{\tt W}( \pi^1)
\end{align}
and
\begin{align}\label{condBBB}
\phi^1 = 
\CALC \circ \textup{forget}(\sfD^\downarrow_0) = \ORTH^{\tt W}( \CALC \circ \textup{forget}(\sfD^\downarrow_1)  ) = \ORTH^{\tt W}( \pi^0).
\end{align}

We now consider four possible cases.\\

\noindent{\em Case 1. Assume first that \({\beps}^\uparrow(\sfD)_{\tt W} = \beps^\downarrow(\sfD)_{\tt W} = 1\).} It follows then that the \({\tt W}\)th (last) column of \(\sfD^\uparrow_1\) has length one, so \({\tt W}\) appears in the first column of the spar tableau for \(\pi^1 = \CALC \circ \textup{forget}(\sfD^\uparrow_1)\), and so \({\tt w}(\pi^1) = {\tt W}\). 
We also have that the \({\tt W}\)th (last) column of \(\sfD^\downarrow_0\) has length one, so \({\tt W}\) appears in the first column of the spar tableau for \(\phi^1 = \CALC \circ \textup{forget}(\sfD^\downarrow_0)\), and so \({\tt w}(\phi^1) = {\tt W}\). Since \(\pi^0 = \ORTH^{\tt W}(\phi^1)\), this implies that \({\tt w}(\pi^0) < {\tt W}\). We also have that 
\begin{align}\label{condCCC}
\pi^\delta = \textup{halve} \circ \textup{forget}(\sfD^\delta_1) = \textup{halve} \circ \textup{forget}(\sfD^\delta_0) = \phi^\delta
\end{align}
by (a) above. Since \({\tt w}(\pi^1) = {\tt W} > {\tt w}(\pi^0)\), it follows by Theorem~\ref{amvrecogthm} and (\ref{condAAA}), (\ref{condBBB}), (\ref{condCCC}) that \((\pi| \phi) \in \aMV\).\\

\noindent{\em Case 2. Assume that \({\beps}^\uparrow(\sfD)_{\tt W} = \beps^\downarrow(\sfD)_{\tt W} = 0\).} This case plays out similar to Case 1, where now \({\tt w}(\pi^0) = {\tt W} < {\tt w}(\pi^1)\) and (\ref{condCCC}) still holds, so \((\pi| \phi) \in \aMV\).\\

\noindent{\em Case 3. Assume that \({\beps}^\uparrow(\sfD)_{\tt W} = 1, \beps^\downarrow(\sfD)_{\tt W} = 0\).} It follows as in Case 1 that \({\tt w}(\pi^1) = {\tt W}\) and \({\tt w}(\pi^0)= {\tt W}\). We also have by (b) above that \(\pi^\delta = \textup{halve} \circ \textup{forget}(\sfD^\delta_1)\) is achieved by deleting the top (\({\tt W}\)-length) row of \(\textup{halve} \circ \textup{forget}(\sfD^\delta_0) = \phi^\delta\). Then again it follows from Theorem~\ref{amvrecogthm} that \((\pi | \phi) \in \aMV\).\\

\noindent{\em Case 4. Assume that \({\beps}^\uparrow(\sfD)_{\tt W} = 0, \beps^\downarrow(\sfD)_{\tt W} = 1\).} It follows as in Case 1 that \({\tt w}(\phi^1) = {\tt W}\) and \({\tt w}(\phi^0)= {\tt W}\), and therefore \({\tt w}(\pi^1), {\tt w}(\pi^0) < {\tt W}\) by (\ref{condAAA}), (\ref{condBBB}). We also have by (b) above that \(\phi^\delta = \textup{halve} \circ \textup{forget}(\sfD^\delta_1)\) is achieved by deleting the top (\({\tt W}\)-length) row of \(\textup{halve} \circ \textup{forget}(\sfD^\delta_0) = \pi^\delta\). Then again it follows from Theorem~\ref{amvrecogthm} that \((\pi | \phi) \in \aMV\).
\end{proof}

\begin{lemma}\label{Ydef}
For \((\pi | \phi) \in \aMV\), we have \(\mathcal{Y}_1(\pi) = \mathcal{Y}_0(\phi)\), and therefore the map
\begin{align*}
\mathcal{Y}: \aMV \to \UL, \qquad \mathcal{Y}(\pi | \phi) = \mathcal{Y}_1(\pi) = \mathcal{Y}_0(\phi)
\end{align*}
is well-defined.
\end{lemma}
\begin{proof}
For a given \(i \in \Z_2\), it is clear that \(\mathcal{Y}_i: \Pi \to \UL\) and \(\mathcal{X}_i: \UL \to \Pi\) are mutual inverses since \(\ORD\)/\(\CALC\) and \(\STACK_i\)/\(\TRIP_i\) are mutual inverses. Thus we have that \(
(\pi | \phi) = (\mathcal{X}_1 \mathcal{Y}_1 \pi \, | \, \mathcal{X}_0 \mathcal{Y}_0 \phi).
\)
On the other hand by Lemma~\ref{Xsidesmatch} we have that 
\(
(\mathcal{X}_1 \mathcal{Y}_1 \pi \, | \, \mathcal{X}_0 \mathcal{Y}_1 \pi) \in \aMV
\).
Therefore
\(
\mathcal{X}_0 \mathcal{Y}_0 \phi = (\mathcal{X}_1 \mathcal{Y}_1 \pi )_\lozenge = \mathcal{X}_0 \mathcal{Y}_1 \pi
\)
by Theorem~\ref{exunloz}, 
which implies that \(\mathcal{Y}_0 \phi = \mathcal{Y}_1 \pi\) since \(\mathcal{X}_0\) is injective.
\end{proof}

\begin{theorem}\label{mainamvulthm}
The maps \(\mathcal{X}: \UL \to \aMV\) and \(\mathcal{Y}: \aMV \to \UL\) are mutually inverse \({\tt A}_1^{(1)}\)-bicrystal isomorphisms.
\end{theorem}
\begin{proof}
That the maps are inverse functions is clear thanks to Lemmas~\ref{Xsidesmatch} and \ref{Ydef}, since \(\mathcal{X}_i\) and \(\mathcal{Y}_i\) are inverse functions, as noted in the proof of Lemma~\ref{Ydef}. It thus remains to check that \(\mathcal{X}\) is a bicrystal map. 
We will show that \(f_0 \mathcal{X}(\sfD) = \mathcal{X}(f_0 \sfD)\); the proofs that \(f_1, f_1^*, f_0^*\) commute with \(\mathcal{X}\) follow in a similar fashion. 
Write \(\tilde{\sfD}:= f_0 \sfD\), \(\pi = \mathcal{X}_1(\sfD)\), and \(\varpi = \mathcal{X}_1(\tilde{\sfD})\). In view of \S\ref{aMVcrysdef}, it suffices to check that \(\varpi= \pi_{+0:1}\).

Recall that \(\tilde{\sfD}\) is achieved by adding a box colored \(0\) to the bottom of the leftmost unarced \(1\)-bottom box in \(\sfD[0]\), if it exists, and otherwise by appending a box to the right end of the uppermost \(0\)-peak row in \(\sfD\). In either case, the new box is an unarced \(0\)-bottom box in \(\tilde{\sfD}[0]\), and all other bottom boxes are arced in \(\tilde{\sfD}[0]\) if and only if they were arced in \(\sfD[0]\).
It immediately follows that \(\TRIP_1(\tilde{\sfD}) = (\sfD^\uparrow_1, \sfD^\delta_1, f_0 \sfD^\downarrow_1)\), and so \(\varpi^1 = \pi^1\) and \(\varpi^\delta = \pi^\delta\).

Write \(\brho = \beps^\downarrow(\sfD)\),  \(\tilde{\brho} = \beps^\downarrow(\tilde{\sfD})\). 
It follows from definitions that the \(k\)th bottom box from the left in \(\sfD\) is 
\begin{enumerate}
\item[(i)] \(1\)-colored and unarced in \(\sfD[0]\) if and only if \(k \in \PL_1(\brho^\textup{rev})\);
\item[(ii)] \(0\)-colored and unarced in \(\sfD[0]\) if and only if \(k \in \PR_0(\brho^\textup{rev})\), 
\end{enumerate}
and similarly for \(\tilde{\sfD}\) and \(\tilde \brho\).
Writing \(\bnu = \mathcal{J}\brho^\textup{rev}\) and \(\tilde{\bnu} = \mathcal{J} \tilde \brho^\textup{rev}\), we have by (i), (ii) and  (\ref{Jdef})  that 
\begin{align*}
\tilde{\nu}_1^0 = \#\PR_0(\tilde{\brho}^\textup{rev}) =  \#\PR_0(\brho^\textup{rev}) +1 = \nu_1^0 + 1.
\end{align*}
We moreover have by the above discussion and (\ref{Tmapdef}) that  \(\mathcal{T}_0 \tilde{\brho}^\textup{rev} = 1^x \mathcal{T}_0 \brho^\textup{rev}\), where \(x = 0\) if the new box in \(\tilde{\sfD}\) was added below a \(1\)-bottom box in \(\sfD\), and \(x=1\) if it was added to the uppermost \(0\)-peak row in \(\sfD\) (hence increasing the width of the diagram by one).
So it follows from  (\ref{Jdef})  that \(\tilde{\nu}_k^0 = \nu_k^0\) for all \(k >1\). Writing \((\lambda^1, \lambda^0) = \textup{mc} (\bnu)\) and \((\tilde\lambda^1, \tilde\lambda^0) = \textup{mc} (\tilde\bnu)\), we have then that \(n_1(\tilde \lambda^0) = n_1(\lambda^0) + 1\), and \(n_k(\tilde \lambda^0) = n_k(\lambda^0)\) for \(k >1\). By Lemma~\ref{mvbiparsdesc}, note that \(\rho_k = j\) (resp. \(\tilde{\rho}_k = j\)) if and only if \(k\) appears in the first column of \(\SPAR(\lambda^j)\) (resp. \(\SPAR(\tilde{\lambda}^j)\)). 

Now, we have \(\pi^0 = \CALC(\sfD^\downarrow_1)\) by definition. Then \(k\) appears in the first column of \(\SPAR(\pi^0)\) if and only if the \(k\)th column of \(\sfD^\downarrow_1\) has odd length, which happens if and only if the \(k\)th bottom box from the left in \(\sfD^\downarrow_1\) is colored 0, which happens if and only if \(\beps^\downarrow(\sfD^\downarrow_1)_k = 0\). We further have \(\brho = \beps^\downarrow(\sfD) = \beps^\downarrow(\sfD^\downarrow_1)1^y\) for some \(y \in \ZZ_{\geq 0}\). Thus \(\rho_k = 0\) if and only if \(k\) appears in the first column of \(\SPAR(\pi^0)\). But then we have \(\ORD(\pi^0) = \ORD(\lambda^0)\), which implies \(\pi^0 = \lambda^0\) since \(\ORD\) is injective by Theorem~\ref{tallyprop}. In similar fashion we see that \(\varpi^0 = \tilde \lambda^0\), so it follows that 
\(n_1(\varpi^0) = n_1(\pi^0) + 1\) and \(n_k(\varpi^0) = n_k(\pi^0)\) for \(k> 1\), so \(\varpi = \pi_{+0:1} \). 

Finally, the fact that the statistics \(\textup{wt}\), \(\varepsilon_i\), \(\varepsilon^*_i\), \(\varphi_i\), \(\varphi^*_i\) commute with \(\mathcal{X}\) is easily checked by induction, taking advantage of the fact that \(\UL\), \(\aMV\) are highest weight.
\end{proof}

\subsection{Establishing the crystal isomorphisms in \S\ref{kmulisomsec}}
Recall now the maps \(\mathcal{Z}^{\bkap}, \mathcal{W}^{\bkap}\), and associated gluing and splitting operations from \S\ref{kmulisomsec}.

\subsubsection{Freezing boxes in upper ledge diagrams}
Let \(\usfD = (\sfD_1, \ldots, \sfD_{m+1}) \in \UL^{m+1}\) be a sequence of upper ledge diagrams. We will consider the diagrams arranged horizontally so \(\sfD_{t+1}\) is to the right of \(\sfD_t\). Then the {\em frozen} boxes of \(\usfD\) are defined as follows. For \(t \in [1,m]\), let \(v_t\) be the rightmost box in the uppermost peak row in \(\sfD_t\). Then for each \(t=m, m-1, \ldots, 1\), freeze the nearest unfrozen bottom box to the right of \(v_t\) whose color is opposite \(v_t\), and then freeze the nearest unfrozen bottom box to the right whose color is opposite that, and so on, greedily freezing the nearest unfrozen opposite-color bottom box to the right at each step until exhaustion. We say boxes frozen in the \(t\)th step are \(t\)-frozen when it is useful to be specific.

\begin{lemma}\label{splittrackbottomfrozen1}
Let \(\sfD \in \UL\), \(i \in \Z_2\), and set \(\usfD = (\sfD_1, \sfD_2) = \SPLIT^i(\sfD)\). Then the unfrozen bottom boxes in \((\sfD_1, \sfD_2)\), read from left to right, are exactly the images under the \(\SPLIT^i\) map of the bottom boxes in \(\sfD\), read from left to right.
\end{lemma}
\begin{proof}
The statement is trivial if \(\sfD_1\) or \(\sfD_2\) are empty, so let us assume otherwise.
We utilize the terminology from \S\ref{defthesplitmap}. 
Let \(u\) be the uppermost box of color \(i\) in the first column of \(\sfD\). Let \(r\) be the ray extending directly southeast from the centroid of \(u\). Let \(\sfL\) be the set of boxes in \(\sfD\) which are weakly southwest of \(r\), and let \(\sfU\) be the set of boxes in \(\sfD\) which are strictly northeast of \(r\). Let \(\sfU'\) be the ledge diagram obtained by left-justifying all boxes in \(\sfU\). Let \(\sfU'' = \FLOAT(\sfU')\). Then we have \(\sfD_1 = \sfL\) and \(\sfD_2 = \sfU''\). 
Let \(x\) be the leftmost among the lowest row of boxes in \(\sfU\), and let \(y\) be the rightmost box in the uppermost peak row of \(\sfD\). See Figure~\ref{trackingboxes676767} for an example.\\

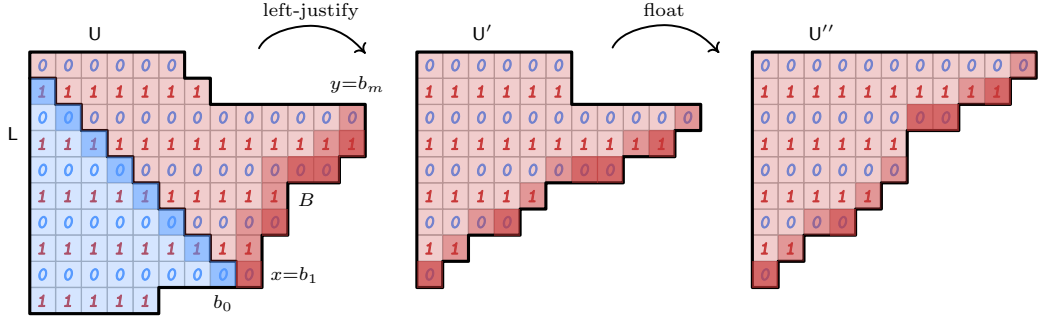
\begin{figure}[h]
\begin{align*}
\\
\hackcenter{}
\hackcenter{
\begin{overpic}[height=38.5mm]{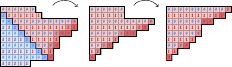}
  \put(-1, 18){\makebox(0,0)[r]{$\scriptstyle \sfL$}}    
    \put(6.6, 27.3){\makebox(0,0)[b]{$\scriptstyle \sfU$}}    
      \put(45, 27.3){\makebox(0,0)[b]{$\scriptstyle \sfU'$}}    
        \put(78.6, 27.3){\makebox(0,0)[b]{$\scriptstyle \sfU''$}}    
        \put(28.4,11.4){\makebox(0,0)[r]{$\scriptstyle B$}}    
          \put(19.3,2){\makebox(0,0)[t]{$\scriptstyle b_0$}}    
           \put(24,4){\makebox(0,0)[l]{$\scriptstyle x=b_1$}}    
             \put(32.6,22){\makebox(0,0)[b]{$\scriptstyle y=b_m$}}    
              \put(28,29.3){\makebox(0,0)[b]{$\scriptstyle \textup{left-justify}$}}   
              \put(63,29.6){\makebox(0,0)[b]{$\scriptstyle \FLOAT$}}   
\end{overpic}
}
\end{align*}
\caption{At left, the upper ledge diagram \(\sfD\). The diagram \(\sfL\) is shaded in blue, with the boxes lying on \(r\) shaded in darker blue. The diagram \(\sfU\) is shaded red, \(B\) shaded darker red, and the bottom-box members of \(B\) shaded darkest red. At center is \(\sfU'\), and at right is \(\sfU''\)---note that the darkest red boxes become the unfrozen bottom boxes in \((\sfL, \sfU'')\).}
\label{trackingboxes676767}
\end{figure}

\noindent{\em Claim 1. The box \(y\) is weakly northeast of \(x\).} Clearly \(y\) is weakly east of \(x\).  Assume by way of contradiction that \(x\) is strictly north of \(y\), i.e., in a row above the uppermost peak row of \(\sfD\). Then by the definition of upper ledge diagrams, the row below \(x\) is necessarily longer than the row containing \(x\), so the box \(z\) directly southeast of \(x\) must be in \(\sfD\). Since \(x\) is strictly northeast of \(r\), this implies \(z\) is strictly northeast of \(r\) as well, so \(z \in \sfU\), in contradiction of the definition of \(x\).\\

By Claim 1, there must be a path \(B\) in \(\sfD\) from \(x\) to \(y\) given by stepping east whenever possible, and north otherwise. In fact, \(B\) is contained in \(\sfU\) since every box traversed in \(B\) is weakly northeast of \(x\), which is strictly northeast of \(r\). Write \(x= b_1, b_2, \ldots, b_m = y\) for the boxes in \(B\) traversed in northeasterly order. If \(b_1\) is bottom, we define \(b_0\) to be the box to its left, and if \(b_1\) is non-bottom, we define \(b_0\) to be the box below it.
See again Figure~\ref{trackingboxes676767} for an example.\\

\noindent{\em Claim 2. If \(b \in B\), then there is no box directly southeast of \(b\) in \(\sfD\).}
We prove the claim for \(b_t\) by induction on \(t\). First note that if there were a box \(z\) directly southeast of \(b_1 = x\), then since \(x\) is strictly northeast of \(r\), \(z\) would be as well, so \(z \in \sfU\), in contradiction of the definition of \(x\). Thus the claim holds for \(b_1\). 
Now assume the claim holds for \(b_t\), for some \(1 \leq t < m\). If \(b_{t+1}\) is east of \(b_t\), then there is no box directly below \(b_{t+1}\), and therefore there cannot be a box directly southeast of \(b_{t+1}\) by the definition of ledge diagrams. If \(b_{t+1}\) is north of \(b_t\), then by the definition of the path \(B\), there is no box to the east of \(b_t\), and thus again no box directly southeast of \(b_{t+1}\).\\

\noindent{\em Claim 3. If \(z \in \sfU \backslash B\), then there is a box directly southeast of \(z\) in \(\sfD\).} Assume by way of contradiction that \(z \in \sfU \backslash B\) and there is not a box directly southeast of \(z\) in \(\sfD\). Let \(R\) be the row containing \(z\). Note then that the row below \(R\) is therefore weakly shorter than \(R\), so it must be that \(R\) is weakly below the uppermost peak row of \(\sfD\), and weakly above the row containing \(x\). Thus \(R\) must contain some box \(b \in B\). By the definition of the path \(B\), every box in \(R\) to the east of \(b\) is in \(B\), so it must be that \(z\) is west of \(b\). Since there is no box directly to the southeast of \(z\) in \(\sfD\), it follows that no box in \(R\) east of \(z\) has a box below it. Since \(z \notin B\), this implies that all steps in \(B\) were eastward steps prior to \(b\), so \(b\) must be in the same row as \(x\). By construction, the path \(B\) contains all boxes in \(R\) weakly east of \(x\), so it must be that \(z\) is strictly west of \(x\), a contradiction of the definition of \(x\).\\

\noindent{\em Claim 4. The box \(b_0\) is the lowest box in \(\sfD\) which lies along \(r\).} First note that \(b_0\) must be in \(\sfL\) by the definition of \(x\). Since \(x\) is strictly northeast of \(r\), it must be then that \(b_0\) lies along \(r\). If \(x\) is bottom, then \(b_0\) is to the left of \(x\) and there is no box directly southeast of \(b_0\), so \(b_0\) is lowest box in \(\sfD\) which lies along \(r\). Thus we may assume that \(b_0\) is non-bottom. Then \(b_0\) is directly below \(x\). There is no box directly southeast of \(b_1\) by Claim 2, so there is no box directly to the right of \(b_0\). Then since the row containing \(b_0\) is weakly lower than the uppermost peak row of \(\sfD\), it must be that there is no box directly to the southeast of \(b_0\). Thus \(b_0\) is the lowest box in \(\sfD\) which lies along \(r\).\\

\noindent{\em Claim 5. Reading from left to right, the bottom boxes in \(\sfU'\) (and thus in \(\sfU''\)) are exactly \(b_1, \ldots, b_m\).}
Let \(z, w\) be boxes in \(\sfU\). Then \(w\) is directly southeast of \(z\) in \(\sfU\) if and only if \(w\) is directly under \(z\) in \(\sfU'\), since the left justification process shifts every row in \(\sfU\) one notch to the right relative to the row under it. Then by Claims 1 and 2, a box \(z \in \sfU\) will be bottom in \(\sfU'\) if and only if \(z \in B\). Again, since \(\sfU'\) is obtained by shifting every row in \(\sfU\) one notch to the right relative to the row under it, any two boxes arranged such that one is weakly northeast of the other in  \(\sfU\) will retain that relationship in \(\sfU'\). Thus since \(b_1, \ldots, b_m\) are in weakly increasing order of northeasternness, they retain that relationship in \(\sfU'\), completing the proof.\\

\noindent{\em Claim 6. For \(t \in [1,m]\), the box \(b_t\) is bottom in \(\sfD\) if and only if \(b_{t-1}\) is of the same color.} By construction, each \(b_t\) is either a north or east step from \(b_{t-1}\), and thus \(b_t\) will have the same color as \(b_{t-1}\) if and only if it is an east step from \(b_{t-1}\). Thus we must show that \(b_t\) is bottom if and only if \(b_t\) is an east step from \(b_{t-1}\).
If \(b_t\) is a north step from \(b_{t-1}\), then \(b_t\) is immediately not bottom. If \(t >1\) and \(b_t\) is an east step from \(b_{t-1}\), then there is no box directly southeast of \(b_{t-1}\) by Claim 2, so \(b_t\) is bottom. Finally, assume \(b_1\) is an east step from \(b_0\). Then by Claim 4 there is no box directly southeast of \(b_0\), so again \(b_1\) is bottom. \\

\noindent{\em Claim 7. The set of bottom boxes in \(\sfD\) which belong to \(\sfU\), read from left to right in \(\sfD\), is the subsequence of \(b_1, \ldots, b_m\) consisting of those \(b_t\) which match the color of \(b_{t-1}\).} Let \(z \in \sfU\) be bottom in \(\sfD\). Then since there is no box directly below \(z\), there can be no box directly to the southeast of \(z\), so \(z =b_t\) for some \(t \in [1,m]\) thanks to Claim 3. Thus the bottom boxes of \(\sfD\) which belong to \(\sfU\) are exactly those \(b_t\) which match the color of \(b_{t-1}\) by Claim 6, and thus the claim follows since the sequence \(b_1, \ldots, b_m\) proceeds in a northeasterly direction.
\\

\noindent{\em Claim 8. The unfrozen bottom boxes in \((\sfL, \sfU'')\) read from left to right are exactly the bottom boxes in \(\sfD\) read from left to right.} Since no boxes are frozen in \(\sfL\), it is enough to show that the bottom boxes in \(\sfU\) read from left to right are exactly the bottom boxes in \(\sfU''\) read from left to right. Note that since \(b_0\) is the lowest element of \(\sfD\) that lies along \(r\), we have that \(b_0\) is in the uppermost peak row of \(\sfL\), and by Claim 5, the bottom boxes in \(\sfU''\) read from left to right are \(b_1, \ldots, b_m\). Then by definition of the freezing process, the box \(b_t\) will be frozen in \(\sfU''\) if and only if \(b_t\) is of opposite color to \(b_{t-1}\). Thus the unfrozen bottom boxes in \(\sfU''\) read from left to right is the subsequence of \(b_1, \ldots, b_m\) consisting of those \(b_t\) which match the color of \(b_{t-1}\), and thus Claim 7 gives the result.
\end{proof}

Let \(\bkap\) be a multicharge, and let \(\usfD = (\sfD_1, \ldots, \sfD_m) \in \UL^m\). We define operators on such sequences:
\begin{align*}
s^{\bkap}(\usfD) &= (\sfD_1, \ldots, \sfD_{m-1}, \SPLIT^{\kappa_m}(\sfD_m)_\downarrow, \SPLIT^{\kappa_m}(\sfD_m)_\uparrow) \in \UL^{m+1};\\
g^{\bkap}(\usfD) &= (\sfD_1, \ldots, \sfD_{m-2}, \GLUE(\sfD_{m-1}, \sfD_{m})) \in \UL^{m-1} \qquad (m>1).
\end{align*}

\begin{lemma}\label{bigfreezeUL}
Let \(\bkap\) be a multicharge, and \(\sfD \in \UL\). Let \(\usfD = (\sfD_1, \ldots, \sfD_{m+1}) = (s^{\bkap})^{\circ m}(\sfD)\). Then the unfrozen bottom boxes in \(\usfD\), read from left to right, are exactly the images under \((s^{\bkap})^{\circ m}\) of the bottom boxes in \(\sfD\), read from left to right.
\end{lemma}
\begin{proof}
We go by induction on \(m\). The base case \(m=0\) is trivial, so assume \(m > 0\) and the claim holds for \(m-1\). Write \(\usfC = (s^{\bkap})^{\circ(m-1)}(\sfD)\). 
By Lemma~\ref{splittrackbottomfrozen1}, the bottom boxes in \(\usfD\) which are not \(m\)-frozen, read from left to right, are exactly the images under \(s^{\bkap}\) of the bottom boxes in \(\usfC\), read from left to right. 
Thus, by definition of the freezing process,  the unfrozen bottom boxes in \(\usfD\), read from left to right, are the images under \(s^{\bkap}\) of the unfrozen bottom boxes in \(\usfC\), read from left to right. 
By induction assumption, the unfrozen bottom boxes in \(\usfC\), read from left to right, are exactly the images under \((s^{\bkap})^{\circ(m-1)}\) of the bottom boxes in \(\sfD\), read from left to right, which yields the result.
\end{proof}

\subsubsection{Frozen arc diagrams for multipartitions}
Let \(\bkap\) be a multicharge, and \(\blam \in \MP^{\bkap}\), of level \(\ell\). The {\em frozen} boxes of \(\blam\) are defined as follows. For \(t \in \ZZ_{>0}\), let \(a_t\) be the uppermost addable position in \(\lambda^{(t)}\), and let \(r_t = \textup{res}(a_t)\). We now iteratively {\em freeze} removable boxes in \(\blam\): for each \(t=\ell-1, \ell-2, \ldots, 1\), freeze the nearest unfrozen removable box of residue \(r_t\) above \(a_t\), and then freeze the nearest unfrozen removable box with residue \(\hat r_t\) above that, and so on, alternating residue at each step, until exhaustion. We say boxes frozen in the \(t\)th step are \(t\)-{\em frozen} when it is useful to be specific.

Let \(i \in \Z_2\). We define the {\em frozen \(i\)-arc diagram} \(\blam[i]^{\,\,\snowflake\,}\) of \(\blam\) as follows. 
First we freeze removable boxes in \(\blam\) according to the process above, subsequently ignoring these boxes. Now reading from top to bottom, whenever an \(i\)-removable unfrozen box is followed by an \(\hat{i}\)-removable unfrozen box with no unarced unfrozen removable boxes between them, draw an arc connecting them. Repeat until no more such arcs are possible. The resulting diagram is \(\blam[i]^{\,\,\snowflake\,}\). An example is shown at the left in Figure~\ref{bigexampleoffreezingfig}.

\begin{figure}[h]
\begin{align*}
\\
\hackcenter{}
\hackcenter{
\begin{overpic}[height=120mm]{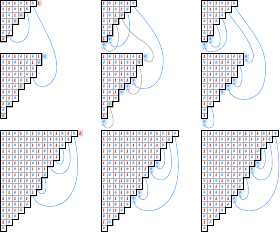}
\end{overpic}
}
\end{align*}
\caption{Various arc diagrams related to a multipartition \(\blam \in \MPres^{\bkap}\). At left, the frozen \(0\)-arc diagram \(\blam[0]^{\,\,\snowflake\,}\) with frozen boxes shaded in gray. In the center, the thawed \(0\)-arc diagram \(\blam[0]^{\waterdrop}\) with type A arcs in blue and type B arcs in gray. At right, the usual \(0\)-arc diagram \(\blam[0]\). }
\label{bigexampleoffreezingfig}
\end{figure}

\begin{lemma}\label{firstfrozenlem}
Let \(\sfD \in \UL\), \(i \in \ZZ_2\). 
Let \(x,y\) be boxes in \(\sfD\), and \(\tilde{x}, \tilde{y}\) be their images in \(\blam = \mathcal{W}^{\bkap}(\sfD)\). 
\begin{enumerate}
\item The box \(x\) is a bottom box in \(\sfD\) if and only if \(\tilde x\) is unfrozen removable in  \(\blam[i]^{\,\,\snowflake\,}\).
\item If \(x,y\) are bottom boxes in \(\sfD\) with \(x\) left of \(y\), then \(\tilde x\) is lower than \(\tilde y\) in \(\blam\).
\item The box \(x\) is unarced bottom in \(\sfD[i]\) if and only if \(\tilde x\) is unarced unfrozen removable in  \(\blam[i]^{\,\,\snowflake\,}\).
\item The boxes \(x,y\) are bottom and share an arc in \(\sfD[i]\) if and only if \(\tilde x, \tilde y\) are unfrozen removable  and share an arc in \(\blam[i]^{\,\,\snowflake\,}\).
\end{enumerate}
\end{lemma}
\begin{proof}
Say that \(\blam\) is of level \(\ell\). Then we have that \(\blam = \mathcal{W}^{\bkap}(\sfD)\) is obtained by computing \(\usfD = (s^{\bkap})^{\circ \ell}(\sfD)\), and then shifting the \(k\)th column in each component up by \(k-1\) notches. It is clear then that a box is bottom in \(\usfD\) if and only if its image is removable in \(\blam\), and a bottom box \(y\) is right of a bottom box \(x\) in \(\usfD\) if and only if the image of \(y\)  is above the image of \(x\) in \(\blam\). Then, in consideration of the respective freezing processes and Lemma~\ref{bigfreezeUL}, claims (1) and (2) immediately follow. Then (1), (2) immediately imply (3) and (4), since, reading left to right, the bottom boxes in \(\sfD\) appear in the same sequence as their images (the unfrozen removable boxes) appear in \(\blam\), read bottom to top. 
\end{proof}

\begin{lemma}\label{candelbottoms}
Let \(\bkap\) be a multicharge and \(\sfD \in \UL\). Let \(x\) be a bottom box in \(\sfD\) and let \(\tilde x\) be its image in \(\mathcal{W}^{\bkap}(\sfD)\). Assume \(\sfD' \in \UL\) is obtained by removing \(x\) from \(\sfD\). Then \(\mathcal{W}^{\bkap}(\sfD')\) is obtained by removing \(\tilde x\) from \(\mathcal{W}^{\bkap}(\sfD)\).
\end{lemma}
\begin{proof}
Let \(y\) be a box in \(\sfD\), and let \((\sfD_1, \sfD_2) = \SPLIT^i(\sfD)\). It is straightforward to see from the definition of the \(\SPLIT^i\) map that whether \(y\) lands in \(\sfD_1\) or \(\sfD_2\), and what its position is relative to the northwesternmost box in the component in which it lands, is dependent only on the set of boxes in \(\sfD\) which are weakly northwest of \(y\). Since \(x\) is bottom and \(\sfD' \in \UL\), it follows that \(x\) is not weakly northwest of any other box in \(\sfD\), so no other box's position depends on \(x\). Thus we have that \(\SPLIT^i(\sfD')\) is obtained by deleting the image of \(x\) from \(\SPLIT^i(\sfD)\). Moreover, we have by Lemma~\ref{splittrackbottomfrozen1} that the image of \(x\) is bottom in \(\SPLIT^i(\sfD)\). Iterating on these facts then, we have that \((s^{\bkap})^{\circ m}(\sfD')\) is obtained by deleting the image of \(x\) in \((s^{\bkap})^{\circ m}(\sfD')\) for any \(m \in \Z_{\geq 0}\). Finally, if \(\blam\) is of level \(\ell\), then \(\mathcal{W}^{\bkap}(\sfD)\) (resp. \(\mathcal{W}^{\bkap}(\sfD')\)) is obtained by shifting the \(k\)th column in each component of \((s^{\bkap})^{\circ \ell}(\sfD)\) (resp. \((s^{\bkap})^{\circ \ell}(\sfD')\)) up by \(k-1\) notches, so the result follows.
\end{proof}

Let \(\blam \in \MPres^{\bkap}\). 
We form the {\em thawed \(i\)-arc diagram}  \(\blam[i]^{\waterdrop}\)  now as follows. First construct \(\blam[i]^{\,\,\snowflake\,}\). Note that every \(\hat i\)-removable box must have an \(i\)-addable box directly beneath it, since every component in \(\blam\) is 2-restricted. For every arced \(\hat i\)-removable box in \(\blam[i]^{\,\,\snowflake\,}\) , move the bottom of the arc down one notch so that it connects to the \(i\)-addable box directly beneath it instead (we'll call these {\em type A} arcs). Now, reading from top to bottom, whenever an unarced \(i\)-removable box (frozen or not) is followed by an unarced \(i\)-addable box with no unarced \(i\)-removable or \(i\)-addable boxes between them, draw an arc connecting them (we'll call these {\em type B} arcs). Repeat until no more such arcs are possible. The resulting diagram is  \(\blam[i]^{\waterdrop}\). An example is shown in the middle in Figure~\ref{bigexampleoffreezingfig}.

\begin{lemma}\label{secondfrozenlem}
Let \(\blam \in \MPK^{\bkap}\) and \(i \in \Z_2\). Let \(u\) be an \(i\)-removable box, and \(v\) be an \(i\)-addable box in \(\blam\).
\begin{enumerate}
\item If \(u\) is unfrozen and unarced in \(\blam[i]^{\,\,\snowflake\,}\), then \(u\) is unarced in \(\blam[i]\).
\item If \(u\) is frozen or arced in \(\blam[i]^{\,\,\snowflake\,}\), then \(u\) is arced in \(\blam[i]^{\waterdrop}\).
\item If \(u,v\) are arced in \(\blam[i]^{\waterdrop}\), then every \(i\)-addable or \(i\)-removable box between them is arced in \(\blam[i]^{\waterdrop}\).
\item The box \(u\) is arced in \(\blam[i]^{\waterdrop}\)  if and only if \(u\) is arced in \(\blam[i]\).
\item The box \(v\) is arced in  \(\blam[i]^{\waterdrop}\)  if and only if \(v\) is arced in \(\blam[i]\).
\end{enumerate}
\end{lemma}

\begin{proof}
For (1), let \(u\) be unfrozen and unarced in \(\blam[i]^{\,\,\snowflake\,}\). 
Let \(X\) be the set of \(i\)-addable boxes below \(u\), and \(Y\) be the set of \(i\)-removable boxes below \(u\). We define now a function \(q:X \to Y\), considering separately the three possible cases. Let \(x \in X\).\\

\noindent{\em Case 1. Say \(x\) is in the top row of \(\lambda^{(t)}\).} Then set \(q(x)\) to be the lowest \(t\)-frozen box---which must be in \(Y\) since \(u\) is unfrozen, above \(x\), and the freezing process greedily \(t\)-freezes the lowest unfrozen \(i\)-removable box above \(x\).\\

\noindent{\em Case 2. Say \(x\) is directly beneath an \(\hat i\)-removable \(t\)-frozen box \(x'\).} Then set \(q(x)\) to be the lowest \(t\)-frozen \(i\)-removable box above \(x'\)---which again must be in \(Y\) since \(u\) is unfrozen, above \(x'\), and the freezing process greedily \(t\)-freezes the lowest unfrozen \(i\)-removable box above \(x'\).\\

\noindent{\em Case 3. Say \(x\) is directly beneath an \(\hat i\)-removable unfrozen box \(x'\).}
Then \(x'\) cannot be unarced in  \(\blam[i]^{\,\,\snowflake\,}\) since it is below \(u\) and \(u\) is unarced. Set \(q(x)\) to be the unfrozen \(i\)-removable box which is arced to \(x'\) in \(\blam[i]^{\,\,\snowflake\,}\)---which must be in \(Y\) since this arc cannot by construction bypass the unarced \(u\).\\

It is clear from construction that \(q\) is injective and \(q(x)\) is above \(x\) for all \(x \in X\), so we have that 
\begin{align}\label{zzy8sf}
\#\{x' \in X \mid x' \textup{ is above } x\} < \# \{y \in Y \mid y \textup{ is above } x\}
\end{align}
for all \(x \in X\). Now, say by way of contradiction that \(u\) is arced in \(\blam[i]\) to some \(x \in X\). Then every \(i\)-removable and \(i\)-addable box strictly between \(u\) and \(x\) must be arced in \(\blam[i]\), so
\begin{align*}
\#\{x' \in X \mid x' \textup{ is above }x\} = \# \{y \in Y \mid y \textup{ is above } x\}
\end{align*}
in contradiction of (\ref{zzy8sf}). Thus \(u\) is unarced in \(\blam[i]\), completing the proof of (1).

For (2), let \(Z\) be the set of frozen \(i\)-removable boxes in \(\blam\). Let \(W\) be the set of \(i\)-addable boxes in \(\blam\) which do not belong to a type A arc in  \(\blam[i]^{\waterdrop}\).
Let \(w\in W\), and assume that \(w\) is arced by a type B arc to some \(i\)-removable box \(x\) in \(\blam[i]^{\waterdrop}\). We claim that \(x \in Z\). Indeed, say by way of contradiction that \(x\) is unfrozen. Letting \(w'\) be the box directly above \(w\), we note that \(w'\) is by definition not arced in \(\blam[i]^{\,\,\snowflake\,}\) (else \(w\) would belong to a type A arc), so cannot be arced to \(x\) in \(\blam[i]^{\,\,\snowflake\,}\). If \(x\) were arced to some \(w'' \neq w'\) in \(\blam[i]^{\,\,\snowflake\,}\), then \(x\) would be arced via type A arc to the box directly beneath \(w''\) in \(\blam[i]^{\waterdrop}\) instead, so this cannot be the case either. Hence both \(w'\) and \(x\) are unarced in \(\blam[i]^{\,\,\snowflake\,}\), which again is impossible since \(w'\) is \(\hat i\)-removable and \(x\) is \(i\)-removable and \(w'\) is below \(x\), completing the claim.

We define now a function \(p: Z \to W\), considering the two possible cases. Let \(z \in Z\).\\

\noindent{\em Case 1. Say \(z\) is the lowest \(t\)-frozen box.} Then set \(p(z)\) to be the \(i\)-addable box in the top row of \(\lambda^{(t)}\)---such a box must exist and be below \(z\) by definition of the freezing process.\\

\noindent{\em Case 2. Say \(z \) is not the lowest \(t\)-frozen box.} Let \(w'\) be the nearest \(t\)-frozen box below \(z\), and set \(p(z)\) to be the addable box directly beneath \(w'\)---note that by definition of the freezing process that \(p(z)\) will have residue \( i\) and \(z'\) is not arced in \(\blam[i]^{\,\,\snowflake\,}\) so \(p(z) \in W\).\\

It is clear from construction that \(p\) is injective and \(p(z)\) is below \(z\) for all \(z \in Z\), so we have that 
\begin{align}\label{zzy8sggg}
\#\{w \in W \mid w \textup{ is below } z\} > \# \{z' \in Z \mid z' \textup{ is below } z\}
\end{align}
for all \(z \in Z\). Now assume that there exists a lowest \(z \in Z\) which is unarced in \(\blam[i]^{\waterdrop}\). Then every \(w \in W\) which is below \(z\) must be arced to some \(i\)-removable \(x\) below \(z\)---and by the claim above we also have that such \(x \in Z\). Thus it follows that
\begin{align}
\#\{w \in W \mid w \textup{ is below } z\} \leq \# \{z' \in Z \mid z' \textup{ is below } z\}
\end{align}
in contradiction of (\ref{zzy8sggg}). Thus no such \(z\) can exist and so every frozen box is arced in  \(\blam[i]^{\waterdrop}\). Finally, if \(u\) is arced in \(\blam[i]^{\,\,\snowflake\,}\) then by construction \(u\) is arced in \(\blam[i]^{\waterdrop}\), completing the proof of (2).

For (3), note that if \(u\) and \(v\) are connected by a type B arc in \(\blam[i]^{\waterdrop}\), then by construction there cannot be any unarced \(i\)-addable or \(i\)-removable boxes between \(u\) and \(v\). So we may assume \(u\) and \(v\) are connected by a type A arc in \(\blam[i]^{\waterdrop}\). Let \(v'\) be the unfrozen \(\hat i\)-removable box directly above \(v\). Then there is an arc from \(u\) to \(v'\) in \(\blam[i]^{\,\snowflake\,}\).
Let \(M\) be the set of \(i\)-addable boxes between \(u\) and \(v\), and let \(N\) be the set of \(i\)-removable boxes between \(u\) and \(v\). We define a map \(s: M \to N\) as follows. 
Letting \(m \in M\), we consider the three possible cases:\\

\noindent{\em Case 1. Say that \(m\) is in the top row of \(\lambda^{(t)}\).} Set \(s(m)\) to be the lowest \(t\)-frozen box---since \(m\) has residue \(i\), and \(u\) is unfrozen, it must be that \(s(m)\) has residue \(i\) and lies below \(u\), so \(s(m)\) is in \(N\).\\

\noindent{\em Case 2. Say there is a removable box \(m'\) directly above \(m\) which is \(t\)-frozen. } Set \(s(m)\) to be the nearest \(t\)-frozen box above \(m'\)---since \(m'\) has residue \(\hat i\) and \(u\) is unfrozen, it must be that \(s(m)\) lies below \(u\), so \(s(m)\) is in \(N\).\\

\noindent{\em Case 3. Say there is a removable box \(m'\) directly above \(m\) which is unfrozen.} Then since \(m'\) is unfrozen and falls within the arc connecting \(u\) and \(v'\) in \(\blam[i]^{\,\,\snowflake\,}\), it must be that \(m'\) is arced to some \(i\)-removable unfrozen box \(n\) below \(u\) in  \(\blam[i]^{\,\snowflake\,}\). Therefore \(m\) shares a type A arc with \(n\) in \(\blam[i]^{\waterdrop}\). Set \(s(m) = n\).\\

It is clear from construction that \(s\) is injective and \(s(m)\) is above \(m\) for all \(m \in M\), so we have that 
\begin{align}\label{yuyuy67}
\#\{n \in N \mid n \textup{ is above } m\} > \# \{m' \in M \mid m' \textup{ is above } m\}
\end{align}
for all \(m \in M\). Now if any \(m \in M\) is unarced in \(\blam[i]^{\waterdrop}\), it must be that every \(n \in N\) above \(m\) is arced to some \(m' \in M\) above \(m\), and so 
\begin{align*}
\#\{n \in N \mid n \textup{ is above } m\} \leq \# \{m' \in M \mid m' \textup{ is above } m\}
\end{align*}
in contradiction of (\ref{yuyuy67}). Thus every \(m \in M\) is arced in \(\blam[i]^{\waterdrop}\). 

Now let \(n \in N\). If \(n\) is frozen, then it is arced in \(\blam[i]^{\waterdrop}\) by (2). If \(n\) is unfrozen, then since \(n\) lies between \(u\) and \(v'\), which are arced in  \(\blam[i]^{\,\snowflake\,}\), it follows that \(n\) must be arced in \(\blam[i]^{\,\snowflake\,}\) and therefore belongs to a type A arc in \(\blam[i]^{\waterdrop}\). This completes the proof of (3).

For (4) and (5), we will generally have arcs which cross in \(\blam[i]^{\waterdrop}\). We may iteratively refine this diagram, replacing any oriented crossing \(\cross\) with \(\uncross\)\,, until no more crossings remain. Call this diagram \(\textup{Unc}(\blam[i]^{\waterdrop})\). We note the following features of \(\textup{Unc}(\blam[i]^{\waterdrop})\), which follow directly from previous results and consideration of the uncrossing process.
\begin{enumerate}
\item[(a)] All arcs in \(\textup{Unc}(\blam[i]^{\waterdrop})\) are oriented downward from an \(i\)-removable to an \(i\)-addable box in \(\textup{Unc}(\blam[i]^{\waterdrop})\).
\item[(b)] There are no unarced \(i\)-addable or \(i\)-removable boxes trapped within any arc in \(\textup{Unc}(\blam[i]^{\waterdrop})\) (which follows from (3)).
\item[(c)] There is no unarced \(i\)-addable box below an unarced \(i\)-removable box in \(\textup{Unc}(\blam[i]^{\waterdrop})\) (since this does not occur in \(\blam[i]^{\waterdrop}\)).
\item[(d)] No arcs cross in \(\textup{Unc}(\blam[i]^{\waterdrop})\).
\end{enumerate}

It is straightforward to check that the features (a)--(d) of \(\textup{Unc}(\blam[i]^{\waterdrop})\) noted above uniquely distinguish it as the diagram \(\blam[i] = \textup{Unc}(\blam[i]^{\waterdrop})\). As the arcedness/unarcedness of \(i\)-addable and \(i\)-removable boxes in \(\blam[i]^{\waterdrop}\) is unchanged by the uncrossing process that obtains \(\textup{Unc}(\blam[i]^{\waterdrop})\), claims (4) and (5) of the lemma statement immediately follow. 
\end{proof}

\begin{theorem}\label{WKisommainproof}
We have \(\mathcal{W}^{\bkap}(\UL) \subseteq \MPK^{\bkap}\), and  \(\mathcal{W}^{\bkap}: \UL \to \MPK^{\bkap}\) and \(\mathcal{Z}^{\bkap}: \MPK^{\bkap} \to \UL\) are mutually inverse crystal isomorphisms.
\end{theorem}
\begin{proof}
Let \(\sfD \in \UL\). Let \(x\) be a box of \(\sfD\) and \(\tilde x\) be its image in \(\blam = \mathcal{W}^{\bkap}(\sfD)\). We claim that \(x\) is unarced \(i\)-bottom in \(\sfD[  i]\) if and only if \(x\) is unarced \(i\)-removable in \(\blam[i]\). Indeed, if \(x\) is an unarced \(i\)-bottom box in \(\sfD\), then \(\tilde x\) is unarced unfrozen \(i\)-removable in \(\blam[i]^{\,\,\snowflake\,}\) by Lemma~\ref{firstfrozenlem}(3), and so \(\tilde x\) is unarced in \(\blam[i]\) by Lemma~\ref{secondfrozenlem}(1).
On the other hand, if \(\tilde x\) is unarced \(i\)-removable in \(\blam[i]\), then by Lemma~\ref{secondfrozenlem}(2),(4) \(\tilde x\) is unfrozen and unarced in \(\blam[i]^{\,\,\snowflake\,}\), and so \(x\) is an unarced \(i\)-bottom box in \(\sfD\) by Lemma~\ref{firstfrozenlem}(3).

Thus by Lemma~\ref{firstfrozenlem}(2) \(x\) is the rightmost unarced \(i\)-bottom box in \(\sfD[ i]\) if and only if \(\tilde x\) is the uppermost unarced \(i\)-removable box in \(\blam[i]\). Therefore \(e_i \sfD = 0\) if and only if \(e_i \blam = 0\). On the other hand if \(e_i \sfD, e_i \blam \neq 0\), then by Lemma~\ref{candelbottoms} we have that \(\mathcal{W}^{\bkap}(e_i \sfD) = e_i \blam = e_i \mathcal{W}^{\bkap}(\sfD)\).

Assume that \(\mathcal{W}^{\bkap}(f_{i_k} \cdots f_{i_1} \varnothing) = f_{i_k} \cdots f_{i_1} \varnothing \) for some \(i_1, \ldots, i_k, i_{k+1} \in \Z_2\).
Then we have
\begin{align*}
e_{i_{k+1}} f_{i_{k+1}} \cdots f_{i_{1}} \varnothing &= f_{i_k} \cdots f_{i_{1}} \varnothing = \mathcal{W}^{\bkap}( f_{i_k} \cdots f_{i_{1}} \varnothing)
= \mathcal{W}^{\bkap}(e_{i_{k+1}} f_{i_{k+1}} \cdots f_{i_{1}} \varnothing)
= e_{i_{k+1}}\mathcal{W}^{\bkap}( f_{i_{k+1}} \cdots f_{i_{1}} \varnothing),
\end{align*}
and therefore 
\begin{align*}
 f_{i_{k+1}} \cdots f_{i_{1}}  = 
f_{i_{k+1}} e_{i_{k+1}} f_{i_{k+1}} \cdots f_{i_{1}} 
=
f_{i_{k+1}}e_{i_{k+1}}\mathcal{W}^{\bkap}( f_{i_{k+1}} \cdots f_{i_{1}} \varnothing)
=\mathcal{W}^{\bkap}( f_{i_{k+1}} \cdots f_{i_{1}} \varnothing).
\end{align*}
Thus, since \(\mathcal{W}^{\bkap}(\varnothing) = \varnothing\) by definition, it follows from induction that \(\mathcal{W}^{\bkap}(f_{i_m} \cdots f_{i_1} \varnothing) = f_{i_m} \cdots f_{i_1} \varnothing\) for all \(m \in \Z_{\geq 0}\) and all \(i_1, \ldots, i_m \in \ZZ_2\). Then since \(\UL\) and \(\MPK^{\bkap}\) are highest weight \({\tt A}^{(1)}_1\)-crystals generated by \(\varnothing\), it follows that \(\mathcal{W}^{\bkap}(\UL) \subseteq \MPK^{\bkap}\) and \(\mathcal{W}^{\bkap}\) commutes with \(f_i\) for \(i \in \Z_2\).

To see injectivity of \(\mathcal{W}^{\bkap}\), assume by way of contradiction that \(\sfD , \sfD' \in \UL\) are of minimal rank such that \(\sfD \neq \sfD'\) and \(\mathcal{W}^{\bkap}(\sfD) = \mathcal{W}^{\bkap}(\sfD')\). There exists \(i \in \Z_2\) such that \(e_i\mathcal{W}^{\bkap}(\sfD) \neq 0\), and by above we have
\begin{align*}
\mathcal{W}^{\bkap}( e_i \sfD) = e_i \mathcal{W}^{\bkap}(\sfD) = e_i \mathcal{W}^{\bkap}(\sfD') = \mathcal{W}^{\bkap}( e_i \sfD'), 
\end{align*}
which implies by minimality that \(e_i \sfD = e_i \sfD'\), and so \(\sfD = f_i e_i \sfD = f_i e_i \sfD' = \sfD'\), giving the desired contradiction.

That \(\mathcal{W}^{\bkap}\) commutes with \(\textup{wt}\), \(\varphi_i\), \(\varepsilon_i\) follows from definitions, so  \(\mathcal{W}^{\bkap}: \UL \to \MPK^{\bkap}\) is a crystal isomorphism. By construction, the map \(\GLUE\) is a left inverse to \(\SPLIT^i\), so \(\mathcal{Z}^{\bkap}\) is a left inverse to \(\mathcal{W}^{\bkap}\), which completes the proof.
\end{proof}

\subsection{Recognizing Kleshchev multipartitions}
For \(\sfD \in \UL\), we will write \(\textup{ht}(\sfD)\) for the {\em height} (number of nonzero rows) of \(\sfD\). We will write \(\textup{wid}(\sfD)\) for the {\em width} (length of a peak row) of \(\sfD\).
For a multicharge \(\bkap\), we will say \(\usfD \in \UL^m\) is {\em \(\bkap\)-good} provided it satisfies the following criteria. For \(t \in [1,m]\), set \(\tilde{\kappa}_t(\usfD) = \kappa_t\) if \(\sfD_t = \varnothing\), and otherwise set \(\tilde{\kappa}_t(\usfD)\) to be the color of the top row of \(\sfD_t\). 
\begin{enumerate}
\item For all \(t \in [1,m-1]\), \( \tilde{\kappa}_t(\usfD) = \kappa_t\).
\item For all \(t \in [1,m-1]\), \(\sfD_t\) is {\em staircase}; the top boxes of \(\sfD_t\) lie along a \(-45^\circ\) diagonal.
\item For all \(t \in [1,m-1]\), \(\textup{wid}(\sfD_t) - \textup{ht}(\sfD_{t+1}) \geq \delta_{\tilde{\kappa}_t(\usfD) , \tilde{\kappa}_{t+1}(\usfD) } - 1\).
\item For all \(s \in [1,m-2]\), \(u \in [s+2,m]\) such that \(\sfD_u \neq \varnothing\) and \(\tilde{\kappa}_s(\usfD)  \neq \tilde{\kappa}_{s+1}(\usfD)  = \cdots = \tilde{\kappa}_{u-1}(\usfD)  \neq \tilde{\kappa}_u(\usfD) \), there exists some \(t \in [s,u-1]\) such that \(\textup{wid}(\sfD_t) - \textup{ht}(\sfD_{t+1}) \geq \delta_{\tilde{\kappa}_t(\usfD) , \tilde{\kappa}_{t+1}(\usfD) } + 1.\)
\end{enumerate}

\begin{lemma}\label{kappagoodlemma}
Let \(\usfD \in \UL^m\) be \(\bkap\)-good. Then we have:
\begin{enumerate}
\item[(a)]\(s^{\bkap}(\usfD)\) is \(\bkap\)-good and \(g^{\bkap} \circ s^{\bkap}(\usfD) = \usfD\).
\item[(b)] If \(m>1\) then \(g^{\bkap}(\usfD)\) is \(\bkap\)-good and \(s^{\bkap} \circ g^{\bkap}(\usfD) = \usfD\).
\end{enumerate}
\end{lemma}
\begin{proof}
We first record some easy observations about the split function. For \(\sfD \in \UL\) and \(i \in \Z_2\) we have:
\begin{enumerate}
\item[(i)] If the top row of \(\sfD\) is colored \(i\), then \(\textup{ht}(\SPLIT^i(\sfD)_\downarrow) = \textup{ht}(\sfD)\) and \(\textup{ht}(\SPLIT^i(\sfD)_\uparrow) \leq \textup{wid}(\SPLIT^i(\sfD)_\downarrow) \), with equality only if the top row of \(\SPLIT^i(\sfD)_\uparrow\) is colored \(i\).
\item[(ii)] If the top row of \(\sfD\) is colored \(\hat i\), then \(\textup{ht}(\SPLIT^i(\sfD)_\downarrow) =\textup{ht}(\sfD) - 1\) and we have \(\textup{ht}(\SPLIT^i(\sfD)_\uparrow) \in \{\textup{wid}(\SPLIT^i(\sfD)_\downarrow), \textup{wid}(\SPLIT^i(\sfD)_\downarrow) + 1\}\) and the top row of \( \SPLIT^i(\sfD)_\uparrow \) is colored \(\hat i\).
\end{enumerate}

Now we consider (a). It is straightforward to check that \(\GLUE\) is a left inverse of \(\SPLIT^i\) for all upper ledge diagrams, so the latter claim is immediate. Assume \(\usfC \in \UL^{m-1}\) is \(\bkap\)-good, and set \(\usfD = s^{\bkap}(\sfD) \in \UL^{m}\). We check that properties (1)--(4) hold for \(\usfD\). Note first that since \(\usfC\) is \(\bkap\)-good and \(\sfD_t = \sfC_t\) for \(t \in [1,m-2]\), we have that (1),(2) hold when \(t \in [1,m-2]\), (3) holds when \(t \in [1,m-3]\), and (4) holds when \(u \in [1,m-3]\). Moreover, since \(\sfD_{m-1} = \SPLIT^{\kappa_{m-1}}(\sfC_{m-1})_{\downarrow}\), we have that (1),(2) hold for \(t=m-1\) by definition of the map \(\SPLIT^{\kappa_{m-1}}\). Thus it remains to check (3) when \(t \in \{m-2,m-1\}\) and (4) when \(u \in \{m-2,m-1\}\).
We consider this in cases.\\

\noindent{\em Case 1. Assume \(\tilde{\kappa}_{m-1}(\usfC) = \kappa_{m-1}\)}. In this case we have by (i)  that 
\begin{align*}
\textup{ht}(\sfD_{m-1}) = \textup{ht}(\sfC_{m-1}) \leq \textup{wid}(\sfC_{m-2}) +1 - \delta_{\tilde{\kappa}_{m-2}(\usfC), \tilde{\kappa}_{m-1}(\usfC)} = \textup{wid}(\sfD_{m-2}) + 1 - \delta_{\tilde{\kappa}_{m-2}(\usfD), \tilde{\kappa}_{m-1}(\usfD)},
\end{align*}
where the inequality follows from the fact that \(\usfC\) is \(\bkap\)-good. Thus (3) holds when \(t=m-1\). We also have by (i) that 
\begin{align}\label{923847}
\textup{wid}(\sfD_{m-1})  \geq 
 \textup{ht}(\sfD_m) + 1 - \delta_{\tilde{\kappa}_m(\usfD), \kappa_{m-1}} = \textup{ht}(\sfD_m) + 1 - \delta_{\tilde{\kappa}_m(\usfD), \tilde{\kappa}_{m-1}(\usfD)} 
\end{align}
which implies that (3) holds when \(t=m\). 

Now assume that the assumptions of (4) are met in \(\usfD\), with \(u=m-1\). Then the same assumptions are met in \(\usfC\) since \(\tilde{\kappa}_{m-1}(\usfD) = \kappa_{m-1} = \tilde{\kappa}_{m-1}(\usfC) \). But then (4) holds for \(u=m-1\) in \(\usfD\) because it holds in \(\usfC\) and \(\textup{ht}(\sfD_{m-1}) = \textup{ht}(\sfC_{m-1})\). Now assume the conditions of (4) are met in \(\usfD\), with \(u=m\). Then \(\tilde{\kappa}_m(\sfD) \neq \tilde{\kappa}_{m-1}(\sfD)\), so by (\ref{923847}) we have that \(\textup{wid}(\sfD_{m-1}) - \textup{ht}(\sfD_m) \geq 1\), and thus (4) holds for \(\usfD\) when \(u=m\).\\

\noindent{\em Case 2. Assume \(\tilde{\kappa}_{m-1}(\usfC) \neq \kappa_{m-1}\)}. In this case we have that by (ii) that
\begin{align*}
 \textup{ht}(\sfD_{m-1}) &= \textup{ht}(\sfC_{m-1}) - 1 \leq \textup{wid}(\sfC_{m-2}) - \delta_{ \tilde{\kappa}_{m-2}(\usfC), \tilde{\kappa}_{m-1}(\usfC)} = \textup{wid}(\sfD_{m-2}) -1+
 \delta_{ \tilde{\kappa}_{m-2}(\usfD), \tilde{\kappa}_{m-1}(\usfD)}
\end{align*}
where the inequality follows from the fact that \(\usfC\) is \(\bkap\)-good, and the last equality follows from the fact that  \( \tilde{\kappa}_m(\usfD)  = \tilde{\kappa}_{m-1}(\usfC) \neq \kappa_{m-1} = \tilde{\kappa}_{m-1}(\usfD)\). Thus (3) holds when \(t= m-1\). We also have by (ii) that 
\begin{align*}
\textup{wid}(\sfD_{m-1}) \geq \textup{ht}(\sfD_m) -1,
\end{align*}
which implies that (3) holds when \(t=m\).

Now we assume the assumptions of (4) are met in \(\usfD\), with \(u=m-1\). Then \(\tilde{\kappa}_{m-2}(\usfC) = \tilde{\kappa}_{m-2}(\usfD) \neq \tilde{\kappa}_{m-1}(\usfD) \neq \tilde{\kappa}_{m-1}(\usfC)\), and since \(\usfC\) is \(\bkap\)-good, we have that 
\begin{align}\label{xxyxyx}
\textup{wid}(\sfD_{m-2}) = \textup{wid}(\sfC_{m-2}) \geq \textup{ht}(\sfC_{m-1}) = \textup{ht}(\sfD_{m-1}) + 1.
\end{align}
thus (4) holds for \(u= m-1\) in \(\usfD\).

Now assume the conditions of (4) are met in \(\usfD\), with \(u=m\). If \(s = m-2\), then \(\tilde{\kappa}_{m-2}(\usfC) = \tilde{\kappa}_{m-2}(\usfD) \neq \tilde{\kappa}_{m-1}(\usfD) \neq \tilde{\kappa}_{m-1}(\usfC)\), and thus (4) holds again as in (\ref{xxyxyx}).
Thus we may assume \(s\leq m-3\). Then we have that
\begin{align*}
\tilde{\kappa}(\usfC)_s \neq \tilde{\kappa}(\usfC)_{s+1} = \cdots  =  \tilde{\kappa}(\usfC)_{m-2} \neq \tilde{\kappa}(\usfC)_{m-1},
\end{align*}
and since \(\usfC\) is \(\bkap\)-good, there exists some \(t \in [s,m-2]\) such that \(\textup{wid}(\sfC_t) - \textup{ht}(\sfC_{t+1}) \geq \delta_{\tilde{\kappa}_t(\usfC), \tilde{\kappa}_{t+1}(\usfC)}+1\). If \(t \leq m-3\), then this immediately implies that \(\textup{wid}(\sfD_t) - \textup{ht}(\sfD_{t+1}) \geq \delta_{\tilde{\kappa}_t(\usfD), \tilde{\kappa}_{t+1}(\usfD)}+1\), as desired. If \(t=m-2\), then 
\begin{align*}
\textup{wid}(\sfD_{m-2}) = 
 \textup{wid}(\sfC_{m-2}) \geq \textup{ht}(\sfC_{m-1}) + 1 = \textup{ht}(\sfD_{m-1}) + 2,
\end{align*}
as desired. Thus (4) holds for \(u=m\) in \(\usfD\), which completes the proof of (a).

Now we consider (b). Assume \(\usfC \in \UL^{m+1}\) is \(\bkap\)-good, and set \(\usfD = g^{\bkap}(\usfC)\). 
Since \(\usfC\) is \(\bkap\)-good, we have by (3) that \(\textup{wid}(\sfC_m) \geq \textup{ht}(\sfC_{m+1}) + \delta_{\tilde{\kappa}_m(\usfC), \tilde{\kappa}_{m+1}(\usfC)} - 1\). This, together with the fact that \(\sfC_m\) is staircase, ensures that, when lined up as in Figure~\ref{gengluefig} to compute  \(\GLUE(\sfC_m, \sfC_{m+1})\), the bottom row of \(\sfC_{m+1}\) does not extend below the uppermost peak row of \(\sfC_m\). From this fact, it is straightforward to see that \(\SPLIT^{\kappa_m} \circ \GLUE(\sfC_m, \sfC_{m+1}) = (\sfC_m, \sfC_{m+1})\), and so \(s^{\bkap} \circ g^{\bkap}( \usfC)  = \usfC\), proving the latter part of (b).

We now check that properties (1)--(4) hold for \(\usfD\). Note first that since \(\usfC\) is \(\bkap\)-good and \(\sfD_t = \sfC_t\) for \(t \in [1,m-1]\), we have immediately that (1), (2) hold, (3) holds when \(t \in [1,m-2]\), and (4) holds when \(u \in [1,m-1]\). Thus it remains to check (3) when \(t=m-1\) and (4) when \(u = m\). We do so now, in three separate cases. \\

\noindent{\em Case 1. Assume \(\tilde{\kappa}_m(\usfC) = \tilde{\kappa}_{m+1}(\usfC)\) or \(\textup{ht}(\sfC_{m+1})< \textup{wid}(\sfC_m)\).} 
By definition of the \(\GLUE\) map and the fact that \(\sfC_m\) is staircase, we have \(\tilde{\kappa}_m(\usfD) = \tilde{\kappa}_m(\usfC)\), and 
\begin{align*}
\textup{ht}(\sfD_m)= \textup{ht}(\sfC_m) 
& \leq \textup{wid}(\sfC_{m-1}) - \delta_{ \tilde{\kappa}_{m-1}(\usfC), \tilde{\kappa}_m(\usfC)   } + 1
= \textup{wid}(\sfD_{m-1}) - \delta_{ \tilde{\kappa}_{m-1}(\usfD), \tilde{\kappa}_m(\usfD)   } + 1,
\end{align*}
where the inequality follows from the fact that \(\usfC\) is \(\bkap\)-good. Therefore (3) holds when \(t=m-1\). Note moreover that since \(\textup{ht}(\sfD_m) = \textup{ht}(\sfC_m)\) and \(\sfD_t = \sfC_t\) for \(t \in [1,m-1]\), if the conditions for (4) are met in \(\usfD\) with \(s=m\), then they are similarly met in \(\usfC\), and so (4) holds for \(\usfD\) since \(\usfC\) is \(\bkap\)-good.
\\

\noindent{\em Case 2. Assume \(\tilde{\kappa}_m(\usfC) \neq \tilde{\kappa}_{m-1}(\usfC) =\tilde{\kappa}_{m+1}(\usfC)\) and \(\textup{ht}(\sfC_{m+1}) \geq \textup{wid}(\sfC_m)\).} Since \(\usfC\) is \(\bkap\)-good, we have that \(\textup{ht}(\sfC_{m+1}) \in \{\textup{wid}(\sfC_m), \textup{wid}(\sfC_m)+1\}\). As \(\sfC_m\) is staircase, the definition of the \(\GLUE\) map gives us that \(\tilde{\kappa}_m(\usfD) = \tilde{\kappa}_{m+1}(\usfC) = \tilde{\kappa}_{m-1}(\sfC) = \tilde{\kappa}_{m-1}(\sfD)\) (so condition (4), \(u=m\) cannot be met in \(\usfD\)), and \(\textup{ht}(\sfD_m) = \textup{ht}(\sfC_m) + 1\). Moreover, since \(\usfC\) is \(\bkap\)-good, we have by (4) then that  \(\textup{wid}(\sfC_{m-1}) - \textup{ht}(\sfC_m) \geq 1\), and thus
\begin{align*}
\textup{wid}(\sfD_{m-1}) = \textup{wid}(\sfC_{m-1}) \geq \textup{ht}(\sfC_m) + 1 = \textup{ht}(\sfD_m),
\end{align*}
so (3) holds for \(\usfD\) when \(t=m-1\). \\

\noindent{\em Case 3. Assume \(\tilde{\kappa}_{m-1}(\usfC) = \tilde{\kappa}_m(\usfC) \neq \tilde{\kappa}_{m+1}(\usfC)\) and \(\textup{ht}(\sfC_{m+1}) \geq \textup{wid}(\sfC_m)\).} As above, we have that \(\textup{ht}(\sfC_{m+1}) \in \{\textup{wid}(\sfC_m), \textup{wid}(\sfC_m)+1\}\). As \(\sfC_m\) is staircase, the definition of the \(\GLUE\) map gives us that \(\tilde{\kappa}_m(\usfD) = \tilde{\kappa}_{m+1}(\usfC) \neq \tilde{\kappa}_{m-1}(\sfC) = \tilde{\kappa}_{m-1}(\sfD)\), and \(\textup{ht}(\sfD_m) = \textup{ht}(\sfC_m) + 1\). Therefore since \(\usfC\) is \(\bkap\)-good, we have by (3) that 
\begin{align*}
\textup{wid}(\sfD_{m-1}) = \textup{wid}(\sfC_{m-1}) \geq \textup{ht}(C_m) = \textup{ht}(\sfD_m) - 1, 
\end{align*}
so (3) holds for \(\usfD\) when \(t=m-1\). 

Now assume the conditions of (4) are met in \(\usfD\), with \(u=m\). Then we must have \(s\leq m-2\), and
\begin{align*}
\tilde{\kappa}_{s}(\usfC) \neq \tilde{\kappa}_{s+1}(\usfC) = \cdots = \tilde{\kappa}_m(\usfC) \neq \tilde{\kappa}_{m+1}(\usfC)
\end{align*}
and so since \(\usfC\) is \(\bkap\)-good, we have by (4) that \(\textup{wid}(\sfC_t) - \textup{ht}(\sfC_{t+1}) \geq \delta_{\tilde{\kappa}_t(\usfC), \tilde{\kappa}_{t+1}(\usfC)} + 1\) for some \(t \in [s, m-1]\) (the case \(t=m\) being already ruled out above). If \(t \leq m-2\), then we immediately have that (4) holds for \(\usfD\), so assume finally that \(t=m-1\). Then
\begin{align*}
\textup{wid}(\sfD_{m-1}) = \textup{wid}(\sfC_{m-1}) \geq \textup{ht}(\sfC_m) + 2 = \textup{ht}(\sfD_m) + 1,
\end{align*}
and so (4) holds for \(\usfD\), completing the proof.
\end{proof}

\begin{theorem}\label{Kleshtestappend}
Let \(\bkap\) be a multicharge, and let \(\blam \in \MP^{\bkap}\) be a multipartition of level \(\ell\). Then \(\blam \in \MPK^{\bkap}\) if and only if it satisfies the following conditions:
\begin{enumerate}
\item The partition \(\lambda^{(t)}\) is 2-restricted for all \(t \in [1,\ell]\).
\item For all \(t \in [1,\ell-1]\), we have \(\lambda^{(t)}_1 - (\lambda^{(t+1)})'_1 \geq \delta_{\kappa_t, \kappa_{t+1}} - 1\).
\item If \(s \in [1,\ell-2]\), \(u \in [s+2,\ell]\) are such that \(\lambda^{(u)} \neq \varnothing\) and \(\kappa_s \neq \kappa_{s+1} = \cdots =  \kappa_{u-1} \neq \kappa_u\), then there exists  \(t \in [s,u-1]\) such that  \(\lambda^{(t)}_1 - (\lambda^{(t+1)})'_1 \geq \delta_{\kappa_t, \kappa_{t+1}}+1\). 
\end{enumerate}
\end{theorem}
\begin{proof}
Assume \(\blam \in \MPK^{\bkap}\). Then by Theorem~\ref{WKisommainproof}, we have \(\blam = \mathcal{W}^{\bkap}(\sfD)\) for some \(\sfD \in \UL\). In view of the definition of \(\mathcal{W}^{\bkap}\), we have that \(\blam\) is achieved by computing
\(
\usfD= (\sfD_1, \ldots, \sfD_\ell, \varnothing) := s^{\kappa_{\ell }} \circ \cdots \circ s^{\kappa_1}( \sfD),
\)
and then shifting the \(k\)th column in each component up by \(k-1\) notches. Now, the single upper ledge diagram \(\sfD\) is vacuously a \(\bkap\)-good sequence, and so it follows from Lemma~\ref{kappagoodlemma} that \(\usfD\) is \(\bkap\)-good. It is immediate that column-shifting a staircase upper ledge diagram in the above fashion yields a 2-restricted partition, so (1) is satisfied for \(\blam\). Moreover, we have \(\textup{wid}(\sfD_t) = \lambda^{(t)}_1\) and \(\textup{ht}(\sfD_t) = (\lambda^{(t)})'_1\) for all \(t \in [1,\ell]\), so (2), (3) are satisfied for \(\blam\) as well since \(\usfD\) is \(\bkap\)-good.

Now assume \(\blam\) satisfies conditions (1)--(3) above. Let \(\usfD = (\sfD_1, \ldots, \sfD_\ell, \varnothing)\) be the sequence of diagrams given by shifting the \(k\)th column in each component down by \(k-1\) notches. By (1), each component is staircase upper ledge, and (2), (3) ensure that \(\usfD\) is \(\bkap\)-good. Therefore, setting \(\sfD:= g^{\kappa_1} \circ \cdots \circ g^{\kappa_\ell}(\usfD) \in \UL\), we have by Lemma~\ref{kappagoodlemma} that \(\usfD = s^{\kappa_\ell} \circ \cdots \circ s^{\kappa_1}(\sfD)\), and so \(\blam = \mathcal{W}^{\bkap}(\sfD)\). Thus by Theorem~\ref{WKisommainproof} we have that \(\blam \in \MPK^{\bkap}\).
\end{proof}

\section{Representations of KLR algebras}\label{bigKLRsec}

\subsection{The KLR algebra of type \({\tt A}_1^{(1)}\).}
Fix a field \(\k\) of characteristic \(p \geq 0\). Let \(\theta \in \ZZ_{\geq 0}I\), and set \(n= \height(\theta) \). We write \(I^\theta\) for the set of binary words \(\bi = i_1 \cdots i_n \in \Z_2^n\) such that \(\sum_{k=1}^n \alpha_{i_k} = \theta\). 
As in \cite{KhovLauda1, Rouq1}, the {\em KLR} (or {\em quiver Hecke}) algebra of type \({\tt A}_{1}^{(1)}\) is the \(\Z\)-graded \(\k\)-algebra \(R_\theta\) generated by elements
\begin{align*}
\{1_{\bi} \mid \bi \in I^\theta\} \cup \{y_1, \dots, y_n\} \cup \{\psi_1, \dots, \psi_{n-1}\}, 
\end{align*}
subject to the relations:
\begin{align*}
1_{\bi} 1_{\bj} &= \delta_{\bi, \bj} 1_{\bi};
\qquad
\sum_{\bi \in I^\theta}1_{\bi} = 1;
\qquad
y_r 1_{\bi} = 1_{\bi} y_r;
\qquad
y_r y_s = y_s y_r;
\qquad
\psi_r 1_{\bi} = 1_{s_r \bi} \psi_r;
\end{align*}
\begin{align*}
(y_r \psi_r - \psi_r y_{r+1}) 1_{\bi} = \delta_{i_r, i_{r+1}} 1_{\bi} = (y_{r+1} \psi_r - \psi_r y_r);
\qquad
\psi_r^2 1_{\bi}
=- \delta_{i_r, \hat{i}_{r+1}}(y_r^2 - 2y_r y_{r+1} + y_{r+1}^2)1_{\bi};
\end{align*}
\begin{align*}
(\psi_{r+1} \psi_r\psi_{r+1} - \psi_r\psi_{r+1}\psi_r )1_{\bi} = -\delta_{i_r,i_{r+2}} \delta_{i_r, \hat{i}_{r+1}}(y_r - 2y_{r+1} + y_{r+2}) 1_{\bi},
\end{align*}
for all admissible \(\bi, \bj \in I^\theta\) and \(r,s \in [1,n]\).
The \(\Z\)-grading is given by setting \(\deg(1_{\bi}) = 0\), \(\deg(y_r) = 2\), \(\deg(\psi_r 1_{\bi}) = 2 - 4\delta_{i_r, i_{r+1}}\). To each permutation \(\sigma \in \mathfrak{S}_n\), we associate a choice of distinguished reduced expression \(\sigma = s_{j_1} \dots s_{j_k}\) and define an associated element \(\psi_\sigma := \psi_{j_1} \dots \psi_{j_k} \in R_\theta\).

\subsection{Representation theory of \(R_\theta\)}
We consider the category \(R_\theta\)-mod of finitely generated \(\ZZ\)-graded \(R_\theta\)-modules. We will use the \(\cong\) symbol to indicate a (degree-preserving) isomorphism of \(R_\theta\)-modules, and \(\approx\) to indicate an isomorphism of \(R_\theta\)-modules up to some grading shift.
For \(\theta_1,\ldots, \theta_k \in \Z_{\geq 0}I\), there is a natural inclusion 
\begin{align*}
R_{\theta_1, \ldots, \theta_k}:= R_{\theta_1} \otimes \cdots \otimes  R_{\theta_k} \hookrightarrow R_{\theta_1 + \cdots + \theta_k},
\end{align*} 
and we write \(1_{\theta_1, \ldots, \theta_k}\) for the image of the identity under this inclusion.
This yields accompanying exact induction and restriction functors
\begin{align*}
&\Ind_{\theta_1,\ldots, \theta_k}^{\theta_1 + \cdots + \theta_k}
:= R_{\theta_1 + \cdots + \theta_k}1_{\theta_1, \ldots, \theta_k}\otimes_{R_{\theta_1, \ldots, \theta_k}}?
\colon \;R_{\theta_1, \ldots, \theta_k}\textup{-mod} \to R_{\theta_1 + \cdots + \theta_k}\textup{-mod},
\\
&\Res_{\theta_1,\ldots, \theta_k}^{\theta_1 + \cdots + \theta_k}
:=
1_{\theta_1, \ldots, \theta_k} R_{\theta_1 + \cdots + \theta_k} \otimes_{R_{\theta_1 + \cdots + \theta_k} }?
\colon \;R_{\theta_1 + \cdots + \theta_k}\textup{-mod} \to R_{\theta_1, \ldots, \theta_k}\textup{-mod}.
\end{align*}
For \(M_i \in R_{\theta_i}\) we will sometimes use the shorthand \(M_1 \circ \cdots \circ M_k:= \Ind_{\theta_1,\ldots, \theta_k}^{\theta_1 + \cdots + \theta_k}(M_1 \boxtimes \cdots \boxtimes M_k)\).

\subsubsection{Chevalley functors}\label{transfuncsec}
The KLR algebras \(R_{0}, R_{\alpha_1}, R_{\alpha_0}\) each have a single one-dimensional simple module concentrated in degree zero, which we denote as \(\k_0, \k_{\alpha_1}, \k_{\alpha_0}\), respectively. For \(i \in \Z_2\) we denote the functors
\begin{align*}
&E_i: R_{\theta+ \alpha_i}\textup{-mod} \to R_{\theta}\textup{-mod}, \quad M \mapsto \Res_{\theta, \alpha_i}^{\theta+\alpha_i}M\\
&F_i: R_{\theta}\textup{-mod} \to R_{\theta+\alpha_i}\textup{-mod}, \quad M \mapsto \Ind_{\theta, \alpha_i}^{\theta+ \alpha_i}(M \boxtimes \k_{\alpha_i})\\
&E_i^*: R_{\theta+ \alpha_i}\textup{-mod} \to R_{\theta}\textup{-mod}, \quad M \mapsto \Res_{\alpha_i, \theta}^{\theta+\alpha_i}M\\
&F_i^*: R_{\theta}\textup{-mod} \to R_{\theta+\alpha_i}\textup{-mod}, \quad M \mapsto \Ind_{ \alpha_i, \theta}^{\theta+ \alpha_i}(\k_{\alpha_i}\boxtimes M),
\end{align*}
where we interpret \(E_iM\) (resp. \(E_i^*M\)) as an \(R_\theta\)-module via the restriction induced by the inclusion \(R_\theta \hookrightarrow R_{\theta, \alpha_i}\) (resp. \(R_\theta \hookrightarrow R_{\alpha_i, \theta}\)).

\begin{remark}
The KLR algebras \(R_\theta\) categorify the lower half of the quantum group associated to \(\hat{\mathfrak{sl}}_2(\C)\), in the sense that Grothendieck rings of the categories of finitely generated projective and finite dimensional graded KLR modules recover (integral forms of) \(U_q^-(\hat{\mathfrak{sl}}_2)\) and its dual, respectively \cite{KhovLauda1, Rouq1}. The functors \(E_i, F_i, E_i^*, F_i^*\) satisfy categorified Serre relations, and the set of simple KLR modules taken together with the action of these functors underwrites an \({\tt A}^{(1)}_1\)-crystal structure isomorphic to \(\mathcal{B}(\infty)\), see  \cite{LVcrystals} and \S\ref{branchingpoly}.
\end{remark}

\subsubsection{Graded characters}
Let \(\theta \in \Z_{\geq 0}I\) and \(M \in R_\theta\textup{-mod}\). For \(d \in \Z\) we denote the \(d\)th degree subspace in \(M\) by \(M_d\), and use \(q^cM\) to denote grading shift by \(c\), so \((q^cM)_m = M_{m-c}\). The {\em graded character} of \(M\) is 
\begin{align*}
\textup{ch}\, M := \sum_{\bi \in I^\theta} \sum_{d \in \Z} \dim(1_{\bi} M_d)q^d  \bi \in \Z_{\geq 0}[q^{\pm 1}] I^\theta.
\end{align*}
We say that \(\bi \in I^\theta\) is a {\em word in \(M\)} provided that \(1_{\bi} M \neq 0\), i.e., \(\bi\) is in the support of \(\textup{ch}\, M\).

\subsection{Specht modules}
Much of KLR representation theory can be couched in the language of (skew) Specht modules, which have a nicely combinatorial structure and graded character.

\subsubsection{Skew diagrams}

A {\em skew diagram} \(\mu\) is a collection of boxes (with inherent residues) that can be realized as the set difference of Young diagrams \(\mu = \lambda \backslash \nu\) for some charge \(\kappa \in \Z_2\) and some \(\lambda \supseteq \nu \in \Par^\kappa\). We say \(\mu\) is a {\em ribbon} if it is connected and contains no \(2 \times 2\) squares. We extend the notion of content to skew diagrams in the obvious way.
See Figure~\ref{figskewdiags} for an example.

\subsubsection{Standard tableaux and residue sequences}
Given a skew diagram \(\mu\), a {\em \(\mu\)-tableau} is a bijective labeling \({\tt t}: [1, |\mu|] \to \mu\) for the boxes of \(\mu\). We say \({\tt t}\) is {\em standard} if labels increase left to right along rows and top to bottom down columns. We write \(\textup{Std}(\mu)\) for the set of standard \(\mu\)-tableaux. The {\em leading tableau} \({\tt t}^{\mu}\) is the standard tableau which labels boxes row-by-row, left to right, beginning with the top row (see Figure~\ref{figskewdiags}). If \(\mu\) has content \(\theta \in \Z_{\geq 0}I\), then the {\em residue sequence} of a \(\mu\)-tableau \({\tt t}\) is 
\(
\bi^{\tt t} = \textup{res}({\tt t}(1)) \cdots \textup{res}({\tt t}(|\mu|)) \in I^\theta
\). 
We write \(\bi^{\mu} := \bi^{{\tt t}^\mu}\) for the residue sequence of the leading tableau. The symmetric group \(\mathfrak{S}_{|\mu|}\) has a left action on \(\mu\)-tableaux given by \((\omega \cdot {\tt t})(x) = {\tt t}(\omega^{-1} \cdot x)\); we write \(\omega_{\tt t}\) for the permutation such that \(\omega_{\tt t} \cdot {\tt t}^\mu = {\tt t}\), and set \(\psi_{\tt t}:= \psi_{\omega_{{\tt t}}} \).

\subsubsection{Garnir elements}
For \(d \in \Z_{\geq 0}\) and \(r \in [1,d-1]\), we define
\begin{align*}
\tau_r := 1+ \psi_{2r}\psi_{2r-1}\psi_{2r+1}\psi_{2r} \in R_{d\delta}.
\end{align*}
Recalling that for each \(\omega \in \mathfrak{S}_d\) we have a chosen reduced expression \(\omega = s_{r_1} \cdots s_{r_m}\), we associate \(\omega\) with the element \(\tau_\omega := \tau_{r_1} \cdots \tau_{r_m} \in R_{d\delta}\). For \(r+s = d\), we write \(\mathscr{D}^{r,s}_d\) for the set of  minimal coset representatives in \(\mathfrak{S}_d /(\mathfrak{S}_r \times \mathfrak{S}_s)\), and set
\begin{align*}
g_{r,s} := \sum_{\omega \in \mathscr{D}_d^{r,s}} \tau_\omega  \in R_{d \delta}.
\end{align*}

Let \(\mu\) be a skew diagram of content \(\theta\). Assume \(u\) is a box in \(\mu\) with residue \(i\) and there is a box \(v\) directly below \(u\). We call such \(u\) a {\em Garnir box for \(\mu\)}. The {\em Garnir belt of \(u\)} is the set of boxes to the right of \(u\) and to the left of \(v\) (inclusive). Assume \([x,y]\) is the interval of labels in the Garnir belt in the leading tableau \({\tt t}^\mu\).
The {\em Garnir tableau} \({\tt g}^\mu_u\) is the tableau defined as follows. Outside of the Garnir belt, all labels match those in \({\tt t}^\mu\). If the lower row of the Garnir belt is an odd number of boxes, label the first box in this row with \(x\). If the upper row of the Garnir belt is an odd number of boxes, label the last box in this row with \(y\). The remaining interval of labels \([x', y']\) then are used to fill in the \(2r\) boxes in the top row of the belt first, followed by the \(2s\) boxes in the bottom row to complete \({\tt g}^\mu_{u}\). Then the {\em Garnir element for \(u\)} is  
\begin{align*}
\textup{Gar}^\mu_u 1_{\bi^\mu}= (1_{\bj} \otimes g_{r,s}1_{(i \hat i)^{r+s}} \otimes 1_{\bk}) \psi_{{\tt g}^\mu_u} 1_{\bi^\mu} \in R_{\theta}.
\end{align*}
where \(\bj, \bk\) are the residue sequences associated with the labels \([1,x'-1]\) and \([y'+1, |\mu|]\) in \({\tt t}^\mu\), respectively. See Figure~\ref{figskewdiags}.

\begin{figure}[h]
\begin{align*}
\hackcenter{}
\hackcenter{
\begin{overpic}[height=17.5mm]{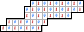}
\end{overpic}
}
\;\;
\hackcenter{}
\hackcenter{
\begin{overpic}[height=17.5mm]{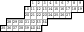}
\end{overpic}
}
\;\;
\hackcenter{}
\hackcenter{
\begin{overpic}[height=17.5mm]{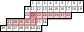}
\end{overpic}
}
\end{align*}
\caption{From left to right, a skew diagram \(\mu\) of content \(20 \alpha_0 + 21 \alpha_1\); the leading tableau \({\tt t}^\mu\), and the Garnir tableau \({\tt g}^\mu_u\) for a box \(u\) with Garnir belt highlighted in red.}
\label{figskewdiags}
\end{figure}

\subsubsection{Specht modules}
We now define Specht modules for the KLR algebra associated to skew diagrams, following \cite{KMR, MuthSkew}.

\begin{definition}
Let \(\mu\) be a skew diagram of content \(\theta \in \Z_{\geq 0}I\). The {\em (skew) Specht module} \(\bS(\mu) \in R_\theta\)-mod is the left ideal \(R_\theta 1_{\bi^\mu}\) modulo the left submodule generated by the elements
\begin{enumerate}
\item \(y_r 1_{\bi^\mu}\), for every \(r \in [1, |\mu|]\);
\item \(\psi_r 1_{\bi^\mu}\), for every pair of labels \(r,r+1\) in the same row of \({\tt t}^\mu\);
\item \(\textup{Gar}^\mu_u 1_{\bi^\mu}\), for every Garnir box \(u\) in \(\mu\).
\end{enumerate}
\end{definition}
We will henceforth drop the `skew' and refer to \(\bS(\mu)\) merely as a Specht module, heedless of whether \(\mu\) is a skew diagram or a proper Young diagram. Specht modules have a combinatorial basis and graded character indexed by the standard tableaux:

\begin{theorem}\cite[Theorem 5.1]{MuthSkew}
\label{spechtbasis}
Let \(\mu\) be a skew diagram.
The Specht module \(\bS(\mu)\) has homogeneous \(\k\)-basis given by 
\begin{align*}
\{1_{\bi^{\tt t}} \psi_{{\tt t}} 1_{\bi^\mu} \mid {\tt t} \in \textup{Std}(\mu)\}.
\end{align*}
and graded character
\(
\textup{ch}\, \bS(\mu) = \sum_{{\tt t} \in \textup{Std}(\mu)} q^{\deg(\psi_{\tt t} 1_{\bi^\mu})}\bi^{\tt t}
\).
\end{theorem}

\subsection{Cuspidal systems and classification of simple \(R_\theta\)-modules}\label{stratasec}
Throughout this section we fix \(i \in \Z_2\), and recall that this corresponds to a choice of convex order \(\succitext\) from \S\ref{cvxordersec}. Following \cite{KleshCusp, McNaffine, KMstrat, TW, MNSS}, to each \(\beta \in \Phi_+^\re\) we associate a self-dual simple finite dimensional {\em \(\succitext\)-cuspidal} module \(L_{i,\beta} \in R_\beta\textup{-mod}\), and to each integer partition \(\lambda\) we associate a self-dual simple finite dimensional {\em \(\succitext\)-semicuspidal} module \(L_{i,\lambda} \in R_{|\lambda|\delta}\textup{-mod}\).
These cuspidal/semicuspidal modules are key building blocks in the representation theory of KLR algebras. We refer the reader to the references above for the precise definition of (semi)cuspidal modules and their properties; for our purposes it will be enough to directly construct these modules in the next subsection.

\subsubsection{Cuspidal/semicuspidal modules via Specht modules} Following \cite{ADMPSS, MNSS} we may explicitly construct the cuspidal/semicuspidal modules via Specht modules associated with certain skew diagrams as follows. Recall that we have fixed \(i \in \Z_2\) which determines a choice of convex order \(\succitext\). For \(j \in \Z_2\), \(k \in \Z_{>0}\), let \(\xi_i(\alpha_{j:k})\) be the unique ribbon skew diagram of content \(\alpha_{j:k}\) such that all boxes of content \(i\) are (southeast)-removable. For a partition \(\lambda\), let \(\xi_i(\lambda)\) be the skew diagram formed by doubling the length of every column of \(\lambda\), shifting the \(t\)th column up \(t-1\) units, and choosing residues so that (southeast)-removable boxes have content \(i\). In Figure~\ref{skewcuspex} we show some example diagrams in the case of the convex order \(\succatext\).

\begin{figure}[h]
\hrulefill
\;\;
{\em Skew diagrams for real cuspidal modules}
\;
\hrulefill
\vspace{2mm}
\begin{align*}
\hackcenter{}
\xi_1(\alpha_{1:1}) = 
\hackcenter{
\begin{overpic}[height=3.5mm]{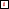}
\end{overpic}
}
\qquad
\xi_1(\alpha_{1:2})= 
\hackcenter{
\begin{overpic}[height=7mm]{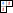}
\end{overpic}
}
\qquad
\xi_1(\alpha_{1:3})= 
\hackcenter{
\begin{overpic}[height=10.5mm]{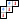}
\end{overpic}
}
\qquad
\xi_1(\alpha_{1:4}) = 
\hackcenter{
\begin{overpic}[height=14mm]{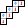}
\end{overpic}
}
\qquad
\cdots
\qquad
\qquad
\end{align*}
\begin{align*}
\hackcenter{}
\qquad
\qquad
\cdots
\qquad
\xi_1(\alpha_{0:4})= 
\hackcenter{
\begin{overpic}[height=14mm]{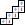}
\end{overpic}
}
\qquad
\xi_1(\alpha_{0:3})= 
\hackcenter{
\begin{overpic}[height=10.5mm]{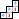}
\end{overpic}
}
\qquad
\xi_1(\alpha_{0:2})= 
\hackcenter{
\begin{overpic}[height=7mm]{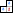}
\end{overpic}
}
\qquad
\xi_1(\alpha_{0:1})= 
\hackcenter{
\begin{overpic}[height=3.5mm]{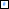}
\end{overpic}
}
\end{align*}

\vspace{3mm}
\hrulefill
\;\;
{\em Skew diagrams for imaginary semicuspidal modules}
\;
\hrulefill
\vspace{1mm}
\begin{align*}
\hackcenter{}
\xi_1
\hspace{-1mm}
\left(
\,
\hackcenter{
\begin{overpic}[height=2.4mm]{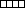}
\end{overpic}
}
\,
\right)
=
\hackcenter{
\begin{overpic}[height=17.5mm]{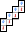}
\end{overpic}
}
\qquad
\xi_1
\hspace{-1mm}
\left(
\,
\hackcenter{
\begin{overpic}[height=4mm]{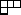}
\end{overpic}
}
\,
\right)
=
\hackcenter{
\begin{overpic}[height=21mm]{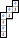}
\end{overpic}
}
\qquad
\xi_1
\hspace{-1mm}
\left(
\,
\hackcenter{
\begin{overpic}[height=4mm]{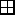}
\end{overpic}
}
\,
\right)
=
\hackcenter{
\begin{overpic}[height=17.5mm]{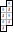}
\end{overpic}
}
\qquad
\xi_1
\hspace{-1mm}
\left(
\,
\hackcenter{
\begin{overpic}[height=6mm]{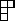}
\end{overpic}
}
\,
\right)
=
\hackcenter{
\begin{overpic}[height=24.5mm]{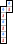}
\end{overpic}
}
\qquad
\xi_1
\hspace{-1mm}
\left(
\,
\hackcenter{
\begin{overpic}[height=7.5mm]{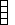}
\end{overpic}
}
\hspace{-1.5mm}
\right)
=
\hackcenter{
\begin{overpic}[height=28mm]{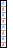}
\end{overpic}
}
\end{align*}
\caption{Examples of cuspidal/semicuspidal skew diagrams for the convex order \hspace{-0.5mm}$\texorpdfstring{\begin{smallmatrix}{}\\ \succ \\ {}^1\end{smallmatrix}}{a,b}$.}
\label{skewcuspex}
\end{figure}

\begin{theorem}\label{cuspsasspechts}
For \(i, j \in \Z_2\), \(k \in \Z_{>0}\), and \(\lambda \in \Par\) we have that
\begin{align*}
L_{i, \alpha_{j:k}} \cong q^{k-1} \bS(\xi_i(\alpha_{j:k}))
\qquad
\textup{and}
\qquad
L_{i, \lambda} \cong  \textup{hd} \,q^{c(\lambda)} \bS(\xi_i(\lambda)),
\end{align*}
where
\begin{align*}
c(\lambda) =2 |\lambda| - 2 p(\lambda) - 2
 \sum_{m=1}^{p(\lambda)} \min\left\{m-1, \left\lfloor \frac{\lambda_m - 1}{2} \right\rfloor \right\}. 
\end{align*}
Moreover, \([\bS(\xi_i(\lambda)) : L_{i,\lambda}] = 1\) and if \([\bS(\xi_i(\lambda)): M] \neq 0\) for some simple module \(M\), then \(M \approx L_{i, \mu}\) for some \(\mu \trianglelefteq \lambda\).
\end{theorem}
\begin{proof}
Up to grading shift, these facts are recorded in \cite[Theorem~6.13]{ADMPSS} and \cite[Theorem~9.15]{MNSS}. Now write \(\bi(\alpha_{j:k})=\hat i^{k-1} i^k \in I^{\alpha_{j:k}} \) and \(\bi(\lambda) =\hat i^{\lambda_1} i^{\lambda_1} \cdots \hat i^{\lambda_{p(\lambda)}} i^{\lambda_{p(\lambda)}} \in I^{|\lambda|\delta}\). Self-duality implies that the graded dimensions of \(1_{\bi(\alpha_{j:k})}L_{i, \alpha_{j:k}}\) and \(1_{\bi(\lambda)}L_{i,\lambda}\) are invariant under the bar involution \(q \leftrightarrow q^{-1}\).
Moreover, by \cite[Theorem~6.13]{ADMPSS} and \cite[Lemma~5.13, Theorem~9.15]{MNSS} the graded dimensions of \(1_{\bi(\alpha_{j:k})} \bS(\xi_i(\alpha_{j:k}))\) and \(1_{\bi(\lambda)} \bS(\xi_i(\lambda))\) agree with those of \(1_{\bi(\alpha_{j:k})}L_{i, \alpha_{j:k}}\) and \(1_{\bi(\lambda)}L_{i,\lambda}\) respectively, up to some explicit shift. By directly calculating the graded dimensions of \(1_{\bi(\alpha_{j:k})} \bS(\xi_i(\alpha_{j:k}))\) and \(1_{\bi(\lambda)} \bS(\xi_i(\lambda))\) using Theorem~\ref{spechtbasis}, and finding the requisite shifts to induce bar invariance, we arrive at the shifts in the theorem statement.
\end{proof}

\subsection{Classification of simple \(R_\theta\)-modules via cuspidal systems}\label{cuspclasssec}
There is a bilexicographic partial order \(\geq \) on \(\Pi(\theta)\) defined as follows (see \cite[Definition~9.17]{MNSS}). Set \(\pi > \phi\) provided that:
\begin{enumerate}
\item For \(j \in \Z_2\), there exists \(k \in \Z_{>0}\) where \(n_k(\pi^j) > n_k(\phi^j)\) and \(n_{k'}(\pi^j) = n_{k'}(\phi^j)\) for all \(k' < k\), or;
\item For \(j \in \Z_2\), we have \(\pi^j = \phi^j\) and \(\pi^\delta \triangleright \phi^\delta\).
\end{enumerate}

Let \(\theta \in \Z_{\geq 0}I\) and \(\pi = \{\pi^1, \pi^\delta, \pi^0\} \in \Pi(\theta)\). We associate to \(\pi\) a {\em proper standard module} \(\bar\Delta_i(\pi)\) which is the \(\succitext\)-ordered induction product of simple cuspidal and semicuspidal modules as defined in \cite[(6.5)]{KleshCusp}:
\begin{align}\label{DeltaDef}
\bar\Delta_i(\pi) := 
q^{\textup{sh}(\pi)}
L_{i,\alpha_{i:1}}^{\circ \,n_1(\pi^i)} \circ L_{i,\alpha_{i:2}}^{\circ \,n_2(\pi^i)}  \circ \cdots
\circ L_{i,\pi^\delta} \circ \cdots \circ
L_{ i,\alpha_{\;\hat i:2}}^{\circ \,n_2(\pi^{\hat i})} \circ L_{i,\alpha_{\;\hat i:1}}^{\circ \,n_1(\pi^{\hat i})},
\end{align}
where the grading shift \(\textup{sh}( \pi)\) is given by
\begin{align}\label{DeltaShift}
\textup{sh}(\pi) = \frac{1}{2} \sum_{\substack{j \in \Z_2 \\ t \in \Z_{>0}   }} n_t( \pi^j)[n_t(\pi^j) - 1].
\end{align}

The key result about cuspidal systems is the following:
\begin{theorem}\label{longthm}
Let \(\theta \in \Z_{\geq 0}I\) and \(i \in \Z_2\).
\begin{enumerate}
\item For \(\pi \in \Pi(\theta)\), the proper standard module \(\bar\Delta_i(\pi)\) has self-dual simple head \(L_i(\pi):= \textup{hd} \,\bar\Delta_i(\pi)\).
\item The set  
\begin{align}\label{rootparclassification}
\{ L_i(\pi) \mid \pi \in \Pi(\theta)\}
\end{align}
is a complete and irredundant list of simple \(R_\theta\)-modules up to isomorphism and grading shift. 
\item For all \(\pi, \phi \in \Pi(\theta)\) we have \([\bar \Delta_i(\pi): L_i(\pi)] = 1\) and \([\bar \Delta_i(\pi): L_i(\phi)] \neq 0\) only if \(\pi \geq \phi\).
\item For all \((\pi | \phi) \in \aMV(\theta)\), we have that 
\(L_1(\pi) \cong L_{0}( \phi)\), and we may alternatively index the simple modules in (\ref{rootparclassification}) as 
\begin{align*}
\{L(\pi | \phi) := L_1(\pi) \cong L_0(\phi) \mid (\pi | \phi) \in \aMV(\theta)\}.
\end{align*}
\item For \((\pi | \phi) \in \aMV\), the geometric polytope for \((\pi | \phi)\) is the convex hull of the words of \(L(\pi | \phi) \) considered as paths in the root lattice.
\end{enumerate}
\end{theorem}
\begin{proof}
(1)--(3) are the content of \cite[Theorem~4.1]{KleshCusp} and \cite[Theorem~2.19]{TW}. By \cite[Lemma~5.13]{MNSS}, our simple semicuspidal modules \(L_{i, \lambda}\) are uniquely identified by having the {\em Gelfand-Graev} word \(\boldsymbol{g}^\lambda = \hat{i}^{\lambda_1} i^{\lambda_1} \cdots \hat{i}^{\lambda_{p(\lambda)}} i^{\lambda_{p(\lambda)}}\) and no word \(\boldsymbol{g}^\mu\) for \(\mu\) more dominant. Then, as noted in \cite[Remark~3.43]{TW}, our choice of labelings for the collection of imaginary semicuspidal simple modules agrees with the choice made in \cite[Definition~3.38]{TW}, so (4) and (5) are the content of \cite[Theorem~B, Corollary 3.11]{TW}. 
\end{proof}

\begin{remark}
The cuspidal system approach to KLR representation theory has proven to be a powerful organizing principle, yielding deep structural results and interactions with other areas of representation theory.
The simple modules \(L_i(\pi)\) and proper standard modules \(\bar \Delta_i(\pi)\) (as well as their infinite dimensional counterparts, the projective covers \(P_i(\pi)\) and standard modules \(\Delta_i(\pi)\)) are the key ingredients in a complete {\em stratification} framework for \(R_\theta\)-mod, replete with associated filtrations, BGG reciprocity, and other homological results, see \cite{KMstrat, McNaffine, Murata}. The {\em imaginary} strata in this setup are linked via Howe duality to the classical Schur algebra, and through Morita equivalence to RoCK blocks of Hecke algebras and to `zigzag deformations' of affine Hecke and Schur algebras, see \cite{AffWeb, KMimag, KMaffzig, FaceFunctors, MNSS}. Through the lens of categorification, the projective and standard modules (and respectively, simple and proper standard modules) encode higher symmetries in the theory of  crystal  and PBW bases (and respectively, their duals) for \(\mathcal{U}^-_q(\hat{\mathfrak{sl}}_2)\), see \cite{KatoKLR, LVcrystals, TW, VV}.
\end{remark}

\subsection{Real root functors}\label{rootfuncsec}
For \(i \in \Z_2\), \(k \in \Z_{>0}\), we denote the functors
\begin{align*}
&E_{i:k} \colon R_{\theta+ \alpha_{i:k}}\textup{-mod} \to R_{\theta}\textup{-mod}, \quad M \mapsto \Hom_{R_{\theta + \alpha_{i:k}}}(R_\theta \circ L_{\hat i, \alpha_{i:k}}, M) \\
&F_{i:k}\colon R_{\theta}\textup{-mod} \to R_{\theta+\alpha_{i:k}}\textup{-mod}, \quad M \mapsto \Ind_{\theta, \alpha_{i:k}}^{\theta+ \alpha_{i:k}}(M \boxtimes L_{\hat i, \alpha_{i:k}})\\
&E_{i:k}^*\colon R_{\theta+ \alpha_{i:k}}\textup{-mod} \to R_{\theta}\textup{-mod}, \quad M \mapsto \Hom_{R_{\theta + \alpha_{i:k}}}(L_{i, \alpha_{i:k}} \circ R_\theta, M)\\
&F_{i:k}^*\colon R_{\theta}\textup{-mod} \to R_{\theta+\alpha_{i:k}}\textup{-mod}, \quad M \mapsto \Ind_{ \alpha_{i:k}}^{\theta+ \alpha_{i:k}}(L_{i, \alpha_{i:k}}\boxtimes M)
\end{align*}
These functors generalize the Chevalley functors of \S\ref{transfuncsec} from simple roots to all real roots; we have \(E_i \cong E_{i:1}, F_i \cong F_{i:1},E_i^* \cong E^*_{i:1}\) and \(F_i^* \cong F^*_{i:1}\) since \(L_{i, \alpha_i} \cong L_{\hat i, \alpha_i} \cong \k_{\alpha_i}\) is the unique one-dimensional self-dual simple module for \(R_{\alpha_i}\). 

\subsection{Augmented branching rules for simple modules labeled by affine MV polytopes}
It will be useful to define the following statistics on polytopes \((\pi | \phi) \in \aMV\):
\begin{align*}
\begin{array}{lllllllllll}
m_1(\pi|\phi) = \min(\phi^1);\quad
&m_1^*(\pi|\phi) = \min(\pi^1);\quad
&m_0(\pi|\phi) = \min(\pi^0);\quad
&m_0^*(\pi|\phi) = \min(\phi^0).\\
n_{1:k}(\pi|\phi) = n_k(\phi^1);
&n^*_{1:k}(\pi|\phi) = n_k(\pi^1);\quad
&n_{0:k}(\pi|\phi) = n_k(\pi^0);\quad
&n_{0:k}^*(\pi|\phi) = n_k(\phi^0).
\end{array}
\end{align*}
Recall the operators \(e_{i:k}, e^*_{i:k}, f_{i:k}, f^*_{i:k}\) on \(\aMV\) defined in \S\ref{aMVcrysdef}.

\begin{proposition}\label{rootfuncbranch}
Let \((\pi|\phi) \in \aMV\), \(i \in \Z_2\), and \(k \in \Z_{>0}\). 
\begin{enumerate}
\item If \(m_{ i}(\pi|\phi)>k\) then \(E_{ i:k} L(\pi |\phi)= 0\), and if \(m_{ i}(\pi|\phi) = k\) then 
\begin{align*}
q^{1-n_{i:k}(\pi | \phi)} \textup{soc}\, E_{ i:k} L(\pi |\phi) \cong L(e_{i:k}(\pi |\phi))
\end{align*}
\item If \(m^*_{ i}(\pi|\phi)>k\) then \(E^*_{ i:k} L(\pi| \phi)= 0\), and if \(m^*_{ i}(\pi|\phi) = k\) then 
\begin{align*}
q^{1-n_{i:k}^*(\pi | \phi)} \textup{soc}\, E^*_{ i:k} L(\pi| \phi) \cong  L(e^*_{i:k}(\pi |\phi)).
\end{align*}
\item If \(m_{  i}(\pi|\phi) \geq k\) then 
\begin{align*}
q^{n_{i:k}(\pi| \phi)} \textup{hd}\, F_{ i:k} L(\pi | \phi) \cong L(f_{i:k} (\pi| \phi)).
\end{align*}
\item If \(m^*_{  i}(\pi | \phi) \geq k\) then 
\begin{align*}
q^{n_{i:k}^*(\pi| \phi)} \textup{hd}\, F^*_{ i:k}L(\pi | \phi) \cong L(f^*_{i:k} (\pi| \phi)).
\end{align*}
\end{enumerate}
\end{proposition}
\begin{proof}
For sake of space throughout the proof, write \(\alpha:= \alpha_{i:k}\).
We prove (4) first in the \(i=1\) case, noting that the \(i=0\) case, as well as (3), is proved in symmetrical fashion. We have then that \(L(\pi | \phi) = L_i(\pi)\).
Since \(m_i^*(\pi | \phi) = \min(\pi^i) \geq k\), 
, it follows from (\ref{DeltaDef}, \ref{DeltaShift}) that
\(
\bar{\Delta}_i(\pi_{+i:k}) \cong 
q^{n_k(\pi^i)}   L_{ i, \alpha} \circ \bar{\Delta}_i(\pi).
\)
By Theorem~\ref{longthm}(1) we have a surjection \(L_{ i, \alpha} \boxtimes \bar{\Delta}_i(\pi)  \twoheadrightarrow  L_{ i, \alpha} \boxtimes L_i(\pi) \), and so by exactness of induction we have a  surjection
\begin{align}\label{littleindarg}
\bar{\Delta}_i(\pi_{+i:k}) \cong 
q^{n_k(\pi^i)}   L_{ i, \alpha} \circ  \bar{\Delta}_i(\pi) \twoheadrightarrow q^{n_k(\pi^i)}  L_{ i, \alpha} \circ L_i(\pi) = q^{n_k(\pi^i)} F^*_{i:k} L_i(\pi). 
\end{align}
But \(\bar{\Delta}_i(\pi_{+i:k})\) has simple head \(L_i(\pi_{+i:k})\) by Theorem~\ref{longthm}(1), which yields result (4).

Now we prove (2) in the \(i=1\) case as well, noting that the \(i=0\) case, as well as (1), is proved in symmetrical fashion. We again have \(L(\pi| \phi) = L_i(\pi)\). 
By \cite[Theorem~4.1]{KleshCusp}(v), every simple factor of \(\Res_{\alpha, \theta}^{\alpha + \theta} L_i(\pi)\) is of the form \(L_{i, \alpha} \boxtimes L_i(\sigma)\), where \(\sigma \leq \pi_{-i:k}\). Then for such \(\sigma\) we have:
\begin{align*}
[\textup{soc}\, \Res_{\alpha, \theta}^{\alpha + \theta}L_i(\pi) : L_{i, \alpha} \boxtimes L_i(\sigma)] &=
\dim_q \Hom_{R_\alpha \otimes R_\theta}(L_{i, \alpha} \boxtimes L_i(\sigma), \Res_{\alpha, \theta}^{\alpha + \theta} L_i(\pi))\\
&=
\dim_q \Hom_{R_{\alpha + \theta}}(L_{i, \alpha} \circ L_i(\sigma), L_i(\pi))\\
&= 
\dim_q \Hom_{R_{\alpha + \theta}}(F^*_{i:k} L_i(\sigma), L_i(\pi)) \\
&= 
\dim_q \Hom_{R_{\alpha + \theta}}(q^{-n_k(\sigma^{i})}L_i(\sigma_{+i:k}), L_i(\pi)) 
= \delta_{\sigma, \pi_{-i:k}}q^{n_k(\pi_{-i:k}^i)},
\end{align*}
where the second to last equality follows since \(L_i(\pi)\) is simple and \(F^*_{i:k} L_i(\sigma)\) has simple head \(q^{-n_k(\sigma^i)}L_i(\sigma_{+i:k})\) by part (4), and the last equality follows by Theorem~\ref{longthm}(2). Since \(n_k(\pi_{-i:k}^i) = n_k(\pi^i) - 1\), we have \(\textup{soc}\, \Res_{\alpha, \theta}^{\alpha + \theta} L_i(\pi) \cong q^{n_k(\pi^i)-1}L_{i, \alpha} \boxtimes L_i(\pi_{-i:k})\) if \(n_k(\pi^i) >0\), and zero otherwise. Note that, as \(R_\theta\)-modules, 
\begin{align*}
\Hom_{R_\alpha}(L_{i,\alpha}, \textup{soc} \, \Res_{\alpha, \theta}^{\alpha + \theta} L_i(\pi)) 
\cong \Hom_{R_\alpha}(L_{i,\alpha}, q^{n_k(\pi^i)-1} L_{i, \alpha} \boxtimes L_i(\pi_{-i:k}))
\cong q^{n_k(\pi^i)-1}L_i(\pi_{-i:k}).
\end{align*}

Now let the homomorphism \(\varphi \in \Hom_{R_\alpha}(L_{i,\alpha}, \Res_{\alpha, \theta}^{\alpha + \theta} L_i(\pi))\) be any generator of a simple \(R_\theta\)-submodule in \(\Hom_{R_\alpha}(L_{i,\alpha}, \Res_{\alpha, \theta}^{\alpha + \theta} L_i(\pi))\). Then there is a nonzero \(R_\alpha \otimes R_\theta\)-module evaluation map \(L_{i, \alpha} \boxtimes (R_\theta \cdot \varphi) \to \Res_{\alpha, \theta}^{\alpha + \theta} L_i(\pi)\), 
and since \(L_{i, \alpha} \boxtimes (R_\theta \cdot \varphi)\) is a simple \(R_\alpha \otimes R_\beta\)-module, this map must be injective with image  in \(\textup{soc}\, \Res_{\alpha, \theta}^{\alpha + \theta} L_i(\pi)\). Therefore \(\varphi(L_{i, \alpha}) \subseteq \textup{soc}\, \Res_{\alpha, \theta}^{\alpha + \theta} L_i(\pi)\), and so \(\varphi \in \Hom_{R_\alpha}(L_{i, \alpha}, \textup{soc}\, \Res_{\alpha, \theta}^{\alpha + \theta} L_i(\pi))\). Thus
\begin{align*}
\textup{soc}\, \Hom(L_{i, \alpha}, \Res_{\alpha, \theta}^{\alpha + \theta} L_i(\pi)) 
\subseteq
\Hom(L_{i, \alpha},  \textup{soc}\, \Res_{\alpha, \theta}^{\alpha + \theta} L_i(\pi)).
\end{align*}
Therefore, as \(R_\theta\)-modules, we have
\begin{align*}
\textup{soc}\, E^*_{i:k} L_i(\pi) 
&= \textup{soc} \,\Hom_{R_{\alpha+\theta}}(L_{i,\alpha} \circ R_\theta, L_i(\pi))
\cong \textup{soc}\, \Hom_{R_{\alpha} \otimes R_\theta}( L_{i, \alpha} \boxtimes R_\theta, \Res_{\alpha, \theta}^{\alpha + \theta} L_i(\pi))\\
&\cong \textup{soc}\, \Hom_{R_\alpha}(L_{i,\alpha}, \Res_{\alpha, \theta}^{\alpha + \theta} L_i(\pi))
\subseteq
\Hom(L_{i, \alpha},  \textup{soc} \,\Res_{\alpha, \theta}^{\alpha + \theta} L_i(\pi))
\cong q^{n_k(\pi^i)-1} L_i(\pi_{-i:k}).
\end{align*}
Since \(\Hom(L_{i, \alpha},  \textup{soc} \,\Res_{\alpha, \theta}^{\alpha + \theta} L_i(\pi)) \neq 0\), it follows that \(\Hom(L_{i, \alpha},   \Res_{\alpha, \theta}^{\alpha + \theta} L_i(\pi)) \neq 0\), and thus we have that \(\textup{soc}\, \Hom_{R_\alpha}(L_{i,\alpha}, \Res_{\alpha, \theta}^{\alpha + \theta} L_i(\pi)) \neq 0\). Therefore \(\textup{soc}\, E^*_{i:k} L_i(\pi)\) is a nonzero \(R_\theta\)-submodule of the simple module \(q^{n_k(\pi^i)-1} L_i(\pi_{-i:k})\), and result (2) follows.
\end{proof}

This yields the immediate corollary (which appears in ungraded form in \cite[Proposition~3.13]{TW}):

\begin{corollary}\label{branchingpoly}
 Let \((\pi | \phi) \in \aMV(\theta)\). Recalling the \({\tt A}_1^{(1)}\)-bicrystal structure of \(\aMV\) from \S\ref{aMVcrysdef}, we have 
\begin{enumerate}
\item If \(\varepsilon_i(\pi | \phi) = 0\) then \(E_i L(\pi | \phi) = 0\), and if \(\varepsilon^*_i(\pi | \phi) = 0\) then \(E_i^* L(\pi | \phi) = 0\). Otherwise
\begin{align*}
q^{1-\varepsilon_i(\pi | \phi)}  \textup{soc}\, E_i L(\pi | \phi) \cong L({e}_i (\pi | \phi) )
\quad
\textup{and}
\quad
q^{1-\varepsilon^*_i(\pi | \phi)}\textup{soc} \,E_i^* L(\pi | \phi) \cong L({e}^*_i (\pi | \phi) ).
\end{align*}
\item We have
\begin{align*}
q^{\varepsilon_i(\pi | \phi)} \textup{hd} \, F_i L(\pi | \phi) \cong L({f}_i (\pi | \phi) )
\quad
\textup{and}
\quad
q^{\varepsilon^*_i(\pi | \phi)} \textup{hd} \,F_i^* L(\pi | \phi) \cong L({f}^*_i (\pi | \phi) ).
\end{align*}
\end{enumerate}
\end{corollary}

\subsection{Cyclotomic KLR algebras}
Let \(\theta \in \Z_{\geq 0}I\) and \(\Lambda = t_0 \Lambda_0 + t_1 \Lambda_1\) be a dominant integral weight. Then we define the {\em cyclotomic KLR algebra} \(R_\theta^\Lambda\) to be the quotient of \(R_\theta\) by the two-sided ideal generated by \(\{y_1^{t_{i_1}}1_{ \bi} \mid \bi \in I^\theta\}\). We write \(\textup{infl}^\Lambda: R_\theta^\Lambda\textup{-mod} \to R_\theta\textup{-mod}\) and \(\textup{pr}^\Lambda: R_{\theta}\textup{-mod} \to R_{\theta}^\Lambda\textup{-mod}\) for the usual inflation and truncation functors. The algebra \(R_\theta^\Lambda\) is always finite dimensional.

\begin{remark}
Cyclotomic KLR algebras occupy a central position in modern representation theory due to their myriad connections to categorification and other algebraic structures. To name a few, the celebrated result \cite{BKisom} of Brundan-Kleshchev establishes that cyclotomic KLR algebras are isomorphic to blocks of cyclotomic Hecke algebras, allowing the latter---and modular representations of symmetric groups in particular (see \cite{KleshSurvey})---to be studied through the (graded) lens of cyclotomic KLR algebras. Furthermore, Kang-Kashiwara proved \cite{KangKash} that the cyclotomic KLR algebras \(R^\Lambda_\theta\) categorify the highest weight representations \(V(\Lambda)\) for the quantum group \(U_q(\widehat{\mathfrak{sl}}_n)\), with the simple \(R^\Lambda_\theta\)-modules categorifying the crystal \(\mathcal{B}(\Lambda)\), as shown by Varagnolo-Vasserot \cite{VV}.
\end{remark}

\subsubsection{Specht modules for multipartitions}
Let \(\bkap = (\kappa_1, \kappa_2, \ldots)\) be a multicharge and \(\blam = (\lambda^{(1)}, \lambda^{(2)}, \ldots ) \in \MP^{\bkap}(\theta)\) be a multipartition of level \(\ell\). Then we may consider each constituent partition \(\lambda^{(t)}\) as a (skew) diagram with residues as determined by \(\kappa_t\). Following \cite{KMR} (and expressly calculating the grading shift in \cite[(2.3)]{KMR}) we set
\begin{align*}
S^{\bkap}(\blam) := q^{\textup{sh}^{\bkap}(\blam)}\bS(\lambda^{(\ell)}) \circ  \bS(\lambda^{(\ell-1)}) \circ \cdots \circ  \bS(\lambda^{(1)}) \in R_\theta\textup{-mod},
\end{align*}
where the grading shift is given by
\begin{align*}
\textup{sh}^{\bkap}(\blam) = \sum_{\substack{1 \leq t \leq \ell\\ 1 \leq r \leq p(\lambda^{(t)})}} \left \lfloor   \frac{\lambda_r^{(t)}}{2} \right \rfloor + 
\sum_{\substack{1 \leq s< t \leq \ell\\ 1 \leq r \leq p(\lambda^{(t)})}} 
\left \lfloor
\frac{\lambda_r^{(t)} + ((r + \delta_{\kappa_s, \hat \kappa_t})\hspace{-2mm} \mod 2)}{2}
\right\rfloor.
\end{align*}
If \(\blam \in \MP^{(\bkap, \ell)}(\theta)\), then \(S^{\bkap}(\blam)\) naturally descends to an \(R^{\Lambda(\bkap, \ell)}_\theta\)-module.

\begin{theorem}\cite[Theorem~5.10]{BKdecomp} \cite[Main Theorem]{HuMathasCycA} \label{MPlabelthmcyc}
Let \(\theta \in \Z_{\geq 0}I\). Let \(\Lambda\) be a dominant weight, and assume that the multicharge
 \(\bkap\) and \(\ell \in \Z_{>0}\) are such that \(\Lambda = \Lambda(\bkap, \ell)\). Then the algebra \(R_\theta^{\Lambda}\) is graded cellular, with poset \(\MP^{(\bkap, \ell)}(\theta)\) under the dominance order \(\trianglerighteq\) and cell modules \(\{S^{\bkap}(\blam) \mid \blam \in \MP^{(\bkap, \ell)}(\theta)\}\). For all {\em Kleshchev} multipartitions \(\blam \in \MPK^{(\bkap,\ell)}(\theta)\), the module \(S^{\bkap}(\blam)\) has self-dual simple head \(D^{\bkap}(\blam):= \textup{hd} \,S^{\bkap}(\blam)\), and 
\begin{align*}
\{ D^{\bkap}(\blam) \mid \blam \in \MPK^{(\bkap, \ell)}(\theta)\}
\end{align*}
is a complete and irredundant list of simple \(R_\theta^{\Lambda}\)-modules up to isomorphism and grading shift.
\end{theorem}

\subsubsection{Cyclotomic branching rules for simple modules labeled by Kleshchev multipartitions}
Following \cite[\S4.4]{BKdecomp} we define (biadjoint, exact) cyclotomic \(i\)-restriction and \(i\)-induction functors
\begin{align*}
&E_i^{\Lambda}\colon R^{\Lambda}_{\theta+\alpha_i}\textup{-mod} \to R^{\Lambda}_{\theta}\textup{-mod},
\hspace{5mm}M \mapsto 1_{\theta, \alpha_i} R^{\Lambda}_{\theta+\alpha_i} \otimes_{R^{\Lambda}_{\theta+\alpha_i} } M
\\
&F_i^{\Lambda}\colon R^{\Lambda}_{\theta}\textup{-mod} \to R^{\Lambda}_{\theta+\alpha_i}\textup{-mod},
\hspace{5mm}M \mapsto R_{\theta+\alpha_i}^{\Lambda} 1_{\theta, \alpha_i}\otimes_{R_{\theta}^{\Lambda}} M
\end{align*}

\begin{theorem}  \cite[Theorem~4.12]{BKdecomp}\label{cycbranch}
Let \(\theta = a_0 \alpha_0 + a_1 \alpha_1\in \Z_{\geq 0}I\). Let \(\Lambda=t_0 \Lambda_0 + t_1 \Lambda_1\) be a dominant integral weight, and assume that the multicharge
 \(\bkap\) and \(\ell \in \Z_{>0}\) are such that \(\Lambda = \Lambda(\bkap, \ell)\). Let \(i \in \Z_2\) and \(\blam \in \MPK^{(\bkap, \ell)}(\theta)\). Then recalling the \({\tt A}_1^{(1)}\)-crystal structure of \(\MPK^{(\bkap, \ell)}\) from \S\ref{cycKleshcrystal}, we have
\begin{enumerate}
\item If \(\varepsilon_i^{(\bkap, \ell)} (\blam) = 0\), then \(E_i^{\Lambda} D^{\bkap}(\blam) = 0\). Otherwise
\begin{align*}
q^{\varepsilon_i (\blam)-1}
 \textup{hd}\left( E_i^{\Lambda}  D^{\bkap}(\blam)\right)
\cong D^{\bkap}(e_i^{(\bkap, \ell)}  \blam)
\cong 
q^{1-\varepsilon_i (\blam)}
 \textup{soc}\left( E_i^{\Lambda} D^{\bkap}(\blam)\right)
\end{align*}
\item If \(\varphi^{(\bkap, \ell)}_i (\blam) = 0\), then \(F_i^{\Lambda} D^{\bkap}(\blam) = 0\). Otherwise
\begin{align*}
q^{2a_i - 2a_{\hat i} - t_i + \varphi_i^{(\bkap, \ell)}(\blam)} 
 \textup{hd}\left(F_i^{\Lambda}D^{\bkap}(\blam) \right)
\cong D^{\bkap}(f_i^{(\bkap, \ell)}  \blam)
\cong 
q^{2a_i - 2a_{\hat i} - t_i - \varphi_i^{(\bkap, \ell)}(\blam)+2}
\textup{soc}\left( F_i^{\Lambda}D^{\bkap}(\blam) \right).
\end{align*}
\end{enumerate}
\end{theorem}

\subsection{Classification of simple \(R_\theta\)-modules via Kleshchev multipartitions}\label{kleshclasssec}
In a limiting procedure, we may lift the classification of simple modules for cyclotomic KLR algebras \(R^\Lambda_\theta\) in Theorem~\ref{MPlabelthmcyc} to the full affine KLR algebra \(R_\theta\).

\begin{theorem}\label{MPlabelthm}
Let \(\theta \in \Z_{\geq 0}I\) and let \(\bkap\) be a multicharge. The set
\begin{align}\label{affinemultclass}
\{ D^{\bkap}(\blam) = \textup{hd}\, S^{\bkap}(\blam) \mid \blam \in \MPK^{\bkap}(\theta)\}
\end{align}
is a complete and irredundant set of simple \(R_\theta\)-modules up to isomorphism and grading shift. Moreover, \([S^{\bkap}(\blam): D^{\bkap}(\blam)] = 1\) and \([S^{\bkap}(\bmu): D^{\bkap}(\blam)] \neq 0\) for some \(\bmu \in \MP^{\bkap}(\theta)\) only if \(\bmu \trianglerighteq \blam\).
\end{theorem}
\begin{proof}
This is essentially \cite[Theorem~6.34]{GeLi}; we provide the proof here since our setup is somewhat different. Simplicity of the modules in (\ref{affinemultclass}) follows from Theorem~\ref{MPlabelthmcyc}. Every self-dual simple module \(L \in R_{\theta}\textup{-mod}\) is finite dimensional and thus \(y_1^{m}L = 0\) for some \(m \gg 0\) by degrees. Then, letting \(M\) be large enough that the list \(\kappa_1, \ldots, \kappa_M\) contains at least \(m\) 0's and 1's, it follows that \(L\) descends to a simple \(R^{\Lambda(\bkap, M)}\)-module, and hence \(L \approx D^{\bkap}(\blam)\) for some \(\blam \in \MPK^{(\bkap, M)}(\theta) \subset \MPK^{\bkap}(\theta)\) by Theorem~\ref{MPlabelthmcyc}. Thus the list (\ref{affinemultclass}) is complete. For irredundancy, we note that if \(D^{\bkap}(\blam) \approx D^{\bkap}(\bmu)\) for some \(\blam, \bmu \in \MPK^{\bkap}(\theta)\), then again choosing \(M \gg 0 \) such that \(\lambda^{(m)} = \mu^{(m)} = \varnothing\) for \(m > M\), we would have that \(\blam, \bmu \in \MPK^{(\bkap, M)}(\theta)\) and \(D^{\bkap}(\blam) \approx D^{\bkap}(\bmu) \in R_\theta^{\Lambda(\bkap, M)}\textup{-mod}\) and thus \(\blam = \bmu\) by Theorem~\ref{MPlabelthmcyc}. The final statement in the theorem follows likewise from consideration of these modules in the cyclotomic setting, see \cite[Theorem~4.12]{BKdecomp}.
\end{proof}

\subsubsection{(Affine) branching rules for simple modules labeled by Kleshchev multipartitions}

\begin{proposition}\label{EFaffmulti}
Let \(\theta \in \Z_{\geq 0}I\) and let \(\bkap\) be a multicharge. Let \(i \in \Z_2\) and \(\blam \in \MPK^{\bkap}(\theta)\). 
Recalling the Chevalley functors from \S\ref{transfuncsec} and the \({\tt A}_1^{(1)}\)-crystal structure of \(\MPK^{\bkap}\) from \S\ref{Kleshcrys}, we have:
\begin{enumerate}
\item If \(\varepsilon_i(\blam) = 0\) then \(E_i D^{\bkap}(\blam) = 0\). Otherwise
\begin{align*}
q^{\varepsilon_i(\blam)-1}  \textup{hd}\, E_i D^{\bkap}(\blam) \cong D^{\bkap}({e}_i \blam) \cong 
q^{1-\varepsilon_i(\blam)}  \textup{soc}\, E_i D^{\bkap}(\blam).
\end{align*}
\item We have
\begin{align*}
q^{\varepsilon_i(\blam)} \textup{hd} \, F_i D^{\bkap}(\blam) \cong D^{\bkap}({f}_i \blam ).
\end{align*}
\end{enumerate}
\end{proposition}
\begin{proof}
It follows from definitions that \(\textup{infl}^{\Lambda(\bkap, \ell)} \circ E_i^{\Lambda(\bkap, \ell)} = E_i \circ \textup{infl}^{\Lambda(\bkap, \ell)}\), and thus the right side of Theorem~\ref{cycbranch}(1) implies (1) in the proposition. By Theorem~\ref{MPlabelthm} and Theorem~\ref{longthm}(4) there is some bijection \(B:\MPK^{\bkap} \to  \aMV \) such that \(D^{\bkap}(\blam) \cong L(B(\blam))\) for all \(\blam \in \MPK^{\bkap}\). Now, note that by Proposition~\ref{branchingpoly} we have
\begin{align*}
L(B(\blam))
\approx
D^{\bkap}(\blam) 
=
D^{\bkap}({e}_i {f}_i \blam)
\approx
\textup{hd}\, E_i D^{\bkap} ({f}_i \blam)
\approx
\textup{hd} \, E_i L(B({f}_i \blam)) 
\approx
 L({e}_i B({f}_i\blam)),
\end{align*}
so by Theorem~\ref{longthm}(4) we have that \({e}_i B({f}_i\blam) = B(\blam)\), and thus \(B({f}_i \blam) = {f}_i{e}_i B({f}_i\blam) = {f}_iB(\blam)\), and so \(B\) is a crystal isomorphism. Thus the rest of the claims in the proposition follow from Proposition~\ref{branchingpoly}.
\end{proof}

\subsection{Translating between labeling regimes for simple KLR modules}\label{simpletranslations}
In this subsection we apply the combinatorial results of \S\ref{crystalisomsec} to help work with, and translate between, the various classification regimes for simple KLR modules.

\subsubsection{Working with cuspidal systems}
When studying the representation theory of \(R_\theta\) from the cuspidal system point of view, it is useful to be able to efficiently switch between labeling systems associated to the two choices of convex order \(\succatext\) and \(\succbtext\), particularly for the purposes of calculating with the Chevalley functors of \S\ref{transfuncsec}. Recall the combinatorial description of the involution \(\lozenge: \Pi(\theta) \to \Pi(\theta)\) from Theorem~\ref{rpartranslate}, and the fact that for \(\pi \in \Pi(\theta)\), we have that \((\pi | \pi_\lozenge)\) and \((\pi_\lozenge | \pi)\) are affine MV polytopes. Then the next proposition follows from Theorem~\ref{longthm}(4).

\begin{proposition}\label{lozengetranslate}
Let \(\theta \in \Z_{\geq 0}I\). For all \(i \in \Z_2\) and \(\pi \in \Pi(\theta)\) we have \(L_i(\pi) \cong L_{\hat i}(\pi_\lozenge)\).
\end{proposition}

\begin{remark}\label{rootparchevappl}
In conjunction with Theorem~\ref{rpartranslate}, Proposition~\ref{rootfuncbranch} and Corollary~\ref{branchingpoly}, Proposition~\ref{lozengetranslate} allows one to efficiently and combinatorially compute with the Chevalley functors of \S\ref{transfuncsec} (and their generalized real root analogues in \S\ref{rootfuncsec}) in the root partition setting; for instance we have
\begin{align*}
\textup{hd} \,F_1 L_1(\pi) \cong \textup{hd}\, F_1 L_0(\pi_\lozenge) \cong q^{-n_1(\pi_\lozenge^1)}L_0((\pi_\lozenge)_{+1:1}) \cong q^{-n_1(\pi_\lozenge^1)}L_1(((\pi_\lozenge)_{+1:1})_\lozenge).
\end{align*}
\end{remark}

\subsubsection{Translating between cuspidal system and Kleshchev multipartition labelings}
The ability to move efficiently between the cuspidal- and multipartition-theoretic approaches to KLR representations allows us to combine the advantages of each. The former carries with it a stratification theory with significant homological control and reciprocity results, while the latter connects to the cellular structure of cyclotomic quotients and the classical theory of Specht modules and their homomorphisms. Different aspects of representation-theoretic information are encoded in the combinatorics of the labels themselves: we know, for instance, that the character of a simple \(R_\theta\)-module is bounded by its labeling polytope, while the multiplicity of words in the simple is bounded by their occurrence as residue sequences in the labeling multipartition. The crystal isomorphisms of \S\ref{crystalisomsec} let us make this connection directly.

\begin{theorem}\label{roottomultiisomthm}
Let \(\theta \in \Z_{\geq 0}I\), let \(\bkap\) be a multicharge and let \(i \in \Z_2\). Then we have
\begin{align*}
L_i(\pi) \cong D^{\bkap}( \mathcal{W}^{\bkap} \circ  \mathcal{X}_i (\pi))
\qquad
\textup{and}
\qquad
D^{\bkap}(\blam) \cong L_i(    \mathcal{Y}_i \circ \mathcal{Z}^{\bkap} (\blam))
\end{align*}
for all \(\pi \in \Pi(\theta)\) and \(\blam \in \MPK^{\bkap}(\theta)\).
\end{theorem}
\begin{proof}
As noted in the proof of Proposition~\ref{EFaffmulti}, there is a crystal isomorphism \(B:\MPK^{\bkap} \xrightarrow{\sim} \aMV\) such that \(L(B(\blam)) \cong D^{\bkap}(\blam)\) for all \(\blam \in \MPK^{\bkap}\). Then it follows by Corollary~\ref{PolyMPthm} that \(B =   \mathcal{Y} \circ \mathcal{Z}^{\bkap} \) and \(B^{-1} = \mathcal{W}^{\bkap} \circ  \mathcal{X}\), so the statement follows from Theorem~\ref{PolyOH}.
\end{proof}

\begin{remark}\label{Dlabelrem} In view of Theorem~\ref{roottomultiisomthm}, we may consider \(\UL\) as yet another index for simple KLR modules, writing
\(
L(\sfD) := L( \mathcal{Y}(\sfD)) 
\)
for all \(\sfD \in \UL\), noting then that \(L(\sfD) \cong L_i(\mathcal{Y}_i(\sfD)) \cong D^{\bkap}(\mathcal{W}^{\bkap} \sfD)\) for any \(i \in \Z_2\), and multicharge \(\bkap\).
\end{remark}

\subsection{Translating between multicharges}\label{transmulticharges}
For a given integral dominant weight \(\Lambda = t_0 \Lambda_0 + t_1 \Lambda_1\), there are in general \({t_0 + t_1 \choose t_0}\) inequivalent choices of multicharges \(\bkap\) which define a cellular structure and labeling for the simple modules of \(R^\Lambda_\theta\). Our next result shows that translating between these choices amounts to `gluing-then-splitting':

\begin{proposition}
Let \(\bkap, \bom\) be multicharges. Then \(D^{\bkap}(\blam) \cong D^{\bom}( \mathcal{W}^{\bom} \circ \mathcal{Z}^{\bkap}(\blam))\) for all \(\blam \in \MPK^{\bkap}\).
\end{proposition}
\begin{proof}
In view of Theorem~\ref{roottomultiisomthm} and Theorem~\ref{PolyOH} we have
\begin{align*}
D^{\bkap}(\blam) \cong L_1(\mathcal{Y}_1 \circ \mathcal{Z}^{\bkap}(\blam)) \cong D^{\bom}(\mathcal{W}^{\bom} \circ \mathcal{X}_1 \circ \mathcal{Y}_1\circ \mathcal{Z}^{\bkap}(\blam)) \cong D^{\bom}(\mathcal{W}^{\bom} \circ \mathcal{Z}^{\bkap}(\blam)),
\end{align*}
as desired.
\end{proof}

\subsection{The jump statistic via upper ledge diagrams}\label{jumpstatsec}
For \(i \in \Z_2\), \(M \in R_\theta\textup{-mod}\), let us set  
\begin{align*}
\tilde{F}_i(M) = \textup{hd}(F_i M); \qquad
\tilde{E}_i^*(M) = \textup{soc}(E_i^*M); \qquad
\tilde{\varepsilon}_i^*(M) = \max\{ m \in \Z_{>0} \mid (\tilde{E}_i^*)^m(M) \neq 0\}.
\end{align*}
Then the {\em jump} statistic on simple modules in \(R_\theta\textup{-mod}\) defined in \cite[\S6.1]{LVcrystals} is given by
\begin{align*}
\textup{jump}_i(M) := \max\{J \in \Z_{\geq 0} \mid \varepsilon_i^*( \tilde{F}_i^J M) = \varepsilon_i^*(M)\}.
\end{align*}
This statistic was a technical innovation that allowed Lauda--Vazirani to verify the crystal axioms for induction/restriction, and has played a key role in subsequent categorification projects  \cite{KvingeVazirani2018, Vazirani2016}. In general computing \(\textup{jump}_i(M)\) is at best an iterative process, even with the labeling crystal data for \(M\) in hand. If one works as in Remark~\ref{Dlabelrem} with the labeling data of \(\UL\) instead however, it can be computed directly:

\begin{proposition}\label{jumpprop}
Let \(\sfD \in \UL\). Then we have
\begin{align}\label{jumpforumlaD}
\textup{jump}_i(L(\sfD)) = \#\{\textup{unarced \(\hat i\)-bottom boxes in }\sfD[i]\} + \#\{\textup{unarced \(\hat i\)-top boxes in }\sfD[i]\}.
\end{align}
\end{proposition}
\begin{proof}
Say \(\sfD\) has \(u_b\) unarced \(\hat i\)-bottom boxes and \(u_t\) unarced \(\hat i\)-top boxes. Then by \S\ref{ULcrystaldefworkings},  \(f_i^{u_b}\sfD\) is obtained by placing an \(i\)-box beneath each one of these \(\hat i\)-bottom boxes, so that \((f_i^{u_b}\sfD)[i]\) has no unarced \(\hat i\)-bottom boxes, has a peak row of color \(i\), and the top boxes in \(f_i^{u_b}\sfD\) are unchanged from \(\sfD\). Then by \S\ref{ULcrystaldefworkings} again, \(f_i^{u_b + m}\sfD\) is obtained by adding \(m\) boxes of color \(i\)---which are both \(i\)-top and \(i\)-bottom---to the uppermost \(i\)-peak row of \(f_i^{u_b}\sfD\). Each of these \(m\) new \(i\)-colored boxes will be top-arced in \((f_i^{u_b + m}\sfD)[i]\) if and only if \(u_t \geq m\), so we have \(\varepsilon_i^*(f_i^{u_b + m} \sfD) = \varepsilon_i^*(\sfD)\) if and only if \(u_t \geq m\). Then the result follows from Theorems~\ref{mainamvulthm} and Corollary~\ref{branchingpoly}.
\end{proof}

\subsection{Augmented branching rules for simple modules labeled by Kleshchev multipartitions}\label{augmentedmulti}
Thanks to Theorem~\ref{roottomultiisomthm} and Proposition~\ref{rootfuncbranch}, we can describe the action of root functors of \S\ref{rootfuncsec} on simple modules labeled by Kleshchev multipartitions---we simply translate from \(D^{\bkap}(\blam)\) to the isomorphic module \(L(\pi | \phi)\) via the map \(\mathcal{Y} \circ \mathcal{Z}^{\bkap}\), utilize Proposition~\ref{rootfuncbranch} and then translate back via the map \(\mathcal{W}^{\bkap} \circ \mathcal{X}\). For reader convenience we detail an example of this computation now.

\begin{example}\label{crazylongextext}
Assume \(\bkap = (1,0,\ldots)\), and \(\blam = ((5,4^2,3^4,2,1),(4^2,3,2^2,1))\), shown at left in Figure~\ref{craylongexfix67}. Consider the application of the functor \(F_{0:3}\) to \(D^{\bkap}(\blam)\). 
Writing \((\pi | \phi) := \mathcal{Y} \circ \mathcal{Z}^{\bkap}(\blam)\), we have by Theorem~\ref{roottomultiisomthm}  that
\(
 D^{\bkap}(\blam) \cong L(\pi | \phi) \cong L_1(\pi)
\), with 
\begin{align*}
m_0(\pi | \phi) = \min(\pi^0); \quad n_{0:3}(\pi|\phi) = n_3(\pi^0);  \quad f_{0:3}(\pi | \phi) = (\pi_{+0:3} \mid (\pi_{+0:3})_\lozenge); \quad L(f_{0:3} (\pi | \phi)) \cong L_1(\pi_{+0:3}),
\end{align*}
so by Theorem~\ref{roottomultiisomthm} and Proposition~\ref{rootfuncbranch} we will have
\begin{align}\label{longDexpres}
q^{n_3(\pi^0)} \textup{hd} F_{0:3} D^{\bkap}(\blam) \cong D^{\bkap}(\bmu); \qquad \bmu :=
\mathcal{W}^{\bkap} \circ   \mathcal{X}_1 ( (\mathcal{Y}_1 \circ \mathcal{Z}^{\bkap}(\blam))_{+0:3}),
\end{align}
provided \(\min(\pi^0) \geq 3\). Now we walk through the (slightly streamlined) computation of \(\bmu\), as in Figure~\ref{craylongexfix67}.
\begin{enumerate}
\item 
Compute \(\sfD:= \mathcal{Z}^{\bkap}(\blam) \) by shifting/gluing the components of \(\blam\) as shown at left.
\item Separate into the \(1\)-triple decomposition \(\TRIP_1(\sfD) = (\sfD^\uparrow_1, \sfD^\delta_1, \sfD^\downarrow_1)\). 
\item Find \(\pi^0 = \CALC \circ \textup{forget}(\sfD^\downarrow_1)\) as shown at bottom---note \(n_3(\pi^0) = 0\) 
 and \(\min(\pi^0) = 4\) so (\ref{longDexpres}) holds. 
\item 
Now \(\pi_{+0:3}\) is given by adding a row of length 3 to the bottom of \(\pi^0\) (without changing \(\pi^1\) or \(\pi^\delta\)).
\item Replace \(\sfD^\downarrow_1\) with the new diagram \(\tilde{\sfD}^\downarrow_1 = \textup{color}_0 \circ \ORD(\pi^0_{+0:3})\).
\item Regather upper ledge diagrams, setting
\(
\tilde{\sfD} := \mathcal{X}_1 ( (\mathcal{Y}_1 \circ \mathcal{Z}^{\bkap}(\blam))_{+0:3}) = \STACK_1 ( \sfD^\uparrow_1, \sfD^\delta_1, \tilde{\sfD}^\downarrow_1),
\)
\item Compute \(\bmu = \mathcal{W}^{\bkap}(\tilde{\sfD}) = ((6,5^2,4,3^3,2,1),(5,4,3,2^2,1))\) by splitting/shifting \(\tilde{\sfD}\) as shown at right. 
\end{enumerate}
Then we arrive at
\begin{align*}
\textup{hd} F_{0:3} D^{\bkap}(\blam) \cong D^{\bkap}(\bmu).
\end{align*}
\end{example}

\begin{figure}[h]
\begin{align*}
\\
\hackcenter{}
\hackcenter{
\begin{overpic}[height=53.5mm]{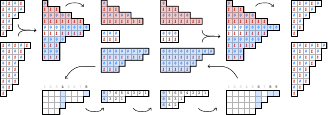}
 \put(-3.5,21){\makebox(0,0)[l]{$\scriptstyle \lambda^{(1)}$}}   
  \put(-3.5,34){\makebox(0,0)[l]{$\scriptstyle \lambda^{(2)}$}}   
   \put(85.5,21){\makebox(0,0)[l]{$\scriptstyle \mu^{(1)}$}}   
  \put(85.5,34){\makebox(0,0)[l]{$\scriptstyle \mu^{(2)}$}}   
   \put(11,34){\makebox(0,0)[l]{$\scriptstyle\sfD$}} 
    \put(28.5,34){\makebox(0,0)[l]{$\scriptstyle\sfD^\uparrow_1$}} 
     \put(28.5,25){\makebox(0,0)[l]{$\scriptstyle\sfD^\delta_1$}} 
      \put(28.5,19){\makebox(0,0)[l]{$\scriptstyle\sfD^\downarrow_1$}} 
      \put(46.6,34){\makebox(0,0)[l]{$\scriptstyle\sfD^\uparrow_1$}} 
     \put(46.6,25){\makebox(0,0)[l]{$\scriptstyle\sfD^\delta_1$}} 
      \put(46.6,19){\makebox(0,0)[l]{$\scriptstyle\tilde{\sfD}^\downarrow_1$}} 
         \put(67.3,34){\makebox(0,0)[l]{$\scriptstyle\tilde{\sfD}$}} 
          \put(32,8.3){\makebox(0,0)[b]{$\scriptstyle\pi^0$}}   
              \put(51.5,7.8){\makebox(0,0)[b]{$\scriptstyle\pi^0_{+0:3}$}}   
            \put(9,26.5){\makebox(0,0)[b]{$\scriptstyle\mathcal{Z}^{\bkap}$}}     
                   \put(22.8,34.5){\makebox(0,0)[b]{$\scriptstyle\TRIP_1$}}     
                     \put(81,34.7){\makebox(0,0)[b]{$\scriptstyle \mathcal{W}^{\bkap}$}}     
                       \put(26,14.2){\makebox(0,0)[t]{$\scriptstyle\textup{forget}$}}     
                         \put(67.7,14.2){\makebox(0,0)[t]{$\scriptstyle\textup{color}_0$}}     
                          \put(65,25){\makebox(0,0)[b]{$\scriptstyle\STACK_1$}}     
                  \put(28.7,0.5){\makebox(0,0)[t]{$\scriptstyle\CALC$}}             
                   \put(43.5,0.5){\makebox(0,0)[t]{$\scriptstyle\textup{add 3-row}$}}     
                     \put(63.5,0.5){\makebox(0,0)[t]{$\scriptstyle\ORD$}}     
\end{overpic}
}
\end{align*}
\caption{The computation yielding \(\textup{hd}\, F_{0:3}D^{\bkap}(\blam) \cong D^{\bkap}(\bmu)\) as detailed in Example~\ref{crazylongextext}.}
\label{craylongexfix67}
\end{figure}

\subsection{A tile factor theorem}\label{tilingsec}
James’s celebrated regularization theorem (see \S\ref{Restrictizationsec}) gives, for any level one Specht module, a purely combinatorial rule for describing a unique factor of this module, and a dominance bound on all of its other simple factors.
In \cite[Theorem~9.23]{MNSS} a cuspidal-flavored analogue of this result was provided for Specht modules. In this section we can provide a streamlined version of this analogue in level one thanks to the isomorphisms in \S\ref{crystalisomsec}.

Recalling \S\ref{cuspclasssec}, the partial order \(\geq\) on \(\Pi(\theta)\) induces a partial order \(\geq_i\) on \(\MPK^{\bkap}(\theta)\); we write \(\blam \geq_i \bmu\) provided \(\mathcal{Y}_i \circ \mathcal{Z}^{\bkap}(\blam) \geq \mathcal{Y}_i \circ \mathcal{Z}^{\bkap}(\bmu)\).
Let \(\lambda \in \Par^{\kappa}\) be a (level one) partition, and let \(\sigma\) be the largest staircase partition (i.e., of the form \((m, m-1, \ldots, 1)\) for some \(m\)) contained in \(\lambda\). If \(\sigma\) contains the rightmost \(i\)-residue box in the top row of \(\lambda\), then we say \(\lambda\) is {\em \(i\)-quasirestricted}.

Assume \(\lambda \in \Par^\kappa\) is \(i\)-quasirestricted. Then there is a unique `\(i\)-tiling' of \(\lambda\), which we compute in a greedy fashion; first tile all the southeast-removable \(\xi_1(\alpha_{\hat i:1})\) ribbons in \(\lambda\), then (ignoring these boxes) tile all the southeast-removable \(\xi_1(\alpha_{\hat i:2})\)'s, and then \(\xi_1(\alpha_{\hat i:3})\)'s, and so on until the remaining diagram has only removable boxes of residue \(i\) (and hence no more such tiles can be found). Call the  remaining partition \(\bar \lambda\). Let \(\chi^i\) be the partition such that \(n_k(\chi^i)\) is the number of \(\xi_i(\alpha_{\hat i: k})\) tiles in the \(i\)-tiling of \(\lambda\). Then we construct a partition \(\lambda^{\mathsf{T}(i)} \in \Parres^{\kappa}\) by appending \(c_r(\ORD(\chi^i))\) boxes to the \(r\)th column of \(\bar{\lambda}\) for all \(r \in \Z_{>0}\) (in essence, shifting the columns of \(\ORD(\chi^i)\) up to meet \(\bar \lambda\)). See Figure~\ref{tilingex1123fig} for an example.

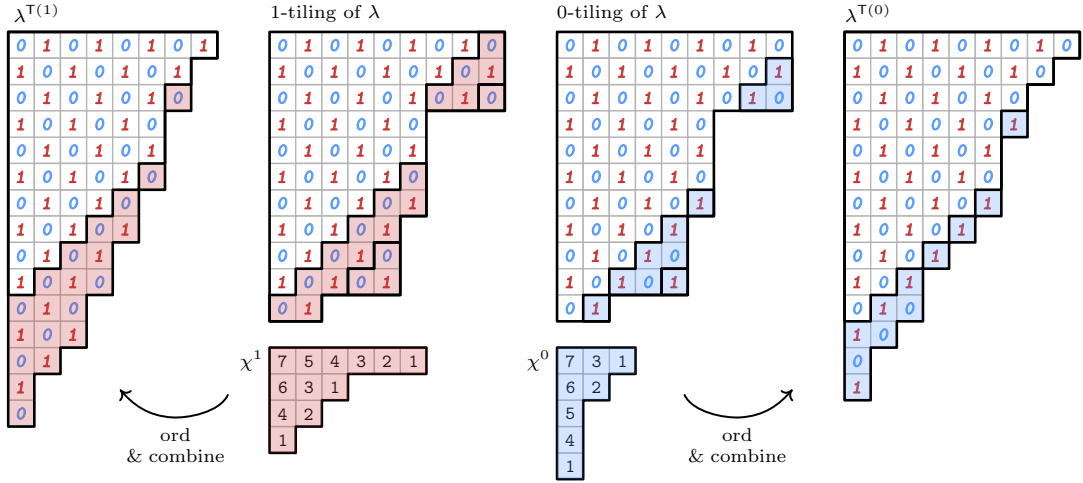
\begin{figure}[h]
\begin{align*}
\\
\hackcenter{}
\hackcenter{
\begin{overpic}[height=59.5mm]{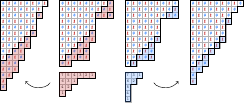}
 \put(0.6,42.7){\makebox(0,0)[l,b]{$\scriptstyle \lambda^{\mathsf{T}(1)}$}}   
  \put(78,42.7){\makebox(0,0)[l,b]{$\scriptstyle \lambda^{\mathsf{T}(0)}$}}   
 \put(24.5,42.7){\makebox(0,0)[l,b]{$\scriptstyle 1\textup{-tiling of } \lambda$}}    
  \put(51.2,42.7){\makebox(0,0)[l,b]{$\scriptstyle 0\textup{-tiling of } \lambda$}}    
    \put(23.9,11.2){\makebox(0,0)[r]{$\scriptstyle \chi^1$}}    
      \put(50.6,11.2){\makebox(0,0)[r]{$\scriptstyle \chi^0$}}    
      \put(15.7,4.3){\makebox(0,0)[u]{$\scriptstyle \ORD$}}  
        \put(15.7,2.3){\makebox(0,0)[u]{$\scriptstyle \&\; \textup{combine}$}}  
        \put(68,4.3){\makebox(0,0)[u]{$\scriptstyle \ORD$}}  
        \put(68,2.3){\makebox(0,0)[u]{$\scriptstyle \&\; \textup{combine}$}}  
\end{overpic}
}
\end{align*}
\caption{In the top center, the two \(i\)-tilings of an \(i\)-quasirestricted partition \(\lambda \in \Par^0\). Below them are the partitions \(\chi^i\) which record the ribbons used in the tiling, and their spar tableaux \(\SPAR(\chi^i)\). To the sides, the associated 2-restricted partition \(\lambda^{\mathsf{T}(i)}\). }
\label{tilingex1123fig}
\end{figure}

\begin{proposition}\label{tilingresult}
Let \(i, \kappa \in \Z_2\), and assume \(\lambda \in\Par^\kappa\) is an \(i\)-quasirestricted partition. Then \(D^{\kappa}(\lambda^{\mathsf{T}(i)})\) arises exactly once as a simple factor in \(S^{\kappa}(\lambda)\) up to some shift, and every simple factor of \(S^{\kappa}(\lambda)\) is of the form \(D^{\kappa}(\mu)\) up to shift, where \(\mu \in \Parres^\kappa\) and \(\mu \leq_i\lambda^{\mathsf{T}(i)}\).
\end{proposition}
\begin{proof}
The condition that \(\lambda\) is \(i\)-quasirestricted ensures that \(\lambda\) has a {\em \(\xi(\pi)\)-tiling} for \(\succitext\) in the language of \cite[\S9]{MNSS}, for some \(\pi \in \Pi(\theta)\). By \cite[Theorem~9.23]{MNSS} the simple module \(L_i(\pi)\) arises once as a factor of \(S^{\kappa}(\lambda)\), and \(\phi \leq_i \pi\) for every factor \(L_i(\phi)\) of \(S^{\kappa}(\lambda)\) (up to shifts). By \cite[Proposition~4.12]{MNSS}, the tiles in \(\xi_i(\pi)\) of content \(\succieqtext \delta\) compose the partition \(\bar \lambda\), and it follows by Corollary~\ref{PolyMPthm} and consideration of the associated combinatorial maps that \(\lambda^{\mathsf{T}(i)} = \mathcal{W}^{\bkap} \circ \mathcal{X}_i(\pi)\). Then the result follows from Theorem~\ref{roottomultiisomthm}.
\end{proof}

\begin{remark}
When \(\lambda\) is itself 2-restricted, the partition \(\lambda^{\mathsf{T}(i)}\) is recovered simply by removing all removable \(\hat i\) boxes and then adding them back on in the lowest possible \(\hat{i}\)-addable positions. In this setting Proposition~\ref{tilingresult} is covered by \cite[Theorem~4.4]{TTdecomp}.
\end{remark}

\begin{remark}
While Proposition~\ref{tilingresult} is of a similar flavor to James's regularization theorem, we note that in general the information offered by each is distinct; the simple factor guaranteed by our tiling theorem is not generally the same as the one coming from James's theorem (as may be seen by considering the Figure~\ref{tilingex1123fig} example), and James's dominance partial order \(\trianglerighteq\) and our cuspidal-flavored partial order \(\geq_i\) do not coincide.
\end{remark}

\subsection{A level two analogue of James's regularization theorem (conjecturally)} \label{conjjames2}

\subsubsection{Restrictization}\label{Restrictizationsec} Let \(\lambda \in \Par\). The {\em restrictization} \(\lambda_\textup{R}\) is defined by sliding boxes in \(\lambda\) as far southwest as possible along the northeast/southwest antidiagonals (or {\em ladders}) they belong to.
James's celebrated {\em regularization theorem} (better called in our setting a {\em restrictization} theorem), extended by Fayers--Lyle--Martin, establishes the following powerful result (presented here for \({\tt A}_1^{(1)}\) but holding analogously for all types \({\tt A}_{\geq 1}^{(1)}\)), which describes a specific composition factor appearing in any level one Specht module, and gives a dominance bound on other factors.
\begin{theorem}\cite[Theorem~A]{JamesDecompII}, \cite[Theorem~1.7]{FLM}\label{Jreg}
Let \(\kappa\) be a (level one) charge, and let \(\lambda \in \Par^\kappa\). Then:
\begin{enumerate}
\item There exists a nonzero homomorphism \(S^\kappa(\lambda_\textup{R}) \to S^\kappa(\lambda)\);
\item We have \([S^\kappa(\lambda): D^\kappa(\lambda_\textup{R})] = 1\);
\item If \([S^\kappa(\lambda): D^\kappa(\mu)] >0\) for \(\mu \in \Parres^\kappa\), then \(\mu \trianglelefteq \lambda_\textup{R}\).
\end{enumerate}
\end{theorem}

It is beguiling to attempt to extend James's theorem to higher levels. In type \({\tt A}_1^{(1)}\) level three and higher, a completist's dream generalization will be impossible---it is easy to find Specht modules \(S^{\bkap}(\blam)\) which have multiple maximally dominant factors, so a result along the lines of Theorem~\ref{Jreg} (3) cannot be hoped for. In types \({\tt A}_{> 1}^{(1)}\) matters are even worse, as similar issues begin popping up even in level two. Thus if one wants a full-strength generalization of Theorem~\ref{Jreg}, the only setting where it seems possible to accomplish is in type \({\tt A}_1^{(1)}\) level two. We give a conjecture in this direction below.

\subsubsection{Kleshchevization} 
Let \(\bkap = (\kappa_1, \kappa_2)\) be a bicharge, and \(\blam = (\lambda^{(1)}, \lambda^{(2)}) \in \MP^{\bkap}\). We define the {\em Kleshchevization} \(\blam_\textup{K} \in \MPK^{\bkap}\) of \(\blam\) by setting \(\blam_\textup{K} = \mathcal{Z}^{\bkap} \circ \mathcal{W}^{\bkap}(\blam)\). Note that by Theorem~\ref{isommpul} we have \(\blam_\textup{K} = \blam\) if and only if \(\blam \in \MPK^{\bkap}\).

Unpacking the particulars of the maps \(\mathcal{W}^{\bkap}\) and \(\mathcal{Z}^{\bkap}\), one sees that Kleshchevization amounts to ``greedy ladder-filling'' as follows.
First position \(\lambda^{(2)}\) to the northeast of \(\lambda^{(1)}\), so that the northwesternmost box(es) of residue \(\kappa_1\) in \(\lambda^{(2)}\)  lie along the same antidiagonal as the northwesternmost (\(\kappa_1\)-residue) box in \(\lambda^{(1)}\).
Then \(\blam_\textup{K}\) is the \(\bkap\)-bipartition obtained by sliding boxes directly to the southwest along antidiagonals in this arrangement as far as possible. Note that this procedure specializes to the usual restrictization of \(\lambda^{(1)}\) when \(\lambda^{(2)} = \varnothing\). See Figure~\ref{Kleshchevification} for an illustration. 

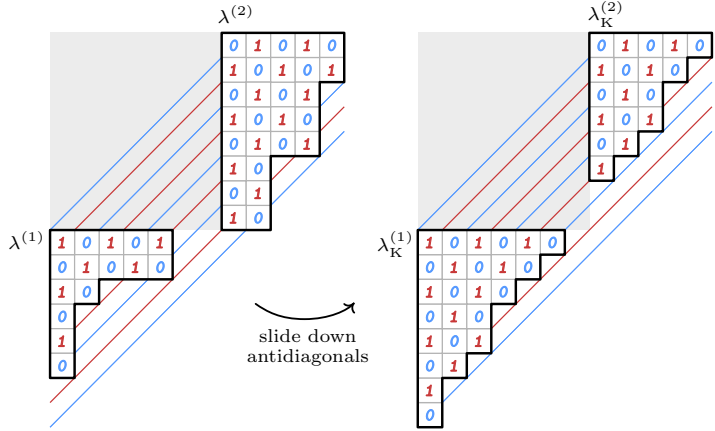
\begin{figure}[h]
\begin{align*}
\\
\hackcenter{}
\hackcenter{
\begin{overpic}[height=52.5mm]{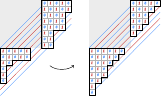}
 \put(-0.5,27){\makebox(0,0)[r, b]{$\scriptstyle \lambda^{(1)}$}}   
  \put(31,61){\makebox(0,0)[r, b]{$\scriptstyle \lambda^{(2)}$}}   
   \put(55,25.8){\makebox(0,0)[r, b]{$\scriptstyle \lambda^{(1)}_\textup{K}$}}   
  \put(86.7,60.5){\makebox(0,0)[r, b]{$\scriptstyle \lambda^{(2)}_\textup{K}$}}   
     \put(39,15){\makebox(0,0)[t]{$\scriptstyle \textup{slide down}$}}  
        \put(39,12){\makebox(0,0)[t]{$\scriptstyle \textup{antidiagonals}$}}  
\end{overpic}
}
\end{align*}
\caption{Let \(\bkap = (1,0)\) and \(\blam = ((5^2, 2, 1^3), (5^2,4^3,2^3))\). We arrange \(\lambda^{(1)}, \lambda^{(2)}\) so that the first antidiagonal of \(1\)-residue boxes in \(\lambda^{(2)}\) is aligned with the top left (residue 1) box in in \(\lambda^{(1)}\), and then slide boxes all the way down antidiagonals (passing through but not landing in the gray zone) to get \(\blam_\textup{K} = ((6,5,4,3^2,2,1^2),(5,4,3^2,2,1))\).}
\label{Kleshchevification}
\end{figure}

\begin{conjecture}\label{Jregconj}
Let \(\bkap = (\kappa_1, \kappa_2)\) be a bicharge, and let \(\blam \in \MP^{\bkap}\). Then:
\begin{enumerate}
\item There exists a nonzero homomorphism \(S^{\bkap}(\blam_\textup{K}) \to S^{\bkap}(\blam)\);
\item We have \([S^\kappa(\blam): D^\kappa(\blam_\textup{K})] = 1\);
\item If \([S^\kappa(\blam): D^\kappa(\bmu)] >0\) for \(\bmu \in \MPK^{\bkap}\), then \(\bmu \trianglelefteq \blam_\textup{K}\).
\end{enumerate}
\end{conjecture}

We have computationally verified that this result holds in small ranks. An intrepid mathematician may hypothesize that the Kleshchevization process can reasonably be extended to  levels greater than two (at the expense of weakening part (3) of the conjecture statement). Naively defining \(\blam_\textup{K} = \mathcal{Z}^{\bkap} \circ \mathcal{W}^{\bkap}(\blam)\) in higher levels does not succeed without some adjustment, however, and we are currently unable to compute with large enough ranks in higher levels to make any concrete conjectures in this direction.

\subsection{An alternate classification regime for KLR algebras (conjecturally)}\label{conjlabelreg}
In \S\ref{cuspclasssec},  \ref{kleshclasssec} we recalled how simple modules for the KLR algebra \(R_\theta\) may be constructed as heads of explicit (skew)  Specht modules directly associated to the combinatorial aspects of vertices in the \({\tt A}_1^{(1)}\)-crystals \(\aMV\) and \(\MPK^{\bkap}\). We expect that there should be such a construction associated to our `intermediary' \({\tt A}_1^{(1)}\)-crystal \(\UL\) as well, and make a conjecture in this direction now.

We associate a skew diagram \(\xi(\sfD)\) to any upper ledge diagram \(\sfD \in \UL\) as follows. Let \(u\) be the rightmost box in the uppermost peak row of \(\sfD\), and extend a ray \(r\) from the centroid of \(u\) directly to the northwest (i.e., at a \(135^\circ\) angle). Now form a diagram \(\sfD'\) by shifting every row above the uppermost peak row in \(\sfD\) to the right until the rightmost box in each of these rows lies along \(r\) (forming a staircase along the northeast side of \(\sfD'\)). Next, construct \(\xi(\sfD)\) by shifting the \(k\)th column in \(\sfD'\) up \(k-1\) steps. See Figure~\ref{magicshapeexfig} for an example.

\begin{figure}[h]
\begin{align*}
\\
\hackcenter{}
\hackcenter{
\begin{overpic}[height=45.5mm]{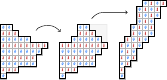}
 \put(0.5,34.3){\makebox(0,0)[l,b]{$\scriptstyle \sfD$}}    
  \put(47,34.3){\makebox(0,0)[l,b]{$\scriptstyle \sfD'$}}    
   \put(86,49){\makebox(0,0)[l,b]{$\scriptstyle \xi(\sfD)$}}    
    \put(28.4,36){\makebox(0,0)[b]{$\scriptstyle \textup{shift}$}}   
       \put(28.4,33.5){\makebox(0,0)[b]{$\scriptstyle \textup{rows}$}}   
    \put(67,43.7){\makebox(0,0)[b]{$\scriptstyle \textup{shift}$}}   
     \put(67,41.2){\makebox(0,0)[b]{$\scriptstyle \textup{columns}$}}   
\end{overpic}
}
\end{align*}
\caption{Computing the skew diagram \(\xi(\sfD)\) associated to an upper ledge diagram \(\sfD\). Note \(\xi(\sfD)\) is equal to \(\lambda \backslash \mu\), where \(\lambda = (8,7^2,6^3,5,4,3^2,2,1^2)\), \( \mu = (4,3^3,2^2,1) \in \Parres^0\).}
\label{magicshapeexfig}
\end{figure}
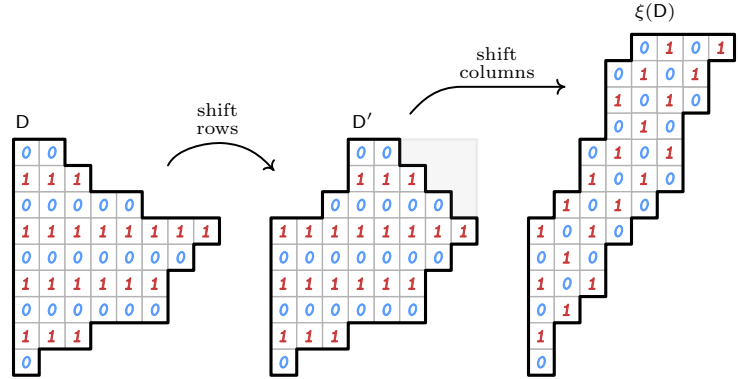

\begin{conjecture}\label{Dconj}
Let \(\theta \in \Z_{\geq 0}I\). For \(\sfD \in \UL(\theta)\), the Specht module \(\bS(\xi(\sfD))\) has simple head isomorphic to \(L(\sfD)\) up to some shift.
\end{conjecture}
As in the \(\aMV\), \(\MPK^{\bkap}\) settings, one would also expect some control over the other simple factors of \(\bS(\xi(\sfD))\) and related homological results.
Considering the combinatorics at work, one sees (as in Figure~\ref{magicshapeexfig}) that the skew diagrams \(\xi(\sfD)\) are exactly those realized by set differences of properly nested 2-restricted partitions. Defining the set of skew diagrams:
\begin{align*}
\mathcal{NP}_{\textup{2res}}(\theta) = \{ \lambda\backslash\mu \mid \kappa \in \Z_2, \, \lambda, \mu \in \Parres^\kappa, \, \lambda_1 > \mu'_1,\,
\textup{cont}(\lambda\backslash \mu) = \theta\}
\end{align*}
allows a reformulation of Conjecture~\ref{Dconj}:

\begin{conjecture}\label{NP2conj}
Let \(\theta \in \Z_{\geq 0}I\). 
For \(\lambda\backslash \mu \in \mathcal{NP}_{\textup{2res}}(\theta)\), the Specht module \(\bS(\lambda \backslash \mu)\) has simple head and 
\begin{align*}
\{ L(\lambda \backslash \mu) := \textup{hd}\, \bS(\lambda \backslash \mu) \mid \lambda \backslash \mu \in \mathcal{NP}_{\textup{2res}}(\theta)\}
\end{align*}
is a complete and irredundant list of simple \(R_\theta\)-modules up to isomorphism and grading shift.
\end{conjecture}

Theorem~\ref{MPlabelthmcyc} shows that \(\bS(\lambda \backslash \varnothing)\) has simple head for every \(\lambda \in \Parres^\kappa\), and it was shown in \cite{MNSS} that deleting a core (staircase in this \({\tt A}_1^{(1)}\) setting) partition \(\rho\) from \(\lambda\) has nice enough functorial behavior that \(\bS(\lambda \backslash \rho)\) should maintain simple head as well. One cannot push this too far; it can be seen for instance that \(\bS(\lambda \backslash \mu)\) is decomposable over a field of characteristic zero when we take \(\lambda = (2^2,1^3), \mu = (1^3) \in \Parres^\kappa\). Our hope is that the condition \(\lambda_1 > \mu'_1\) for \(\lambda, \mu \in \Parres^\kappa\) is restrictive enough to dissuade this rude behavior.

\subsubsection{A conjecture in general affine type {\tt A}}
For readers familiar with the setup of KLR algebras and associated combinatorics in general type \({\tt A}^{(1)}_{e-1}\), we offer one final conjecture, motivated by small-rank computations and large-rank optimism. Let \(e \geq 2\), and consider the associated KLR algebra \(R_\theta({\tt A}_{e-1}^{(1)})\) for some \(\theta\) in the root lattice of type \({\tt A}_{e-1}^{(1)}\), over an arbitrary field. For \(i \in \Z_e\), let \(\MPK^{(i,i)}(\theta)\) be the set of Kleshchev \({\tt A}_{e-1}^{(1)}\)-bipartitions of bicharge \((i,i)\), with \(\textup{cont}(\lambda^{(1)} \backslash \lambda^{(2)}) = \theta\) (see \cite{Ariki2001, HuMathasCycA}).

\begin{conjecture}\label{NP3conj}
For \(i \in \Z_e\), \((\lambda^{(1)}, \lambda^{(2)}) \in \MPK^{(i,i)}\), the Specht module \(\bS(\lambda^{(1)} \backslash \lambda^{(2)})\) has simple head and 
\begin{align*}
\{\textup{hd}\, \bS(\lambda^{(1)} \backslash \lambda^{(2)}) \mid (\lambda^{(1)}, \lambda^{(2)}) \in \MPK^{(i,i)}(\theta)\}
\end{align*}
is a complete (albeit redundant) list of simple \(R_\theta({\tt A}_{e-1}^{(1)})\)-modules up to isomorphism and grading shift.
\end{conjecture}

\bibliographystyle{eprintamsplain}
\bibliography{Biblio}

@book{Kac1990,
  author    = {Kac, Victor G.},
  title     = {{I}nfinite {D}imensional {L}ie {A}lgebras},
  edition   = {3rd},
  publisher = {Cambridge University Press},
  address   = {Cambridge},
  year      = {1990},
}

@article{Kamnitzer2010,
  author  = {Kamnitzer, Joel},
  title   = {Mirkovi\'c--{V}ilonen cycles and polytopes},
  journal = {Ann. of Math. (2)},
  volume  = {171},
  number  = {1},
  pages   = {245--294},
  year    = {2010},
  doi     = {10.4007/annals.2010.171.245},
  url     = {https://doi.org/10.4007/annals.2010.171.245},
}

@article{MuthiahTingley2014,
  author  = {Muthiah, Dinakar and Tingley, Peter},
  title   = {Affine {PBW} bases and {MV} polytopes in rank 2},
  journal = {Selecta Math. (N.S.)},
  volume  = {20},
  number  = {1},
  pages   = {237--260},
  year    = {2014},
  doi     = {10.1007/s00029-012-0117-z},
  url     = {https://doi.org/10.1007/s00029-012-0117-z},
}

@article{MuthiahTingley2018,
  author  = {Muthiah, Dinakar and Tingley, Peter},
  title   = {Affine {PBW} bases and affine {MV} polytopes},
  journal = {Selecta Math. (N.S.)},
  volume  = {24},
  number  = {5},
  pages   = {4781--4810},
  year    = {2018},
  doi     = {10.1007/s00029-018-0436-9},
  url     = {https://doi.org/10.1007/s00029-018-0436-9},
}

@article{Fayers2010,
  author    = {Fayers, Matthew},
  title     = {An {LLT}-type algorithm for computing higher-level canonical bases},
  journal   = {J. Pure Appl. Algebra},
  volume    = {214},
  number    = {12},
  pages     = {2186--2198},
  year      = {2010},
  doi       = {10.1016/j.jpaa.2010.02.021},
  url       = {https://doi.org/10.1016/j.jpaa.2010.02.021},
}

@incollection{Vazirani2016,
  author    = {Vazirani, Monica},
  title     = {Categorifying the tensor product of a level 1 highest weight and perfect crystal in type {A}},
  booktitle = {Lie algebras, Lie superalgebras, vertex algebras and related topics},
  series    = {Proc. Sympos. Pure Math.},
  volume    = {92},
  pages     = {293--324},
  publisher = {Amer. Math. Soc., Providence, RI},
  year      = {2016},
}

@article{KvingeVazirani2018,
  author  = {Kvinge, Henry and Vazirani, Monica},
  title   = {A combinatorial categorification of the tensor product of the {K}irillov-{R}eshetikhin crystal {$B^{1,1}$} and a fundamental crystal},
  journal = {Algebr. Represent. Theor.},
  volume  = {21},
  number  = {6},
  pages   = {1277--1331},
  year    = {2018},
  doi     = {10.1007/s10468-017-9747-3},
  url     = {https://doi.org/10.1007/s10468-017-9747-3},
}

@article{FLM,
    author = {Fayers, Matthew and Lyle, Sin\'ead and Martin, Stuart},
    title = {$p$-restriction of partitions and homomorphisms 
             between {S}pecht modules},
    journal = {J. Algebra},
    volume = {306},
    number = {1},
    year = {2006},
    pages = {175--190},
    doi = {10.1016/j.jalgebra.2005.11.040},
    url = {https://doi.org/10.1016/j.jalgebra.2005.11.040},
}

@book{FayersBook,
  author = {Fayers, Matthew},
  title  = {\(e\)-modular combinatorics of partitions},
  note   = {In preparation},
  year   = {2026}
}

@article{JaconLecouvey,
    author = {Jacon, Nicolas and Lecouvey, C\'edric},
    title = {Crystal isomorphisms for irreducible highest weight 
             {$\mathcal{U}_v(\widehat{\mathfrak{sl}}_e)$}-modules 
             of higher level},
    journal = {Algebr. Represent. Theory},
    volume = {13},
    number = {4},
    year = {2010},
    pages = {467--489},
    doi = {10.1007/s10468-009-9133-x},
    url = {https://doi.org/10.1007/s10468-009-9133-x},
}

@article{FLOTW,
    author = {Foda, Omar and Leclerc, Bernard and Okado, Masato 
              and Thibon, Jean-Yves and Welsh, Trevor},
    title = {Branching functions of {$A_{n-1}^{(1)}$} and 
             {J}antzen--{S}eitz problem for {A}riki--{K}oike algebras},
    journal = {Adv. Math.},
    volume = {141},
    number = {2},
    year = {1999},
    pages = {322--365},
    doi = {10.1006/aima.1998.1783},
    url = {https://doi.org/10.1006/aima.1998.1783},
}

@article{KimMono,
    author = {Kim, Jeong-Ah},
    title = {Monomial realization of crystal graphs for 
             {$U_q(A_n^{(1)})$}},
    journal = {Math. Ann.},
    volume = {332},
    number = {1},
    year = {2005},
    pages = {17--35},
    doi = {10.1007/s00208-004-0613-3},
    url = {https://doi.org/10.1007/s00208-004-0613-3},
}

@article{SSSrigB,
    author = {Salisbury, Ben and Scrimshaw, Travis},
    title = {A rigged configuration model for 
             {$B(\infty)$}},
    journal = {J. Combin. Theory Ser. A},
    volume = {133},
    year = {2015},
    pages = {29--57},
    doi = {10.1016/j.jcta.2015.01.008},
    url = {https://doi.org/10.1016/j.jcta.2015.01.008},
}

@article{KangYW,
    author = {Kang, Seok-Jin},
    title = {Crystal bases for quantum affine algebras and 
             combinatorics of {Y}oung walls},
    journal = {Proc. London Math. Soc. (3)},
    volume = {86},
    number = {1},
    year = {2003},
    pages = {29--69},
    doi = {10.1112/S0024611502013734},
    url = {https://doi.org/10.1112/S0024611502013734},
}

@article{KimShinYW,
    author = {Kim, Jeong-Ah and Shin, Dong-Uy},
    title = {Generalized {Y}oung walls and crystal bases 
             for quantum affine algebra of type {$A_n^{(1)}$}},
    journal = {Proc. Amer. Math. Soc.},
    volume = {138},
    number = {11},
    year = {2010},
    pages = {3877--3889},
    doi = {10.1090/S0002-9939-2010-10428-8},
    url = {https://doi.org/10.1090/S0002-9939-2010-10428-8},
}

@incollection{Uglov,
    author = {Uglov, Denis},
    title = {Canonical bases of higher-level $q$-deformed 
             {F}ock spaces and {K}azhdan--{L}usztig polynomials},
    booktitle = {Physical Combinatorics},
    series = {Progr. Math.},
    volume = {191},
    publisher = {Birkh\"auser},
    address = {Boston},
    year = {2000},
    pages = {249--299},
    doi = {10.1007/978-1-4612-1378-9_8},
    url = {https://doi.org/10.1007/978-1-4612-1378-9_8},
}

@article{Gerber,
    author = {Gerber, Thomas},
    title = {Crystal isomorphisms in {F}ock spaces and 
             {S}chensted correspondence in affine type {$A$}},
    journal = {Algebr. Represent. Theory},
    volume = {18},
    number = {4},
    year = {2015},
    pages = {1009--1046},
    doi = {10.1007/s10468-015-9530-2},
    url = {https://doi.org/10.1007/s10468-015-9530-2},
}

@incollection{Kashiwara1995,
  author    = {Kashiwara, Masaki},
  title     = {On crystal bases},
  booktitle = {Representations of groups ({B}anff, {AB}, 1994)},
  series    = {CMS Conf. Proc.},
  volume    = {16},
  pages     = {155--197},
  publisher = {Amer. Math. Soc., Providence, RI},
  year      = {1995},
}

@article{BKW,
    author = {Brundan, Jonathan and Kleshchev, Alexander and Wang, Weiqiang},
    title = {Graded {S}pecht modules},
    journal = {J. Reine Angew. Math.},
    volume = {655},
    year = {2011},
    pages = {61--87},
    doi = {10.1515/crelle.2011.033},
    note = {arXiv:0901.0218},
    url = {https://doi.org/10.1515/crelle.2011.033}
}

@book{pynchon2006,
  author    = {Pynchon, Thomas},
  title     = {Against the {D}ay},
  publisher = {Penguin Press},
  address   = {New York},
  year      = {2006},
}

@article {KangKash,
    AUTHOR = {Kang, Seok-Jin and Kashiwara, Masaki},
     TITLE = {Categorification of highest weight modules via
              {K}hovanov-{L}auda-{R}ouquier algebras},
   JOURNAL = {Invent. Math.},
  FJOURNAL = {Inventiones Mathematicae},
    VOLUME = {190},
      YEAR = {2012},
    NUMBER = {3},
     PAGES = {699--742},
      ISSN = {0020-9910,1432-1297},
   MRCLASS = {17B67 (17B37)},
  MRNUMBER = {2995184},
MRREVIEWER = {Volodymyr\ Mazorchuk},
       DOI = {10.1007/s00222-012-0388-1},
       URL = {https://doi.org/10.1007/s00222-012-0388-1},
}

@inproceedings {KleshSurvey,
    AUTHOR = {Kleshchev, Alexander},
     TITLE = {Modular representation theory of symmetric groups},
 BOOKTITLE = {Proceedings of the {I}nternational {C}ongress of
              {M}athematicians---{S}eoul 2014. {V}ol. {III}},
     PAGES = {97--120},
 PUBLISHER = {Kyung Moon Sa, Seoul},
      YEAR = {2014},
      ISBN = {978-89-6105-806-3; 978-89-6105-803-2},
   MRCLASS = {20C30 (17B37 20C08)},
  MRNUMBER = {3729020},
MRREVIEWER = {Himmet\ Can},
}

@article {BKisom,
    AUTHOR = {Brundan, Jonathan and Kleshchev, Alexander},
     TITLE = {Blocks of cyclotomic {H}ecke algebras and {K}hovanov-{L}auda
              algebras},
   JOURNAL = {Invent. Math.},
  FJOURNAL = {Inventiones Mathematicae},
    VOLUME = {178},
      YEAR = {2009},
    NUMBER = {3},
     PAGES = {451--484},
      ISSN = {0020-9910,1432-1297},
   MRCLASS = {20C08 (16G99 20F55)},
  MRNUMBER = {2551762},
MRREVIEWER = {Andrew\ Mathas},
       DOI = {10.1007/s00222-009-0204-8},
       URL = {https://doi.org/10.1007/s00222-009-0204-8},
}

@article{MisraMiwa,
    AUTHOR = {Misra, Kailash and Miwa, Tetsuji},
     TITLE = {Crystal base for the basic representation of
              {$U_q(\hat{\mathfrak{sl}}(n))$}},
   JOURNAL = {Comm. Math. Phys.},
    VOLUME = {134},
      YEAR = {1990},
    NUMBER = {1},
     PAGES = {79--88},
       DOI = {10.1007/BF02102090},
       URL = {https://doi.org/10.1007/BF02102090},
}

@article {KatoKLR,
    AUTHOR = {Kato, Syu},
     TITLE = {Poincar\'e-{B}irkhoff-{W}itt bases and
              {K}hovanov-{L}auda-{R}ouquier algebras},
   JOURNAL = {Duke Math. J.},
  FJOURNAL = {Duke Mathematical Journal},
    VOLUME = {163},
      YEAR = {2014},
    NUMBER = {3},
     PAGES = {619--663},
      ISSN = {0012-7094,1547-7398},
   MRCLASS = {17B37 (16T20)},
  MRNUMBER = {3165425},
MRREVIEWER = {Sonia\ Natale},
       DOI = {10.1215/00127094-2405388},
       URL = {https://doi.org/10.1215/00127094-2405388},
}

@article {VV,
    AUTHOR = {Varagnolo, M. and Vasserot, E.},
     TITLE = {Canonical bases and {KLR}-algebras},
   JOURNAL = {J. Reine Angew. Math.},
  FJOURNAL = {Journal f\"ur die Reine und Angewandte Mathematik. [Crelle's
              Journal]},
    VOLUME = {659},
      YEAR = {2011},
     PAGES = {67--100},
      ISSN = {0075-4102,1435-5345},
   MRCLASS = {17B37 (16T20)},
  MRNUMBER = {2837011},
MRREVIEWER = {Nicolas\ Jacon},
       DOI = {10.1515/CRELLE.2011.068},
       URL = {https://doi.org/10.1515/CRELLE.2011.068},
}

@article {KMimag,
    AUTHOR = {Kleshchev, Alexander and Muth, Robert},
     TITLE = {Imaginary {S}chur-{W}eyl duality},
   JOURNAL = {Mem. Amer. Math. Soc.},
  FJOURNAL = {Memoirs of the American Mathematical Society},
    VOLUME = {245},
      YEAR = {2017},
    NUMBER = {1157},
     PAGES = {xvii+83},
      ISSN = {0065-9266,1947-6221},
      ISBN = {978-1-4704-2249-3; 978-1-4704-3603-2},
   MRCLASS = {17B67 (05E10 17B22 20G43)},
  MRNUMBER = {3589160},
MRREVIEWER = {Aleksandr\ Nikolaevich\ Sergeev},
       DOI = {10.1090/memo/1157},
       URL = {https://doi.org/10.1090/memo/1157},
 }

@article{Ariki2001,
  author  = {Ariki, Susumu},
  title   = {On the classification of simple modules for cyclotomic {Hecke} algebras of type {$G(m,1,n)$} and {Kleshchev} multipartitions},
  journal = {Osaka J. Math.},
  volume  = {38},
  year    = {2001},
  pages   = {827--837},
}

@article {ArikiMathas,
    AUTHOR = {Ariki, Susumu and Mathas, Andrew},
     TITLE = {The number of simple modules of the {H}ecke algebras of type
              {$G(r,1,n)$}},
   JOURNAL = {Math. Z.},
  FJOURNAL = {Mathematische Zeitschrift},
    VOLUME = {233},
      YEAR = {2000},
    NUMBER = {3},
     PAGES = {601--623},
      ISSN = {0025-5874,1432-1823},
   MRCLASS = {20C08 (16G99 17B67 20F55)},
  MRNUMBER = {1750939},
MRREVIEWER = {Gerhard\ Hiss},
       DOI = {10.1007/s002090050489},
       URL = {https://doi.org/10.1007/s002090050489},
}

@article{KleshBranchII,
    author = {Kleshchev, Alexander},
    title = {Branching rules for modular representations of symmetric groups. {II}},
    journal = {J. Reine Angew. Math.},
    volume = {459},
    year = {1995},
    pages = {163--212},
    doi = {10.1515/crll.1995.459.163},
    url = {https://doi.org/10.1515/crll.1995.459.163}
}

@article {KleshBranchI,
    AUTHOR = {Kleshchev, Alexander},
     TITLE = {Branching rules for modular representations of symmetric
              groups. {I}},
   JOURNAL = {J. Algebra},
  FJOURNAL = {Journal of Algebra},
    VOLUME = {178},
      YEAR = {1995},
    NUMBER = {2},
     PAGES = {493--511},
      ISSN = {0021-8693,1090-266X},
   MRCLASS = {20C30 (20C20 20G05)},
  MRNUMBER = {1359899},
MRREVIEWER = {Jens\ C.\ Jantzen},
       DOI = {10.1006/jabr.1995.1362},
       URL = {https://doi.org/10.1006/jabr.1995.1362},
}

@article {LLT,
    AUTHOR = {Lascoux, Alain and Leclerc, Bernard and Thibon, Jean-Yves},
     TITLE = {Hecke algebras at roots of unity and crystal bases of quantum
              affine algebras},
   JOURNAL = {Comm. Math. Phys.},
  FJOURNAL = {Communications in Mathematical Physics},
    VOLUME = {181},
      YEAR = {1996},
    NUMBER = {1},
     PAGES = {205--263},
      ISSN = {0010-3616,1432-0916},
   MRCLASS = {17B37 (20C20)},
  MRNUMBER = {1410572},
MRREVIEWER = {Thomas\ M.\ Halverson},
DOI = {10.1007/BF02101678},
       URL = {http://doi.org/10.1007/BF02101678},
}

@article {AKT,
    AUTHOR = {Ariki, Susumu and Kreiman, Victor and Tsuchioka, Shunsuke},
     TITLE = {On the tensor product of two basic representations of
              {$U_v(\mathfrak{sl}_e)$}},
   JOURNAL = {Adv. Math.},
  FJOURNAL = {Advances in Mathematics},
    VOLUME = {218},
      YEAR = {2008},
    NUMBER = {1},
     PAGES = {28--86},
      ISSN = {0001-8708,1090-2082},
   MRCLASS = {17B37 (05E10)},
  MRNUMBER = {2409408},
MRREVIEWER = {Olga\ Bershtein},
       DOI = {10.1016/j.aim.2007.11.018},
       URL = {https://doi.org/10.1016/j.aim.2007.11.018},
}

@article {Jacon,
    AUTHOR = {Jacon, Nicolas},
     TITLE = {Kleshchev multipartitions and extended {Y}oung diagrams},
   JOURNAL = {Adv. Math.},
  FJOURNAL = {Advances in Mathematics},
    VOLUME = {339},
      YEAR = {2018},
     PAGES = {367--403},
      ISSN = {0001-8708,1090-2082},
   MRCLASS = {20C08 (05E10 17B37)},
  MRNUMBER = {3866901},
MRREVIEWER = {Jun\ Hu},
       DOI = {10.1016/j.aim.2018.09.038},
       URL = {https://doi.org/10.1016/j.aim.2018.09.038},
}

@incollection {MathasAK,
    AUTHOR = {Mathas, Andrew},
     TITLE = {Simple modules of {A}riki-{K}oike algebras},
 BOOKTITLE = {Group representations: cohomology, group actions and topology
              ({S}eattle, {WA}, 1996)},
    SERIES = {Proc. Sympos. Pure Math.},
    VOLUME = {63},
     PAGES = {383--396},
 PUBLISHER = {Amer. Math. Soc., Providence, RI},
      YEAR = {1998},
      ISBN = {0-8218-0658-0},
   MRCLASS = {20C20 (16G99)},
  MRNUMBER = {1603195},
MRREVIEWER = {Hiroshi\ Naruse},
       DOI = {10.1090/pspum/063/1603195},
       URL = {https://doi.org/10.1090/pspum/063/1603195},
}

@article {BessSyl,
    AUTHOR = {Bessenrodt, Christine},
     TITLE = {A bijection for {L}ebesgue's partition identity in the spirit
              of {S}ylvester},
   JOURNAL = {Discrete Math.},
  FJOURNAL = {Discrete Mathematics},
    VOLUME = {132},
      YEAR = {1994},
    NUMBER = {1-3},
     PAGES = {1--10},
      ISSN = {0012-365X,1872-681X},
   MRCLASS = {11P81 (05A17)},
  MRNUMBER = {1297366},
MRREVIEWER = {Richard\ P.\ Lewis},
       DOI = {10.1016/0012-365X(94)90228-3},
       URL = {https://doi.org/10.1016/0012-365X(94)90228-3},
}

@incollection {eulerbijections,
    AUTHOR = {Alladi, Krishnaswami},
     TITLE = {Euler's partition theorem and refinements without appeal to
              infinite products},
 BOOKTITLE = {Algorithmic combinatorics: enumerative combinatorics, special
              functions and computer alge},
    SERIES = {Texts Monogr. Symbol. Comput.},
     PAGES = {9--23},
 PUBLISHER = {Springer, Cham},
      YEAR = {2020},
      ISBN = {978-3-030-44558-4; 978-3-030-44559-1},
   MRCLASS = {05A17},
  MRNUMBER = {4300202},
       DOI = {10.1007/978-3-030-44559-1\_2},
       URL = {https://doi.org/10.1007/978-3-030-44559-1_2},
}

@article {Paksurvey,
    AUTHOR = {Pak, Igor},
     TITLE = {Partition bijections, a survey},
   JOURNAL = {Ramanujan J.},
  FJOURNAL = {Ramanujan Journal. An International Journal Devoted to the
              Areas of Mathematics Influenced by Ramanujan},
    VOLUME = {12},
      YEAR = {2006},
    NUMBER = {1},
     PAGES = {5--75},
      ISSN = {1382-4090,1572-9303},
   MRCLASS = {05A17 (05A30 11P81)},
  MRNUMBER = {2267263},
MRREVIEWER = {George\ E.\ Andrews},
       DOI = {10.1007/s11139-006-9576-1},
       URL = {https://doi.org/10.1007/s11139-006-9576-1},
}

@article{BKT,
    author = {Baumann, Pierre and Kamnitzer, Joel and Tingley, Peter},
    title = {Affine {M}irkovi\'{c}--{V}ilonen polytopes},
    journal = {Publ. Math. Inst. Hautes \'{E}tudes Sci.},
    volume = {120},
    year = {2014},
    pages = {113--205},
    doi = {10.1007/s10240-013-0057-y},
    note = {arXiv:1110.3661},
    url = {https://www.numdam.org/item/PMIHES_2014__120__113_0/},
}

@article {aMVrank2,
    AUTHOR = {Baumann, Pierre and Dunlap, Thomas and Kamnitzer, Joel and
              Tingley, Peter},
     TITLE = {Rank 2 affine {MV} polytopes},
   JOURNAL = {Represent. Theory},
  FJOURNAL = {Representation Theory. An Electronic Journal of the American
              Mathematical Society},
    VOLUME = {17},
      YEAR = {2013},
     PAGES = {442--468},
      ISSN = {1088-4165},
   MRCLASS = {17B67 (05E10 52B20)},
  MRNUMBER = {3084418},
MRREVIEWER = {Kyu-Hwan\ Lee},
       DOI = {10.1090/S1088-4165-2013-00438-7},
       URL = {https://doi.org/10.1090/S1088-4165-2013-00438-7},
}

@article {JamesDecompII,
    AUTHOR = {James, G. D.},
     TITLE = {On the decomposition matrices of the symmetric groups. {II}},
   JOURNAL = {J. Algebra},
  FJOURNAL = {Journal of Algebra},
    VOLUME = {43},
      YEAR = {1976},
    NUMBER = {1},
     PAGES = {45--54},
      ISSN = {0021-8693},
   MRCLASS = {20C30},
  MRNUMBER = {430050},
MRREVIEWER = {A.\ Kerber},
       DOI = {10.1016/0021-8693(76)90143-5},
       URL = {https://doi.org/10.1016/0021-8693(76)90143-5},
}

@article {TTdecomp,
    AUTHOR = {Tan, Kai Meng and Teo, Wei Hao},
     TITLE = {Sign sequences and decomposition numbers},
   JOURNAL = {Trans. Amer. Math. Soc.},
  FJOURNAL = {Transactions of the American Mathematical Society},
    VOLUME = {365},
      YEAR = {2013},
    NUMBER = {12},
     PAGES = {6385--6401},
      ISSN = {0002-9947,1088-6850},
   MRCLASS = {17B37 (20C30)},
  MRNUMBER = {3105756},
MRREVIEWER = {Guiyu\ Yang},
       DOI = {10.1090/S0002-9947-2013-05860-6},
       URL = {https://doi.org/10.1090/S0002-9947-2013-05860-6},
}

@article {FaceFunctors,
    AUTHOR = {McNamara, Peter J. and Tingley, Peter},
     TITLE = {Face functors for {KLR} algebras},
   JOURNAL = {Represent. Theory},
  FJOURNAL = {Representation Theory. An Electronic Journal of the American
              Mathematical Society},
    VOLUME = {21},
      YEAR = {2017},
     PAGES = {106--131},
      ISSN = {1088-4165},
   MRCLASS = {17B37},
  MRNUMBER = {3670026},
MRREVIEWER = {Zhankui\ Xiao},
       DOI = {10.1090/ert/496},
       URL = {https://doi.org/10.1090/ert/496},
}

@article {LVcrystals,
    AUTHOR = {Lauda, Aaron D. and Vazirani, Monica},
     TITLE = {Crystals from categorified quantum groups},
   JOURNAL = {Adv. Math.},
  FJOURNAL = {Advances in Mathematics},
    VOLUME = {228},
      YEAR = {2011},
    NUMBER = {2},
     PAGES = {803--861},
      ISSN = {0001-8708,1090-2082},
   MRCLASS = {17B37},
  MRNUMBER = {2822211},
MRREVIEWER = {Peter\ W.\ Tingley},
       DOI = {10.1016/j.aim.2011.06.009},
       URL = {https://doi.org/10.1016/j.aim.2011.06.009},
}

@article {GeLi,
    AUTHOR = {Li, Ge},
     TITLE = {Integral basis theorem of cyclotomic
              {K}hovanov-{L}auda-{R}ouquier algebras of type {A}},
   JOURNAL = {J. Algebra},
  FJOURNAL = {Journal of Algebra},
    VOLUME = {482},
      YEAR = {2017},
     PAGES = {1--101},
      ISSN = {0021-8693,1090-266X},
   MRCLASS = {20C08 (16G30)},
  MRNUMBER = {3646285},
MRREVIEWER = {J.\ Matthew\ Douglass},
       DOI = {10.1016/j.jalgebra.2017.02.022},
       URL = {https://doi.org/10.1016/j.jalgebra.2017.02.022},
}

@article {BKdecomp,
    AUTHOR = {Brundan, Jonathan and Kleshchev, Alexander},
     TITLE = {Graded decomposition numbers for cyclotomic {H}ecke algebras},
   JOURNAL = {Adv. Math.},
  FJOURNAL = {Advances in Mathematics},
    VOLUME = {222},
      YEAR = {2009},
    NUMBER = {6},
     PAGES = {1883--1942},
      ISSN = {0001-8708,1090-2082},
   MRCLASS = {20C08 (17B37 20C20 20F55)},
  MRNUMBER = {2562768},
MRREVIEWER = {Andrew\ Mathas},
       DOI = {10.1016/j.aim.2009.06.018},
       URL = {https://doi.org/10.1016/j.aim.2009.06.018},
}

@article {HuMathasCycA,
    AUTHOR = {Hu, Jun and Mathas, Andrew},
     TITLE = {Graded cellular bases for the cyclotomic
              {K}hovanov-{L}auda-{R}ouquier algebras of type {$A$}},
   JOURNAL = {Adv. Math.},
  FJOURNAL = {Advances in Mathematics},
    VOLUME = {225},
      YEAR = {2010},
    NUMBER = {2},
     PAGES = {598--642},
      ISSN = {0001-8708,1090-2082},
   MRCLASS = {20C08 (16S99 17B37 20C30)},
  MRNUMBER = {2671176},
MRREVIEWER = {Anton\ Cox},
       DOI = {10.1016/j.aim.2010.03.002},
       URL = {https://doi.org/10.1016/j.aim.2010.03.002},
}

@misc{Rouq1,
      title={2-{K}ac-{M}oody algebras}, 
      author={Rapha\"el Rouquier},
      year={2008},
      eprint={0812.5023},
      archivePrefix={arXiv},
      primaryClass={math.RT},
      url={https://arxiv.org/abs/0812.5023}, 
}

@misc{AffWeb,
      title={Superalgebra deformations of web categories: {A}ffine and cyclotomic webs}, 
      author={Davidson, Nicholas and Kujawa,  Jonathan R. and Muth, Robert},
      year={2025},
      eprint={2511.21671},
      archivePrefix={arXiv},
      primaryClass={math.RT},
      url={https://arxiv.org/abs/2511.21671}, 
}

@misc{Murata,
      title={Affine highest weight structures on module categories over quiver Hecke algebras}, 
      author={Haruto Murata},
      year={2024},
      eprint={2412.12903},
      archivePrefix={arXiv},
      primaryClass={math.RT},
      url={https://arxiv.org/abs/2412.12903}, 
}

@article {KhovLauda1,
    AUTHOR = {Khovanov, Mikhail and Lauda, Aaron D.},
     TITLE = {A diagrammatic approach to categorification of quantum groups.
              {I}},
   JOURNAL = {Represent. Theory},
  FJOURNAL = {Representation Theory. An Electronic Journal of the American
              Mathematical Society},
    VOLUME = {13},
      YEAR = {2009},
     PAGES = {309--347},
      ISSN = {1088-4165},
   MRCLASS = {17B37},
  MRNUMBER = {2525917},
MRREVIEWER = {Fan\ Xu},
       DOI = {10.1090/S1088-4165-09-00346-X},
       URL = {https://doi.org/10.1090/S1088-4165-09-00346-X},
}

@article {ADMPSS,
    AUTHOR = {Abbasian, Dina and Difulvio, Lena and Muth, Robert and
              Pasternak, Gabrielle and Sholtes, Isabella and Sinclair,
              Frances},
     TITLE = {Cuspidal ribbon tableaux in affine type {A}},
   JOURNAL = {Algebr. Comb.},
  FJOURNAL = {Algebraic Combinatorics},
    VOLUME = {6},
      YEAR = {2023},
    NUMBER = {2},
     PAGES = {285--319},
      ISSN = {2589-5486},
   MRCLASS = {17B22 (05E10 20C30 20G42)},
  MRNUMBER = {4591589},
MRREVIEWER = {Stuart\ Martin},
       DOI = {10.5802/alco.260},
       URL = {https://doi.org/10.5802/alco.260},
}

@article {MNSS,
    AUTHOR = {Muth, Robert and Nicewicz, Thomas and Speyer, Liron and
              Sutton, Louise},
     TITLE = {A skew {S}pecht perspective of rock blocks and cuspidal
              systems for {KLR} algebras in affine type {A}},
   JOURNAL = {Represent. Theory},
  FJOURNAL = {Representation Theory. An Electronic Journal of the American
              Mathematical Society},
    VOLUME = {29},
      YEAR = {2025},
     PAGES = {718--788},
      ISSN = {1088-4165},
   MRCLASS = {20C08 (05E10 05E16 17B22 20C30)},
  MRNUMBER = {4967672},
       DOI = {10.1090/ert/698},
       URL = {https://doi.org/10.1090/ert/698},
}

@article {TW,
    AUTHOR = {Tingley, Peter and Webster, Ben},
     TITLE = {Mirkovi\'c-{V}ilonen polytopes and
              {K}hovanov-{L}auda-{R}ouquier algebras},
   JOURNAL = {Compos. Math.},
  FJOURNAL = {Compositio Mathematica},
    VOLUME = {152},
      YEAR = {2016},
    NUMBER = {8},
     PAGES = {1648--1696},
      ISSN = {0010-437X,1570-5846},
   MRCLASS = {17B37 (16G10 17B67 19A49 52B20)},
  MRNUMBER = {3542489},
MRREVIEWER = {Csaba\ Sz\'ant\'o},
       DOI = {10.1112/S0010437X16007338},
       URL = {https://doi.org/10.1112/S0010437X16007338},
}

@article {KMstrat,
    AUTHOR = {Kleshchev, Alexander and Muth, Robert},
     TITLE = {Stratifying {KLR} algebras of affine {ADE} types},
   JOURNAL = {J. Algebra},
  FJOURNAL = {Journal of Algebra},
    VOLUME = {475},
      YEAR = {2017},
     PAGES = {133--170},
      ISSN = {0021-8693,1090-266X},
   MRCLASS = {17B67 (05E10 16G30 17B22 20C30 20G43)},
  MRNUMBER = {3612467},
       DOI = {10.1016/j.jalgebra.2016.07.006},
       URL = {https://doi.org/10.1016/j.jalgebra.2016.07.006},
}

@article {McNaffine,
    AUTHOR = {McNamara, Peter J.},
     TITLE = {Representations of {K}hovanov-{L}auda-{R}ouquier algebras
              {III}: symmetric affine type},
   JOURNAL = {Math. Z.},
  FJOURNAL = {Mathematische Zeitschrift},
    VOLUME = {287},
      YEAR = {2017},
    NUMBER = {1-2},
     PAGES = {243--286},
      ISSN = {0025-5874,1432-1823},
   MRCLASS = {16W50 (16T05)},
  MRNUMBER = {3694676},
MRREVIEWER = {Volodymyr\ Mazorchuk},
       DOI = {10.1007/s00209-016-1825-4},
       URL = {https://doi.org/10.1007/s00209-016-1825-4},
}

@article {KleshCusp,
    AUTHOR = {Kleshchev, Alexander},
     TITLE = {Cuspidal systems for affine {K}hovanov-{L}auda-{R}ouquier
              algebras},
   JOURNAL = {Math. Z.},
  FJOURNAL = {Mathematische Zeitschrift},
    VOLUME = {276},
      YEAR = {2014},
    NUMBER = {3-4},
     PAGES = {691--726},
      ISSN = {0025-5874,1432-1823},
   MRCLASS = {20C08 (20C30)},
  MRNUMBER = {3175157},
MRREVIEWER = {Anton\ Cox},
       DOI = {10.1007/s00209-013-1219-9},
       URL = {https://doi.org/10.1007/s00209-013-1219-9},
}

@article {MuthSkew,
    AUTHOR = {Muth, Robert},
     TITLE = {Graded skew {S}pecht modules and cuspidal modules for
              {K}hovanov-{L}auda-{R}ouquier algebras of affine type {A}},
   JOURNAL = {Algebr. Represent. Theory},
  FJOURNAL = {Algebras and Representation Theory},
    VOLUME = {22},
      YEAR = {2019},
    NUMBER = {4},
     PAGES = {977--1015},
      ISSN = {1386-923X,1572-9079},
   MRCLASS = {20C08 (05E10 20C30)},
  MRNUMBER = {3985148},
MRREVIEWER = {Anton\ Cox},
       DOI = {10.1007/s10468-018-9808-2},
       URL = {https://doi.org/10.1007/s10468-018-9808-2},
}

@article {KMR,
    AUTHOR = {Kleshchev, Alexander and Mathas, Andrew and Ram, Arun},
     TITLE = {Universal graded {S}pecht modules for cyclotomic {H}ecke
              algebras},
   JOURNAL = {Proc. Lond. Math. Soc. (3)},
  FJOURNAL = {Proceedings of the London Mathematical Society. Third Series},
    VOLUME = {105},
      YEAR = {2012},
    NUMBER = {6},
     PAGES = {1245--1289},
      ISSN = {0024-6115,1460-244X},
   MRCLASS = {20C08 (05E10 20C30)},
  MRNUMBER = {3004104},
MRREVIEWER = {Chi\ Kin\ Mak},
       DOI = {10.1112/plms/pds019},
       URL = {https://doi.org/10.1112/plms/pds019},
}

@article{KMaffzig,
	author = {Kleshchev, Alexander and Muth, Robert},
	doi = {10.1090/tran/7464},
	fjournal = {Transactions of the American Mathematical Society},
	journal = {Trans. Amer. Math. Soc.},
	mrclass = {20C08, 17B10, 05E10},
	pages = {4535--4583},
	title = {Affine zigzag algebras and imaginary strata for {KLR} algebras},
	url = {https://doi.org/10.1090/tran/7464},
	volume = {371},
	year = {2019}}

@misc{G,
	author = {Gould, Henry},
	title = {Tables of Combinatorial Identities},
	url = {https://math.wvu.edu/~hgould/},
	urldate = {2020-06-29},
	year = 2010}

\end{document}